%% file: main.tex
\documentclass[reqno,oneside,12pt]{amsbook}

\usepackage[main=english]{babel}
\usepackage[utf8]{inputenc}
\usepackage{tikz-cd}
\usepackage{amsfonts,amsmath,amsthm}
\usepackage{stmaryrd}
\usepackage{thmtools,thm-restate}
\usepackage[T1]{fontenc}
\usepackage{enumitem}
\usepackage{varioref}
\usepackage{units}
\usepackage[hypertexnames=false]{hyperref}
\usepackage{graphicx}
\usepackage{amssymb}
\usepackage[left=3.5cm, right=3.5cm]{geometry}
\usepackage{mathabx}
\usepackage{times,mathptm, mathtools}
\usepackage[colorinlistoftodos,prependcaption]{todonotes}
\usepackage{etoolbox}
\usepackage{mathrsfs}
\usepackage{mathdots}

\def\do#1{\csdef{c#1}{\ensuremath{\mathcal{#1}}}}
\docsvlist{A,B,C,D,E,F,G,H,I,J,K,L,M,N,O,P,Q,R,S,T,U,V,W,X,Y,Z}
\def\do#1{\csdef{s#1}{\ensuremath{\mathscr{#1}}}}
\docsvlist{A,B,C,D,E,F,G,H,I,J,K,L,M,N,O,P,Q,R,S,T,U,V,W,X,Y,Z}
\def\do#1{\csdef{f#1}{\ensuremath{\mathsf{#1}}}}
\docsvlist{A,B,C,D,E,F,G,H,I,J,K,L,M,N,O,P,Q,R,S,T,U,V,W,X,Y,Z}

\newcommand{\G}{\mathbb G}

\newcommand{\HH}{\mathbb H}
\newcommand{\CS}{\mathbb S}
\newcommand{\w}{\omega}
\renewcommand{\div}{\operatorname{div}}
\newcommand{\OO}{\mathcal O}
\newcommand{\bb}{\mathfrak b}
\renewcommand{\aa}{\mathfrak a}
\DeclareMathOperator{\rank}{rank}
\DeclareMathOperator{\Crit}{Crit}
\DeclareMathOperator{\pr}{pr}

\DeclareMathOperator{\QAlb}{QAlb}

\DeclareMathOperator{\Supp}{Supp}
\DeclareMathOperator{\NS}{NS}

\DeclareMathOperator{\cNS}{Cartier-NS (X_0)}
\DeclareMathOperator{\ccNS}{Cartier-NS}
\DeclareMathOperator{\wNS}{Weil-NS(X_0)}
\DeclareMathOperator{\wwNS}{Weil-NS}

\DeclareMathOperator{\Weil}{Weil}
\DeclareMathOperator{\Cartier}{Cartier}
\DeclareMathOperator{\Winf}{Weil_\infty (X_0)}
\DeclareMathOperator{\Cinf}{Cartier_\infty (X_0)}
\DeclareMathOperator{\CX}{Cartier (X_0)}
\DeclareMathOperator{\WX}{Weil (X_0)}
\DeclareMathOperator{\L2}{L^2(X_0)}
\DeclareMathOperator{\Hom}{Hom}
\DeclareMathOperator{\LPLus}{\Hom (\Cinf, \R)_{(+)}}

\DeclareMathOperator{\PSL}{PSL}
\DeclareMathOperator{\PGL}{PGL}
\DeclareMathOperator{\Exc}{Exc}
\DeclareMathOperator{\Bir}{Bir}
\DeclareMathOperator{\Tr}{Tr}
\DeclareMathOperator{\Frac}{Frac}
\DeclareMathOperator{\Gal}{Gal}

\DeclareMathOperator{\Sing}{Sing}
\DeclareMathOperator{\Vect}{Vect}
\DeclareMathOperator{\Far}{Far}
\DeclareMathOperator{\qm}{qm}

\newcommand{\Sinf}{\cS_\infty (X_0)}
\newcommand{\Vinf}{\cV_\infty}

\newcommand{\BD}{\partial_X X_0}
\renewcommand{\triangle}{\triangleright}

\DeclareMathOperator{\car}{char}

\newcommand{\an}{{an}}
\newcommand{\oM}{\overline M}
\DeclareMathOperator{\Div}{Div}
\newcommand{\DivInf}{\Div_\infty}
\DeclareMathOperator{\id}{id}

\DeclareMathOperator{\Aut}{Aut}
\DeclareMathOperator{\Pic}{Pic}

\DeclareMathOperator{\ord}{ord}

\DeclareMathOperator{\Top}{top}
\DeclareMathOperator{\End}{End}
\DeclareMathOperator{\spec}{Spec}
\DeclareMathOperator{\supp}{Supp}

\newcommand{\R}{\mathbf{R}}
\newcommand{\C}{\mathbf{C}}
\newcommand{\Q}{\mathbf{Q}}
\newcommand{\Z}{\mathbf{Z}}
\newcommand{\N}{\mathbf{N}}
\newcommand{\A}{\mathbf{A}}

\newcommand{\K}{\mathbf{K}}
\renewcommand{\k}{\mathbf k}
\renewcommand{\P}{\mathbf{P}}
\renewcommand{\phi}{\varphi}
\renewcommand{\epsilon}{\varepsilon}
\renewcommand{\tilde}{\widetilde}
\renewcommand{\hat}{\widehat}
\newcommand{\m}{\mathfrak m}
\newcommand{\p}{\mathfrak p}

\newtheorem{thm}{Theorem}[chapter]
\newtheorem{thm*}{Theorem}
\newtheorem{bigthm}{Theorem}
\newtheorem{prop}[thm]{Proposition}
\newtheorem{cor}[thm]{Corollary}
\newtheorem{lemme}[thm]{Lemma}

\newtheorem{claim}[thm]{Claim}
\theoremstyle{definition}
\newtheorem{dfn}[thm]{Definition}
\newtheorem{ex}[thm]{Example}
\newtheorem{rmq}[thm]{Remark}

\theoremstyle{remark}
\counterwithin*{equation}{chapter}

\AtBeginEnvironment{dfn}{%
  \setlist[enumerate]{label={(\roman*)}}
}

\AtBeginEnvironment{thm}{%
  \setlist[enumerate,1]{label={(\arabic*)}}
}

\AtBeginEnvironment{prop}{%
  \setlist[enumerate,1]{label={(\arabic*)}}
}

\AtBeginEnvironment{lemme}{%
  \setlist[enumerate,1]{label={(\arabic*)}}
}

\numberwithin{section}{chapter}
\numberwithin{subsection}{section}
\numberwithin{subsubsection}{subsection}
\numberwithin{figure}{section}

\makeatletter

\renewcommand\subsubsection{\@secnumfont}{\bfseries}%
\renewcommand\subsubsection{\@startsection{subsubsection}{3}
  \z@{.5\linespacing\@plus.7\linespacing}{-.5em}%
  {\normalfont\bfseries}}

  \makeatother

\makeatletter
\newcommand*{\da@rightarrow}{\mathchar"0\hexnumber@\symAMSa 4B }
\newcommand*{\da@leftarrow}{\mathchar"0\hexnumber@\symAMSa 4C }
\newcommand*{\xdashrightarrow}[2][]{%
  \mathrel{%
    \mathpalette{\da@xarrow{#1}{#2}{}\da@rightarrow{\,}{}}{}%
  }%
}
\newcommand{\xdashleftarrow}[2][]{%
  \mathrel{%
    \mathpalette{\da@xarrow{#1}{#2}\da@leftarrow{}{}{\,}}{}%
  }%
}
\newcommand*{\da@xarrow}[7]{%
  \sbox0{$\ifx#7\scriptstyle\scriptscriptstyle\else\scriptstyle\fi#5#1#6\m@th$}%
  \sbox2{$\ifx#7\scriptstyle\scriptscriptstyle\else\scriptstyle\fi#5#2#6\m@th$}%
  \sbox4{$#7\dabar@\m@th$}%
  \dimen@=\wd0 %
  \ifdim\wd2 >\dimen@
    \dimen@=\wd2 %
  \fi
  \count@=2 %
  \def\da@bars{\dabar@\dabar@}%
  \@whiledim\count@\wd4<\dimen@\do{%
    \advance\count@\@ne
    \expandafter\def\expandafter\da@bars\expandafter{%
      \da@bars
      \dabar@ 
    }%
  }%
  \mathrel{#3}%
  \mathrel{%
    \mathop{\da@bars}\limits
    \ifx\\#1\\%
    \else
      _{\copy0}%
    \fi
    \ifx\\#2\\%
    \else
      ^{\copy2}%
    \fi
  }%
  \mathrel{#4}%
}
\makeatother

\begin{document}
\author{Marc Abboud \\ Université de Neuchâtel}
\title{On the dynamics of endomorphisms of affine surfaces}
\maketitle
\newpage
\tableofcontents

\input{intro}
\input{chapitre1}
\input{chapitre2}

\bibliographystyle{alpha}
\bibliography{biblio}

\end{document}

%% file: intro.tex
\chapter{Introduction}

In \cite{favreEigenvaluations2007}, Favre and Jonsson developed tools from valuative theory to study the dynamics of a
dominant endomorphism of the complex affine plane. We extend this theory to the case of any affine surface, over any
field. We give a new method to construct an eigenvaluation of an endomorphism. We
generalize the result of Favre and Jonsson and show that the first dynamical degree of a dominant endomorphism of a
normal affine surface is an algebraic integer of degree $\leq 2$. The general method is to construct a compactification
where our endomorphisms admit a fixed point at infinity where the dynamical degree can be computed by studying the local
dynamics at this point. We then apply this construction to the study of the dynamics of automorphisms where we are able
to say much more. In particular, we obtain a new kind of rigidity result: the set of
first dynamical degrees of loxodromic automorphisms of a given affine surface must be contained in the set of integers or in the
set of algebraic integers of degree 2.

\section{Dynamical degrees}
Let $X$ be a smooth projective variety over an algebraically closed field and let $d$ be its dimension.
For $d$ Cartier divisors $D_1, \cdots, D_d$ of $X$ we can define the intersection product $D_1 \cdots
D_d \in \Z$ (see \cite{lazarsfeldPositivityAlgebraicGeometry2004}). If $f: X \dashrightarrow X$ is a
dominant rational transformation of $X$, we define for $0 \leq \ell \leq d$ the $\ell$-th dynamical degree of
$f$ by
\begin{equation}
  \lambda_\ell (f) := \lim_{n \rightarrow \infty} \left( (f^n)^* H^\ell \cdot H^{d-\ell} \right)^{1/n},
  \label{<+label+>}
\end{equation}
where $H$ is an ample divisor over $X$. One can show that these quantities are well defined and do not
depend on the choice of $H$. Furthermore, the dynamical degrees are
birational invariants: if $\phi: X
\dashrightarrow Y$ is a birational map, then
\begin{equation}
  \lambda_l (f) = \lambda_l ( \phi \circ f \circ {\phi}^{-1}), \quad \forall 0 \leq l \leq d.
  \label{<+label+>}
\end{equation}
We have that $\lambda_d (f)$ is the topological degree of $f$ and $\lambda_0 (f) = 1$. The Khovanskii-Teissier inequalities
(see \cite{gromovConvexSetsKahler1990}, \cite{dinhMixedHodgeRiemannBilinear2005}) imply that the sequence
$(\lambda_l)_{0 \leq l \leq d}$ is log-concave; i.e
\begin{equation}
  \frac{\log \lambda_{l-1} + \log
  \lambda_{l+1}}{2} \leq \log \lambda_l, \quad \forall 1 \leq l \leq d-1.
\end{equation}
In particular, one has $\forall 1 \leq l \leq d, \lambda_1 (f)^l \geq \lambda_k (f)$.

Let $X_0$ be a smooth affine variety of dimension $d$ and $f : X_0 \rightarrow X_0$ a dominant endomorphism of
$X_0$. We define the dynamical degrees of $f$ as follows. A \emph{completion} of $X_0$ is a smooth
projective variety $X$
equipped with an open immersion $\iota: X_0 \hookrightarrow X$ such that $\iota(X_0)$ is dense in $X$.
The endomorphism $f$ induces a dominant rational transformation of $X$ via $\tilde f = \iota \circ f
\circ {\iota}^{-1}$ and we define the dynamical degrees
\begin{equation}
  \lambda_l (f) := \lambda_l (\tilde f).
  \label{<+label+>}
\end{equation}
As the dynamical degrees are birational invariants, these quantities do not depend on the choice of the
completion $X$.
The data of these dynamical degrees gives information on the dynamical system. For example over $\C$, Dinh
and Sibony showed in \cite{dinhBorneSuperieurePour2003} that for all dominant rational transformation $f: X
\dashrightarrow X$
\begin{equation}
  h_{\Top} (f) \leq \max_{0 \leq l \leq d} \log (\lambda_l)
  \label{EqInegaliteEntropie}
\end{equation}
where $h_{\Top}$ is the topological entropy of $f$, Gromov showed this result for endomorphisms of $\P^N$
in \cite{gromovEntropyHolomorphicMaps2003}. Yomdin showed in \cite{yomdinVolumeGrowthEntropy1987} that we
have an equality if $f$ is an endomorphism. The inequality is strict in general (see
\cite{guedjEntropieTopologiqueApplications2005}). Recently, Favre, Truong and Xie showed in
\cite{favreTopologicalEntropyRational2022} that the inequality \eqref{EqInegaliteEntropie} still holds in
the non archimedean case; however the equality does not hold even for endomorphisms.

\section{Dynamical degrees on projective surfaces}
A natural question is to ask what numbers can appear as the first dynamical degree of a rational
transformation of a projective surface. For the topological degrees, it is easy to check that any integer $k$ is the
topoological degree of a dominant rational transformation of $\P^2$ (consider $f(x,y) = (x^k, y^k)$ over $\C^2$).

In 2021, Bell, Diller and Jonsson showed in \cite{bellTranscendentalDynamicalDegree2020} that there exists
a dominant rational transformation $\sigma: \P^2 \dashrightarrow \P^2$ such that $\lambda_1 (\sigma)$ is
transcendental. The authors with Krieger showed in \cite{bellTranscendentalDynamicalDegree2020} this
example can be generalised to give an example of a birational transformation of
$\P^N, N \geq 3$ with a transcendental first dynamical degree. However in dimension 2, there are strong
constraints on $\lambda_1(f)$ for $f$ birational. In
\cite{dillerDynamicsBimeromorphicMaps2001}, Diller and Favre showed that the first dynamical degree of a
birational transformation of a projective surface is an algebraic integer, but with arbitrary large degree. Indeed,
Bedford, Kim and McMullen have given
in \cite{bedfordPeriodicitiesLinearFractional2006} and \cite{mcmullenDynamicsBlowupsProjective2007}
examples of birational transformations of projective surfaces with first dynamical degree an algebraic
integer of arbitrary large degree. In particular, Theorem 1.1 of
\cite{mcmullenDynamicsBlowupsProjective2007} states that for all $d \geq 10$ we can find a smooth complex
projective surface with an automorphism with first dynamical degree an algebraic integer of degree $d$. This also holds
in positive characteristic by the main theorem of \cite{cantatRationalSurfacesLarge}. Blanc and Cantat showed in
\cite{blancDynamicalDegreesBirational2013} that the set of all first dynamical degrees of elements of $\Bir(\P^2_K)$ is
a well ordered set if $K$ is infinite.

\section{Dynamical degrees of endomorphisms of affine surfaces and Perron numbers}
The first example of an affine surface is the complex
affine plane $\C^2$. An endomorphism is then a polynomial transformation. Even in that case, the first
dynamical degree is not necessarily an integer. Indeed, let
\begin{equation}
  A = \begin{pmatrix}
    a & b \\
    c & d
  \end{pmatrix}
\end{equation}
be a matrix with nonnegative integer coefficients such that $ad - bc \neq 0$. Consider the following
monomial transformation
\begin{equation}
  f (x,y) = \left( x^a y^b, x^c y^d \right),
  \label{<+label+>}
\end{equation}
then $f^N$ is the monomial transformation where the monomials are given by the coefficients of $A^N$ and
$\lambda_1(f)$ is equal to the spectral radius of $A$. Hence, $\lambda_1(f)$ is an algebraic integer of
degree 2 because it satisfies the equation
\begin{equation}
  \lambda_1 (f)^2 - \Tr (A) \lambda_1 (f) + \det (A) = 0.
  \label{<+label+>}
\end{equation}
It is in fact a \emph{Perron number}. A \footnote{In the litterature, the Galois conjugates of a Perron number
have a \emph{strictly} smaller modulus but we want to include square roots of integers in our definition.}{(weak)} Perron number
is a real algebraic integer $\alpha \geq 1$ such that all its Galois conjugates have complex modulus $\leq \left| \alpha
\right|$.
Thus, there exist polynomial transformations $f$ of the affine plane with $\lambda_1 (f)$ an integer or a Perron number
of degree 2. Favre and Jonsson showed that these are the only two possibilities.

\begin{thm}\cite{favreEigenvaluations2007}
  Let $f : \C^2 \rightarrow \C^2$ be a dominant polynomial transformation, then $\lambda_1 (f)$ is a
  Perron number of degree $\leq 2$.
\end{thm}

The first result of this memoir is to extend this result to all \emph{normal affine surfaces}, in characteristic zero.
Even if the semigroup of endomorphisms can change drastically when one changes the affine surface. For example,
Blanc and Dubouloz, in \cite{blancAffineSurfacesHuge2013}, build smooth affine surfaces with a big group
of automorphisms, much bigger than the one of the affine plane. Bot used this construction to show the
existence of smooth complex rational affine surfaces with uncountably many real forms (see
\cite{botSmoothComplexRational2023}). The results in this paper show that, even though structure wise these
groups are a lot more complicated, from the point of view of the dynamics of a single element, this is not
the case.

\begin{bigthm}\label{BigThmDynamicalDegreesEng}
  Let $X_0$ be a normal affine surface over a field $\k$ of characteristic zero. If $f : X_0 \rightarrow
  X_0$ is a dominant endomorphism, then $\lambda_1(f)$ is a Perron number of degree $\leq 2$.
\end{bigthm}
We will also deal with positive
characteristic in the paper but the result is not as easily statable as there are more cases to treat. In particular, Theorem
\ref{BigThmDynamicalDegreesEng} holds for the affine plane $\A^2_\k$ over any field $\k$ (see Theorem
\ref{BigThmExistenceValuationPropreEng}).

The proof uses valuative techniques which we describe in the next section. We also obtain results on the dynamics of $f$. For any completion $X$ of
$X_0$, the endomorphism $f$
extends to a rational transformation of $X$, it has a finite number of indeterminacy points at infinity i.e on $X
\setminus X_0$. One cannot hope in general to find a
completion $X$ such that $f$ extends to a regular endomorphism of $X$. The strategy of proof consists of studying the
dynamics of $f$ at infinity. More specificallly, we find good
completions where $f$ has an attracting fixed point at infinity, i.e a point $p \in X \setminus X_0$ where the lift
$f: X \dashrightarrow X$ of $f$ is defined at $p$ and $f(p) = p$ we can then study the local dynamics at $p$ to compute
the first dynamical degree of $f$. Theorem \ref{BigThmAutomorphismIntroEng} below provides
a precise statement in the case of automorphisms; the most general results will be described in Chapter
\ref{ChapterAutomorphisms}.

\section{The dynamical spectrum of the algebraic torus}
If $V$ is an algebraic variety, let $\End(V)$ be the semigroup of dominant endomorphisms of $V$. We define the
\emph{dynamical spectrum} of $V$ by
\begin{equation}
  \Lambda (V) := \left\{ \lambda_1 (f) : f \in \End (V) \right\}.
  \label{<+label+>}
\end{equation}
As every $2\times 2$ matrix with integer coefficients induces a monomial endomorphism of the algebraic torus $\G_m^2$,
we have that $\Lambda (\G_m^2)$ is the set of Perron numbers of degree $\leq 2$. By Theorem
\ref{BigThmDynamicalDegreesEng}, this shows that $\Lambda (\G_m^2)$ is maximal among the dynamical spectra of normal
affine surfaces. One might wonder if this is a characterization of the algebraic torus but we show that this is not
the case. In fact, we show that any Perron number of degree 2 is the dynamical degree of a monomial transformation of
$\A^2_\k$ with positive coefficients. Thus we prove the following result.

\begin{thm}\label{BigThmDynamicalSpectrum}
  For any field $\k$,
  \begin{equation}
    \Lambda (\A^2_\k) = \Lambda (\G_{m,\k}^2).
    \label{<+label+>}
  \end{equation}
\end{thm}

\section{Existence of an eigenvaluation}
Let $A$ be the ring of regular functions of a normal affine surface $X_0$ over an algebraically closed
field $\k$. A \emph{valuation} is a map $v : A
\rightarrow \R \cup \left\{ \infty \right\}$ such that
\begin{enumerate}
  \item $v(PQ) =v(P) +v(Q)$;
  \item $v (P+Q) \geq \min (v(P),v(Q))$;
  \item $v(0) = \infty$;
  \item $v_{| \k^\times} = 0$
\end{enumerate}
Two valuations $v$ and $\mu$ are \emph{equivalent} if there exists $t >0$ such that $v = t \mu$.
For example, if $X$ is a completion of $X_0$, for each irreducible curve $E \subset X$, the map $\ord_E$
defined by $\ord_E(P)$ being the order of vanishing of $P$ along $E$ is a valuation. Any valuation of the
form $\lambda \ord_E$ with $\lambda > 0$ is called \emph{divisorial}. If $f$ is an endomorphism of $X_0$,
then $f$ induces a ring homomorphism $f^* : A \rightarrow A$. We can then define the pushforward $f_* v$
of a valuation $v$ by
\begin{equation}
  f_* v (P) = v (f^* P).
  \label{<+label+>}
\end{equation}

We say that a valuation is \emph{centered at infinity} if there exists $P \in A$ such that $v(P) <0$. If
$X$ is a completion of $X_0$, the divisorial valuations centered at infinity are exactly the one
corresponding to the irreducible components of $X \setminus X_0$. Let $\cV_\infty$ the set of valuations
centered at infinity and $\hat
\cV_\infty$ the set of valuations centered at infinity modulo equivalence. Suppose for the sake of
simplicity that $f$ is an automorphism of $X_0$, then $f_*$ induces a bijection of $\cV_\infty$ and of
$\hat \cV_\infty$ which will in fact be a homeomorphism for a topology that will be described in Chapter \ref{ChapterValuations}.

If $X_0$ is the complex affine plane, then Favre and Jonsson proved the existence of a valuation $v_* \in \cV_\infty$
such that $f_* v_* = \lambda_1 (f) v_*$. Such a valuation is called an \emph{eigenvaluation} of $f$. To
do so, they show in \cite{favreValuativeTree2004} that $\hat \cV_\infty$ has a real tree structure and
$f_*$ is compatible with this structure. The existence of $v_*$ follows from a fixed point theorem on
trees. The existence of this eigenvaluation has a big impact on the dynamics of $f$. In particular, it
allows one to find a good completion $X$ of $\C^2$ and a point $q \in X \setminus \C^2$ (a point at infinity) which is
an attracting fixed point for the dynamics of $f$ (extended to $X$ as a rational map). Xie uses this construction to
prove the Zariski dense orbits conjecture and the dynamical
Mordell-Lang conjecture for polynomial endomorphisms of the complex affine plane
(\cite{xieExistenceZariskiDense2017}). Jonsson and Wulcan use these techniques to build canonical heights
for polynomial endomorphisms of the complex affine plane with small topological degree in \cite{jonssonCanonicalHeightsPlane2012}.

\begin{bigthm}\label{BigThmExistenceValuationPropreEng} \label{BIGTHMEXISTENCEVALUATIONPROPREENG}
  Let $X_0$ be a normal affine surface over an algebraically closed field $\k$ (of any characteristic) and
  let $f$ be a dominant endomorphism of $X_0$. Suppose that
  \begin{enumerate}
    \item $\k[X_0]^\times = \k^\times$.
    \item For any completion $X$ of $X_0$, $\Pic^0 (X) = 0$.
    \item $\lambda_1 (f)^2 > \lambda_2(f)$.
  \end{enumerate}
  Then, there exists a nef valuation $v_*$, unique up to equivalence, such that
  \begin{equation}
    f_* (v) = \lambda_1(f) v_*
  \end{equation}
  and $\lambda_1 (f)$ is a Perron number of degree $\leq 2$.
\end{bigthm}

The techniques we use do not use the global geometry of $\hat \cV_\infty$ because it is not necessarily a tree
anymore. If $X$ is a completion of $X_0$ and $v$ is a valuation centered at infinity, we associate in a canonical way a divisor
$Z_{v, X}$ of $X$ supported outside of $X_0$. If $\pi: Y \rightarrow X$ is
another completion of $X_0$ obtained from blowing up points of $X$ at infinity, and $\pi_* Z_{v, Y} = Z_{v, X}$
(see Proposition \ref{PropValuationForCartierDivisorOverOneCompletion}), a valuation is nef if for every $X$
$Z_{v,X}$ is nef. This construction involves the
Picard-Manin space of $X_0$. We give a brief description of this space. Consider the direct limit
\begin{equation}
  \cNS = \varinjlim_{X} \NS (X)_\R
  \label{<+label+>}
\end{equation}
indexed by all the completions of $X_0$. This is an infinite dimension real vector space. The intersection form can be
extended in a natural way to $\cNS$ and the Hodge Index theorem states that it is a non-degenerate form of Minkowski
type, i.e its signature is $(1, \infty)$. We can use the intersection form to define a norm on $\cNS$ and the Picard
Manin space of $X_0$ will be the completion of $\cNS$ with respect to this norm. It has a structure of a Hilbert space
and any dominant endomorphism $f$ of $X_0$ induces two bounded operators $f^*, f_*$ on it.
The spectral analysis of the operators $f_*, f^*$  (see
\cite{boucksomDegreeGrowthMeromorphic2008, cantatGroupesTransformationsBirationnelles2011}) allows one to
construct the eigenvaluation $v_*$ and show its uniqueness. Namely, $\lambda_1(f)$ is the spectral radius of $f^*$ and
$f_*$ and when $\lambda_1^2 > \lambda_2$, there is a spectral gap property. The eigenvalue $\lambda_1$ is simple for
$f^*$ and $f_*$. This process is similar to the techniques of
\cite{dangSpectralInterpretationsDynamical2021} \S 6. These techniques were used by Gignac and Ruggiero in
\cite{gignacLocalDynamicsNoninvertible2021} to study the local dynamics of non invertible germs near a normal
singularity in dimension 2. This memoir can be considered to be the global counterpart to the local techniques developed
by these two authors. Our construction of the valuation $v_*$ is however different.

\section{Discussion on the assumptions of Theorem \ref{BigThmExistenceValuationPropreEng}}
The assumptions of Theorem \ref{BigThmExistenceValuationPropreEng} may seem arbitrary but they are not
restrictive. Indeed, if assumption (1) or (2) is not satisfied, then one can show that $f$ preserves a
fibration over a \footnote{a quasi-abelian variety is an algebraic group such that there exists an
algebraic torus $T$ and an abelian variety $A$ such that the sequence of algebraic groups $0 \to T \to X \to A \to
0$ is exact.}{\emph{quasi-abelian variety}}. We can decompose the dynamics of $f$ with this fibration and it
becomes easier to study. This is done in Chapter \ref{ChapterGeneralCase}, we show that the only case with interesting dynamics
is when $X_0$ is the algebraic torus $\G_m^2$.

\begin{bigthm}\label{BigThmClassificationQAlbNonTrivial}
  Let $X_0$ be a normal affine surface over an algebraically closed field of characteristic zero. Suppose that $X_0$ does not satisfy
  Conditions (1) or (2) of Theorem \ref{BigThmExistenceValuationPropreEng}, then either
  \begin{enumerate}
    \item $X_0$ is of log general type. Every dominant endomorphism of $X_0$ is an automorphism and $\Aut(X_0)$ is a
      finite group. Thus, $\lambda_1(f) = 1$ for every dominant endomorphism $f$ of $X_0$.
      \item There exists a curve $C$ and a fibration $\pi : X_0 \dashrightarrow C$ preserved by every endomorphism of
        $X_0$. This means that for every endomorphism $f$ of $X_0$, there exists an endomorphism $g : C \rightarrow C$
        such that $\pi \circ f = g \circ \pi$. In that case $\lambda_1(f)$ is always an integer.
      \item $X_0 \simeq \G_m^2$.
    \end{enumerate}
\end{bigthm}

For positive characteristic such a classification also holds if we consider \emph{separable} endomorphisms. However, our proof does
not go through for endomorphisms as we will use log Kodaira-Iitaka
fibration and in positive characteristic a dominant endomorphism need not preserve the log Kodaira-Iitaka fibration
because the pull-back of a nonseparable differential form could be zero.

If Assumption (3) of Theorem \ref{BigThmExistenceValuationPropreEng} is not satisfied, then we have $\lambda_1 (f)^2 = \lambda_2(f)$. In that case,
$\lambda_1(f)$ is automatically a Perron number of degree $\leq 2$ because $\lambda_2(f)$ is the
topological degree of $f$, hence an integer. In the case of the complex affine plane, Favre and Jonsson
manage to classify all polynomial transformations of the complex affine plane for which $\lambda_1^2 =
\lambda_2$: either they preserve a rational fibration, or there exists a completion $X$ of $\A^2_\C$
with at most quotient singularities at infinity such that $f$ extends to an endomorphism of $X$. We
expect that such a classification should exist in general, with one exceptional counterexample: The monomial
transformations of $(\G_m^2)$ that cannot be made algebraically stable (see \cite{favreApplicationsMonomialesDeux2003}).
We conjecture that the only counterexamples of this classification should come from equivariant quotient of these
monomial maps. In the local case of dyamics near a normal singularity Gignac and Ruggiero
(\cite{gignacLocalDynamicsNoninvertible2021}) showed such a classification.
One can notice that in the invertible case, a classification exists: By
\cite{gizatullinApplicationNumericalCriterion1969}, \cite{dillerDynamicsBimeromorphicMaps2001} and
\cite{cantatDynamiqueAutomorphismesSurfaces2001}, every birational transformation $\sigma : X \rightarrow X$ of a smooth
projective surface such that $\lambda_1(\sigma) = 1$ lifts to an automorphism or preserves a rational or elliptic
fibration.

\section{Statement of the theorem in the case of automorphisms}
In the case of loxodromic automorphism (i.e with $\lambda_1 > 1$), we obtain informations on the dynamics. For
this introduction we state the result in the complex case.

\begin{bigthm}\label{BigThmAutomorphismIntroEng}
  Let $X_0$ be a normal affine surface over $\C$ such that $\C [X_0]^\times = \C^\times$. If $f$ is an
  automorphism of $X_0$ such that $\lambda_1 (f) > 1$, then there exists a completion
  $X$ of $X_0$ such that $X$ is smooth outside the singular points of $X_0$ and
  \begin{enumerate}
    \item $f$ admits a unique attracting fixed point $p \in X (\C) \setminus X_0(\C)$ at infinity.
  \item An iterate of $f$ contracts $X \setminus X_0$ to $p$.
    \item There exists local analytic coordinates centered at $p$ such that $f$ is locally of the form
  \begin{enumerate}
      \item \begin{equation}
        f (z,w) = (z^a w^b, z^c w^d)         \label{EqFormeNormaleMonomiale}
      \end{equation}
      with $a,b,c,d$ integers $\geq 1$, in that case $\lambda_1(f)$ is the spectral radius of $\begin{pmatrix}
        a & b \\
        c & d
      \end{pmatrix}$. In particular, $\lambda_1 (f) \in \R \setminus \Q$, it is a Perron number of degree 2.
    \item or
      \begin{equation}
        f (z,w) = (z^a, \lambda z^c w + P(z))
        \label{EqFormeNormaleInfSingEng}
      \end{equation}
  with $a \geq 2, c \geq 1$ and $P \not \equiv 0$ a polynomial, in that case $\lambda_1(f) =a$ is an integer.
  \end{enumerate}
\item The attracting fixed points of $f$ and ${f}^{-1}$ are distinct.
\item The local normal form of ${f}^{-1}$ at its attracting fixed point is of the same form as $f$.
\end{enumerate}
\end{bigthm}

Theorem \ref{BigThmAutomorphismIntroEng} holds in fact for any complete algebraically closed field (in any characteristic) but
we cannot be as precise with local normal forms in general, see Theorem
\ref{ThmLocalFormOfMapAndDynamicalDegreAreQuadraticInteger} and \ref{ThmAutomorphismCaseDynamicalCompactifications}.

The cases (3)(a) et (3)(b) are mutually exclusive in the following way
\begin{bigthm}\label{BigThmDichotomyAutomorphismEng}
  Let $X_0$ be a normal affine surface over $\C$ such that $\C [X_0]^\times = \C^\times$ and $f \in \Aut(X_0)$
  a loxodromic automorphism. We have the following dichotomy
  \begin{itemize}
    \item If $\lambda_1 (f) \in \Z_{\geq 0}$, then for any loxodromic automorphism $g$ of $X_0$, we have
      $\lambda_1(g) \in \Z_{\geq 0}$ and the local normal form of $g$ at its attracting fixed point is
      given by \eqref{EqFormeNormaleInfSingEng}.
    \item If $\lambda_1 (f) \not \in \Z_{\geq 0}$ then it is a Perron number of degree 2 and this
      holds for any loxodromic automorphism $g$ of $X_0$. In particular, the local normal form of $g$ at
      its attracting fixed point is given by \eqref{EqFormeNormaleMonomiale}.
  \end{itemize}
\end{bigthm}
This result actually holds for any algebraically closed field (of any characteristic).

\subsection*{Plan of the memoir}
This memoir is divided into two parts. In the first part we establish the main definitions and results needed for the
proofs of the theorems stated in this introduction. In Chapter \ref{ChapterDefinitions}, we define completions of an
affine surface and introduce the Picard
Manin space of an affine surface. In Chapter \ref{ChapterValuations}-\ref{ChapterTopologiesValuations}, we define
valuations and explain the geometry of the space of valuations centered at
infinity of an affine
surface.  The
main result of this part is that a valuation induces a linear form with special properties on the space of divisors at
infinity and that this process is bijective. This is the goal of Chapters
\ref{ChapterValuationsAsLinearForms}-\ref{ChapterBijection}.  

The second part is dedicated to the proofs of the theorems of this introduction using the results established in the
first part. We construct the eigenvaluation and prove Theorems \ref{BigThmDynamicalDegreesEng},
\ref{BigThmDynamicalSpectrum} and \ref{BigThmExistenceValuationPropreEng} in Chapters \ref{ChapterGeneralCase} and
\ref{ChapterDynamicsQuasiAlbTrivial}. The main part is to show that the space of valuations centered at infinity embeds
into the Picard-Manin space of $X_0$. This is done in Chapter \ref{ChapterDynamicsQuasiAlbTrivial}. This construction
uses the Hodge index theorem, we associate for every valuation
$v$ a $b$-divisor $Z_v$ and we give a numerical criterion to show when $Z_v$ is nef. This is the global counterpart of
\cite{gignacLocalDynamicsNoninvertible2021} where Gignac and Ruggiero developped these techniques to study the dynamics
of germs of normal surfaces singularities. This construction only works when the \emph{Quasi-Albanese} variety of $X_0$
is trivial (which is for example the case for the affine plane) so we need to deal with the case where it is not. This
is a case which does not appear for the affine plane and is treated in Chapter \ref{ChapterGeneralCase}. In Chapter
\ref{ChapterLocalNormalForms}, we show that the eigenvaluation
constructed is an attracting fixed point in the space of valuations and derive results on the dynamics at infinity of
our endomorphisms. This is the same strategy as in \cite{favreEigenvaluations2007} but our techniques differ. The work
of Favre and Jonsson is based on the fact that the space of valuations centered at infinity of the affine plane is a
real tree and that the action of an endomorphism of the affine plane on this tree is sufficiently regular to find an
attracting fixed point.
Our approach is based on \cite{boucksomDegreeGrowthMeromorphic2008} where it is shown in the case
$\lambda_1^2 (f) >\lambda_2 (f)$, the spectral radius of the operator $f_*$ on the Picard-Manin space of $X_0$ is equal
to $\lambda_1$ and furthermore it is a simple eigenvalue. We show that the associated eigenvector $\theta_*$ is of the
form $Z_{v_*}$ for $v_*$ centered at infinity and this is the desired eigenvaluations. We study examples in Chapter
\ref{ChapterExamples}. For examples over the complex affine plane, we
refer to \cite{favreEigenvaluations2007} and \cite{favreDynamicalCompactificationsMathbf2011}. For examples of affine
surfaces with interesting automorphisms group, we refer to \cite{blancAffineSurfacesHuge2013}. We then apply these
results to the study of the dynamics of loxodromic automorphisms of affine surfaces using previous results from
Gizatullin. Theorems \ref{BigThmAutomorphismIntroEng} and \ref{BigThmDichotomyAutomorphismEng} are proven in Chapter
\ref{ChapterAutomorphisms}, where we apply the techniques of Chapter \ref{ChapterLocalNormalForms}.

One caveat of the approach in this memoir is that we do not obtain informations on the dynamics when $\lambda_1^2 = \lambda_2$, in that case the
techniques presented here do not allow one to construct an eigenvaluation. This is because in that case the eigenvalue
$\lambda_1$ is not simple anymore and the associated eigenvectors are not necessarily attractive. What is
probably necessary to do is to generalise the techniques of \cite{gignacLocalDynamicsNoninvertible2021} to affine surfaces.

\subsection*{Acknowledgments}
This work was done during my PhD under the supervision of Serge Cantat and Junyi Xie whom I thank deeply. I would also
like to thank Charles Favre for answering some of my questions when I showed him this work. Part of this memoir was
written during my visit at Beijing International Center for Mathematical Research which I thank for
its welcome. During this visit, I had very fruitful discussions with Matteo Ruggiero which helped improving some
results of the initial version of this memoir. I also thank the France 2030 framework programme Centre Henri
Lebesgue ANR-11-LABX-0020-01 and
European Research Council (ERCGOAT101053021) for creating an attractive mathematical environment. Finally, the author
acknowledges support by the Swiss National Science Foundation Grant “Birational transformations of higher dimensional
varieties” 200020-214999. I finally thank the anonymous reviewer for his/her very thorough report on a first version of
this memoir which has helped improved considerably the presentation.

%% file: chapitre1.tex
\part{Valuations and Algebraic geometry}
This first part will establish all the prelimary tools and results necessary for the proofs of the theorems stated in
the introduction. The general idea is that we start with a dynamical system given by a normal affine surface $X_0$ and a
dominant endomorphism $f$, the endomorphism $f$ induces a transformation on several spaces associated to $X_0$.

For every completion $X,Y$ of $X_0$, $f$ lifts to a rational map $f:X \dashrightarrow Y$ which induces a pullback and a
pushforward on the Néron-Séveri spaces. There is not a canonical choice of a completion of $X_0$, so the right space to
consider is (a completion of) the direct limit of the Néron-Severi groups of completion of $X_0$. This space is called
the \emph{Picard-Manin} space of $X_0$ and $f$ defines a pullback and a pushforward operator on it. The first dynamical
degree is the spectral radius of the operator induced by $f$ on the Picard-Manin space. This is done in Chapter
\ref{ChapterDefinitions}.

In Chapter \ref{ChapterValuations}, we study the set of valuations over the ring of regular
functions of $X_0$. We will focus on a special subset called the space of valuations centered at infinity. The main
example of valuations centered at infinity are the valuations induced by the irreducible component of $X \setminus X_0$
where $X$ is a completion of $X_0$. These ones are called \emph{divisorial} valuations. The
endomorphism $f$ defines a pushforward operator on this space. In Chapter \ref{ChapterValuationTree} and
\ref{ChapterTopologiesValuations} we describe the different topologies on this space and discuss its local tree
structure following the work of Favre and Jonsson on the valuative tree in \cite{favreValuativeTree2004}.

In Chapter \ref{ChapterValuationsAsLinearForms}, Chapter \ref{ChapterLinearFormsToValuations} and Chapter
\ref{ChapterBijection} we show how these two spaces are related. Namely, a valuation centered at infinity induces a linear
form with special properties on divisor supported at infinity for every completion $X$ of $X_0$ and this process is
bijective. The action of $f$ on the space of valuations and on the space of divisors are compatible via this bijection.

\chapter{Results from Algebraic Geometry}
In this chapter, we recall several results from Algebraic Geometry that will be used throughout this memoir.
Let $\k$ be an algebraically closed field. A \emph{variety} is an integral scheme of finite type
over $\k$. A surface is a variety of dimension 2. An affine variety over $\k$ is a variety $X_0 = \spec A$ with $A$ a
finitely generated $\k$-algebra. We will denote by $\k[X_0]$ the ring of regular functions of the affine variety $X_0$.

\section{Bertini}
\begin{thm}[Bertini's Theorem, \cite{hartshorneAlgebraicGeometry1977}]\label{ThmBertini}
  Let $X \subset \P^N$ be a smooth quasi-projective variety over an algebraically closed field $\k$. The set of
  hyperplanes $H$ of $\P^N$ such that the intersection $H \cap X$ is a smooth irreducible subvariety of $X$ is a dense
  open subset of $\P \Gamma (\P^N, \OO (1))$.
\end{thm}

\section{Local power series and local coordinates}

Let $X$ be a variety and $x \in X$ a closed point. We will write $\OO_{X, x}$ for the ring of germs of regular
functions at $x$. A \emph{regular sequence} of $\OO_{X,x}$ is a sequence $t_1, \cdots, t_r \in \OO_{X,x}$ such that
$t_1$ is not a zero divisor in $\OO_{X,x}$ and for all $i \geq 2, t_i$ is not a zero divisor in $\OO_{X, x} /
(t_1, \cdots, t_{i-1})$ (see \cite{hartshorneAlgebraicGeometry1977} p.184). The point $x$ is \emph{regular} if the local
ring $\OO_{X, x}$ is regular, i.e there exists a regular sequence of length $\dim \OO_{X,x}$.

\begin{thm}[Cohen's structure theorem, \cite{hartshorneAlgebraicGeometry1977}, Theorem 5.5A]\label{ThmCompletionLocalRing}
  Let $R$ be a complete regular local $\k$-algebra of dimension $n$ with maximal ideal $\m$, then the completion of $R$ with
  respect to the $\m$-adic topology is isomorphic to $\k \left[ \left[ t_1, \cdots, t_n \right] \right]$ where
  $(t_1, \cdots, t_n)$ is a regular sequence of $R$.
\end{thm}

Let $X$ be a surface and $x$ a regular point of $X$. Then, we will say that $(z,w)$ are \emph{local coordinates} at
$x$ if $(z,w)$ is a regular sequence of $\OO_{X,x}$. If $(z,w)$ is a regular sequence of the completion
$\hat{\OO_{X,x}}$ we will say that they are local \emph{formal} coordinates. By Theorem \ref{ThmCompletionLocalRing},
${\hat{\OO_{X,x}}}$ is isomorphic to $\k [ [ z,w ] ]$. Finally, If $\k = \C_v$ is a complete
algebraically closed field characteristic zero,  we consider the local ring of germs of \emph{holomorphic} functions at
$x$, this is the
subring of $\hat{\OO_{X,x}}$ of power series with a positive radius of convergence. We denote it by
$\OO_{X,x}^{hol}$ it is also a local ring of dimension 2, if $(z,w)$ is a regular sequence of $\OO_{X,x}^{hol}$, we
say that $(z,w)$ are local \emph{analytic coordinates}. If $E,F$ are two germs of reduced irreducible curves at $x$
(algebraic, analytic of formal) we will say that $(z,w)$ are \emph{associated} to $(E,F)$ if $z = 0$ is a local equation
of $E$ and $w = 0$ is a local equation of $F$.

\section{Boundary}
\begin{prop}[\cite{goodmanAffineOpenSubsets1969}, Proposition 1 and 2] \label{PropBoundaryIsCurve}
  Let $X_0$ be an affine variety and let $\iota: X_0 \hookrightarrow X$ be an open embedding into a projective
  variety, then the subvariety $X \setminus X_0$ is connected and of pure codimension 1.
\end{prop}

Set
\begin{equation}
  \BD := X \setminus X_0,
\end{equation}
 we call it the \emph{boundary} of $X_0$ in $X$; by Proposition \ref{PropBoundaryIsCurve} it is a possibly reducible curve
 when $X_0$ is a surface.

\begin{thm}[\cite{goodmanAffineOpenSubsets1969}] \label{ThmGoodmanExistenceAmpleDivisorAtInfinity}
    Let $X$ be a normal proper surface and $U$ an open dense affine subset of $X$ (that is an open dense
    subset of $X$ that is also an affine variety) such that every point of $V := X \setminus U$ is factorial (i.e its
    local ring is a UFD), then there exists an ample effective divisor $H$ on $X$ such that $\supp H = V$.
\end{thm}
In particular, the condition on the theorem holds if the singular points of $X$ are contained in $U$.
In fact, Goodman shows that Theorem \ref{ThmGoodmanExistenceAmpleDivisorAtInfinity} holds in higher dimension with the
only difference that you may need to do some blow-ups at infinity to find an ample divisor.

\section{Surfaces}

\begin{thm}[\cite{hartshorneAlgebraicGeometry1977} Proposition 5.3]\label{ThmCompositionBlowUps}
  Let $g: S_1 \rightarrow S_2$ be a birational morphism between smooth projective surfaces. Then, $g$ is a composition of
  blow-ups of points and of an automorphism of $S_2$. Furthermore, if $h : S_1 \dashrightarrow S_2$ is a birational map,
  then there exists a sequence of blow-ups $\pi: S_3 \rightarrow S_1$ such that $h \circ \pi: S_3 \rightarrow S_2$ is
  regular and $S_3$ can be chosen minimal for this property.
\end{thm}

\begin{prop}\label{PropTechniqueContractionDeCourbes}
  Let $g : S_1 \dashrightarrow S_2$ be a birational map. Let $\pi: S_3 \rightarrow S_1$ be a minimal resolution of
  indeterminacies of $g$ such that the lift $h : S_3 \rightarrow S_2$ of $g$ is regular. Then, the first curve
  contracted by $h$ must be the strict transform of a curve in $S_1$.
\end{prop}

Recall the Castelnuovo criterion

\begin{thm}[\cite{hartshorneAlgebraicGeometry1977} Theorem V.5.7]
  Let $C$ be a curve in a projective surface $S$ such that $C \simeq \P^1$ and $C^2 = -1$, then there exists a
  projective surface $S'$, a birational morphism $\pi : S \rightarrow S'$ and a point $p \in S'$ such that $S$ is
  isomorphic via $\pi$ to the blow up of $p$ and $C$ is the exceptional divisor under this isomorphism.
\end{thm}
We will use these results for the study of automorphisms of affine surfaces as they induce birational maps.
Understanding the combinatorics of the blow ups and contractions induced by the automorphism will allow us to understand
their dynamics.

Our work relies heavily on the elimination of indeterminacies for rational morphism. Since we are in dimension 2, it
exists in any characteristic.
\begin{thm}
  Let $f : S_1 \dashrightarrow S_2$ be a dominant rational morphism between projective varieties over an algebraically
  closed field of any characteristic, then there exists a
  sequence of blow-ups $\pi: S \rightarrow S_1$ such that $f \circ \pi: S \rightarrow S_2$ is regular.
\end{thm}

\section{Rigid contracting germs in dimension 2 and local normal forms}\label{SubSecContractingRigidGerm}
Let $(\k, | \cdot |)$ be a complete metrised field of characteristic zero. We define the euclidian topology
on $\k^2$ as the topology induced by the norm $\|(x,y) \| = \max  (\left| x \right|, \left| y \right|)$.
Let $f : (\k^2, 0) \rightarrow (\k^2, 0)$ be the germ of an analytic function fixing the origin, i.e a convergent power
series in the coordinates axis $(x,y)$ such that $f(0,0) = (0,0)$. The \emph{critical
set} $\Crit(f)$ of $f$ is the set where the Jacobian of $f$ vanishes. A germ is said to be \emph{rigid} if the
generalised critical set $\cup_{n \geq 0} f^{-n} (\Crit (f)) = \cup_{n \geq 1} \Crit(f^n)$ is a divisor with simple
normal crossings (see \cite{favreClassification2dimensionalContracting2000}). Since $\car \k = 0$, $\Crit(f)$ is a union
of finitely many irreducible germ of curves.

A germ is \emph{contracting} if there exists an open (euclidian) neighbourhood $U$ of $0$ such that $f (U) \Subset U$. In
\cite{favreClassification2dimensionalContracting2000}, Favre classified all the complex rigid contracting germs in dimension 2
up to holomorphic conjugacy. Ruggiero extended partially the classification in higher dimension in
\cite{ruggieroContractingRigidGerms2013} and showed that, in dimension 2, it holds over any algebraically closed
complete metrised field
of characteristic zero. For this section, $\C$ will denote any algebraically closed complete metrised field of characteristic zero.
In dimension 2, there are 7 possible possibilities which we call \emph{local normal forms}. We are
interested in 3 of them that will appear in this text. However since we do not only work in characteristic zero, we
start by more general local forms that works over any field and show their counterpart over $\C$.

\textbf{First normal form.--} Let $f$ be the germ of a regular map at a smooth point of a projective surface over an
algebraically closed field $\k$. Suppose
that there are local coordinates $(z,w)$ at the origin such that $f$ contracts $\left\{ z=0
\right\}$ with an index of ramification $a \geq 2$, $f$ admits no invariant curves and no other curves is contracted to
the origin or sent to $w=0$, then $f$ is of the form
\begin{equation}
  f(z,w) = (z^a \phi(z,w), z^c (z \psi_1 (z,w) + w \psi_2 (z,w))
  \label{EqPseudoFormeInfSing}
\end{equation}
with $\phi$ invertible, $\psi_1$ regular and $\psi_2 (0,w) \neq 0$ and all these functions are regular. If $\k = \C$,
this germ of regular function is actually rigid and in the classification of Favre this local normal form
corresponds to Class 2 of Table II in \cite{favreClassification2dimensionalContracting2000}. It is analytically conjugated to
\begin{equation}
  f(x,y) = (x^a, \lambda x^c y + P (x))
  \label{EqLocalNormalFormInfinitelySingular}
\end{equation}
with $a \geq 2, c \geq 1, \lambda \in \k^\times$ and $P$ is a polynomial such that $P(0) = 0$.  This is the local normal
form of a Hénon map at its attracting fixed point in $\P^2$ (see \cite{favreClassification2dimensionalContracting2000}
\S 2).

\textbf{Second normal form.--} If $f$ is a germ of a
regular function such that there exists local coordinates $(z,w)$ at the origin with both axis $\left\{ z=0
\right\}$ and $\left\{ w =0 \right\}$ contracted and they are the only two germs of curves contracted. Then, $f$ is
of the following pseudomonomial form
\begin{equation}
  f(z,w) = \left( z^{a_{11}} w^{a_{12}} \phi(z,w), z^{a_{21}} w^{a_{22}} \psi(z,w) \right)
  \label{EqPseudoFormeMonomiale}
\end{equation}
with $\phi, \psi$ invertible regular functions and $a_{ij} \in \Z_{\geq 0}$. Suppose now that $\k = \C$ and that $ad -bc
\neq 0$, then \eqref{EqPseudoFormeMonomiale} is rigid and analytically conjugated to the monomial normal form
\begin{equation}
            f (x,y) = (x^{a_{11}} y^{a_{12}}, x^{a_{21}} y^{a_{22}})
  \label{EqLocalNormalFormMonomial}
\end{equation}
 The germ of curves $\left\{ x = 0 \right\},
 \left\{ y = 0 \right\}$ are contracted to the origin. We have $\Crit (f^n) = \left\{ xy = 0 \right\}$. If $\C$ is the
 field of complex numbers, we can
characterize the matrix $A$ given by $(a_{ij})$ in the following way. The local fundamental group of $ (\C^2, 0)
\setminus \left\{ xy =0 \right\}$ is isomorphic to $\Z^2$. The action of $f_*$ on $\Z^2$ is given by the matrix $A$ and
we have that $\left| \det A \right|$ is equal to the topological degree of $f$. This corresponds to Class 6 of Table II of
\cite{favreClassification2dimensionalContracting2000}.

\textbf{Third normal form.--} The third one is
\begin{equation}
           f (x,y) = (x^a y^b \phi, y^c \psi)   \label{EqLocalNormalFormDivisorial}
\end{equation}
with $a\geq 2, b,c \geq 1$ and $\phi, \psi$ are germs of invertible regular functions
vanishing at the origin. We have that $\left\{ y = 0 \right\}$ is contracted to the origin. The germ $E = \left\{ x = 0
\right\}$ is $f$-invariant with a ramification index equal to $a$. The origin is a noncritical fixed point of
$f_{|E}$ if and only if $c = 1$.

If $\k = \C$, We have $\Crit(f^n) = \left\{ xy = 0 \right\}$. If $c =1$, this
germ is rigid but not necessarily contracting. It is contracting if and only if $\left| \psi(0) \right| < 1$. If the
germ is contracting then the germ is analytically conjugated to this normal form
\begin{equation}
  f(z,w) = \left( z^a w^b, \psi(0) w\right)
  \label{<+label+>}
\end{equation}
with the same numbers $a,b$ as in Equation \ref{EqLocalNormalFormDivisorial}. This corresponds to Class 5 of
Table II in \cite{favreClassification2dimensionalContracting2000}.

\chapter{Divisors at infinity and Picard-Manin space}\label{ChapterDefinitions}
In this chapter, we introduce the notion of completions of an affine surface $X_0$. They are essentialy projective
compatifications of $X_0$ and form a projective set. The Picard-Manin space of $X_0$ will be a completion of the direct limit of
the Néron Sévéri groups of the completions of $X_0$. It is a Hilbert space on which every endomorphism of $X_0$ acts in
a natural way.
  Let $\k$ be an algebraically closed field of any characteristic and let $X_0$ be a normal affine surface over $\k$. We will
  denote by $\k[X_0]$ the ring of regular functions on $X_0$.

\section{Completions and divisors at infinity}\label{SubSecCompletions}

A \emph{completion} of $X_0$ is the data of a projective surface $X$ with an open
embedding $\iota: X_0 \hookrightarrow X$ such that $\iota(X_0)$ is an open dense subset of $X$ and such
that there exists an open smooth neighbourhood of $\BD$ in $X$.  We will say that a
completion is \emph{good} if $\partial_{X} X_0$ is an effective divisor with simple normal
crossings. From any completion of $X$, one obtains a good one by a finite number of blow ups at infinity (i.e on
$\BD$) see for example \cite{hartshorneAlgebraicGeometry1977} Theorem 3.9 p.391.

Let $X$ be a completion of $X_0$ with the
embedding $\iota_{X}: X_0 \rightarrow X$, we will still denote $\iota_{X}(X_0)$ by $X_0$ and we will denote
by $\OO_{X}(X_0)$ the subring of $\k(X)$ of functions $f \in \k (X)$ which are regular on $X_0$.
By Proposition \ref{PropBoundaryIsCurve}, the boundary $\BD$ is a possibly reducible
connected curve. We denote by $\Div (X)$ the group of divisors of $X$ and by
$\Div_\infty (X)$ the subgroup of divisors of $X$ supported on $\partial_X X_0$. For $\A = \Z, \Q,\R$, we set $\Div
(X)_\A := \Div (X) \otimes \A$ and $\DivInf(X)_\A = \DivInf(X) \otimes \A$. Let $E_1, \cdots, E_m$ be the
irreducible components of $\partial_{X} X_0$ (we will call them the \emph{prime divisors at infinity}). Any element of
$\Div_\infty (X)_\A$ is of the form $D = \sum_i a_i(D) E_i$ with $a_i (D) \in \A$. We will write $\ord_{E_i} (D)$ for
$a_i (D)$ of $D$ at $E_i$.
For a family $(D_j)_{j \in J}$ of elements of $\DivInf (X)$ the coefficients $a_i (D)$ are integers; so, using the
natural order on $\Z$, we
define the supremum $\bigvee_{j \in J} D_j$ and the infimum $\bigwedge_{j \in J} D_j$ by

\begin{equation}
  \bigvee_j D_j = \sum_i \sup(\ord_{E_i}(D_j)) E_i \quad \text { and } \quad \bigwedge_j D_j = \sum_i
  \inf(\ord_{E_i}(D_j)) E_i
\end{equation}

It only exists if each $(\ord_{E_i}(D_j))_{j \in J}$ is bounded respectively from above or from below. If
$\bigwedge_j D_j$ (respectively $\bigvee_j D_j$) is well defined we say that the family $(D_j)$ is
\emph{bounded from below (from above)}. Notice that we only define supremum and infimum for family of divisors with
coefficients in $\Z$.

\section{Morphisms between completions, Weil, Cartier divisors}\label{SectionCompletions}
\medskip

  \paragraph{Some notations} If $\pi: Y \rightarrow X$ is a projective birational morphism between
  smooth projective surfaces  and $D_{X}$ is a divisor on $X$, we will denote by $\pi^* D_X$ the \emph{pull-back}
  of $D_{X}$ under $\pi$ and if $D_{X}$ is effective, then $\pi ' (D_{X})$ will be the strict transform of
  $D_{X}$ under $\pi$. For any projective surface $Z$, if $D_{Z}$ is a divisor on $Z$, we will denote by
  $\OO_{Z} (D_{Z})$ the invertible sheaf on $Z$ associated to $D_{Z}$.

Let $X_1, X_2$ be two completions of $X_0$ with their embeddings $\iota_1, \iota_2$. There exists a unique birational
map $\pi: X_1 \dashrightarrow X_2$ such that the diagram
\begin{equation}
  \label{DiagMorphismCompletions}
  \begin{tikzcd}
    X_1 \ar[r, dashed, "{\pi}"] & X_2 \\
    X_0 \ar[u, "{\iota_1}", hook] \ar[r, equal, "{\id}"] & X_0 \ar[u, "{\iota_2}", hook]
  \end{tikzcd}
\end{equation}
commutes. If $\pi$ is a morphism, we call it a \emph{morphism of completions}. In that case we say that $X_1$ is
\emph{above} $X_2$. By Theorem \ref{ThmCompositionBlowUps}, $ {\pi}^{-1}$ is a composition of
blow-ups; since $\pi$ is an isomorphism over $X_0$, the centers of these blowups are above
$\partial_{X_2} X_0$. Conversely, let $X$ be a completion of $X_0$ with an embedding $\iota: X_0 \hookrightarrow
X$, let $\pi: Y \rightarrow X$ be the blowup of $X$ at a point $p \in \BD$, then $Y$ with the embedding
${\pi}^{-1} \circ \iota: X_0 \rightarrow Y$ is a completion of $X_0$ and $\pi$ is a morphism of completions. For a morphism of
completions $\pi :Y \rightarrow X$, we will write $\Exc (\pi) \subset Y$ for the exceptional locus of $\pi$.

\begin{lemme}\label{LemmaProjectiveSystem}
  The system of completions of $X_0$ is a projective system: For any two completions $X_1, X_2$ of $X_0$ there
  exists a completion $X_3$ above $X_1$ and $X_2$.
\end{lemme}
\begin{proof}
  Let $X_1$, $X_2$ be two completions of $X_0$, let $\pi: X_1 \dashrightarrow X_2$ be the birational map from
  Diagram \ref{DiagMorphismCompletions}. By Theorem \ref{ThmCompositionBlowUps}, there exists a sequence of blow-ups
  $\pi_1 : X_3 \rightarrow X_1$ such that $g = \pi_1 \circ \pi: X_3 \rightarrow X_2$ is regular. It is clear
  that $\pi_1$ is a morphism of completions since by definition $\iota_{X_3} =: \iota_3 = \iota_{1} \circ {\pi_1}^-1$. The
  map $g$ is also a morphism of completion because by construction $g = \pi \circ \pi_1$ and $\iota_2 = \pi \circ
\iota_1$, therefore $\iota_3 = {\pi_1}^{-1} \circ \iota_1 =   {g}^{-1} \circ \pi \circ \iota_1 = {g}^{-1} \circ \iota_2 $
\end{proof}

If $\pi : X_1 \rightarrow X_2$ is a morphism of completions. We can define (see \cite{fultonIntersectionTheory1998},
Section 1.4) the pushforward $\pi_*: \Div(X_1)_\A
\rightarrow \Div(X_2)_\A$ and pullback $\pi^*: \Div(X_2)_\A \rightarrow \Div(X_1)_\A$ of divisors. They define
group homomorphisms
\begin{equation}
  \pi_* : \Div_\infty (X_1)_\A \twoheadrightarrow \Div_\infty (X_2)_\A \quad \text{and} \quad \pi^* :
  \Div_\infty (X_2)_\A \hookrightarrow \Div_\infty (X_1)_\A;
  \label{EqHomomorphisAtInfinity}
\end{equation}
the map $\pi^*$ is often called the \emph{total transform}.
Recall that (\cite{hartshorneAlgebraicGeometry1977} Proposition 3.2 p.386)
\begin{equation}\label{EqId}
  \pi_* \pi^* = \id_{\Div(X_2)_\A}.
\end{equation}

Let $X$ be a completion of $X_0$ and $P \in \k[X_0]$, then $({\iota_X}^{-1})^* (P) \in \k(X)$. We set $(\iota_X)_* :=
({\iota_X}^{-1})^*$ and we denote by
  $\div_{X}(P) := \div ((\iota_{X})_* P)$ the divisor of the rational function $P$ in $X$. In
  particular, if $\pi: Y \rightarrow X$ is a morphism of completions above $X_0$, then by Diagram
  \eqref{DiagMorphismCompletions}, one has $\iota_Y =
  {\pi}^{-1} \circ \iota_X$. Therefore $\div_Y (P) = \div (({\pi}^{-1} \circ \iota_X)_* (P)) =  \div(\pi^* \left(
      (\iota_X)_* (P) \right)) = \pi^* \div_X (P)$. We
  will write $\div_{\infty, X} (P) \in
  \DivInf (X)$ the divisor on $X$ supported at infinity such that
  \[
    \div_X (P) = D + \div_{\infty, X} (P)
  \]
  where $D$ is an effective divisor and no components of its support is in $\BD$.

  \begin{ex}
    Let $X_0 = \A^2 = \spec \k [x,y]$ and let $P = xy$. Take the completion $\P^2$ of $\A^2$ with homogeneous coordinates
    $X,Y,Z$ such that $x= X/Y$ and $y = Y/ Z$. Then,
    \begin{equation}
      \div_{\P^2} (P) = \left\{ X = 0 \right\} + \left\{ Y = 0 \right\} - 2 \left\{ Z = 0 \right\}
      \label{<+label+>}
    \end{equation}
    and $\div_{\infty, \P^2} (P) = -2 \left\{ Z = 0 \right\}$.
    Let $\pi : X \rightarrow \P^2$ be the blow-up of $[1:0:0]$, we can take $W$ to be the subscheme of $\P^2
    \times \P^1$ given by the equation
    \begin{equation}
      UZ = V Y
      \label{<+label+>}
    \end{equation}
    where $U,V$ are the homogeneous coordinates of $\P^1$. Then $\pi$ is the projection onto the first factor. We take
    the affine chart $X = 1$ in $\P^2$ with affine coordinates $y' = Y/ X$ and $z ' = Z / X$. Take the
    chart $U= 1$ with affine coordinate $v$ in $\P^1$, then $W \cap \left\{ X =1 \right\} \times \left\{ U = 1
    \right\}$ is an affine chart of $W$  with coordinates $v, y'$ and we have the relation $z' = v y'$; $y' = 0$ is a
  local equation of the exceptional divisor and $v = 0$ is a local equation of the strict transform of $z' = 0$.
    \begin{equation}
      \pi^* (P) = \pi^* \left(\frac{y'}{(z')^2}\right) = \frac{y'}{v^2 (y')^2} = \frac{1}{v^2 y'}
      \label{ }
    \end{equation}
    Therefore,
    \begin{equation}
      \div_X (P) = \pi ' \left\{ X= 0 \right\} + \pi ' \left\{ Y= 0 \right\} - 2 \pi ' \left\{Z = 0\right\} - \tilde E = \pi^*
      (\div_{\P^2} (P))
      \label{<+label+>}
    \end{equation}
    and
    \begin{equation}
      \div_{\infty, X} (P) = -2 \pi ' \left\{ Z = 0 \right\} - \tilde E
      \label{<+label+>}
    \end{equation}
  \end{ex}

  The system of completions of $X_0$ is a projective system by Lemma \ref{LemmaProjectiveSystem}. Consider the system of
  groups $(\Div_\infty (X))_\A$ for $X$ a completion of $X_0$ with compatibility morphisms
  \begin{equation}
    \pi_* : \DivInf (X)
    \rightarrow \DivInf(Y)
  \end{equation}
  for any morphism of
  completions $\pi : X \rightarrow Y$. This is a projective system of groups.
  Analogously, the same system of groups with $\pi^*$ as compatibility morphisms is an inductive system.
  We define the space of Cartier and Weil divisors at infinity of $X_0$ by
  \begin{equation}
    \Cinf_\A = \varinjlim_{X} \DivInf(X)_\A, \text{ and } \Winf_\A = \varprojlim_X \DivInf(X)_\A.
    \label{<+label+>}
  \end{equation}
Concretely, an element $D \in \Winf_\A$ is a collection $D = (D_{X})$ such that $D_{X}$ is
an element of $\DivInf(X)_\A$ for every completion $X$ of $X_0$ and such that for any morphism of completions $\pi:
X \rightarrow Y$, $\pi_* D_{X} = D_{Y}$; $D_X$ is called the \emph{incarnation} of $D$ in $X$. An element of
$\Cinf_\A$ is the data of a completion $X$ and a divisor $D \in \DivInf(X)$ where two pairs $(X, D)$ and $(X', D')$ are
equivalent if there exists a completion $Z$ above $X$ and $X'$ with morphisms of completion $\pi: Z \rightarrow
X, \pi ': Z \rightarrow X'$ such that $\pi^* D = (\pi ')^*  D'$. We will say that $D \in \Cinf_\A$ is
\emph{defined} over a completion $X$ if $D$ is the equivalence class of $(X, D_X)$ for some $D_X \in
\DivInf(X)_\A$. We have a natural inclusion
\begin{equation}
\phi: \Cinf_\A \hookrightarrow \Winf_\A \end{equation}
defined as follows. If $(X,D) \in \Cinf_\A$, then we need to define the incarnation $\phi(D)_Y$ for any completion
$Y$. First of all, set $\phi(D)_X = D$. Then, for any completion $Y$, by Lemma \ref{LemmaProjectiveSystem}, there
exists a completion $Z$ above $Y$ and $X$; denote by $\pi_Y: Z \rightarrow Y$ and $\pi_Z : Z \rightarrow
X$ the respective morphism of completions. We define $\phi (D)_Y := (\pi_Y)_* \pi_X^* D$. This does not depend on the
choice of $Z$ because of Equation \eqref{EqId}. In the rest of the paper, we will drop the notation $\phi(D)$ and
denote by $D$ the image of $(X,D)$ in $\Winf_\A$. We equip $\Winf_\A$ with the projective limit topology.

In the same manner we define $\CX_\A := \varinjlim
\Div(X)_\A$ and $\WX_\A := \varprojlim \Div (X)_\A$ and we have the following commutative diagram
\begin{equation}
    \begin{tikzcd}
      \Cinf_\A \ar[r, hook] \ar[d]{c} & \Winf_\A \ar[d, hook] \\
      \CX_\A \ar[r, hook] & \WX_\A
    \end{tikzcd}
  \label{<+label+>}
\end{equation}

\begin{rmq}
We have that $\Cinf_\A = \Cinf \otimes \A$ but $\Winf_\A $ is strictly larger than $\Winf \otimes \A$ when $\A = \Q, \R$.
Indeed, let $W_1, \ldots, W_r \in \Winf$, $\lambda_1, \ldots, \lambda_r \in \A$ and set $W := \sum_i \lambda_i W_i$.
Then, for every completion $X$ and for every prime divisor $E$ at infinity in $X$ we have
\begin{equation}
  \ord_E (W_X) = \ord_E (\sum_i \lambda_i W_{i, X}) = \sum_i \lambda_i \ord_E (W_{i, X}) \in \Z \lambda_1 + \cdots
  + \Z \lambda_r
  \label{<+label+>}
\end{equation}
In particular, the group $G (W)$ generated by $\left( \ord_E (W_X) \right)_{(X,E)}$ for all completions $X$ and
all prime divisor $E$ at infinity in $X$ is a finitely generated subgroup of $\R$. Now pick a completion $X_1$ and
consider a sequence of blow ups $\pi_n : X_{n+1} \rightarrow X_n$ starting with $X_1$. Let $E_n$ be the
exceptional divisor of $\pi_n$. We still denote by $E_n$ the strict transform of $E_n$ in every $X_m, m \geq n+1$.
Define the Weil divisor $W \in \Winf_\A$ such that its incarnation in $X_{n+1}$ is $W_{X_{n+1}} = \sum_{k = 1}^{n}
\frac{1}{k} E_k$. Then, $G (W)$ is not finitely generated, therefore $W \not \in \Winf \otimes \A$.
\end{rmq}

An element $D$ of $\Winf_\A$ with $\A = \Z, \Q, \R$  is called \emph{effective} (denoted by $D \geq 0$) if its incarnation in
every completion $X$ is effective; if $D$ belongs to $\Cinf_\R$ this is equivalent to $D_X \geq 0$ for
one completion $X$ where $D$ is defined. If $D_1, D_2 \in \Winf_\A$, we will write $W_1 \geq W_2$ for $W_1 - W_2 \geq
0$.

\section{A canonical basis}\label{SubSecCanonicalBasisOfDivisors}
Let $X$ be a completion of $X_0$, we define $\cD_{X,\infty}$ as follows. Elements of $\cD_{X,\infty}$ are equivalence classes of
prime divisors \emph{exceptional above} $X$ at infinity in completions $\pi_Y: Y \rightarrow X$ above $X$
where two prime divisors $E$ and $E'$ belonging respectively to $Y$ and $Y'$ are equivalent if the birational map
${\pi_{Y'}}^{-1} \circ \pi_Y : Y \dashrightarrow Y'$ induces an isomorphism $ {\pi_{Y '}}^{-1} \circ \pi_Y : E
\rightarrow E'$. We call $\cD_{X,\infty}$ the \emph{set of prime divisors above $X$}. We also define $\cD_\infty (X_0)$ as
the set of equivalence classes of prime divisors at infinity modulo the same equivalence relation. We write
$\A^{\cD_{X,\infty}}$ for the set of functions $\cD_{X,\infty} \rightarrow \A$ and $\A^{(\cD_{X,\infty})}$ for the subset of functions with
finite support.

\begin{prop}
  If $X$ is a completion of $X_0$, then
  \begin{equation}
    \Cinf_\A = \DivInf(X)_\A \oplus \A^{(\cD_{X,\infty})}, \quad \text{ and } \Winf_\A = \DivInf(X)_\A \oplus \A^{\cD_{X,\infty}}.
    \label{<+label+>}
  \end{equation}
  This is a homeomorphism with respect to the product topology of $\A^{\cD_{X,\infty}}$.
\end{prop}

\begin{proof}
  Following \cite{boucksomDegreeGrowthMeromorphic2008} Proposition 1.4, for any $E \in \cD_{X,\infty}$ there exists a minimal
  completion $X_E$ above $X$ such that $E$ is a prime divisor in $X_E$. We denote by $\alpha_E \in \Cinf$ the
  element $\alpha_E := (X_E, E)$. Let $E_1, \ldots, E_r$ be the prime divisor at infinity in $X$, then
  \begin{equation}
    (E_0,
    \ldots, E_r) \cup \left\{ \alpha_E: E \in \cD_{X,\infty} \right\}
  \end{equation}
  is a $\A$-basis of $\Cinf_\A$. In the same fashion we
  obtain the second homeomorphism.
\end{proof}

\begin{rmq}
  Since for any completion $X$, one can find a good completion $Y$ above $X$ and the blow up of a
  good completion is still a good completion, the projective system of good
  completions is cofinal in the projective system of completions, so in the rest of the paper any completion that we take
  will be a good completion.
\end{rmq}

If $f : X_0 \rightarrow X_0$ is a dominant endomorphism, then we can define
\begin{equation}
  f^* : \Cinf_\A \rightarrow \Cinf_\A
\end{equation}
 as follows. Let $D = (X, D_X) \in \Cinf_\A$. Let $Y$ be a completion of $X_0$
such that the lift $F: Y \rightarrow X$ of $f$ is regular, then we define
\begin{equation}
  f^* D := (Y, F^* D_X) \in \Cinf_\A.
  \label{<+label+>}
\end{equation}
This does not depend on the choice of $Y$.
The pushforward operator is more tricky to define as $f$ might not be proper. We need to blow-up above $X_0$ to define
it properly. See \S \ref{SecPicardManin}.

\section{Local version of the canonical basis}\label{SubSecLocalCanonicalBasisOfDivisors}
Let $(X, p)$ be the germ of a smooth surface. For any birational morphism $Y\rightarrow (X,p)$ such that $\pi$ induces
an isomorphism $\pi : Y \setminus \pi^{-1}(p) \rightarrow X \setminus \left\{ p \right\}$ we define $\Div_p
(Y)_\A$ as the set of $\A$-exceptional divisors in $Y$. Similar to the previous sections we define
\begin{equation}
  \Cartier(X,p) = \varinjlim_{Y \rightarrow (X,p)} \Div_p (Y)_\A, \quad \Weil(X,p) = \varprojlim_{Y \rightarrow
  (X,p)} \Div_p (Y)_\A
  \label{<+label+>}
\end{equation}
 We can define the set $\cD_{X,p}$ of prime divisors above $p$ as follows. Elements
 of $\cD_{X,p}$ are equivalence classes of prime divisors $E \subset Y$ for any birational model $Y$ of $(X,p)$.

 Now, if $X$ is a completion of $X_0$ and $p$ is a smooth point of $X$, $(X,p)$ defines a germ of a smooth surface and
 we have a natural embedding $\sD_{X,p} \hookrightarrow \sD_X$. Furthermore if $p \not \in X_0$, then we have natural
 embeddings $\Cartier(X,p)_\A \hookrightarrow \Cinf_\A$ and $\Weil(X,p)_\A \hookrightarrow \Winf_\A$.

\begin{prop}
  If $X$ is a completion of $X_0$, then $\cD_{X, \infty} = \bigsqcup_{p \in \BD} \cD_{X, p} $ and
  \begin{align}
    \Cartier(X, p)_\A &= (\A)^{(\cD_{X,p})} \\
    \Weil(X, p)_\A &= (\A)^{\cD_{X,p}}
    \label{<+label+>}
  \end{align}
\end{prop}

\section{Supremum and infimum of divisors}\label{SubSecSupremumAndInfimum}

Let $(D_i)_{i \in I}$ be a family of elements of $\Winf$ such that for all completions $X$, the
family $(D_{i,X})$ is bounded from below, we define $\bigwedge_{i \in I}
D_i$ with its incarnation in $X$ being
\begin{equation}
  \left(\bigwedge D_i\right)_{X} = \bigwedge_i D_{i, X}.
\end{equation}
 We have an analogous definition for $\bigvee_i D_i$ when each $(D_{i,X})$ is bounded from
above.

\begin{lemme}\label{LemmeMinOfCartierIsCartier}
  If $D,D' \in \Cinf$, then $D \wedge D', D \vee D' \in \Cinf$.
\end{lemme}

\begin{proof} It suffices to show that $D \wedge D' \in \Cinf$ because $D \vee D' = - (-D \wedge -D')$. So
  take $D, D' \in \Cinf$, we have to show that $D \wedge D'$ belongs to $\Cinf$.

  Now, it suffices to show this for $D, D'$ effective, indeed let $X$ be a completion such that $D$ and $D'$ are defined
  over $X$. Then, there exists $D_2 \in \Div_\infty (X)$ such that $D - D_2$ and $D' - D_2$ are
  effective. Indeed, take $D_2$ as the Cartier class determined by $D \wedge D'$ in $X$, Then
  \begin{equation}
    D \wedge
    D' = (D - D_2) \wedge (D' - D_2) + D_2.
  \end{equation}

  Therefore, suppose $D, D'$ are effective. Then $\aa = \OO_{X} (-D) + \OO_{X} (- D')$ is a coherent sheaf of ideals
  such that $\aa_{|X_0} = \OO_{X_0}$, let $\pi: Y \rightarrow X$ be the blow-up along $\aa$. Since $\aa_{|X_0}$ is
  trivial, $\pi$ is an isomorphism over $X_0$, therefore $Y$ is a completion
  of $X_0$ with respect to the embedding $\iota_Y := {\pi}^{-1} \circ \iota_X$ and $\pi$ is a morphism of completions.
  The only thing is that $Y$ might not be smooth at infinity, so consider $\omega :Z \rightarrow Y$ a desingularisation
  of $Y$ which is in particular an isomorphism above $X_0$.
  Then, $\pi^* \aa \cdot \OO_{Y}$ is an invertible sheaf over $Y$ trivial over $X_0$, so there exists a divisor
  $D_{Y} \in \DivInf(Y)$ such that $\pi^* \aa = \OO_{Y} (- D_{Y})$. Now, define $D_Z = \omega^* D_Y$.

  \begin{claim}\label{ClaimPullBackOfIdealSheafIsMinimumCartierClass}
    The Cartier class in $\Cinf$ induced by $D_{Z}$ is $D \wedge D'$.
  \end{claim}
  This is shown for example in the proof of \cite[Lemma 2.6]{boucksomDifferentiabilityVolumesDivisors2009}.
  We postpone the proof of this claim to the end of Chapter \ref{ChapterValuationsAsLinearForms}, page \pageref{ProofClaim}.
\end{proof}

\begin{ex}
  Let $X$ be a completion that contains two prime divisors $E, E'$ at infinity in $X$ such that they
  intersect (transversely) at a point $p$. The sheaf of ideals $\aa = \OO_X(-E) + \OO_X (-E')$ is the ideal of regular
  functions vanishing at $p$. The blow up of $\aa$ is exactly the blow up $\pi: Y \rightarrow X$ at $p$ since
  by universal property of the blow-up $\pi^* \aa = \OO_Y (-\tilde E)$ where $\tilde E$ is the exceptional divisor above
  $p$. If we still denote by $E, E',
  \tilde E$ the elements they define in $\Cinf$, then $E \wedge E' = \tilde E$.
\end{ex}

Let $X$ be a good completion of $X_0$. Let $D_1, D_2 \in \DivInf(X)$. Let $E,F$ be two prime divisors at infinity
that intersect. We say that $(D_1, D_2)$ is \emph{well ordered} at $E \cap F$ if
\begin{equation}
  \ord_E (D_1) < \ord_E (D_2) \Leftrightarrow \ord_F(D_1) < \ord_F (D_2).
  \label{<+label+>}
\end{equation}
We say that $(D_1, D_2)$ is a \emph{well
ordered} pair if it is well ordered at $E \cap F$ for every prime divisor $E,F$ at infinity that intersect.
\begin{lemme}
  The class $D_1 \wedge D_2$ is defined in $X$ if and only if $\left( D_1, D_2 \right)$ is a well
  ordered pair if and only if $D_1 \vee D_2$ is well defined in $X$.
\end{lemme}
\begin{proof}
  Suppose for example that $D_1 \vee D_2$ is defined in $X$ and that $D_1, D_2$ is not a well ordered pair and let
  $E,F$ be two prime divisors at infinity that intersect such
  that at $E \cap F, D_i = \alpha_i E + \beta_i F$ with $\alpha_1 < \alpha_2$ and $\beta_1 > \beta_2$. Then, $D_1 \vee
  D_2 = \alpha_2 E + \beta_1 F$. Let $\tilde E$ be the exceptional divisor above $E \cap F$, then we have
  $\ord_{\tilde E} (D_1 \vee D_2) = \alpha_2 + \beta_1$. But
  \begin{equation}
    \ord_{\tilde E} D_i = \alpha_i + \beta_i < \alpha_2 + \beta_1 = \ord_{\tilde E} (D_1 \vee D_2).
  \end{equation}
  This is a contradiction.
\end{proof}

\begin{dfn}\label{DefSInf}
  Let $\Sinf$ be the semigroup of $\Winf$ of elements $D \in \Winf$ such that there exists a (potentially
  uncountable) family $(D_i)_{i \in I} \subset \Cinf$ such that

  \begin{equation}
    D = \bigvee_I D_i
  \end{equation}

\end{dfn}

\begin{prop}
  \begin{enumerate}
    \item $\Cinf \subset \Sinf$.
    \item For $a,b \geq 0$ and $ D, D' \in \Sinf$, one has $a D + b D' \in \Sinf$.
    \item If $D_i \in \Sinf$ for each $i \in I$ and $(D_i)$ is bounded from above then $\bigvee_{i \in I} D_i \in \Sinf$.
    \item If $D, D' \in \Sinf$, then $D \wedge D' \in \Sinf$.
  \end{enumerate}
\end{prop}

\begin{proof}
  The first assertion is trivial as for $D \in \Cinf, D = \bigvee D$. For Property (2), let $X$ be a
  completion of $X_0$ then $\bigvee_i a D_{i, X} + \bigvee_j b D'_{j, X} = \bigvee_{i,j} (a
  D_i + b D'_j)_X$. For Property (3), if $D_i = \bigvee_j D_{i,j}$, then $\bigvee_i D_i = \bigvee_{(i,j)}
  D_{i,j}$. Finally, the fourth assertion is a corollary of Lemma \ref{LemmeMinOfCartierIsCartier}.
\end{proof}

%
%
%

\section{Picard-Manin Space}\label{SecPicardManin}
\subsection{Cartier and Weil classes over a smooth projective surface}
  Let $X$ be a smooth surface and let $\NS (X)$ be the Néron-Severi group of $X$.
We have a perfect pairing given by the intersection form
  \begin{equation}
    \NS (X)_\R \times \NS (X)_\R \rightarrow \R.
    \label{<+label+>}
  \end{equation}

Recall the Hodge index theorem
\begin{thm}[Hodge Index Theorem, \cite{hartshorneAlgebraicGeometry1977} Theorem 1.9 p.364]\label{ThmHodgeIndex}
  Let $X$ be a smooth projective surface over an algebraically closed field. Let $\alpha \in
  \NS (X) \setminus \left\{ 0 \right\}$ and let $H$ be an ample divisor on $X$. If $\alpha \cdot H =0$, then
  \begin{equation}
    \alpha^2 < 0.
    \label{<+label+>}
  \end{equation}
  In particular, the signature of the quadratic form induced by the intersection form is $\left( 1, \rho -1 \right)$
  where $\rho$ is the rank of $\NS (X)$.
\end{thm}

A class $\alpha \in \NS(X)$ is nef if for all irreducible curve $C \subset X, \alpha
\cdot [C] \geq 0$.
If $\pi : Y \rightarrow X$ is a birational morphisms we have two group homomorphisms
  \begin{equation}
    \pi_* : \NS (Y)_\A \rightarrow \NS(X)_\A, \pi^* : \NS (X)_\A \rightarrow \NS (Y)_\A
    \label{<+label+>}
  \end{equation}
  with the following properties
  \begin{enumerate}
    \item $\pi_* \circ \pi^* = \id_{\NS(X)_\A}$
    \item $\pi^* \alpha \cdot \pi^* \beta = \alpha \cdot \beta$
    \item $\pi^* \alpha \cdot \beta = \alpha \cdot \pi_* \beta$ (Projection Formula)
  \end{enumerate}

Furthermore, if $\pi : Y \rightarrow X$ is the blow up of one point, let $\tilde E$ be the exceptional divisor, then
\begin{equation}
  [\tilde E] ^2 = -1, \text{ and } \NS (Y)_\A = \pi^* \NS(X)_\A \operp \A \cdot [\tilde E]
  \label{EqIntersectionAfterBlowUp}
\end{equation}
Therefore, the system of groups $(\NS (Y))_{Y \rightarrow X}$ with compatibility morphisms $\pi_*$ is a projective
system of groups and $(\NS(Y))_{Y \rightarrow X}$ with compatibility morphisms $\pi^*$ is an inductive system of groups.

\begin{dfn}
  The spaces of Cartier and Weil classes of $X$ are defined as
    \begin{equation}
      \ccNS(X)_\A := \varinjlim_{Y \rightarrow X} \NS (Y)_\A, \quad \wwNS(X)_\A = \varprojlim_{Y \rightarrow
      X} \NS (X)_\A.
      \label{<+label+>}
    \end{equation}
\end{dfn}
We equip $\wwNS_\A$ with the topology of the projective limit.
An element of $\wwNS_\A$ is a family $\alpha = (\alpha_Y)_{Y}$ where $\alpha_Y \in
\NS (Y)_\A$ such that for all $\pi : Y \rightarrow X$, we have
\begin{equation*}
  \pi_* \alpha_Y = \alpha_X.
  \label{<+label+>}
\end{equation*}
We call $\alpha_Y$ the \emph{incarnation} of $\alpha$ in $Y$.

An element of $\ccNS(Y)_\A$ is the data of a birational morphism $Y \rightarrow X$ and a class $\alpha \in
\NS(Y)_\A$ with the following equivalence relation: $(Y, \alpha) \simeq (W, \beta)$ if there exists a smooth projective
surface $Z$ with birational morphisms
\begin{equation*}
  \pi_Y : Z \rightarrow Y, \quad \pi_W : Z \rightarrow W
  \label{<+label+>}
\end{equation*}
such that $\pi_Y^* \alpha = \pi_W^* \beta$. We say that the Cartier class is defined (by $\alpha$) in $Y$. We have
a natural embedding
\begin{equation}
\ccNS (X)_\A \hookrightarrow \wwNS(X)_\A.
  \label{<+label+>}
\end{equation}
We have a pairing
  \begin{equation}
    \wwNS(X)_\R \times \ccNS(X)_\R \rightarrow \R
    \label{<+label+>}
  \end{equation}
  given by the following: let $\alpha \in \wwNS(X)_\R$ and $\beta \in \ccNS(X)_\R$; let $Y$ be a model above $X$ such
  that $\beta$ is defined i.e $\beta = (Y, \beta_Y)$; then
  \begin{equation}
    \alpha \cdot \beta := \alpha_Y \cdot \beta_Y.
    \label{<+label+>}
  \end{equation}
  This is well defined because if $\pi : Z \rightarrow Y$ then
  \begin{equation}
    \alpha_Z \cdot \beta_Z = \alpha_Z \cdot \pi^* \beta_Y = \pi_* \alpha_Z \cdot \beta_Y = \alpha_Y \cdot
    \beta_Y    \label{<+label+>}
  \end{equation}
  by the projection formula.

  An element $\alpha \in \wwNS(X)_\R$ is \emph{nef} if for all completion $X$, $\alpha_X$ is nef.

\begin{prop}[\cite{boucksomDegreeGrowthMeromorphic2008} Proposition 1.7]\label{PropPerfectPairing}
  The intersection pairing
  \begin{equation}
    \wwNS(X)_\R \times \ccNS(X)_\R \rightarrow \R
  \end{equation}
  is a perfect pairing and it induces a homeomorphism $\wwNS_\R \simeq \ccNS_\R^*$ endowed with the
  weak-$*$ topology.
\end{prop}

Similarly as in \S\ref{SubSecCanonicalBasisOfDivisors} we have a more explicit
description of the space of Cartier and Weil classes of $X$. Let $\sD_X$ be the set of equivalence class of prime
divisors above $X$, that is all exceptional divisors appearing in models $Y$ above $X$ modulo equivalence.

\begin{prop}\label{PropDescriptionNeronSeveri}
  Let $X$ be a smooth projective surface, then
  \begin{equation}
    \ccNS(X)_\A = \NS (X)_\A \operp \A^{(\cD_X)}, \quad \wwNS(X)_\A = \NS (X)_\A \operp \A^{\cD_X}.
    \label{<+label+>}
  \end{equation}
  Moreover, the intersection product is negative definite over $\A^{(\cD_X)}$ and $\left\{ \alpha_E : E \in \cD_X
  \right\}$ is an orthonormal basis for the quadratic form $\alpha \in \A^{(\cD_X)} \mapsto - \alpha^2$.
\end{prop}

\begin{proof}
  The decomposition follows from Equation \eqref{EqIntersectionAfterBlowUp}. The fact that the intersection form is
  negative definite follows from the existence of an ample divisor on $X$, the Hodge Index theorem and the projection
  formula. The fact that $\left\{ \alpha_E : E \in
  \cD_X \right\}$ is an orthonormal basis is again a consequence of the projection formula and Equation
  \eqref{EqIntersectionAfterBlowUp}.
\end{proof}

\subsection{Local Cartier and Weil classes}
Let $X$ be a smooth projective surface and let $p \in X$ be a closed point. Then, by Proposition
\ref{PropDescriptionNeronSeveri} we have the canonical embeddings
\begin{equation}
  \Cartier(X,p)_\A \hookrightarrow \ccNS(X)_\A, \quad \Weil(X, p)_\A \hookrightarrow \wwNS(X)_\A
  \label{<+label+>}
\end{equation}

\begin{prop}\label{PropLocalPicardManinSpace}
  The space $\Cartier(X,p)_\R$ is an infinite dimensional $\R$-vector space and the intersection product
  defines a negative definite quadratic form over it. The set $\left\{ \alpha_E : E \in \cD_{X,p} \right\}$ is an
  orthonormal basis for the scalar product $\alpha \mapsto - \alpha^2$. Furthermore, the pairing
  \begin{equation}
    \Weil (X,p)_\R \times \Cartier(X,p)_\R \rightarrow \R
    \label{<+label+>}
  \end{equation}
  is perfect.
\end{prop}

\subsection{Functoriality}
  Let $f : X \dashrightarrow X$ be a dominant rational selfmap of $X$. We define $f^*, f_*$ on the spaces of Cartier and
  Weil classes as
  follows. We first define
  \begin{equation}
    f^* : \ccNS(X)_\R \rightarrow \ccNS(X)_\R.
  \end{equation}
  Let $\beta \in \ccNS(X)_\R$ and let $Y$ be a
  model where $\beta$ is defined. Let $Z$
  be a model of $X$ such that the lift $F: Z \rightarrow Y$ is regular, then we define $f^* \beta$ as the
  Cartier class defined in $Z$ by
  \begin{equation}
    f^* \beta := (Z, F^* \beta_Z)
    \label{<+label+>}
  \end{equation}
  this does not depend on the choice of $Z$. Indeed, if $Z'$ is another model such that $F' : Z ' \rightarrow
  X$ is well defined, then there exists a completion $W$ such that we have the following diagram.
  \begin{equation}
    \begin{tikzcd}
      & & W \ar[ld, "\pi_{Z '}"'] \ar[rd, "\pi_Z"] & & \\
      & Z ' \ar[rd, "\pi '"] \ar[ld, "F '"'] & & Z \ar[ld, "\pi"'] \ar[rd, "F"] \\
      Y & & Y \ar[ll, dashed, "f"'] \ar[rr, dashed, "f"] & & Y
    \end{tikzcd}
  \end{equation}
  Then, the lift of $f: W \dashrightarrow X$ is $F \circ \pi_Z = F' \circ \pi_{Z '}$, hence we get
  \begin{equation}
    \pi_Z^* \circ F^* = \pi_{Z '}^* \circ (F')^*
    \label{<+label+>}
  \end{equation} and the pull back of Cartier classes is well defined.

  Next, we define $f_* : \wwNS(X)_\R \rightarrow \wwNS(X)_\R$. Let $\alpha \in
  \wwNS(X)_\R$ and let $Z,Y$ be models of $X$ such that the lift $F : Z
  \rightarrow Y$ is regular, then the incarnation of $f_* \alpha$ in $Y$ is
  \begin{equation}
    (f_* \alpha)_Y := F_* \alpha_Z.
    \label{<+label+>}
  \end{equation}
  Again, this does not depend on the choice of $Z$ by a similar argument as for the pullback. We have the following
  proposition.
  \begin{prop}[\cite{boucksomDegreeGrowthMeromorphic2008} Section 2]\label{PropOperateurSurPicardManin} We have the
    following properties.
    \begin{itemize}
      \item The operator $f^*$ extends to a continuous operator
        \begin{equation}
          f^* : \wwNS(X)_\R \rightarrow \wwNS(X)_\R.
        \end{equation}
      \item the operator $f_*$ restricts to a continuous operator
        \begin{equation}
          f_* : \ccNS(X)_\R \rightarrow \ccNS(X)_\R
        \end{equation}
      \item Let $\alpha \in \wwNS(X)$, let $X, Y$ be completions of $X_0$ such that the lift $f : X \dashrightarrow Y$ does
        not contract any curves, then
    \begin{equation}
      (f^* \alpha)_X = (f^*\alpha_Y)_X
      \label{<+label+>}
    \end{equation}
    \end{itemize}
  \end{prop}

\begin{rmq}
  For a completion $X$, we can also define the
  restriction of $f^*$ and $f_*$ to $\NS (X)$. We denote them respectively by $f_X^*$ and $(f_X)_*$. They are
  defined by
  \begin{equation}
    \forall \beta \in \NS (X), \quad f_X^* \beta = (f^* \beta)_X, \quad (f_X)_* \beta = (f_* \beta)_X
    \label{<+label+>}
  \end{equation}
\end{rmq}

\subsection{The Picard-Manin space and spectral property of the first dynamical degree}
    Consider a smooth projective surface $X$ and $\w \in \NS (X)$ an ample class. By the Hodge
    index theorem, the intersection form on $\ccNS(X) \times \ccNS(X)$ is negative definite on $ \w^\perp$. If $\alpha
    \in \ccNS(X)$,
    the projection of $\alpha$ on $\w^\perp$ is $\alpha - (\alpha \cdot \w)\w$. Consider the quadratic form on
    $\ccNS(X)$ given by
    \begin{equation}
      \forall \alpha \in \ccNS(X), \left|| \alpha^2 \right|| := (\w \cdot \alpha)^2 - \frac{1}{\w^2} (\alpha - (\alpha
    \cdot \w) \w)^2.
  \end{equation}
  This defines a norm on $\ccNS(X)_\R$ and $\ccNS(X)_\R$ is not complete for this norm. We define the \emph{Picard-Manin space}
  of $X$ as the completion of $\ccNS(X)_\R$ with respect to this norm and we denote it by $L^2(X)$; Had we
    chosen a different ample class, we would have gotten an equivalent norm so the space $L^2(X)$ is independent of the
    choice of $\w$. This
    is a Hilbert space and we have
    \begin{prop}[\cite{boucksomDegreeGrowthMeromorphic2008} Proposition 1.10]
      There is a continuous injection
      \begin{equation}
        L^2(X) \hookrightarrow \wwNS(X)
      \end{equation}
      and the topology on $L^2(X)$ induced by
      $\wwNS(X)$ coincides with its weak topology as a Hilbert space. If $\alpha \in \wwNS(X)$ then $\alpha$
      belongs to $L^2(X)$ if and only if $\inf_X (\alpha_X^2) > - \infty$, in which case $\alpha^2 = \inf_X
    (\alpha_X^2)$. Furthermore, the intersection product $\cdot$ defines a continuous bilinear form on
    $L^2(X)$.
    \end{prop}

    \begin{rmq}
      In particular, any nef class belongs to $L^2(X)$. Recall that $\alpha \in \wwNS(X)_\R$ is nef if for every completion $X,
      \alpha_X$ is nef. The cone theorem (\cite{lazarsfeldPositivityAlgebraicGeometry2004} Theorem 1.4.23) states that
      $\alpha_X$ is a limit of ample classes in $\NS (X)_\R$, therefore $(\alpha_X)^2 \geq 0$ and $\alpha \in L^2(X)$.
    \end{rmq}

    Using the canonical basis of exceptional divisors we can have an explicit description of $L^2(X)$. Let $\alpha \in
    \ccNS(X)$ and let $\alpha_X$ be the incarnation of $\alpha$ in $X$. Then, since $\alpha$ is a Cartier class, we have for
    all but finitely many $E \in \cD_X$ that $\alpha \cdot \alpha_E = 0$ and
    \begin{equation}
      \alpha = \alpha_X + \sum_{E \in \cD_X} (\alpha \cdot \alpha_E) \alpha_E.
      \label{}
    \end{equation}
    Therefore,
    \begin{equation}
      \left|| \alpha^2 \right|| = \left|| {\alpha_X} \right||^2 + \sum_{E \in \cD_X} (\alpha \cdot \alpha_E)^2,
      \label{<+label+>}
    \end{equation}
    and
    \begin{equation}
      \alpha^2 = \alpha_X^2 - \sum_{E \in \cD_X} (\alpha \cdot \alpha_E)^2
      \label{<+label+>}
    \end{equation}

    Therefore, $L^2(X)$ is isomorphic to the Hilbert space
    \begin{equation}
      L^2(X) = \NS(X) \obot \ell^2 (\cD_X).
      \label{<+label+>}
    \end{equation}
    We also have the local version of this statement
    \begin{prop}
      Let $X$ be a completion of $X_0$ and $p \in X$ be a point at infinity. Then,
      \begin{equation}
        L^2(X) \cap \Weil (X, p) = \ell^2(\cD_{X,p})
        \label{<+label+>}
      \end{equation}
      and $\left\{ \alpha_E : E \in \cD_{X,p} \right\}$ is a Hilbert basis of this space.
    \end{prop}

    \begin{prop}[\cite{boucksomDegreeGrowthMeromorphic2008}]\label{PropPushForwardAndPullBackOnL2}
      Let $f: X \dashrightarrow X$ be a dominant rational selfmap. The linear maps
    \begin{equation}
      f^*, f_*: \wwNS(X) \rightarrow \wwNS(X)
      \label{<+label+>}
    \end{equation}
    induce continuous operators
    \begin{equation}
       f^*, f_* : L^2(X) \rightarrow L^2(X)
      \label{<+label+>}
    \end{equation}
    Furthermore, we have the following properties in $L^2(X)$.
    \begin{enumerate}
      \item $(f^n)^* = (f^*)^n$;
      \item $\forall \alpha, \beta \in L^2(X), f^* \alpha \cdot \beta = \alpha \cdot f_* \beta$.
      \item $\forall \alpha \in L^2(X), f^* \alpha \cdot f^* \alpha = e(f) \alpha \cdot \alpha$ where $e(f)$
        is the topological degree of $f$.
    \end{enumerate}
  \end{prop}
  In particular, if $f \in \Bir(X)$ then $f^*$ is an isometry of $L^2(X)$ viewed as an infinite dimensional hyperbolic
  space (see \cite{CantatNormalsubgroupsCremona2013}). Boucksom, Favre and Jonsson showed that $\lambda_1 (f)$ is an
  eigenvalue of $f^*$ and $f_*$ and also the spectral radius of these operators. 

  \begin{thm}\label{thm:eigenvectors-picard-manin}
    Let $f: X \dashrightarrow X$ be a dominant rational map. There exists nef Weil classes $\theta^*, \theta_* \in
    \wNS(X)$ such that 
    \begin{equation}
      f^* \theta^* = \lambda_1(f) \theta^*, \quad f_* \theta_* = \lambda_1 \theta_*.
      \label{<+label+>}
    \end{equation}
    In particular, they belong to $L^2 (X)$ and furthermore $\lambda_1$ is the spectral radius of the operators
    $f_*, f^* : L^2 (X) \rightarrow L^2(X).$
  \end{thm}

  When $\lambda_1^2 > \lambda_2$ there is a spectral gap and the eigenvalue $\lambda_1$ is of multiplicity 1.

      \begin{thm}[\cite{boucksomDegreeGrowthMeromorphic2008}]
        \label{ThmEigenclasses}
        Suppose that $\lambda_1(f)^2 > \lambda_2 (f)$, then there exist nef classes $\theta^*, \theta_* \in L^2(X)$
        unique up to multiplication by a positive constant such that

        \begin{enumerate}
          \item $f^* \theta^* = \lambda_1 \theta^*$.
          \item $f_* \theta_* = \lambda_1 \theta_*$.
          \item For all $\alpha \in L^2(X)$,

            \begin{equation}\label{EqAsymptoticPullBack}
              \frac{1}{\lambda_1^n} (f^n)^* \alpha = (\alpha \cdot \theta_*) \theta^* + \footnote{$A = O_\alpha
              (B)$ means that there exists a constant $C(\alpha) > 0$ such that $A \leq C(\alpha) B$.}{O_\alpha} \left(
              \left(\frac{\lambda_2}{\lambda_1^2}\right)^{n/2} \right) ,
            \end{equation}

            \begin{equation} \label{EqAsymptoticPushForward}
              \frac{1}{\lambda_1^n} (f^n)_* \alpha = (\alpha \cdot \theta^*) \theta_* + O_\alpha \left(
              \left(\frac{\lambda_2}{\lambda_1^2}\right)^{n/2} \right) .
            \end{equation}
        \end{enumerate}

            In particular, for all $\alpha, \beta \in L^2(X),$

            \begin{equation} \label{EqAsymptoticIntersection}
              \lim_n \frac{1}{\lambda_1^n} (f^n)^* \alpha \cdot \beta = (\alpha \cdot \theta_*) (\beta \cdot
            \theta^*).
          \end{equation}
          Furthermore, $\theta^*$ and $\theta_*$ satisfy
          \begin{equation}
            (\theta^*)^2 = 0 , \quad \theta_* \cdot \theta^* > 0
            \label{EqIntersectionDeTheta}
          \end{equation}
      \end{thm}
        We call $\theta^*$ and $\theta_*$ the \emph{eigenclasses} of $f$.

\subsection{For a normal affine surface}
Suppose first that $X_0$ is a smooth affine surface. We define the space of Cartier and Weil classes of $X_0$ as
\begin{equation}
  \cNS_\A := \ccNS(X)_\A, \quad \wNS_\A := \wwNS(X)_\A
  \label{<+label+>}
\end{equation}
for any completion $X$ of $X_0$. This does not depend on the choice of the completion $X$ as the space of Cartier and
Weil classes are birational invariants. Similary, we define the \emph{Picard-Manin} space of $X_0$ as
\begin{equation}
  \L2 := L^2 (X).
  \label{<+label+>}
\end{equation}
We describe now more precisely the structure of these spaces.
For any completion $X$ of $X_0$, we have a natural linear map $\tau: \DivInf (X)_\R \rightarrow
\NS(X)_\R$ which extends to a natural linear map $\tau : \Cinf_\R \rightarrow \ccNS(X)_\R$.
  \begin{prop}\label{PropIntersectionFormNonDegenerateAtInfinityGeneralForm}
  The intersection pairing restricted to $\tau (\DivInf (X)_\R)$ is non degenerate.
\end{prop}
\begin{proof}
  Let $D \in \tau \left( \DivInf (X)_\R \right)$, suppose that $D \cdot D' =0$ for all $D' \in \tau\left( \DivInf
  (X)_\R \right).$ Then, by Theorem
  \ref{ThmGoodmanExistenceAmpleDivisorAtInfinity}, there exists $H \in \DivInf (X)$ ample. We have $D \cdot H = 0$. By
  the Hodge index theorem, if $D$ is not numerically equivalent to zero, then $D^2 <0$ and this is a contradiction.
\end{proof}
  Let $V_X \subset \NS (X)$ be the orthogonal subspace of $\tau
(\DivInf (X)_\R)$. Then,
\begin{equation}
  \NS (X)_\R = V_X \operp \tau (\DivInf (X)_\R).
\end{equation}
For example if $X_0 = \A^2$ and $X =\P^2$, then $V_X = 0$. Now, the space of prime divisors $\sD_X$ splits as $\sD_{X_0}
\bigsqcup \sD_{X,\infty}$ where
\begin{equation}
  \sD_{X_0} = \bigsqcup_{p \in X_0} \sD_{X, p}.
  \label{<+label+>}
\end{equation}
And we have
\begin{align}
  \tau(\Cinf_\R) &= \DivInf(X)_\R \oplus \R^{(\sD_{X,\infty})} \\
  \tau(\Winf_\R) &= \DivInf(X)_\R \oplus \R^{\sD_{X,\infty}}.
  \label{<+label+>}
\end{align}
By Proposition \ref{PropIntersectionFormNonDegenerateAtInfinityGeneralForm}, we have that the bilinear product
\begin{equation}
  \tau (\Winf_\R) \times \tau(\Cinf_\R) \rightarrow \R
  \label{<+label+>}
\end{equation}
is non-degenerate. By Proposition \ref{PropDescriptionNeronSeveri}, we have
  \begin{align}
    \cNS_\A &= \A^{(\sD_{X_0})} \operp V_X \operp \tau(\Cinf_\A) \\
    \wNS_\A &= \A^{\sD_{X_0}} \oplus V_X \oplus \tau(\Winf_\A).
    \label{<+label+>}
  \end{align}
  This yields the following description.

\begin{prop}\label{PropDecompositionPicardManin}
  Let $X_0$ be a smooth affine surface, then we have the following splitting
  \begin{align}
    \wNS_\R &= \tau (\Cinf)^{\perp} \oplus \tau(\Winf_\R) \\
    \L2 &= (\tau (\Cinf)^{\perp} \cap \L2) \oplus \left( \tau(\Winf_\R) \cap \L2 \right).
    \label{<+label+>}
  \end{align}
  And we have
  \begin{equation}
    \left( \tau(\Winf_\R) \cap \L2 \right) = \DivInf(X)_\R \oplus \ell^2 (\sD_{X, \infty}).
    \label{<+label+>}
  \end{equation}
  Furthermore, if $f:X_0 \rightarrow X_0$ is a dominant endomorphism, then $\tau(\Winf_\R)$ is $f^*$-invariant and
  $\tau (\Cinf)^{\perp}$ is $f_*$ invariant.
\end{prop}
\begin{proof}
  Indeed, we have a continuous map
  \begin{equation}
    f^* : \Cinf_\R \rightarrow \Cinf_\R
    \label{<+label+>}
  \end{equation} because $f$ is an endomorphism of $X_0$ which
  yields a continuous map
  \begin{equation}
    f^* : \tau (\Cinf_\R) \rightarrow \tau(\Cinf_\R).
    \label{<+label+>}
  \end{equation} Since $\tau(\Winf_\R)$ is the closure of
  $\tau(\Cinf_\R)$ in $\wwNS (X)_\R$ and $f^* : \wwNS (X)_\R \rightarrow \wwNS(X)_\R$ is continuous, we get the
  $f^*$-invariance. By duality, we get that $\tau (\Cinf_\R)^\perp$ is $f_*$-invariant.
\end{proof}

\begin{cor}\label{cor:eigenclass-at-infinity}
  If $f:X_0 \rightarrow X_0$ is a dominant endomorphism such that $\lambda_1(f)^2 > \lambda_2 (f)$, then the eigenclass
  $\theta^*$ of $f$ from Theorem \ref{ThmEigenclasses} belongs to $\tau (\Winf_\R) \cap \L2$. Furthermore, $\theta^* =
  \tau(W)$ where $W \in \Winf_\R$ is $\geq \leq 0$.
\end{cor}
\begin{proof}
  We know that that there exists an ample effective class $H \in \tau (\Cinf_\R)$, so that in particular $H \geq 0$.
  This implies that $H \cdot \theta_* > 0$ by the Hodge index theorem. Thus we have that
  \begin{equation}
    \lim_n \frac{1}{\lambda_1^n} (f^n)^* H \rightarrow (\theta_* \cdot H) \theta^*
    \label{<+label+>}
  \end{equation} in $\L2$. But by Proposition \ref{PropDecompositionPicardManin} we have that for every $n \geq 0$,
  \begin{equation}
    \frac{1}{\lambda_1^n} (f^n)^* H \in \tau (\Cinf_\R).
    \label{<+label+>}
  \end{equation}
  Thus we get that $\theta^* \in \tau (\Winf_\R) \cap \L2$. We also get that $\theta^* = \tau (W)$ for some $W \geq 0$
  because $f^* : \Winf_\R \rightarrow \Winf_\R$ preserves effectiveness.
\end{proof}

If $X_0$ is a normal affine surface, then its singular locus is a finite number of points because it has codimension
$\geq 2$. Let $X$ be a completion of $X_0$, in particular $X \setminus \Sing(X_0)$ is smooth. A \emph{resolution} of
singularities of $X$ is smooth projective surface with a birational morphism $\pi : Y
\rightarrow X$ which induces an isomorphism $\pi : Y \setminus \pi^{-1} (\Sing(X_0)) \rightarrow X \setminus
\Sing(X_0)$. We define the spaces
\begin{equation}
\cNS_\R, \wNS_\R, \L2
  \label{<+label+>}
\end{equation}
\begin{equation}
\ccNS(Y)_\R, \wwNS(Y)_\R, L^2 (Y)_\R
  \label{<+label+>}
\end{equation} respectively. In
particular, $\pi^{-1} (\Sing (X_0)) \cap \pi^{-1} (\BD) = \emptyset$ thus there are natural linear maps
\begin{equation}
\tau : \Cinf_\R \rightarrow \cNS_\R, \quad \tau : \Winf_\R \rightarrow \wNS_\R
  \label{<+label+>}
\end{equation}
and every result in the smooth setting carries through for normal affine surfaces.
\chapter{Valuations}\label{ChapterValuations}
We introduce the notion of valuations and describe some properties. We will especially focus on valuations over the
ring of power series in two variables $\k [ [ x,y] ] $ as they allow one to describe every valuation over $\k[X_0]$ for
$X_0$ a normal affine surface.

\section{Valuations and completions}

Our general reference for the theory of valuations is \cite{vaquieValuations2000}.
Let $R$ be a commutative $\k$-algebra that is also an integral domain, a \emph{valuation} on $R$ is a function $ v
: R \rightarrow \R \cup \{\infty\}$ such that
\begin{enumerate}[label={(\roman*)}]
  \item $ v (\k^*) = 0$;
  \item For all $P,Q \in R,  v(PQ) =  v(P) +  v(Q)$;
  \item For all $P,Q \in R,  v(P + Q) \geq \min ( v(P),  v(Q))$;
  \item $ v(0) = + \infty$.
\end{enumerate}
If $I$ is an ideal of $R$, we set $ v(I) := \min_{i \in I}  v(i)$. If $S \subset I$ is a
set of generators, then
\begin{equation}
  v (I) = \min_{s \in S} v(s).
\end{equation}
\begin{rmq}
  In \cite{abhyankarValuationsCenteredLocal1956} A \emph{valuation} can take the value $+ \infty$ only at $0$ but we do not
require such a property. Let ${\p_v} = \left\{a \in R : v(a) = \infty \right\}$ then ${\p_v}$ is a prime ideal of $R$
that we call the \emph{kernel} of $ v$.
  If $ v$ is a valuation on $R$, it defines naturally a valuation in the sense of
  \cite{abhyankarValuationsCenteredLocal1956} on the quotient field $R / {\p_v}$. Furthermore $ v$ can be
  naturally extended to a valuation on the ring $R_{\p_v}$ via the formula $ v(p/q) =  v(p) -  v(q)$. In
  particular, if $\p_v = \left\{ {0} \right\}$, then $ v$ defines a valuation over $\Frac R$.
\end{rmq}
Let $X$ be a completion of $X_0$ and let $v$ be a valuation over $B := \OO_X (X_0)$. Let $\p_v$ be the kernel of
$v$. Consider $B_{\p_v}$ the localization of $B$ at $\p_v$. Set
\begin{equation}
\OO_v := \left\{ {x \in B_{\p_v}} : {v(x) \geq 0} \right\}.
\end{equation}
This is a subring of $B_{\p_v}$. If $\p_v = \{0 \}$, then this is the classical \emph{valuation ring} of $v$.

\begin{lemme}
  The subring $\OO_v$ is a local ring, its maximal ideal is

  \begin{equation}
  \m_v := \left\{ {x \in \OO_v} : {v(x) >0}  \right\}
\end{equation}
\end{lemme}

\begin{proof}
  It suffices to show that if $v(x) =0$, then $x$ is invertible in $\OO_v$ but this is obvious since $v({x}^{-1}) = -v(x) =
  0$.
\end{proof}
One defines naturally a valuation $v$ on $C := B / {\p_v}$, let $L$ be the fraction field of $C$ and $\OO$ be the valuation
ring of $L$ with respect to $v$. Then, we have the natural isomorphisms
\begin{equation}
L \simeq B_{\p_v} / {\p_v} \text{ and } \OO_v / {\p_v} \simeq \OO.
\end{equation}
Geometrically, the Zariski closure of ${\p_v}$ inside $X$ defines an irreducible closed
subscheme $Y$ of $X$ and $L$ is isomorphic to the field of rational functions on $Y$.

Two valuations $v_1, v_2$ are \emph{equivalent} if there exists a real number $\lambda >0$ such that
$v_1 = \lambda v_2$.
Let $R, R'$ be two integral domains with a homomorphism of schemes $\phi : \spec R' \rightarrow \spec R$;
it induces a ring homomorphism $\phi^*: R \rightarrow R'$. If
$ v$ is a valuation on $R'$ we define $\phi_*  v$ the \emph{pushforward} by $\phi$ of $ v$ by
\begin{equation}
  \forall P \in R, \phi_* v (P) =  v (\phi^*(P)).
\end{equation}

 Let $X_0$ be an affine surface. Denote by
$\cV$ the set of valuations on $\k[X_0]$. We equip this space with the topology of weak convergence, that is the
coarsest topology such that the evaluation map $v \in \cV \mapsto v(P)$ is continuous for all $P \in \k[X_0]$. If $f$ is an
endomorphism of $X_0$, then $f$ induces a continuous map $f_* : \cV \rightarrow \cV$.

  Via the natural isomorphism $\iota_{X}^* : \OO_X (X_0) \rightarrow \k[X_0]$, every $v \in \cV$ induces a valuation
$(\iota_{X})_*  v$ on $\OO_X (X_0)$, namely
\begin{equation}
\forall P \in \OO_X (X_0), \quad (\iota_{X})_*  v (P) :=  v (\iota_{X}^* P).
\end{equation}
We will denote $(\iota_{X})_*  v$ by $v_{X}$ for every valuation $ v$ on $\k[X_0]$.
\begin{rmq}\label{RmqMemeValuationApresEclatement}
  Take a morphism of completions $\pi: X_1 \rightarrow X_2$ and $ v$ a
  valuation on $\k[X_0]$. Then, $(\iota_{X_2})_*  v = ({\pi}^{-1} \circ \iota_{X_1})_*  v$. In particular $\pi_*
  v_{X_2} = v_{X_1}$.
\end{rmq}

\begin{rmq}
  In the language of Berkovich theory, the set $\cV$ is the Berkovich analytification of $X_0$ over $\k$ where we have
  endowed $\k$ with the trivial valuation (see \cite{berkovichSpectralTheoryAnalytic2012}).
\end{rmq}

\begin{ex}[Divisorial valuations]
  Let $X$ be a completion of $X_0$ and $E$ be a prime divisor of $X$. Let $\ord_E$ be the valuation on
  $\k(X)$ such that for any $f \in \k(X), \ord_E (f)$ is the order of vanishing of $f$ along $E$. Any valuation $v$ on
  $\k[X_0]$ such that $v_X$ is equivalent to $\ord_E$ for some prime divisor $E$ in some completion $X$ is called a
  \emph{divisorial} valuation. In that case $\p_v = \{ 0 \}$ and $v$ extends uniquely to a valuation on $\Frac \k[X_0]$. For
  example if $X_0 = \A^2$ and $X = \P^2$, then let $L_\infty$ be the line at infinity, we have $ \forall P \in \k
  [x,y], \ord_{L_\infty} (P) = - \deg (P)$. If instead we take the completion $P^1 \times \P^1$, decompose $\A^2 = \A^1
  \times \A^1$ and let $x,y$ be the affine coordinate of $\A^2$ each being an affine coordinate of $\A^1$. Let $L_x =
  \left\{ \infty \right\} \times \P^1$ and $L_y = \P^1 \times \left\{ \infty \right\}$, then
  \begin{equation}
    \forall P \in \k[x,y], \ord_{L_x} (P) = - \deg_x (P), \quad \ord_{L_y} (P) = - \deg_y (P)
    \label{<+label+>}
  \end{equation}
  where $\deg_x$ (respectively $\deg_y$) is the degree with respect to the variable $x$ (respectively $y$).
\end{ex}

\begin{ex}[Curve valuations]
  Let $X$ be a completion of $X_0$, let $p \in \BD$ and $C$ be the germ of a (formal) curve at $p$. This means that $C$
  is defined as $\phi = 0$ for $\phi$ in the completion $\hat
  \OO_{X,p}$ of the local ring $\OO_{X,p}$ at $p$. If $\psi \in
  \hat \OO_{X,p}$ is another germ of a formal curve at $p$, we define the intersection number at $p$ by
  \begin{equation}
    \left\{ \phi = 0 \right\} \cdot_p \left\{ \psi = 0 \right\} := \dim_\k \hat \OO_{X,p} / \langle
    \phi, \psi \rangle.
    \label{}
  \end{equation}
  This number is equal to $\infty$ exactly when one of the germs is included in the other or equivalently its defining
  equation divides the other's. We first define a valuation $v_{C,p}$ on $\hat \OO_{X,p}$ by

  \begin{equation}
    v_{C,p} (\psi) = \left\{ {\psi=0} \cdot_p C  \right\}.
  \end{equation}
 Suppose $\phi$ is not divisible by the local equation of any component of $\BD$. For any $P
 \in \OO_X (X_0)$, $P$ can be written as $P = \psi_1^{\alpha_1} \cdots \psi_r^{\alpha_r}$ with $\psi_i \in \hat
 \OO_{X,p}$ irreducible and $\alpha_i \in \Z$. We define
 \begin{equation} v_{C,p} (P) := \sum_i \alpha_i v_{C,p} (\psi_i) \in \R \cup \{\infty\}
 \end{equation}
  Then $v_{C,p}$ is a valuation on $\OO_X (X_0)$. Any valuation on $\k[X_0]$ such that $v_X$ is equivalent to $
  v_{C,p}$ is called a \emph{curve} valuation. If $v$ is a valuation such that $\p_v \neq \left\{ 0 \right\}$, then
  $v$ is a curve valuation (see \cite{favreValuativeTree2004} and Proposition \ref{PropClassificationValuation} below).
  We will make the following distinction, if $C$ is defined by $\phi \in \OO_{X, p}$ we will say that $v_{C,p}$ is an
  \emph{algebraic} curve valuation. Otherwise, we will say that it is a \emph{formal} curve valuation.

  If $\phi$ was divisible by the local equation of a component of $\BD$, then $v_{C,p}$ would not define a valuation
  on $\k[X_0]$ as some regular functions $P \in \k[X_0]$ would have a pole along $C$ and $v (P)$ would be equal to $- \infty$.
\end{ex}

\section{Valuations over $\k [ [ x,y ] ]$}\label{SubSecValuationsOnLocalRing}
We recall some results about
valuations from \cite{favreValuativeTree2004} and \cite{favreEigenvaluations2007}. Let $R$ be a regular complete local
ring with maximal ideal $\m$. We say that a valuation on $R$ is \emph{centered} if $v \geq 0$ and $v_{|\m} > 0$. Here we
set $R := \k [ [ x,y] ]$ for our regular complete local ring. Its maximal ideal is $\m := (x,y)$ we will study the set of centered
valuations on $R$.

\begin{prop}[Proposition 2.10 \cite{favreValuativeTree2004}, \cite{spivakovskyValuationsFunctionFields1990}]
  Any valuation on $\k[x,y]$ centered at the origin extends uniquely to a centered valuation on $R$ as follows. Let
  $\phi \in R$ and let $\phi_n$ be the polynomial of degree $n$ such that $\phi = \lim \phi_n$. Then,
  \begin{equation}
    v (\phi) = \lim_{n \rightarrow \infty} \min (v (\phi_n), n).
    \label{<+label+>}
  \end{equation}
\end{prop}

\begin{cor}
  Let $R'$ be a regular local ring of dimension 2 over $\k$, then the  $\m_{R'}$-adic completion $\hat{R'}$ of $R'$ is
  isomorphic to $R$. Any centered valuation on $R'$ extends uniquely to a centered valuation on $\hat{R'}$.
\end{cor}

\begin{proof}
  Let $(x,y)$ be a regular sequence of $R'$, that is $\m_{R'} = (x,y)$.  It
  exists because $R'$ is a regular local ring of dimension 2. Then, by Theorem \ref{ThmCompletionLocalRing}, $\hat R'$
  is isomorphic to $\k [ [ x,y ] ]$. Let $v$
  be a centered valuation on $R'$. We have that $\k [x,y] \subset R'$, so $v$ induces a valuation on $\k [x,y]$ that is
  centered at the origin and we can apply the previous proposition to conclude.
\end{proof}

  Let $p$ be a regular point on a surface $X$ and let $R = \hat{\OO_{X, p}}$ we define 4 types of valuations over $R$.
  \subsection{Divisorial valuations} A valuation $v$ over $R$ is \emph{divisorial} if there exists a sequence of
      blow-up $\pi : (Y, \Exc(\pi)) \rightarrow (X, x)$ such that $v$ is equivalent to $\pi_* \ord_E$ for some
      prime divisor $E \subset \Exc(\pi)$.

      \subsection{Quasimonomial valuations}
      Let $\pi : (Y, \Exc(\pi)) \rightarrow (X, x)$ be a sequence of blow-ups and let $q \in \Exc(\pi)$. A
      \emph{monomial} valuation at $q$ is a valuation $v$ on $\hat{O_{Y, q}}$ such that there exists $s, t > 0$,
      \begin{equation}
        v \left(\sum_{i,j} a_{ij} x^i y^j\right) = \min \left\{ si + tj : a_{ij} \neq 0 \right\}
        \label{<+label+>}
      \end{equation}
      for some local coordinates at $q$. We write $v = v_{s,t}$.

      A valuation over $\hat{\OO_{X, p}}$ is called \emph{quasimonomial} if there exists a sequence of blow-ups $\pi :
      \left( Y, \Exc(\pi) \right) \rightarrow \left( X, p \right)$ such that $v = \pi_* v_{s,t}$. Quasimonomial
      valuations split into two categories: if $s/t \in \Q$, one can show actually that $v$ is divisorial. Otherwise
      $s/t \in \R \setminus \Q$, $v$ is not divisorial and we say that it is \emph{irrational}.

      \subsection{Curve valuations}
      Let $\phi \in \hat{\m_p}$ be irreducible, we define $v_\phi$ by
      \begin{equation}
        \forall \psi \in \hat{\OO_{X, p}}, \quad v_\phi (\psi) = \frac{\left\{ \phi = 0 \right\} \cdot \left\{ \psi = 0
        \right\}}{m(\phi)}
        \label{<+label+>}
      \end{equation}
      where $m(\phi)$ is the order of vanishing of $\phi$ at the origin. A \emph{curve} valuation is a valuation
      equivalent to $v_\phi$ for some $\phi \in \hat{\m_p}$ irreducible.

      \subsection{Infinitely singular valuations}
      These are all the remaining valuations. They have a nice description in term of Puiseux series (see
      \cite{favreValuativeTree2004} Section 4.1 for more details). Briefly, to any valuation $v$ of $\k [ [ x,y ] ]$, one can
      associated a generalised power series
      \begin{equation}
        \hat \phi = \sum_j a_j x^{\beta_j}
        \label{<+label+>}
      \end{equation}
      with $a_j \in \k$ and $\beta_j \in \Q$. The \emph{infinitely singular} valuations are exactly the
      valuations such that $\lim_j \beta_j \neq + \infty$.

\begin{prop}[\cite{favreValuativeTree2004}]\label{PropClassificationValuation}
  There are four types of centered valuations on $R$: divisorial, irrational, curve valuations and infinitely singular
  valuations. The only type of valuation $v$ such that $\p_v = \left\{ v = + \infty \right\} \neq 0$ are curve valuations
\end{prop}

\begin{rmq}\label{RmqValuationDeCourbeKrull}
  Instead of looking at valuations over $R$ with values in $\R$, we can look at valuations with values in a totally ordered
  abelian group $\Gamma$, these are called \emph{Krull valuations} (see \cite{favreValuativeTree2004}, section 1.3) and
  they have the advantage to always extend to $\Frac R$. We
  can make any curve valuation into a Krull valuation by the following procedure (see \cite{favreValuativeTree2004},
  section 1.5.5): Let $\phi \in \m$ and consider the curve valuation $v_\phi$. Let $\Gamma = \Z \times \Q$ with the
  lexicographical order, we define $\hat v_\phi: R \rightarrow \Gamma$ as follows. For any $\psi \in R$, there exists
  an integer $k \in \N$ such that
  \begin{equation}
    \psi = \phi^k \hat \psi
    \label{<+label+>}
  \end{equation}
  where $\hat \psi$ is not divisible by $\phi$. Set
  \begin{equation}
    \hat v (\psi) := (k, v_{\phi} (\hat \psi))
    \label{<+label+>}
  \end{equation}
  Notice that $v_\phi (\psi) = \infty \Leftrightarrow p_1 (\hat v_\phi (\psi)) >0$ where $p_1 : \Gamma \rightarrow \Z$
  is the projection to the first coordinate and if $v_\phi (\psi) < + \infty$, then $\hat v_\phi (\psi) = (0, v_\phi
  (\psi))$.
  We will not need Krull valuations in the rest of the text. But this argument comes in handy for the proof of
  Proposition \ref{PropFiniteNumberOfPreimages} so we state it here.
\end{rmq}

\section{The center of a valuation}

Let $X$ be a completion of $X_0$ and let $v$ be a valuation on $\OO_X (X_0)$.
A \emph{center} of $v$ on $X$ is a scheme-theoretic point $p \in X$ such that $\OO_v$ dominates the local ring
$\OO_{X,p}$ (i.e $\OO_{X,p} \subset \OO_v$ and $\m_p \subset \m_v$). If such a
$p$ exists then $v$ induces a \emph{centered} valuation on $\OO_{X, p}$ (cf \ref{SubSecValuationsOnLocalRing}) and in
particular for any open affine subset $U \subset X$ that contains $p$, $v$ induces a valuation on $\OO_X(U)$ via the
inclusion $\OO_{X}(U)  \subset \OO_{X,p}$.

\begin{lemme}\label{LemmeCentreValuation}
  The center of $v$ on $X$ always exists and is unique.
\end{lemme}

\begin{proof}
  Let $\OO_v$ be the subring of $\k(X)$ where $v$ is $\geq 0$; it contains $\k^*$. Let $L = B_{\p_v} / {\p_v}$ and $\OO =
  \OO_v / {\p_v}$. If $p$ is a center of $v$ on $X$ then we have
  the following commutative diagram of ring homomorphism

  \begin{equation}
    \begin{tikzcd}
      \OO_{X,p} \ar[r, hook] & \OO_v \ar[r,twoheadrightarrow] & \OO \ar[r, hook] & L & \ar[l,
      twoheadrightarrow] B_{\p_v}
    \end{tikzcd};
    \label{<+label+>}
  \end{equation}
  inducing the following commutative diagram of scheme morphisms

  \begin{equation}
    \begin{tikzcd}
      \spec L \ar[d] \ar[rrr] & & & X \ar[d] \\
      \spec \OO \ar[r] & \spec \OO_v \ar[r] & \spec \OO_{X,p} \ar[ru] \ar[r] &  \spec \k
    \end{tikzcd}
  \end{equation}

  Since $X$ is proper over $\k$ (it's a projective variety), the valuative criterion of properness
  (\cite{hartshorneAlgebraicGeometry1977})
  shows that if the center exists, then it is unique.
  For the existence, Let $x \in X$
  be the image of the maximal ideal of $\OO$, then $x$ is the center of $v$ on $X$. Indeed, the image of
  $\spec L$ is the prime ideal $\p_v$ of $\OO_X (X_0)$ and $x$ belongs to its closure, therefore
  $\OO_{X,x} \subset B_{\p_v}$ and the morphism of local rings $\OO_{X,x} \rightarrow \OO$ shows that
  $\OO_v$ dominates $\OO_{X,x}$.

\end{proof}

 The \emph{center} of $v$ on $X$ is the center of $v_{X}$ we will denote it by $c_{X} (v)$.

\begin{ex}
  Let $v$ be a divisorial valuation over $\k[X_0]$ and let $X$ be a completion of $X_0$ such that
  $v_X \simeq \ord_E$ for some prime divisor $E$ of $X$, then the center of $v$ on $X$ is the generic point $x_E$ of
  $E$. Indeed, the ring of regular function at the generic point of $E$ is a discrete valuation ring since $E$ is of
  codimension 1. In that
  case, we will identify the center with its closure and say that the center of $v$ on $X$ is the
  prime divisor $E$.
  In fact a valuation is divisorial if and only if its center on some completion of $X_0$ is
  a prime divisor because if $c_X (v) = x_E$, then $v$ and $\ord_E$ defines the same valuation ring which is a
  discrete valuation ring, therefore they are equivalent.
\end{ex}

\begin{ex}
  If $v$ is a curve valuation and $X$ is a completion of $X_0$ such that $(\iota_X)_* v \simeq
  v_{C,p}$, then the center of $v$ on $X$ is the closed point $p$.
\end{ex}

A valuation over $\k[X_0]$ is \emph{centered at infinity} if there exists a completion $X$ such that $c_X (v) \not \in
X_0$.

\begin{cor}
  Let $X_0$ be a normal affine surface, there are exactly four types of valuations centered at infinity over
  $\k[X_0]$: divisorial valuations, irrational valuations, curve valuations and infinitely singular valuations. If $v$ is a
  valuation such that $\p_v \neq \left\{ 0 \right\}$, then $v$ is a curve valuation.
\end{cor}

\begin{proof}
  let $v$ be a valuation over $\k[X_0]$ and let $c_X (v)$ be its center on some completion $X$. If $c_X (v)$ is a
  prime divisor at infinity then $v$ is divisorial. Otherwise, $c_X (v)$ is a regular point at infinity and $v$
  induces a centered valuation over $\hat{\OO_{X, p}}$. The result follows from the classification of centered
  valuations over $\k [ [ x,y ] ] $ from Proposition \ref{PropClassificationValuation}.
\end{proof}

\begin{rmq}
  One might expect that if $p \in X \setminus X_0$ is a closed point at infinity then every centered valuation on
  $\OO_{X,p}$ defines a valuation over $X_0$ but that is not the case. Namely, for every irreducible component $E$ of the
  boundary $\BD$ such that $p \in E$, the curve valuations $v_{E,p}$ does not define a valuation over $X_0$ because
  there exists $P \in \OO(X_0)$ such that $\ord_E(P) < 0$ and we would get $v_{E,p} (P) = - \infty$ which is not
  allowed.
\end{rmq}

\begin{dfn}
  \begin{itemize}
    \item Let $X$ be a good completion
  of $X_0$ and $p \in
  \BD$ a point at infinity.  Following
  \cite{favreValuativeTree2004}, we say that $p$ is a \emph{free} point if it
  belongs to a unique prime
  divisor at infinity and we say that it is a \emph{satellite point} otherwise, i.e it is the intersection
  point of two prime divisors at infinity.
  \item Let $v$ be a valuation over $\k[X_0]$ centered at infinity. Let $p_1 = c_X (v)$ be its center on $X$ and $X_1 :=
  X$. We define the
  following sequence: If $p_n$ is a prime divisor,
then the sequence stops, else $p_n$ is a closed point of $X_n$ and we define $X_{n+1}$ as the blow up
of $p_n$, then define $p_{n+1} := c_{X_{n+1}} (v)$. This is the \emph{sequence of centers} of $v$ with respect to
$X$.
  \end{itemize}
\end{dfn}

We adopt the following convention: When we write "let $p \in E$ be a free point (at infinity)" this means that $E$
is the unique prime divisor at infinity on which $p$ lies. If we write "let $p = E \cap F$ be a satellite point",
this means that $E$ and $F$ are the two prime divisors at infinity such that $p = E \cap F$ (Recall that we only work
with good completions).

\begin{prop}[\cite{favreValuativeTree2004}, Section 6.2 ] \label{PropSequenceOfInfinitelyNearPoints} Let $v$ be a valuation centered at infinity.
  Let $X$ be a completion of $X_0$ and $(p_n)$ the sequence of centers (above $X$) associated to $v$.
  Then,
  \begin{enumerate}
    \item $v$ is divisorial if and only if the sequence $(p_n)$ is finite.
    \item If $v$ is irrational, then $(p_n)$ contains finitely many free points.
    \item if $v$ is a curve valuation, then $(p_n)$ contains finitely many satellite points.
    \item If $v$ is infinitely singular, then $(p_n)$ contains infinitely many free points.
  \end{enumerate}
\end{prop}

\begin{proof}
  Assertion 1 is clear since the sequence $(p_n)$ stops if and only if $p_n$ is a prime divisor at
  infinity. Assertion 2 and 4 follows from \cite{favreValuativeTree2004} Theorem 6.10 and Assertion 3
  follows from \cite{favreValuativeTree2004} Proposition 6.12.
\end{proof}

\begin{prop}\label{prop:irrational-centers-not-same-divisor}
  Let $X$ be a completion of $X_0$ and let $v$ be an irrational valuation such that $c_X(v) = E \cap F$ where $E,F$ are
  two prime divisors at infinity and let $(p_n)$ be the sequence of centers of $v$, then there exists an $n_0$ such that
  $p_{n_0}$ does not belong to the strict stransform of $E$.
\end{prop} 
\begin{proof}
  If that was not the case then let $(X_n, p_n)$ be the sequence of centers of $v$, we still denote by $E$ the strict
  transform of $E$ in every $p_n$, then we actually have that $p_n = E \cap F_n$ where $F_0 = F$ and $F_n$ is the
  exceptional divisor in $X_n$, the birational morphism $\pi_n : X_n \rightarrow X$ in local coordinates at $p_n$ is of
  the form 
  \begin{equation}
    \pi (z,w) = \left( zw^n \phi, w\psi \right) = (x,y)
    \label{<+label+>}
  \end{equation}
  where $(z,w)$ is associated to $(E, F_n)$ and $(x,y)$ is associated to $(E,F)$ and $\phi, psi$ are invertible. Now, we
  have 
  \begin{equation}
    v(x) = v(z) + n v(w) \geq n v(y).
    \label{<+label+>}
  \end{equation}
  Since $v(y) > 0$ we get that $v(x) = + \infty$ and this is a contradiction.
\end{proof}

  \section{Image of a valuation via an endomorphism}
  Let $f: X_0 \rightarrow X_0$ be a endomorphism of $X_0$, it induces a
  map $f_*$ on the space of valuation $f_*: \cV \rightarrow \cV$ via the formula
  \begin{equation}
    \forall P \in \k[X_0], \forall v \in \cV, \quad f_* v (P) = v (f^* P).
  \end{equation}
  We will denote by $f_\bullet$ the induced map $f_\bullet : \hat \cV \rightarrow \hat \cV$.

  \begin{prop}[Proposition 2.4 of \cite{favreEigenvaluations2007}]\label{PropPreservationOfType}
    Suppose that $f$ is dominant, the map $f_*$ preserves the sets of divisorial, of irrational and of infinitely singular
    valuations. If $v_C$ is a
    curve valuation such that $f$ does not contract $C$, then $f_* v_C$ is a curve valuation. If $f$
    contracts $C$, then $f_* v_C$ is a divisorial valuation.

  \end{prop}
  We will use this proposition in the following context. Let $X, Y$ be
  two completions of $X_0$ such that the lift $F: X \rightarrow Y$ of $f$ is regular. For any
  point $p \in X \setminus X_0$, we have a map $F_*: \cV_X (p) \rightarrow
  \cV_Y (F(p))$ that preserves the type of the valuations. Suppose that $q =: F(p) \in Y \setminus X_0$. The only curve
  that might be contracted by $F$ to
  $q$ are the divisors at infinity because $F$ induces an endomorphism of $X_0$; but the curve valuation that they
  define do not define valuations on $\k[X_0]$.

  Let us explain the pushforward of a divisorial valuation following \S 2.3 of \cite{favreEigenvaluations2007}. Let $X,
  Y$ be completions of $X_0$ and $E \subset \BD$ an irreducible component such that $f : X \dashrightarrow Y$ does not
  contract $E$. Let $E ' \subset Y$ be the image of $E$ by $f$ and suppose that $E ' \subset \partial_Y X_0$. Then,
  \begin{equation}
    f_* \ord_E = k \ord_{E'}
    \label{<+label+>}
  \end{equation}
  where $k = \ord_E (f_{XY}^* E')$.

  \begin{prop}\label{PropFiniteNumberOfPreimages}
    Let $f : X_0 \rightarrow X_0$ be a dominant endomorphism of topological degree $\lambda_2$. Then, every
    valuation $v$ on $\k[X_0]$ has at most $\lambda_2$
    preimages under $f_*$.
  \end{prop}

  \begin{proof}
    Suppose first that the valuation $v$ takes the value $+ \infty$ only for $0$. Therefore, it extends to a valuation
    on $K = \Frac \k[X_0]$.
    The extension $f^* K \hookrightarrow K$ is a finite extension of degree $\lambda_2$. The valuation $v$ induces a
    valuation on $f^* K$ and every valuation $w$ such that $f_* w = v$ is an extension of $v_{|f^* K}$ to $K$. By
    \cite{zariskiCommutativeAlgebra1960} Theorem 19 p.55, there are at most $\lambda_2$ extension of $v_{|f^* K}$.

    If now $\p_v = \left\{ v = +\infty \right\} \neq 0$, then we know that $v$ is a curve
  valuation. By Remark \ref{RmqValuationDeCourbeKrull}, $v$ can be made into a Krull valuation $\hat v$. Since $\hat
v$ is a Krull valuation, it extends to a Krull valuation over $K$ and $f_* v$ extends to a Krull valuation over $f^*K$.
The same argument as above still works as \cite{zariskiCommutativeAlgebra1960} deals with Krull valuations.
\end{proof}

\section{Tamely ramified endomorphisms}\label{SubSecTamelyRamified}
Let $K \hookrightarrow L$ be a field extension, let $v$ be a valuation over $K$ and let $w$ be a valuation over $L$ such
that $w_{|K} = v$. If $\Gamma_v$ and $\Gamma_w$ are the value group of $v$ and $w$ respectively, we have $\Gamma_v
\subset \Gamma_w$ and we define the \emph{ramification index} $e(w|v) = [\Gamma_v : \Gamma_w]$.

If $\OO_v$ is the valuation ring of $v$ and $\OO_w$ the valuation ring of $w$. Let $\kappa_v$ be the residue field of
$v$, then we have a field extension $\kappa_v \hookrightarrow \kappa_w$, the \emph{inertia degree} is defined as
$f(w | v) := [\kappa_w : \kappa_v]$. If $L / K$ is finite of degree $n$, then
\begin{equation}
  e(w | v) f(w | v) \leq n.
  \label{<+label+>}
\end{equation}

Now consider a dominant endomorphism $f : X_0 \rightarrow X_0$, let $L = \k(X_0)$ and $K = f^* L$. Following
\cite{cutkoskyMonomialResolutionsMorphisms2000}, we say that $f$ is
\emph{tamely ramified} if $f$ is separable and for every divisorial valuation $v$ of $X_0, e(v | f_* v)$ is not
divisible by $\car \k$ and the residue field extension $\kappa_v / \kappa_{f_* v}$ is separable.

In particular, if $\car \k = 0$ or $f$ is an automorphism, $f$ is automatically tamely ramified. The composition of
tamely ramified endomorphisms is again tamely ramified.

\begin{lemme}\label{lemme:tamely-ramified-separable-over-curves}
  If $f:X_0 \rightarrow X_0$ is a tamely ramified endomorphism and $f_* \ord_E = m \ord_{E'}$ then $\car \k$ does not
  divide $m$ and the induced map $f : E \rightarrow E'$ is separable.
\end{lemme}
\begin{proof}
  The integer $m$ is exactly the ramification index $e(\ord_E | f_* \ord_E) = m$ and the residue field of $\ord_E$
  (resp. $\ord_{E'}$) is exactly the function field of $E$ (resp. $E'$). The lemma follows by the definition of tamely
  ramified.
\end{proof}

\begin{rmq}\label{rmq:tamely-ramified-at-infinity}
  Since we will only study the action of $f$ on valuations centered at infinity, it could make sense to define a notion
  of \emph{tamely ramified at infinity} where we test the conditions of tamely ramifiedness only on divisorial
  valuations centered at infinity. 
\end{rmq}

\section{Berkovich spaces over a trivially valued field}\label{sec:berkovich-spaces}
Although we won't use this point of view here. Everything appearing in this memoir could be done through the theory of
Berkovich spaces, see \cite{berkovichSpectralTheoryAnalytic2012}. Namely, if $X_0 = \spec A$ is an affine surface over a
field $\K$, then we can equip $\K$ with the trivial valuation and then the space of valuations over $X_0$ is exactly the
Berkovich analytification $X_0^{\an}$ with respect to the trivial valuation. Indeed, Berkovich uses seminorms but one
can check that $v$ is a valuation if and only if $e^{-v}$ is a seminorm. The map $v \mapsto \p_v$ is called the
\emph{contraction map} in the language of Berkovich and the map $r_X$ that sends a valuation $v$ to its center
$c_X(v)$ on a completion $X$ is called the \emph{reduction map}. In particular, $\cV_\infty$ is an open subset as it is
equal to $r_X^{-1} (X \setminus X_0)$ and $r_X$ is anti-continuous. 

\chapter{Tree structure on the space of valuations}\label{ChapterValuationTree}
We show that the space of valuation centered at infinity of a normal affine surface $X_0$ has a local tree structure.
Namely, the set of (normalised) valuations centered at a closed point is isomorphic to the \emph{valuative tree}
constructed in \cite{favreValuativeTree2004}. We recall some of its properties.

\section{Trees}
For this section, we refer to \cite{favreValuativeTree2004} Section 3.1.
Let $(\cT, \leq)$ be a partially ordered set, a subset $\cS \subset \cT$ is \emph{full} if for every $\sigma, \sigma '
\in \cS, \tau \in \cT, \sigma \leq \tau \leq \sigma ' \Rightarrow \tau \in \cS$.
\begin{dfn}
  Let $\Lambda = \N, \Q, \R$. An \emph{interval} in $\Lambda$ is a subset $I \subset \Lambda$ such that for all $x,y, z
  \in \Lambda$, if $x \leq y \leq z$ and $x,z \in I$, then $y \in I$. If $(\cT, \leq)$ be a partially ordered set, then
  $(\cT, \leq)$ is a rooted \emph{$\Lambda$-tree} if
  \begin{enumerate}
    \item $\cT$ has a unique minimal element $\tau_0$ called the \emph{root} of $\cT$.
    \item If $\tau \in \cT$, the set $\left\{ \sigma \in \cT : \sigma \leq \tau \right\}$ is
      \footnote{isomorphic here means that there exists an order preserving bijection.}{isomorphic} to an interval
      in $\Lambda$.
    \item Every full, totally ordered subset of $\cT$ is isomorphic to an interval in $\Lambda$.
    \item Every non-empty subset $\cS$ of $\cT$ admits an infimum.
  \end{enumerate}

  \begin{rmq}\label{rmq:fourth-condition}
    The fourth condition in this definition is not present in \cite{favreValuativeTree2004} but it is necessary for the
    theory to be well behaved as being able to take the minimum of two valuations is crucial. It does
    not follow from the first three conditions, see \cite{novacoskiValuationsCenteredTwodimensional}.
  \end{rmq}

  A \emph{parametrised}-$\Lambda$ tree is a rooted $\Lambda$-tree $\cT$ with a map $\alpha : \cT \rightarrow
  \Lambda \cup \left\{ \infty \right\}$ such that the restriction of $\alpha$ to any full totally ordered subset of
  $\cT$ induces a bijection with an interval in $\Lambda$. The map $\alpha$ is called the \emph{parametrisation}.

  A rooted $\R$-tree is called \emph{complete} if every increasing sequence has an upper bound.
\end{dfn}

A \emph{subtree} $\cS$ of a $\Lambda$-tree $\cT$ is a subset such that $(\cS, \leq_{|\cS})$ is a $\Lambda$-tree.
An \emph{inclusion} of trees is an order preserving injection $\iota: \cS \rightarrow \cT$. Where $\cS$ is a
$\Lambda$-tree, and $\cT$ is a $\Lambda '$-tree, we do not require $\Lambda = \Lambda '$. For example $\N \hookrightarrow
\R$ is an inclusion of trees. In particular, if $\Lambda = \Lambda '$, then $\iota(\cS)$ is a subtree of $\cT$.

If $\cT$ is an $\R$-tree and $\tau_1, \tau_2 \in \cT$, then the \emph{minimum}
$\tau_1 \wedge \tau_2  \in \cT$ exists by completeness of $\R$. We define the set
\begin{equation}
  [\tau_1, \tau_2] := \left\{ \tau \in \cT : \tau_1 \wedge \tau_2 \leq \tau \leq \tau_1 \text{ or } \tau_1 \wedge \tau_2
  \leq \tau \leq \tau_2 \right\}
  \label{<+label+>}
\end{equation}
and we call it a \emph{segment}. The segments $[\tau_1, \tau_2), (\tau_1, \tau_2]$ and $(\tau_1, \tau_2)$ are defined
similarly.  A \emph{finite} subtree of $\cT$ is a subtree that consists of a finite union of segments in $\cT$.

If $\cT$ is an $\R$-tree, a \emph{tangent} vector $\overrightarrow{v}$ at $\tau \in \cT$ is an equivalence class of
elements $\tau ' \in \cT \setminus \left\{ \tau \right\}$ where
\begin{equation}
  \tau' \sim \tau '' \Leftrightarrow (\tau, \tau '] \cap (\tau, \tau ''] \neq \emptyset.
\end{equation}
 We define the \emph{weak} topology on $\cT$ by
the topology generated by the sets
\begin{equation}
  U(\overrightarrow v) := \left\{ \tau ' \in \cT \setminus \left\{ \tau \right\}: \tau ' \text{ represents }
  \overrightarrow v \right\}.
  \label{<+label+>}
\end{equation}

\begin{thm}[\cite{favreValuativeTree2004} Proposition 3.12]
  We have the following
  \begin{itemize}
    \item Every rooted $\R$-tree $\cT$ admits a completion $\overline \cT$ that is a complete rooted $\R$-tree.
    \item Every rooted $\Q$-tree $\cT_\Q$ admits a completion $\cT_\R$ into a complete rooted $\R$-tree, i.e there exists an
      order preserving injection $\iota : \cT_\Q \hookrightarrow \cT_\R$ such that
      \begin{enumerate}
        \item If $\tau_0$ is the root of $\cT_\Q, \iota (\tau_0)$ is the root of $\cT_\R$.
        \item $\iota (\cT_\Q)$ is weakly dense in $\cT_\R$
        \item $\cT_\R$ is minimal for this property.
      \end{enumerate}
    \item If $\alpha_\Q : \cT_\Q \rightarrow \Q_+$ is a parametrisation of $\cT_\Q$, then there exists a unique
      parametrisation $\alpha_\R$ of $\cT_\R$ such that $\alpha_\Q = \alpha_\R \circ \iota$.
  \end{itemize}
\end{thm}

\section{The local tree structure of the space of valuations}

We denote by $\cV_0$ the set of centered valuations on $R$ where $R = \k [ [ x,y ] ]$. Define the
\emph{multiplicity valuation} $v_\m$ by $v_\m (\phi) = \max \left\{ n \geq 0 : \phi \in \m^n
\right\}$. We will sometimes write $m (\phi)$ instead of $v_\m (\phi)$.
Let $\cV_\m \subset \cV_0$ be the set of centered valuations on $R$ such that $v(\m) =1 $ and consider the following order
relation on $\cV_\m$:
\begin{equation}
  v \leq w \iff \forall \phi \in R, v(\phi) \leq w(\phi).
\end{equation}
 With this order relation
$\cV$ becomes a complete rooted $\R$-tree called the \emph{valuative tree} (\cite{favreValuativeTree2004} Theorem 3.14)
rooted in $v_\m$.
 The ends of $\cV_\m$ consist of the curve valuations and the infinitely singular ones. The interior points are all
 quasimonomial valuations, all divisorial valuations are branching points whereas all the irrational valuations are
 regular points (i.e admit only two tangent vectors). Define on $\cV_\m$ the following
 function
\begin{equation}
  \alpha (v) := \sup \left\{\frac{v (\phi)}{m (\phi)} : \phi \in \m \right\}.
  \label{<+label+>}
\end{equation}
It is called the \emph{skewness} function (see \cite{favreValuativeTree2004} \S 3.3)

\begin{prop}[Proposition 3.25 of \cite{favreValuativeTree2004}] \label{PropValuationEnFonctionDeAlpha}
  The skewness function $\alpha: \cV_\m \rightarrow [1, + \infty]$ defines a parametrisation of $\cV_\m$. We have the following
  properties.

  \begin{itemize}
    \item  $\alpha (v) = 1 \Leftrightarrow v= v_\m$.
    \item Let $\phi \in \m$ be irreducible and let $v \in \cV_\m$, then
  \begin{equation}
  \forall \phi \in \m, v (\phi) = \alpha(v \wedge v_\phi) m (\phi). \end{equation}
\item If $v$ is divisorial, then $\alpha (v) \in \Q$.
    \item If $v$ is irrational, then $\alpha (v) \in \R \setminus \Q$.
    \item If $v$ is a curve valuation, then $\alpha (v) = + \infty$.
    \item If $v$ is infinitely singular, then $\alpha (v) \in ( 1 , + \infty]$ and every value is realised.
      \item If $\cV_{\m, \div}$ is the subset of $\cV_\m$ consisting of the divisorial valuations, then
  $(\cV_{\m, \div}, \alpha)$ is a parametrised $\Q$-tree.
    \end{itemize}
\end{prop}

We can define two topologies over $\cV_\m$. The first one is the weak topology being the coarsest topology such that for
all $\phi \in R$, the
evaluations map $v \in \cV_\m \mapsto v(\phi)$ is continuous. The second is the weak topology given by the
$\R$-tree structure on $\cV_\m$.

\begin{prop}[\cite{favreValuativeTree2004}, Theorem 5.1]\label{PropWeakTopologyIsWeakTopologyOfTree}
  The weak topology over $\cV_\m$ given by the evaluation maps $v \in \cV_\m \mapsto v(\phi)$ and the weak topology
  induced by the tree structure of $\cV_\m$ are the same.
\end{prop}

Let $X$ be a good completion of $X_0$ and let $p$ be a smooth point of $X$. Take local coordinates
$z,w$ at $p$, then the completion of the local ring $\OO_{X,p}$ with respect the maximal ideal $\m_p$ is
isomorphic to $\k [ [z,w ] ]$. Let $\cV_X (p)$ be the set of valuations $v$ on $\k[X_0]$ centered at $p$. We
will denote by $\cV_X (p; \m_p)$ the subset of $\cV_X (p)$ of valuations $v$ such that $v(\m_p) = 1$.
The space $\cV_X (p; \m_p)$ is an $\R$-tree isomorphic rooted in $v_{\m_p}$. We make its structure precise.

\begin{prop}\label{PropValuativeTreeIsomorphism}
  The $\R$-tree $\cV_X (p; \m_p)$ is \emph{not} complete.
  \begin{enumerate}
    \item If $p \in E$ is a free point then $\cV_X (p; \m_p)$ is
      isomorphic to $\cV_\m \setminus \left\{ v_z \right\}$ where $z$ is a local equation of $E$.
    \item If $p = E \cap F$ is a satellite point, then $\cV_X (p; \m_p)$ is
      isomorphic to $\cV_\m \setminus \left\{ v_z, v_w \right\}$ where $z,w$ are local coordinates at $p$ with $z$ a
      local equation of $E$ and $w$ a local equation of $F$.
  \end{enumerate}
\end{prop}
\begin{proof}
  If $p \in E$ is a free point, let $z,w$ be local coordinates at $p$ such
  that $z$ is a local equation of $E$. Then, the completion of the local ring at $p$ is isomorphic to $\k [ [ z,w ] ]$
  by Theorem \ref{ThmCompletionLocalRing}. Every $P \in \k[X_0]$ is of the form $P = \frac{\phi}{z^a}$ with $a \geq 0$ and
  $\phi \in \OO_{X,p}$. Hence, a centered valuation on $\k [ [ z, w] ] $ defines a valuation over $\k[X_0]$ if and only if it
  is not the curve valuation $v_z$. Hence we have an isomorphism $\cV_X (p; \m_p) \simeq \cV_\m \setminus \left\{
  v_z \right\}$.

  If $p = E \cap F$ is a satellite point, then let $z,w$ be local coordinates
  at $p$ such that $z$ is a local equation of $E$ and $w$ is a local equation of $F$. Every $P \in \k[X_0]$ is of the form $P
  = \frac{\phi}{z^a w^b}$ where $a,b \geq 0$ and $\phi \in \OO_{X,p}$. Therefore a centered valuation on $\k [ [ z, w ]
  ]$ defines a valuation over $\k[X_0]$ if and only if it is not the curve valuation $v_z$ or $v_w$. Hence we have an
  isomorphism  $\cV_X (p; \m_p) \rightarrow \cV_\m \setminus \left\{ v_z, v_w \right\}$.
\end{proof}

\section{The relative tree with respect to a curve $z=0$}  Let $R = \k \left[ \left[ x,y \right]
\right]$ and let $\m$ be the maximal ideal of $R$. Let $z \in \m$ be irreducible such that $v_\m (z) = 1$. One can
consider the set $\cV_z$ of centered valuations on $R$ such that $v(z) = 1$; we also add the valuation $\ord_z$ to
$\cV_z$ defined by $\ord_z (\phi) = \max \left\{ n  \geq 0 : z^n | \phi \right\}$. (notice that $\ord_z$ is
\emph{not} centered, because for example if $x \neq z, \ord_z (x) = 0$). This is also a tree rooted in $\ord_z$ called
the \emph{relative tree}
(see \cite{favreValuativeTree2004} Proposition 3.61) with the order relation $v \leq_z \mu \Leftrightarrow \forall \phi
\in R, v (\phi) \leq \mu (\phi)$. We can define the weak topology on $\cV_z$ being the coarsest topology such that the
for all $\phi \in R$, the evaluation map $v \in \cV_z \mapsto v(\phi)$ is continuous. There is also the weak topology
given by the tree structure of $\cV_z$.

\begin{prop}[Relative version of \ref{PropWeakTopologyIsWeakTopologyOfTree}]
  The weak topology over $\cV_z$ given by the evaluation maps $v \in \cV_z \mapsto v(\phi)$ and the weak topology
  induced by the tree structure of $\cV_z$ are the same.
\end{prop}

\begin{prop}[\cite{favreValuativeTree2004} Lemma 3.59]\label{PropLocalChartForValuations}
We have an onto map $N_z : \cV_0 \rightarrow \cV_z$ defined by

\begin{align*}
  N_z(v) &= v / v(z) \text{ if } v \neq v_z \\
  N_z(v_z) &= \ord_z.
\end{align*}
This map restricts to a homeomorphism $N_z : \cV_\m \rightarrow \cV_z$ with respect to the weak topology. If $w \in \m$ is
irreducible, then the map $N_{z,w} := N_w \circ {N_w}^{-1} : \cV_z \rightarrow \cV_w$ is a homeomorphism for the weak
topology.

\end{prop}

 The tree $\cV_z$ comes with a skewness function $\alpha_z : \cV_z \rightarrow [0, + \infty]$ and a multiplicity function $m_z
 (\phi) = v_z (\phi)$. The skewness is defined by
\begin{equation}
  \alpha_z (v) := \sup \left\{ \frac{v (\psi)}{m_z (\psi)} | \psi \in \m \right\}
  \label{<+label+>}
\end{equation}

\begin{prop}[Relative version of Proposition
  \ref{PropValuationEnFonctionDeAlpha}]\label{PropValuationEnFonctionDeAlphaRelative}
  The function $\alpha_z : \cV_z \rightarrow [0, + \infty]$ defines a parametrisation of the tree $\cV_z$. We have the
  following properties.
  \begin{itemize}
    \item $\alpha_z (v) = 0 \Leftrightarrow v = \ord_z$.
      \item Let $\phi \in \m$ be irreducible and let $v \in \cV_z$, then
        \begin{equation}
          v (\phi) = \alpha_z (v \wedge N (v_\phi)) m_z (\phi).
        \end{equation}
        \item If $v$ is divisorial or $v = \ord_z$, then $\alpha_z (v) \in \Q$
          \item If $v$ is irrational, then $\alpha_z (v) \in \R \setminus \Q$.
          \item If $v$ is a curve valuation, then $\alpha_z (v) = +\infty$.
          \item If $v$ is infinitely singular, then $\alpha_z (v) \in (0 , + \infty]$ and every value is realised.
            \item If $\cV_{z, \div}$ is the subset of $\cV_z$ consisting of $\ord_z$ and divisorial valuations,
              then $\left( \cV_{z, \div}, \alpha_z \right)$ is a parametrised $\Q$-tree.
  \end{itemize}
     \end{prop}

     \begin{prop}[\cite{favreValuativeTree2004}, Proposition 3.65]\label{PropRelationSkewness}
  We have the following relation
  \begin{equation}
  \forall v \in \cV_0, \quad v (z)^2 \alpha_z \left( \frac{v}{v(z)} \right) = \min\left( v(x), v(y) \right)^2
  \alpha \left( \frac{v}{\min \left( v(x), v(y) \right)} \right)
    \label{<+label+>}
  \end{equation}
  If $w \in \m$ is another irreducible element with $m(w) = 1$, then
  \begin{equation}
    \forall v \in \cV_0, v(z)^2 \alpha_z \left(\frac{v}{v(z)}\right) = v(w)^2 \alpha_w \left( \frac{v}{v(w)} \right).
    \label{<+label+>}
  \end{equation}

\end{prop}

\begin{prop}[\cite{favreValuativeTree2004}, Lemma 3.60 and 6.47]\label{PropCompatibilityOrdre}
  The map $N : \cV_\m \rightarrow \cV_z$ is not an isomorphism of trees. The two orders on $\cV_\m$ and $\cV_z$ are
  compatible except on the segments $[v_\m, v_z]$ and $[\ord_z, N(v_\m)]$ where they are reversed. More precisely,
  \begin{enumerate}
    \item $\forall v, \mu \in [v_\m, v_z] \subset \cV_\m, v \leq_{\m} \mu \Leftrightarrow N(v) \geq_z N (\mu)$.
    \item $\forall v_1, v_2 \in \cV_z \setminus \left\{ \ord_z \right\}, v_1 \leq_z v_2 \Leftrightarrow [{N}^{-1}
      (v_1), v_z ] \subset [{N}^{-1} (v_2), v_x]$.
  \end{enumerate}
\end{prop}

The situation is summed up in Figure \ref{fig:homeo_relative_tree} where we have put arrows on the branches of the tree to
indicate the order.

\begin{figure}[h]
  \centering
  \includegraphics[scale=0.7]{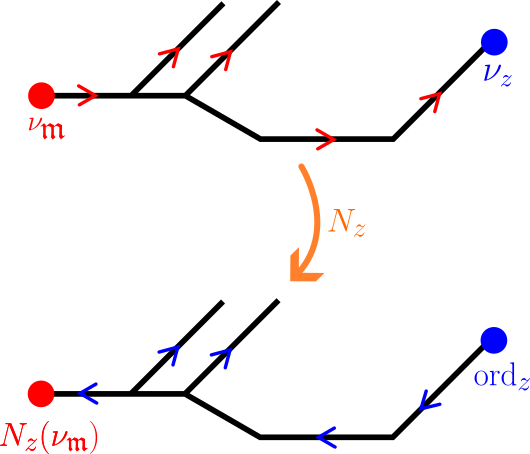}
  \caption{The homeomorphism between $\cV_\m$ and $\cV_z$}
  \label{fig:homeo_relative_tree}
\end{figure}

We will use the relative tree in the following context. Let $E$ be a prime divisor at infinity of some good
completion $X$, let $p$ be a point of
$E$ and let $z,w$ be local coordinates at $p$ such that $E = \left\{ {z = 0} \right\}$. The completion of the local ring at $p$ is
isomorphic to $\k [ [ z, w ] ]$. We define $\cV_X (p; E)$ as
follows; an element of $\cV_X (p;E)$ is either a valuation $v$ on $\k[X_0]$ centered at $p$ such that $v(z) =
1$ or the divisorial valuation $\ord_E$. Notice
that the definition of $\cV_X (p;E)$ does not depend on the local equation $z=0$ of $E$ because the
quotient of two local equations is a regular invertible function.

\begin{prop}\label{PropRelativeValuativeTreeIsomorphism}
  Let $X$ be a completion and let $p \in X$ be a closed point at infinity.
  \begin{enumerate}
    \item If $p \in E$ is a free point, then $\cV_X (p; E)$ is isomorphic to
      $\cV_z$.
    \item If $p = E \cap F$ is a satellite point. Let $z,w$ be local
      coordinates at $p$ such that $z$ is a local equation of $E$ and $w$ a local equation of $F$ then $\cV_X (p; E)$
      is isomorphic to $\cV_z \setminus \left\{ v_w \right\}$ and $\cV_X (p;F)$ is isomorphic to $\cV_w \setminus
      \left\{ v_z \right\}$.
  \end{enumerate}
  The map $N_z : \cV_\m \rightarrow \cV_z$ induces a homeomorphism
  \begin{equation}
    N_{p,E} : \cV_X (p; \m_p) \rightarrow \cV_X
    (p; E) \setminus \left\{ \ord_E \right\}.
  \end{equation}
   Furthermore, if $p = E \cap F$, then the map
  \begin{equation}
    N_{p, F} \circ {N_{p,E}}^{-1}: \cV_X
    (p; E) \setminus \left\{ \ord_E \right\} \rightarrow \cV_X (p; F) \setminus \left\{ \ord_F \right\}
  \end{equation}
  is a homeomorphism.
\end{prop}

\begin{proof}
    If $p \in E$ is a free point. Let $z, w$ be local coordinates at
      $p$ such that $z$ is a local equation of $E$. The completion of the local ring at $p$ is isomorphic to $\k [ [ z,w
      ] ]$ by Theorem \ref{ThmCompletionLocalRing}. For every $P \in \k[X_0]$, $P$ is of the form $P = \frac{\phi}{z^a}$ where
      $a \geq 0$ and $\phi \in \OO_{X,p}$. Therefore, a centered valuation on $\k [ [ z, w ] ] $ defines a valuation
      over $\k[X_0]$ if and only if it is not the curve valuation $v_z$. Since $v_z \not \in \cV_z$ we have that $\cV_X
      (p; E) \simeq \cV_z$. Call $\sigma : \cV_X (p; E) \rightarrow \cV_z$ the isomorphism. We
      define $N_{p,E}$ as follows. Recall by Proposition \ref{PropLocalChartForValuations} that there is a homeomorphism
      $N : \cV_\m \rightarrow \cV_z$ where in particular $N(v_z) = \ord_z$. Here we have that $\ord_z$ is canonically
      identified with $\ord_E$ and $\cV_X (p; \m_p)$ is isomorphic to $\cV_\m \setminus \left\{ v_z \right\}$, call
      $\iota: \cV_X (p; \m_p) \rightarrow \cV_\m \setminus \left\{ v_z \right\}$ the isomorphism. Define
      \begin{equation}
        N_{p,E}:= {\sigma}^{-1} \circ N \circ \tau : \cV_X (p; \m_p) \rightarrow \cV_X (p; E) \setminus \left\{ \ord_E
        \right\},
        \label{<+label+>}
      \end{equation}
      it is a homeomorphism.

    If $p =E \cap F$ is a satellite point. Let $(z,w)$ be local coordinates
      at $p$ such that $z$ is a local equation of $E$ and $w$ is a local equation of $F$. The completion of the local
      ring at $p$ is isomorphic to $\k [ [ z,w ] ] $ by Theorem \ref{ThmCompletionLocalRing}. Every $P \in \k[X_0]$ is of the
      form $P = \frac{\phi}{z^a w^b}$ where $a,b \geq 0$ and $\phi \in \OO_{X,p}$. Therefore a centered valuation on
      $\k [ [ z , w ] ] $ defines a valuation over $\k[X_0]$ if and only if it is not the curve valuation associated to $z$ or
      $w$. Or $v_z$ does not belong to $\cV_z$ but $v_w$ does. Therefore, $\cV_X (p;E)$ is isomorphic to $\cV_z
      \setminus \left\{ v_w \right\}$. If $N_z : \cV_\m \rightarrow \cV_z$ is the map from Proposition
      \ref{PropLocalChartForValuations}, then $N (v_z) = \ord_z$ and $N(v_w) = v_w$. Therefore, $N_w \circ
      {N_z}^{-1} : \cV_z \rightarrow \cV_w$ is a homeomorphism that sends $\ord_z$ to $v_z$ and $v_w$ to $\ord_w$. Fix
      an isomorphism $\tau_E : \cV_X (p;E)  \rightarrow \cV_z \left\{ v_w \right\}$ and $\tau_F : \cV_X (p; F)
      \rightarrow \cV_w \setminus \cV_z$. We have that the map
      \begin{equation}
        N_{p,F} \circ {N_{p,E}}^{-1} = {\tau_F}^{-1} \circ N_w \circ {N_z}^{-1} \circ \tau_E : \cV_X (p; E) \setminus
        \left\{ \ord_E \right\} \rightarrow \cV_X (p; F) \setminus \left\{ \ord_F \right\}
        \label{<+label+>}
      \end{equation}
      is a homeomorphism.
\end{proof}

\begin{prop}\label{PropValuativeTreeIsRelativeValuativeTreeWRTExceptionalDivisor}
  Let $X$ be a completion of $X_0$ and let $E$ be a prime divisor at infinity. If $p_1, p_2 \in E$ are closed points
  with $p_1 \neq p_2$, then $\cV_X (p_1; E) \cap \cV_X (p_2; E) = \left\{ \ord_E \right\}$. Define the set
  $\cV_X (E;E)$ of valuations $v$ such that $c_X (v) \in E$ and $v(z) = 1$ where $z$ is a local equation of $E$
  at $c_X (v)$. Then
  \begin{equation}
    \cV_X (E;E) = \bigcup_{p \in E} \cV_X (p;E)
    \label{<+label+>}
  \end{equation}
  and it has a natural structure of a rooted $\R$-tree rooted in $\ord_E$. The skewness functions $\alpha_E$ glue together to
  give $\cV_X (E;E)$ the structure of a parametrised rooted tree. Every point $p \in E$ defines a tangent vector at
  $\ord_E$ given by $\cV_X (p;E) \setminus \left\{ \ord_E \right\}$.

  Furthermore, Let $Y$ be a completion of $X_0$ and $q \in Y$ a closed point at infinity. Let $\pi : Z \rightarrow
  Y$ be the blow up of $q$ and let $\tilde E$ be the exceptional divisor of $\pi$. Then, for every $\tilde q \in \tilde
  E$, the map $\pi_\bullet: \cV_Z( \tilde q; \tilde E) \rightarrow \cV_Y (q; \m_q)$ is actually equal to $\pi_*$ and
  they glue together to give a map
  \begin{equation}
    \pi_*: \cV_Z (\tilde E; \tilde E) \rightarrow \cV_Y (q; \m_q),
    \label{EqIsomorphismLocalValuativeTreeAndRelativeValuativeTreeWRTExceptionalDivisor}
  \end{equation}
  which is an isomorphism of trees. We have the relation $\alpha_{\m_q} \circ \pi_* = 1 + \alpha_E$ and $b_{\m_q} \circ
  \pi_* = b_E$.
\end{prop}

We postpone the proof to \S \ref{SecGeometricInterpretationValuativeTree}. If $E \simeq \P^1$, this tree is isomorphic
to the tree of normalised valuations centered at infinity over $\A^2$ constructed in \cite{favreEigenvaluations2007}, Appendix.

\section{The monomial valuations centered at an intersection point at
infinity}\label{SecMonomialValuationsAtSatellitepoint} Let $X$ be a good completion of $X_0$
and let $E,F$ be two divisors at infinity that intersect at a
point $p$. Let $(x,y)$ be local coordinates at $p$ such that $E =  \{x=0\}$ and $F = \{y=0\}$. There
are three spaces to consider: $\cV_X (p, \m_p), \cV_X (p; E)$ and $\cV_X (p;F)$. We explain here how
they are related. For $(s,t) \in [0, +\infty]^2 \setminus \left\{ (0,0), \left( \infty, \infty \right)
\right\}$, we denote by $v_{s,t}$ the monomial valuation defined by
\begin{equation}
  v_{s,t} \left(\sum a_{ij} x^i y^j\right) =
  \min \left\{ si + tj | a_{ij} \neq 0 \right\}.
\end{equation}
 Notice that $v_{0,1} = \ord_F, v_{1,0} = \ord_E, v_{1,
\infty} = v_y, v_{\infty,1} = v_x$. We will denote the set of such valuation by $[\ord_E, \ord_F]$. We
use this notation because of the following: $[\ord_E, \ord_F] \cap \cV_X (p; E)$ consists of the
valuations $v_{1, t}$ for $t \in [0, + \infty)$ and $\left[ \ord_E, \ord_F \right] \cap
\cV_X (p;F)$ consists of the valuations $v_{s,1}$ for $s \in [0, + \infty)$. So they define segments in the
respective trees. In
particular we have
\begin{equation}
  N_{p,F} \circ {N_{p,E}}^{-1} (v_{1,t}) = v_{1/t, 1}, \quad \forall t \in (0, + \infty)
  \label{}
\end{equation}

One can show with the definition of the skewness function $\alpha$ that $\alpha_E (v_{1,t}) = t$. Therefore
we show

  \begin{lemme}\label{LemmelevelfunctionValuationMonomiale}
    Let $v$ be a monomial valuation centered at $p = E \cap F$. One has
    \begin{align*}
      \alpha_E \left(\frac{v}{v(x)}\right) = \frac{v (y)}{v(x)} = \frac{t}{s} \text{ if } v = v_{s,t} \\
      \alpha_F \left(\frac{v}{v(y)}\right) = \frac{v(x)}{v(y)} = \frac{s}{t} \text{ if } v = v_{s,t}
    \end{align*}

    In particular we have that $\alpha_E \left(\frac{v}{v(x)} \right) = { \alpha_F \left( \frac{v}{v(y)}
  \right)}^{-1}$ on $]\ord_E, \ord_F[$.
  \end{lemme}

  \begin{lemme}\label{LemmeActionSurMonomiale}
    Let $f : X_0 \rightarrow X_0$ be a dominant endomorphism of $X_0$. Let $X,Y$ be two completions of $X_0$. Let $E
    \cap F =p \in X, E' \cap F' = q \in Y$ be satellite points at infinity such that $f(p) = q$ and
  \begin{equation}
    f(x,y) = \left( x^a y^b \phi, x^c y ^d \psi \right) = \left( z,w \right)
    \label{<+label+>}
  \end{equation}
  with $ad - bc \neq 0$ where $(x,y)$ is associated to $E,F$ and $(z,w)$ to $E', F'$. Then,
  \begin{equation}
    f_* (v_{s,t}) = v_{as + bt, cs + td}
    \label{<+label+>}
  \end{equation}
  \end{lemme}
  \begin{proof}
    We have the map $f_\bullet : \cV_X (p; E) \rightarrow \cV_X (q; E')$ defined by
    \begin{equation}
      f_\bullet (v) = \frac{f_* v}{a + b v(y)}.
      \label{<+label+>}
    \end{equation}
    In particular, for any $t > 0$
    \begin{equation}
      f_\bullet v_{1,t} (w) = \frac{c + dt}{a + b t}.
      \label{<+label+>}
    \end{equation}
    Since $ad-bc \neq 0$, there exists uncountably many irrational $t > 0$ such that $\frac{c + dt}{a + bt} \not \in
    \R$. Thus,
    \begin{equation}
      \frac{c+ dt}{a + bt} = \alpha_{E'} (f_\bullet v_{1,t} \wedge v_w) \not \in \Q.
      \label{<+label+>}
    \end{equation}
    Since irrational valuations are regular points in the valuative tree we must have $f_\bullet (v_{1,t}) =
    v_{1, \frac{c + dt}{a + bt}}$. Now, the set of such $t$ is dense in $\R_{\geq 0}$ so by continuity we get for all $t
    \geq 0$,
    \begin{equation}
      f_\bullet v_{1,t} = v_{1, \frac{c + td}{a + tb}}
      \label{<+label+>}
    \end{equation}
    and the lemma is proven.
  \end{proof}

  \section{Geometric interpretations of the valuative tree}\label{SecGeometricInterpretationValuativeTree}
Let $X$ be a completion of $X_0$ and let $p \in X$ be a closed point at infinity. We
consider in this section only completions above $X$ that are exceptional above $p$. If $\pi : (Y, \Exc(\pi))
\rightarrow (X, p)$ is such a completion, then we call $\Gamma_\pi$ the dual graph which vertices consist of the
exceptional divisors of $\pi$. Two exceptional divisors are linked by an edge if they intersect. The graph $\Gamma_\pi$
is connected without cycles, it is therefore an $\N$-tree. We set the root of $\Gamma_\pi$ to be the exceptional divisor
$\tilde E$ that appears after blowing up $p$.

If $E$ is a prime divisor at infinity of $X$ such that $p \in E$. We define the dual graph
\begin{equation}
  \Gamma_{\pi,
  E} := \Gamma_\pi \cup \{E\}.
\end{equation}
It is also a $\N$-tree. We set the root of $\Gamma_{\pi, E}$ to be $E$.

\begin{lemme}[\cite{favreValuativeTree2004}, Proposition 6.2]
  Let $\pi : Y \rightarrow (X, p)$ be a completion exceptional above $p$. if $\tau : Z \rightarrow Y$ is the
  blow up of a point in the exceptional locus of $\pi$, then there are natural inclusions of $\N$-trees
  \begin{equation}
    \Gamma_\pi \hookrightarrow \Gamma_{\pi \circ \tau}, \quad \Gamma_{\pi, E} \hookrightarrow \Gamma_{\pi \circ \tau,
    E}.
    \label{<+label+>}
  \end{equation}
  Therefore, the direct limits $\Gamma := \varinjlim_{\pi} \Gamma_\pi$, $\Gamma_E := \varinjlim_\pi \Gamma_{\pi, E}$ are
  well defined. The points of $\Gamma$ are in bijection with $\cD_{X,p}$ and $\Gamma_E = \Gamma \cup \{E\}$ and they
  have a structure of $\Q$-trees.
\end{lemme}

\begin{lemme}[\cite{favreValuativeTree2004} Theorem 6.9]
  We have a map $\pi_\bullet: \Gamma_\pi \hookrightarrow \cV_X (p; \m_p)_{\div}$ defined by
  \begin{equation}
    \pi_\bullet (F) = v_F
    \label{<+label+>}
  \end{equation}
  where $v_F$ is the valuation equivalent to $\pi_* \ord_F$ that belongs to $\cV_X (p; \m_p)$. These maps are compatible
with the direct limit and give a map $\Gamma \hookrightarrow \cV_X (p; \m_p)$.
\end{lemme}

\begin{lemme}
  We have a map $\pi_\bullet: \Gamma_{\pi,E} \hookrightarrow \cV_{E, \div}$ defined by
  \begin{equation}
    \pi_\bullet (F) = v_F
    \label{<+label+>}
  \end{equation}
  where $v_F$ is the valuation equivalent to $\pi_* \ord_F$ that belongs to $\cV_X (p;E)$. These maps are compatible
  with the direct limit and give a map $\Gamma_E \hookrightarrow \cV_X (p; E)$.
\end{lemme}

\begin{prop}[\cite{favreValuativeTree2004}, Lemma 6.28]\label{PropOrderRelationAfterOneBlowUp}
  Let $\pi : (Y, \Exc(\pi)) \rightarrow (X, p)$ be a completion exceptional above $p$. Let $q \in Y$ be a closed
  point that belongs to the exceptional component of $\pi$. Let $\tilde F$ be the exceptional divisor above $q$.
\begin{enumerate}
  \item If $q \in F$ with $F \in \Gamma_\pi$, then $v_{\tilde F} > v_{F}$.
  \item If $q = F_1 \cap F_2$ with $F_1, F_2 \in \Gamma_\pi$, suppose that $v_{F_1} < v_{F_2}$, then $v_{F_1} <
    v_{\tilde F} < v_{F_2}$.
  \end{enumerate}
\end{prop}

\begin{prop}[Relative version of Proposition
  \ref{PropOrderRelationAfterOneBlowUp}]\label{PropOrderRelationAfterOneBlowUpRelativeCase}
  Let $\pi : (Y, \Exc(\pi)) \rightarrow (X, p)$ be a completion exceptional above $p$. Let $q \in \Exc(\pi)$. Let
  $\tilde F$ be the exceptional divisor above $q$.
\begin{enumerate}
  \item If $q \in F$ is a free point with $F \in \Gamma_{\pi, E}$, then $v_{\tilde F} > v_{F}$.
  \item If $q = F_1 \cap F_2$ is a satellite point with $F_1, F_2 \in \Gamma_{\pi, E}$, if $v_{F_1} <
    v_{F_2}$, then $v_{F_1} < v_{\tilde F} < v_{F_2}$.
  \item In particular, if $q = E \cap F$, then $\ord_E < v_{\tilde F} < v_F$.
  \end{enumerate}
\end{prop}

\begin{thm}[\cite{favreValuativeTree2004}, Theorem 6.22]
  We have an isomorphism of $\Q$-trees
  \begin{equation}
    \Gamma \simeq \cV_X (p; \m_p)_{\div} , \quad \Gamma_E \simeq \cV_X (p; E)_{\div}
    \label{<+label+>}
  \end{equation}
  given by $F \simeq v_F$.  We can take the completion of the $\Q$-trees to get the isomorphism
  \begin{equation}
    \overline \Gamma \simeq \cV_X (p; \m_p), \quad \overline \Gamma_E \simeq \cV_X (p; E)
    \label{<+label+>}
  \end{equation}
\end{thm}

\begin{prop}\label{PropCompatibilitéOrdreContractionMap}
  Let $X$ be a completion of $X_0$ and let $p \in X$ be a closed point at infinity. Let $\cV_*$ be either
  $\cV_X (p; \m_p)$ or $\cV_X (p;E)$ for some prime divisor $E$ at infinity such that $p \in E$. Let
  $\Gamma_*$ be either $\Gamma$ or $\Gamma_E$. Let $\pi : (Y, \Exc(\pi)) \rightarrow (X, p)$ be a completion
  exceptional above $p$. Let $q \in \Exc (\pi)$ be a closed point. The map $\pi$ induces a map $\pi_* : \cV_Y (q)
  \rightarrow \cV_X (p)$.
  \begin{enumerate}
    \item If $q \in E_q$ is a free point with $E_q \in \Gamma_*$, then we have an inclusion map
    $\pi_\bullet: \cV_Y (q; E_q) \hookrightarrow \cV_*$. The order relation in $\cV_Y (q; E_q)$ and $\cV_*$
    are compatible and $\pi_\bullet$ is an inclusion of trees.
  \item If $q =E_q \cap F_q$ is a satellite point with $E_q,F_q \in \Gamma_*$, then, if $v_{E_q} <_* v_{F_q}$, the
    order relations on $\cV_*$ and $\cV_Y (q; E_q)$ are compatible and $\pi_\bullet : \cV_Y (q; E_q) \hookrightarrow
    \cV_*$ is an inclusion of trees.
  \end{enumerate}
\end{prop}

\begin{proof}
  We only need to show that the orders are compatible on the divisorial valuations of $\cV_Y (q; E_q)$. Therefore we
  show the following,

  \begin{claim}\label{ClaimOrdreCompatible}
    For every completion $\tau: (Z, \Exc(\tau)) \rightarrow (Y, q)$ exceptional above $q$, we have the following
    \begin{enumerate}
      \item For all $ F_1, F_2 \in \Gamma_{\tau, E_q},$
        \begin{equation}
          v_{F_1} <_* v_{F_2} \Leftrightarrow v_{F_1} <_{E_q} v_{F_2}
          \label{<+label+>}
        \end{equation}
      \item If $F \in \Gamma_{\tau, E_q}$ satisfies $F \cap F_q \neq \emptyset$, then \begin{equation}
          v_F < v_{F_q}
          \label{<+label+>}
        \end{equation}
    \end{enumerate}
  \end{claim}
  Here there is a slight abuse of notation as we denote by $v_{F_i}$ the image of $F_i$ both in $\cV_Y (q; E_q)$ and
$\cV_*$. This is done to lighten notations, we believe that it does not provide any confusion.

We prove this by induction on the number of blow ups above $q$. If $\tau = \id$, then $\ord_{E_q}$ is the root of
$\cV_Y (q; E_q)$ and $v_{E_q} < v_{F_q}$ by assumption so there is nothing to do.

Let $\tau : (Z, \Exc (\tau)) \rightarrow (Y, q)$ be a completion exceptional
above $q$ such that Claim \eqref{ClaimOrdreCompatible} is true.
 Let $q' \in \Exc(\tau)$ be a
closed point, let $\tau ' : Z' \rightarrow Z$ be the blow up of $q'$ and let $\tilde F$ be the exceptional divisor
above $q'$.
\begin{itemize}
  \item If $q' \in F$ is a free point with $F \in \Gamma_{\tau, E_q}$, then by Proposition
    \ref{PropOrderRelationAfterOneBlowUpRelativeCase} we have
    \begin{equation}
      v_F <_{E_q} v_{\tilde F}
      \label{<+label+>}
    \end{equation}
    Now we have two possibilities.
    \begin{itemize}
      \item If $q'$ is also a free point with respect to $\Gamma_*$, then by Proposition
        \ref{PropOrderRelationAfterOneBlowUp} and \ref{PropOrderRelationAfterOneBlowUpRelativeCase} we also get
        \begin{equation}
          v_{F} <_* v_{\tilde F}.
          \label{<+label+>}
        \end{equation}
    Since $\tilde F \cap F_q = \emptyset$, Claim \ref{ClaimOrdreCompatible} is shown for $\Gamma_{\tau \circ \tau ', E_q}$.
  \item If $q'$ is the satellite point $F \cap F_q$, then by induction hypothesis we have $v_F <_* v_{F_q}$ and
    therefore $\tilde F \cap F_q \neq \emptyset$ and by Proposition \ref{PropOrderRelationAfterOneBlowUp} and
    \ref{PropOrderRelationAfterOneBlowUpRelativeCase} we get
    \begin{equation}
      v_F <_* v_{\tilde F} <_* v_{F_q}
      \label{<+label+>}
    \end{equation}
    So Claim \ref{ClaimOrdreCompatible} is shown for $\Gamma_{\tau \circ \tau ' , E_q}$.
    \end{itemize}

  \item If $q'$ is a satellite point. Let $F_1, F_2 \in \Gamma_{\tau, E_q}$ such that $q = F_1 \cap F_2$. Suppose
    without loss of generality that $v_{F_1} <_{E_q} v_{F_2}$, then by the induction hypothesis we have $v_{F_1} <_*
    v_{F_2}$ and by Proposition \ref{PropOrderRelationAfterOneBlowUp} and
    \ref{PropOrderRelationAfterOneBlowUpRelativeCase}, we get
    \begin{equation}
      v_{F_1} <_{E_q} v_{\tilde F} <_{E_q} v_{F_2} \text{ and } v_{F_1} <_* v_{\tilde F} <_* v_{F_2}.
      \label{<+label+>}
    \end{equation}
    Since $\tilde F \cap F_q = \emptyset$ we have proven Claim \ref{ClaimOrdreCompatible} for $\Gamma_{\tau \circ \tau
    ', E_q}$.
\end{itemize}
\end{proof}

\begin{proof}[Proof of Proposition \ref{PropValuativeTreeIsRelativeValuativeTreeWRTExceptionalDivisor}]
  Let $Y$ be a completion of $X_0$ and let $q \in Y$ be a closed point at infinity. Let $\pi : Z \rightarrow Y$
  be the blow up of $q$. Let $\tilde E$ be the exceptional divisor and let $\tilde q \in \tilde E$ be a closed point.
  Apply Proposition \ref{PropCompatibilitéOrdreContractionMap} with $\cV_* = \cV_Y (q; \m_q)$. The map $\pi_\bullet :
  \cV_Z (\tilde q; \tilde E) \rightarrow \cV_Y (q; \m_q)$ is an inclusion of trees. There exists local coordinates
  $z,w$ at $q$ and $x,y$ at $p$ such that $\pi (z,w) = (z, zw)$ where $z$ is a local equation of $\tilde E$. We
  therefore get
  \begin{equation}
    v (z) = 1 \Leftrightarrow \min (\pi_* v (x), \pi_* v (y)) = 1.
    \label{<+label+>}
  \end{equation}

  Hence, $\pi_\bullet = \pi_*$ and $\pi_* (\ord_{\tilde E}) = v_{\m_q}$. Therefore we can glue these maps to obtain an
  isomorphism of trees
  \begin{equation}
    \pi_*: \cV_Z (\tilde E; \tilde E) \rightarrow \cV_Y (q; \m_q)
    \label{<+label+>}
  \end{equation}
  We get the relation on the skewness functions by Proposition \ref{PropSkewnessRelationAfterBlowUpContractionMap} which
  will be proven in the next section.
\end{proof}

\section{Properties of skewness}
We have two valuative tree structures. We describe some properties of the skewness function for these two structures and
how they behave after blowing up. Fix a completion $X$, let $p \in X$ be a closed point at infinity and let $E$ be
a prime divisor at infinity in $X$ such that $p \in E$. In
accordance with the notations of the previous section, set $\Gamma = \cD_{X,p}$ and $\Gamma_E = \cD_{X, p} \cup
\left\{ E \right\}$.

\begin{dfn}
  If $F \in \Gamma$ is a prime divisor above $p$, we define the \emph{generic multiplicity} $b(F)$
  inductively as follows.
  \begin{itemize}
    \item $b(\tilde E) = 1$ where $\tilde E$ is the exceptional divisor above $p$.
    \item If $q \in F$ is a free point with $F \in \Gamma$, then $b(\tilde F) = b(F)$ where
      $\tilde F$ is the exceptional divisor above $q$.
    \item If $q = F_1 \cap F_2$ is a satellite point with $F_1, F_2 \in \Gamma$, then $b(\tilde F) = b(F_1) +
      b(F_2)$.
  \end{itemize}
  If $v \in \cV_X (p; \m_p)$ is divisorial then we define $b(v) := b(E)$ where $E$ is the center of $v$ in
  some completion above $X$.
  \label{<+label+>}
\end{dfn}

\begin{dfn}
  If $F \in \Gamma_E$, we define the \emph{relative generic multiplicity} $b_E (F)$ inductively as follows.
  \begin{itemize}
    \item $b_E (E) = 1$.
    \item If $q \in F$ is a free point with $F \in \Gamma_E$, then $b_E (\tilde F) = b_E (F)$.
    \item If $q = F_1 \cap F_2$ is a satellite point with $F_1, F_2 \in \Gamma_E$, then $b_E (\tilde F) = b_E
      (F_1) + b_E (F_2)$.
  \end{itemize}
  If $v \in \cV_X (p; E_z)$ is divisorial, then we set $b_E (v) := b_E (F)$ where $F$ is the center of $v$ in some
  completion above $X$.
\end{dfn}
Figure \ref{fig:algo_generic_multiplicity} sums up the definition of the generic multiplicity.

\begin{figure}[h]
  \centering
  \includegraphics[scale=0.8]{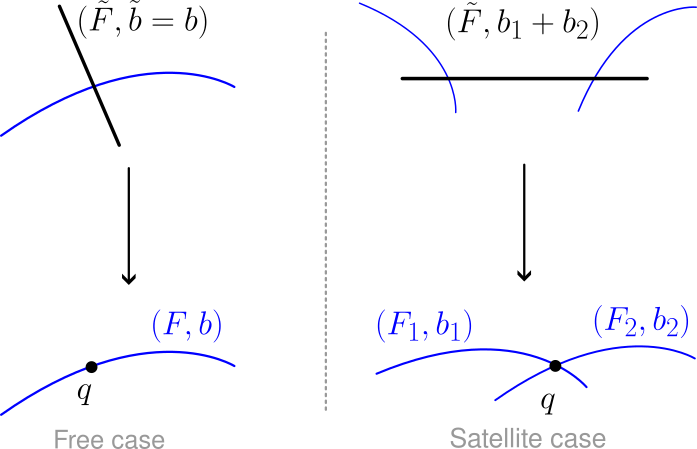}
  \caption{Algorithm for computing the generic multiplicity}
  \label{fig:algo_generic_multiplicity}
\end{figure}

The term \emph{generic multiplicity} is justified by the following proposition.

\begin{prop}[\cite{favreValuativeTree2004} Proposition 6.26]\label{PropInterpretationGenericMultiplicity}
  Let $v \in \cV_X (p; \m_p)$ be divisorial, let $E \in \Gamma$ be the center of $v$ over some completion $\pi: Y
  \rightarrow X$ above $X$. Then,
  \begin{equation}
    \pi_* \ord_E (\m_p) = b (v)
\label{<+label+>}
\end{equation}
\end{prop}

\begin{prop}[Relative version of Proposition
  \ref{PropInterpretationGenericMultiplicity}]\label{PropInterpretationRelativeGenericMultiplicity}
  If $v \in \cV_X (p; E)$ is divisorial, let $F$ be the center of $v$ over some completion $\pi: Y \rightarrow
  X$ above $X$. Then,
  \begin{equation}
    \pi_* \ord_F (z) = b_E (F)
    \label{<+label+>}
  \end{equation}
  where $z \in \OO_{X, p}$ is a local equation of $E$. This means that $\ord_F (\pi^* E) = b_E (F)$.
\end{prop}

From now on we write $\cV_*$ for either $\cV_X (p; \m_p)$ and $\cV_X (p; E)$ and we write $\alpha_*,
b_*$ for the skewness function and the generic multiplicity function associated to the tree structure.

For a valuation $v \in \cV_*$, we define the \emph{infinitely near sequence} of $v$ as follows, set $v_0 = v_*$ the
root of $\cV_*$ and let $p_n$ be the sequence of centers above $X$ associated to $v$. Let $E_n$ be the exceptional
divisor above $p_n$. Set $v_n = \frac{1}{b_* (E_n)} \ord_{E_n}$, if $v$ is quasimonomial $(v_n)$ is the infinitely near
sequence of $v$, in particular it is finite if and only if $v$ is divisorial. If $v$ is a curve valuation or
infinitely singular we define the infinitely near sequence of $v$ as the subsequence of $v_n$ where $c_{X_n} (v)$
is a free point (at infinity).

\begin{prop}\label{PropInfinitelyNearSequence}
  Let $v \in \cV_*$ and let $v_n$ be its infinitely near sequence
  \begin{itemize}
    \item the sequence $v_n := \frac{1}{b_n} \ord_{E_n}$ converges weakly towards $v$.
    \item $\alpha_*(v) = \lim_n \alpha_*(v_n)$.
  \end{itemize}
\end{prop}
\begin{proof}
  The infinitely near sequence is constructed in Section 6.2.2 of \cite{favreValuativeTree2004} (this sequence does not
  have a name in \cite{favreValuativeTree2004}). The fact $v_n$ converges weakly towards $v$ is shown there. To show
  the statement for skewness, we split the proof with respect to the type of $v$.

  If $v$ is a curve valuation or an infinitely singular one, then $v_n < v$ and $v_n$ increases towards $v$. Since
  $\alpha$ induces an order preserving bijection of the segment $[v_*, v]$. We have that $\alpha (v_n) \leq \alpha
  (v)$ and is increasing. So it converges towards a real number $\alpha_0 \in [\alpha_* (v_*), \alpha_* (v)]$. If
  $\alpha_0 < \alpha_* (v)$, then $v_n \leq v_0 < v$ where $\alpha_* (v_0) = \alpha_0$ and this is absurd.

  If $v$ is irrational, then there exists $N_0$ such that for all $n \geq N_0, p_n$ is a satellite point. We can split
  the sequence $(v_n)_{n \geq N_0}$ into two subsequences $(v^+_k), (v^-_k)$ such that $v^+_k$ is increasing and
  converges towards $v$ and $(v^-_k)$ is decreasing and converges towards $v$. We therefore get
  \begin{equation}
    \alpha (v_k^+) < \alpha (v) < \alpha (v_k^-)
    \label{<+label+>}
  \end{equation}
  and it is clear that $\lim_k \alpha (v_k^+) = \lim_k \alpha (v_k^-) = \alpha (v)$.
\end{proof}

We will say that two divisorial valuations $v, v'$ are \emph{adjacent} if there exists a completion $Y$ above $X$
such that the centers of $v$ and $v'$ are both prime divisors and they intersect.

\begin{prop}[\cite{favreValuativeTree2004}, Corollary 6.39]\label{PropSkewnessBlowUp}
  Let $v, v' \in \cV_*$. Assume $v < v'$ and that they are adjacent, then
  \begin{equation}
    \alpha_*(v') - \alpha_*(v) = \frac{1}{b_* (v) b_* (v')}
    \label{<+label+>}
  \end{equation}
\end{prop}

\begin{prop}[\cite{favreValuativeTree2004}, Theorem 6.51]\label{PropSkewnessRelationAfterBlowUpContractionMap}
   Let $\pi: Y \rightarrow X$
  be a completion above $X$ and let $q \in E_q$ be a
  free point of $Y$ such that $\pi (E_q) = p$.  By Proposition
  \ref{PropCompatibilitéOrdreContractionMap}, $\pi_\bullet : \cV_Y (q; E_q) \rightarrow \cV_*$ is
  an inclusion of trees.

  \begin{enumerate}
    \item The normalization of $\pi_* \ord_{E_q}$ (to get a valuation in $\cV_*$) is
      \begin{equation}
        \pi_\bullet \ord_{E_q} =: v_{E_q} = \frac{1}{b_* (E_q)} \pi_* \ord_{E_q}.
        \label{<+label+>}
      \end{equation}
      \item \begin{align}
    \forall v \in \cV_Y (p; E), \quad \alpha_* (\pi_\bullet v) &= \alpha_{*} (v_{E_q}) + \frac{1}{b_*(E_q)^2}
    \alpha_{E_q} (v) \\
    b_* (\pi_\bullet v) &= b_* (E_q) b_{E_1} (v)
    \label{<+label+>}
  \end{align}
  \end{enumerate}
\end{prop}

\begin{proof}
  It suffices to show this formula for every divisorial valuation $v \in \cV_Y (q; E_q)$ and then use infinitely near
  sequences by Proposition
  \ref{PropInfinitelyNearSequence}. We prove the result by induction on the number of blow-ups above $q$. Namely we show
  the following
  \begin{claim}
    For every completion $\tau : (Z, \Exc(\tau)) \rightarrow (Y, q)$ exceptional above $q$, for every $F \in
    \Gamma_{\tau, E_q}$,
    \begin{align}
      b_{*} (F) &= b_{E_q}(F) b_*(E_q) \label{EqGenericMultiplicity} \\
      \alpha_* (v_F) &= \alpha_* (v_{E_q}) + \frac{1}{b_* (E_q)} \alpha_{E_q} (v_F) \label{EqSkewness}
    \end{align}
  \end{claim}
  If $\tau = \id : Y \rightarrow Y$, then $\Gamma_{\tau, E_q} = \left\{ E_q \right\}$. We have by definition that
  $b_{E_q} (E_q) = 1, \alpha_{E_q} (\ord_{E_q}) = 0$. Therefore Equations \eqref{EqGenericMultiplicity} and
  \eqref{EqSkewness} holds.

  Suppose the claim to be true for a completion $\tau: (Z, \Exc(\tau)) \rightarrow (Y, q)$ exceptional above $q$. Let
  $\tau ' : Z ' \rightarrow Z$ be the blow up of a closed point $q ' \in \Exc (\tau)$. Let $\tilde E$ be the
  exceptional divisor above $q'$.

  If $q' \in F$ is a free point with $F \in \Gamma_{\tau, E}$, then $q'$ is also a free point with respect to
  $\Gamma_{*, \pi \circ \tau}$ because $q \in Y$ is a free point. Therefore by definition
  \begin{equation}
    b_*(\tilde E) = b_* (F), \quad b_{E_q} (\tilde E) = b_{E_q} (F)
    \label{<+label+>}
  \end{equation}
  So Equation \eqref{EqGenericMultiplicity} is true for $\tilde E$ by induction. Now, by Proposition \ref{PropSkewnessBlowUp}
  \begin{equation}
    \alpha_* (v_{\tilde E}) = \alpha_* (v_F) + \frac{1}{b_*(F) b_* (E_q)}, \quad \alpha_{E_q} (v_{\tilde E}) =
    \alpha_{E_q} (v_F) + \frac{1}{b_{E_q} (\tilde E) b_{E_q} (F)}
    \label{<+label+>}
  \end{equation}
  By induction, Equation \eqref{EqSkewness} is true for $\tilde E$.

  If $q' = F_1 \cap F_2$ is a satellite point with $F_1, F_2 \in \Gamma_{\tau, E_q}$, then
  \begin{equation}
    b_* (\tilde E) = b_* (F_1) + b_* (F_2), \quad b_{E_q} (\tilde E) = b_{E_q} (F_1) + b_{E_q} (F_2)
    \label{<+label+>}
  \end{equation}
  So by induction Equation \eqref{EqGenericMultiplicity} holds for $\tilde E$. Suppose without loss of generality that
  $v_{F_1} < v_{F_2}$ both in $\cV_*$ and $\cV_Y (q; E_q)$. This is possible by Proposition
  \ref{PropCompatibilitéOrdreContractionMap}. By Proposition \ref{PropSkewnessBlowUp}
  \begin{equation}
    \alpha_* (v_{\tilde E}) = \alpha_* (v_{F_1}) + \frac{1}{b_*(F_1) b_* (\tilde E)}, \quad \alpha_{E_q}
    (v_{\tilde E}) = \alpha_{E_q} (v_{F_1}) + \frac{1}{b_{E_q} (F_1) b_{E_q} (\tilde E)}.
    \label{<+label+>}
  \end{equation}
  Therefore, Equation \eqref{EqSkewness} holds for $\tilde E$. And the claim is shown by induction.
\end{proof}

\begin{prop}\label{PropFiniteSkewnessEverywhere}
  Let $v$ be a valuation over $\k[X_0]$ centered at infinity. Let $X$ be a completion of $X_0$ and let $E$ be a prime
  divisor of $X$ at infinity such that $\tilde v \in \cV_X (E; E)$ for some valuation $\tilde v$ equivalent to
  $v$. If $\alpha_E (\tilde v) < + \infty$, then for every completion $Y$ of $X_0$ if $\tilde v \in \cV_Y
  (F, F)$ for some prime divisor $F$ at infinity in $Y$, then $\alpha_F (\tilde v) < +\infty$.
\end{prop}
\begin{proof}
  If $v$ is quasimonomial, this is immediate as for any prime divisor $E$ at infinity and any closed point $p \in E$,
  we have that $\alpha_E (v) < + \infty$ for $v = \ord_E$ or $v$ quasimonomial centered at $p$. If $v$ is a
  curve valuation, then $\alpha_E (v) = + \infty$ for any prime divisor $E$ of any completion $X$ such that $c_X
  (v) \in E$. So it remains to show the result for $v$ an infinitely singular valuation.

  We show that if $\pi: Y \rightarrow X$ is a completion above $X$, then $\alpha_{E'} (v) < + \infty
  \Leftrightarrow \alpha_E (v) < +\infty$ where $E'$ is a prime divisor
  of $Y$ at infinity such that some multiple of $v$ belongs to $\cV_Y (E', E')$. Let $p = c_X (v)$ and $q =
  c_Y (v)$. Since $v$ is infinitely singular, by Proposition \ref{PropSequenceOfInfinitelyNearPoints} there exists a
  completion $\tau : (Z, \Exc (\tau)) \rightarrow (Y, q)$ exceptional above $q$ such that $c_Z (v)$ is a free
  point $q'$ lying over a unique prime divisor $F$ at infinity. We apply Proposition
  \ref{PropSkewnessRelationAfterBlowUpContractionMap}. We have that
  \begin{align}
    \alpha_E (v) &= \alpha_E (v_{F}) + \frac{1}{b_E (F)^2} \alpha_{F} (v) \\
    \alpha_{E '} (v) &= \alpha_{E'} (v_F) + \frac{1}{b_{E'} (F)^2} \alpha_F (v)
    \label{<+label+>}
  \end{align}
  Thus $\alpha_E (v) < + \infty \Leftrightarrow \alpha_F (v) < +\infty \Leftrightarrow
  \alpha_{E'} (v) < + \infty$.
\end{proof}

\begin{prop}[\cite{favreValuativeTree2004} Proposition 6.35]\label{PropMonomialValuationIsSegment}
   Let $\pi : (Y, \Exc(\pi)) \rightarrow (X, p)$ be a completion exceptional
  above $p$. Let $q = E \cap F \in \Exc(\pi)$ be a satellite point  with $E,F \in
  \Gamma_{*, \pi}$. Define $v_E = \frac{1}{b_* (E)} \pi_* \ord_E$ and $v_F = \frac{1}{b_* (F)} \pi_* \ord_F$. Let $z,w$ be
local coordinates at $q$ associated to ($E$, $F$). Let $v_{s,t}$
  be the monomial valuation centered at $q$ such that $v(z) = s$ and $v(w) = t$. Then, the map $\pi_*$ induces a
  homeomorphism from the set $\left\{ v_{s,t} | s, t \geq 0, s b_*(E) + t b_*(F) = 1 \right\}$ and $[v_E, v_F]
  \subset \cV_*$ for the weak topology.

  Furthermore, the skewness function is given by
  \begin{equation}
    \alpha_* (\pi_* v_{s,t}) = \alpha (v_E) + \frac{t}{b_*(E)}
    \label{EqSkewnessOnSegment}
  \end{equation}
\end{prop}

\begin{proof}
  The first part of the proposition is exactly the content of \cite{favreValuativeTree2004} Proposition 6.35. We compute
  the skewness. It suffices to show \eqref{EqSkewnessOnSegment} for the divisorial valuations in $[v_E, v_F]$ and then
  use infinitely near sequences. We show \eqref{EqSkewnessOnSegment} by induction on the number of blowups. The result
  holds for $v_E = v_{1/b_* (E), 0}$ and $v_F = v_{0, 1/b_* (F)}$. Let $v_{s_1, t_1}, v_{s_2, t_2}$ be adjacent
  divisorial valuations such that \eqref{EqSkewnessOnSegment} holds. Let $E_1, E_2$ be the associated prime divisor and
  let $\tau : (Z, \Exc(\tau) \rightarrow (Y, E \cap F)$ be a completion exceptional above $E \cap F$ such that $E_1,
    E_2$ intersect in $Z$ and let $q = E_1 \cap E_2$. Let $(x,y)$ be local coordinates at $q$ associated to $(E_1, E_2)$
    and let $\tilde E$ be the exceptional divisor above $q$ and $\omega$ be the blow up of $q$. We want to compute $s,t \geq 0$ such that
    \begin{equation}
      (\tau \circ \omega)_* \ord_{\tilde E} = v_{s,t}.
      \label{<+label+>}
    \end{equation}
    To do so, we need to compute $\ord_{\tilde E} ((\tau \circ \omega)^* z)$ and $\left( (\tau \circ \omega)^*
    \ord_{\tilde E} w \right)$
    Let $b_1 = b_* (E_1)$ and $b_2 = b_* (E_2)$. We have by the first part of the proposition that
    \begin{equation}
      v_{E_i} = (\pi \circ \tau)_* \frac{1}{b_i} \ord_{E_i} = \pi_* (v_{s_i, t_i}).
      \label{EqComputation}
    \end{equation}
    Thus, $\tau_* \frac{1}{b_i} \ord_{E_i} = v_{s_i, t_i}$. In local coordinates $(u,v)$, $\omega$ is given by
    \begin{equation}
      \omega (u,v) = (u ,uv)
      \label{<+label+>}
    \end{equation}
    where $u = 0$ is a local equation of $\tilde E$ and $v = 0$ is a local equation of the strict transform of $E_2$. By
    \eqref{EqComputation}, we get up to multiplication by invertible germs of functions that
    \begin{equation}
      \omega^* (\tau^* z) = \omega^* \left( x^{s_1b_1} y^{s_2b_2} \right) = u^{s_1 b_1 + s_2 b_2 } v^{s_2 b_2}.
      \label{<+label+>}
    \end{equation}
    and
    \begin{equation}
      \omega^* (\tau^* w) = u^{t_1 b_1 + t_2 b_2} v^{t_2 b_2}
      \label{<+label+>}
    \end{equation}
    Thus, $s = s_1 b_1 + s_2 b_2$ and $t = t_1 b_1 + t_2 b_2$. This implies that
    \begin{equation}
      \pi_* v_{\frac{s_1 b_1 + s_2 b_2 }{ b_1 + b_2}, \frac{t_1 b_1 + t_2 b_2 }{b_1 + b_2}} = v_{\tilde E}.
      \label{<+label+>}
    \end{equation}
    We compute the skewness, by Proposition \ref{PropSkewnessBlowUp} we have that
    \begin{equation}
      \alpha_* (v_{\tilde E}) = \frac{b_1 \alpha_* (v_{E_1}) + b_2 \alpha_* (v_{E_2})}{b_1 + b_2}
      \label{<+label+>}
    \end{equation}
    and by induction, we get
    \begin{equation}
      \alpha_* (v_{\tilde E}) = \frac{b_1 (\alpha (v_E) + \frac{t_1}{b(E)}) + b_2 (\alpha (v_E) +
      \frac{t_2}{ b(E)})}{b_1 + b_2 } = \alpha (v_E) + \frac{t_1 b_1 + t_2 b_2}{ b_E (b_1 + b_2)}
      \label{<+label+>}
    \end{equation}
    and the result is shown by induction.
\end{proof}

\chapter{Different topologies over the space of valuations}\label{ChapterTopologiesValuations}
We define two topologies on the space of valuations centered at infinity. We saw in the previous chapter that the space
of valuations centered at infinity can be viewed as a space with an atlas of open subsets given by valuation trees. The
valuation tree comes with a weak and a strong topology and they glue together to define the weak and the strong topology
on the whole space of valuations centered at infinity.
\section{The weak topology}
Let $X_0$ be an affine surface and let $\cV_\infty$ be the space of valuations centered at infinity. We define
$\hat{\cV_\infty}$ to be the space of valuations centered at infinity modulo equivalence and $\eta : \cV_\infty
\rightarrow \hat{\cV_\infty}$ the quotient map. We define the weak
topology over $\cV_\infty$ as follows. A basis for the topology is given by
\begin{equation}
  \left\{ v \in \cV_\infty : t < v (P) < t' \right\}
  \label{<+label+>}
\end{equation}
for some $t, t' \in \R, P \in \k[X_0]$. A sequence $v_n$ of $\cV_\infty$ converges towards $v$ if and only if for every $P
\in \k[X_0]$, the sequence $v_n (P)$ converges towards $v(P)$. We define the weak topology over $\hat{\cV_\infty}$ to be the
thinnest topology such that $\eta : \cV_\infty \rightarrow \hat{\cV_\infty}$ is continuous with respect to the weak topology.

\begin{prop}\label{PropConvergenceDesCentres}
  Let $X$ be a completion of $X_0$. Let $v \in \cV_\infty$ and $(v_n)$ a sequence of elements of $\cV_\infty$.
  Suppose that $v_n \rightarrow v$ with respect to the weak topology. Then,
  \begin{itemize}
    \item If $c_X (v) = p$ is a closed point at infinity, then for all $n$ large enough $c_X (v_n) = p$.
    \item If $c_X (v) = E$ is a prime divisor at infinity, then for all $n$ large enough $c_X (v_n) \in E$.
  \end{itemize}
\end{prop}

\begin{proof}
  Suppose first that $c_X (v) = p$ is a closed point at infinity. Let $(x,y)$ be local coordinates at $p$. By
  definition of the center we have $v(x), v(y) >0$. We can
  find $P_1, P_2, Q_1, Q_2 \in \OO_X (X_0)$ such that $x = P_1 / Q_1, y = P_2 / Q_2$ and such that $v(Q_1), v (Q_2) \neq
  \infty$. Indeed by Lemma \ref{LemmeCentreValuation}, $\OO_{X,p}$ is a subring of $\OO_X (X_0)_{\p_v}$ where $\p_v
  = \{v = +\infty\}$. Now, we have that $v_n (P_i) \rightarrow v(P_i)$ and $v_n (Q_i) \rightarrow
  v(Q_i)$ as $n \rightarrow \infty$, therefore for all n large enough
  \begin{equation}
    v_n (x), v_n (y) > 0.
    \label{<+label+>}
  \end{equation}
  Thus, for all $n$ large enough $c_X (v_n) = p$.

  If $c_X (v) = E$, then $v = \lambda \ord_E $ for some $\lambda > 0$. Let $U$ be an open affine subset of $X$
  such that $U \cap E \neq \emptyset$. Let $z$ be a local equation of $E$ over $U$. Similarly, we can write $z = P / Q$
  with $v(Q) \neq \infty$. Since $v_n(P) \rightarrow v(Q)$ and $v_n (Q) \rightarrow v(Q)$, we get that $v_n (z)
  \rightarrow v(z) >0$. Therefore for $n$ large enough, $v_n (z) >0$ and therefore $c_X (v_n) \in E$.
\end{proof}

\begin{prop}\label{PropConvergenceSurFonctionRegulierePareilQueConvergenceLocale}
  Let $X$ be a completion and let $p \in X$ be a closed point at infinity. Let $v \in \cV_X (p)$ and $v_n \in
  \cV_X (p)$. Then, $v_n \rightarrow v$ weakly if and only if for every $\phi \in \OO_{X,p}, v_n (\phi)
  \rightarrow v(\phi)$.
\end{prop}

\begin{proof}
  Indeed, every $\phi \in \OO_{X,p}$ can be written as $\phi = \frac{P}{Q}$ with $v(Q) \neq \infty$. This shows one
  implication. Conversely, every $P \in \k[X_0]$ is of the form $\frac{\phi}{\psi}$ where $\phi, \psi \in \OO_{X,p}$.
  Furthermore, if $p \in E$ is a free point then $\psi = u^a$ where $a \in \Z_{\geq 0}$ and $u$ is a
  local equation of $E$. If $p = E \cap F$ is a satellite point, then $\psi = u^a v^b$ where $uv$ is a local
  equation of $E \cup F$. Now since $v_n$ and $v$ are valuations over $\k[X_0]$, they cannot be the curve valuations
  associated to a prime divisor at infinity. Therefore, for all $n, v_n (\psi) \neq \infty$ and $v(\psi) \neq \infty$.
  This shows the other implication.
\end{proof}

\begin{prop}\label{PropConvergenceLocale}
  Let $X$ be a completion of $X_0$ and let $p \in X$ be a closed point. Let $E$ be a prime divisor at infinity in
  $X$ such that $p \in X$. Let $\eta_p : \cV_X (p) \rightarrow \cV_X (p; E)$ be the natural map defined by
  $\eta_p (v) = \frac{v}{v(z)}$ where $z \in \OO_{X,p}$ is a local equation of $E$. Let $(v_n)$ be a sequence of
  $\cV_X (p)$ and let $v \in \cV_{X}(p)$. If $v_n \rightarrow v$ for the weak topology of $\cV_\infty$, then
  $\eta_p (v_n) \rightarrow \eta_p (v)$ for the weak topology of $\cV_X (p;E)$.
\end{prop}
\begin{proof}
  If $v_n \rightarrow v$ for the weak topology, then, $v_n (z) \rightarrow v(z)$ by Proposition
  \ref{PropConvergenceSurFonctionRegulierePareilQueConvergenceLocale}. Therefore $\eta_p (v_n) \rightarrow \eta_p
  (v)$, again by Proposition \ref{PropConvergenceSurFonctionRegulierePareilQueConvergenceLocale}. This shows the first
  implication.
\end{proof}

\begin{thm}\label{ThmWeakTopologyAtlas}
  Let $X$ be a completion of $X_0$. The weak topology on $\hat{\cV_\infty}$ is the topology induced by the open subsets
  $\cV_X (E; E)$ for all prime divisor $E$ at infinity.
\end{thm}

\begin{proof}
  Let $X$ be a completion at infinity and let $E$ be a prime divisor at infinity. Let $\cV_X (E)$ be the set of
  valuations $v$ over $\k[X_0]$ such that $c_X (v) \in E$ (this includes $c_X (v) = E$, i.e $v = \ord_E$). We have
  that
  \begin{equation}
    \cV_X (E)  = \left\{ \ord_E \right\} \cup \bigcup_{p \in E} \cV_X (p).
  \end{equation}
  Let $U_1, \cdots, U_r$ be a finite open affine cover of $E$ such that for every $i = 1,\cdots, r$ there exists $z_i
  \in \OO_X (U_i)$ a local equation of $E$. Then, every $z_i$ is of the form $z_i = P_i / Q_i$ with $P_i, Q_i \in \k[X_0]$.
  Then,
  \begin{equation}
  \cV_X (E)  = \bigcup_i \left\{ v(Q_i) < + \infty, v(P_i) - v(Q_i) > 0 \right\}
  \label{<+label+>}
\end{equation}
and, it follows that $\cV_X (E)$ is an open subset of $\cV_\infty$.
Set $\hat{\cV_\infty} (p) := \eta (\cV_X (p))$. Define a map $\sigma_p:
\hat{\cV_\infty} (p) \rightarrow \cV_X (p; E) \setminus \left\{ \ord_E \right\} \subset \cV_X (p)$ by
\begin{equation}
  \sigma_p ([v]) = \eta_p (v)
\label{<+label+>}
\end{equation}
where $\eta_p$ is the map from Proposition \ref{PropConvergenceLocale} and $[v]$ is the class of $v$ in
$\hat{\cV_\infty}$. By Proposition \ref{PropConvergenceLocale},
$\sigma_p$ is a continuous section of $\eta_{|\cV_X (p)} : \cV_X (p) \rightarrow \hat{\cV_\infty} (p)$. Still by
Proposition \ref{PropConvergenceLocale}, the map $\sigma_p: [\ord_E] \cup \hat{\cV_\infty} (p) \rightarrow \cV_X
(p; E)$ extended by $\sigma_p ([\ord_E]) = \ord_E$ is also a continuous section of $\eta : \left\{ \lambda \ord_E :
\lambda > 0 \right\} \cup \cV_X (p) \rightarrow \left\{ [\ord_E] \right\} \cup \hat{\cV_\infty} (p)$. These maps
$\sigma_p$ glue together to give a continuous section $\sigma_E : \hat{\cV_\infty} (E) \rightarrow \cV_X (E;E) \subset
\cV_X (E)$ of $\eta : \cV_X (E) \rightarrow \hat{\cV_\infty} (E)$.

To finish the proof we need to understand the behaviour of $\sigma_F, \sigma_E$ on
\begin{equation}
  \hat{\cV_\infty} (E) \cap
  \hat{\cV_\infty} (F) = \hat{\cV_\infty} (p)
\end{equation}
for $p = E
\cap F$ where $E,F$ are two prime divisors at infinity. By Proposition \ref{PropRelativeValuativeTreeIsomorphism}, we have that
the map $N_{p,F} \circ {N_{p,E}}^{-1} : \cV_X (p; E) \setminus \left\{ \ord_E \right\} \rightarrow \cV_X (p; F)
\setminus \left\{ \ord_F \right\}$ is a homeomorphism and we have
\begin{equation}
  (\sigma_F)_{|\hat{\cV_\infty} (p)} = (N_{p,F} \circ {N_{p,E}}^{-1}) \circ (\sigma_E)_{|\hat{\cV_\infty} (p)}
\label{<+label+>}
\end{equation}
\end{proof}

\section{The strong topology}
Let $R = \k [ [ x,y ] ]$ and let $\m = (x,y)$. Let $\cV_*$ be the valuative tree with either the normalization by
$\m$ or with respect to a curve $z$. We will write $\alpha_*$ for the skewness function over $\cV_*$. We
consider a stronger topology on $\cV_*$. Let $\cV_*^{\qm}$ be the
subset of quasimonomial valuations. We define the following distance
\begin{equation}
  d(v_1, v_2) = \alpha(v_1) - \alpha(v_1 \wedge v_2) + \alpha (v_2) - \alpha (v_1 \wedge v_2).
  \label{<+label+>}
\end{equation}
The topology induced by this distance is called the \emph{strong} topology.

\begin{lemme}\label{LemmeWeakStrongTopologiesSameOnSegments}
  Let $v_1 < v_2$ be two quasimonomial valuations in $\cV_*$. Then, the weak topology and the strong topology on $[v_1,
  v_2]$ are the same.
\end{lemme}
\begin{proof}
  Let $\phi \in \m$ be irreducible such that the curve valuation $v_\phi$ satisfies $v_\phi > v_2$. Then we have by
  Proposition \ref{PropValuationEnFonctionDeAlpha} and \ref{PropValuationEnFonctionDeAlphaRelative} that
  \begin{equation}
    v(\phi) = \alpha_* (v) m_* (\phi).
    \label{<+label+>}
  \end{equation}
  This shows that the topologies are the same on the segment $[v_1, v_2]$.
\end{proof}

\begin{prop}[\cite{favreValuativeTree2004} Proposition 5.12] We have the following
  \begin{itemize}
    \item The strong topology is stronger than the weak topology.
    \item The closure of $\cV_{*}^{\qm}$ with respect to the strong topology is the subspace of $\cV_*$ consisting of
      valuations of finite skewness.
  \end{itemize}
\end{prop}

\begin{prop}\label{PropHomeoForStrongTopology}
  Let $R = \k [ [ z , w ] ] $ and let $\cV_\m, \cV_z, \cV_w$ be the three valuation trees. Let $\cV_\m ', \cV_z ', \cV_w
  '$ be the three subtrees of valuations of finite skewness. Then, the maps
  \begin{equation}
    N_z : \cV_\m ' \rightarrow \cV_z ' \setminus \left\{ \ord_z \right\}, \quad N_w \circ {N_z}^{-1} : \cV_z '
    \rightarrow \cV_w '
    \label{<+label+>}
  \end{equation}
  are homeomorphisms with respect to the strong topology.
\end{prop}
This follows from Proposition \ref{PropRelationSkewness}.

Let $\cV_\infty '$ be the subset of $\cV_\infty$ of valuations of finite skewness, this set is well defined thanks to
Proposition \ref{PropFiniteSkewnessEverywhere}. We define the \emph{strong topology} on $\cV_\infty '$ as follows. First
define the strong topology on $\hat{\cV_\infty} ' := \eta (\cV_\infty ')$ using the notations from the proof of Theorem
\ref{ThmWeakTopologyAtlas}.
Consider the map $\sigma_E : \hat{\cV_\infty} ' \cap \hat{\cV_\infty} (E) \rightarrow \cV_X (E; E) '$. We define the
strong topology on $\hat{\cV_\infty} ' \cap \hat{\cV_\infty}(E)$ as the coarsest topology such that $\sigma_E$ is
continuous for the strong topology on $\cV_X (E; E) '$. This defines a topology on $\hat{\cV_\infty} '$ thanks to
Proposition \ref{PropHomeoForStrongTopology}.

\begin{cor}\label{CorApproximatingSequenceStrongConvergence}
  Let $v$ be a valuation centered at infinity, let $X$ be a completion of $X_0$ and let $(v_n)$ be the infinitely near
  sequence of $v$ from Proposition \ref{PropInfinitelyNearSequence}. If $v \in \cV_\infty '$, then $\eta(v_n)$
  converges towards $\eta(v)$ with respect to the strong topology.
\end{cor}

\begin{proof}
  Let $p = c_X (v)$ and we can suppose that $v_n, v \in \cV_X
  (p; E)$ for some prime divisor $E$ at infinity with $p \in E$. Then, we have $v_n \leq v$ for all $n$ and
  $\alpha(v_n) \rightarrow \alpha (v)$. Therefore \begin{equation}
    d(v_n, v) = \alpha(v) - \alpha(v_n) \xrightarrow[n \rightarrow \infty]{} 0
    \label{<+label+>}
  \end{equation}
\end{proof}

\chapter{Valuations as Linear forms}\label{ChapterValuationsAsLinearForms}
As done in \cite{jonssonValuationsAsymptoticInvariants2012}, we can view valuations on $X_0$ as
\begin{itemize}
    \item linear forms with values in $\R$ over the space of integral Cartier Divisors over $X$ supported at infinity
      \item as real-valued functions over the set of coherent fractional ideal sheaves of $X$ co-supported at
        infinity.
  \end{itemize}
  We recall how to do so. For a divisor $D$, we denote by $H^0 (X, \OO_X(D))$ the set of global sections of the line
  bundle $\OO_X (D)$ and
  \begin{equation}
    \Gamma(X, \OO_X (D)) = \left\{ h \in \k(X)^\times : D + \div(h) \geq 0 \right\}.
  \end{equation}

  \section{Valuations as linear forms over $\DivInf (X)$}
  \begin{lemme}\label{LemmeGoodFormOfLocalEquation}
    Let $D \in \Div(X)$ such that the negative part (if any) of $D$ is supported in $\partial_{X} X_0$. For
    any point $p \in X$, there exists an open neighbourhood $U$ of $p$ such that a local equation of
    $D$ on $U$ is of the form $\phi = P \cdot \psi$ with $P \in \OO_X (X_0)$ and $\psi \in \OO_X(U)$.
  \end{lemme}

  \begin{proof}
    Let $\phi \in \k(U')^* = \k(X)^*$ be a local equation of $D$ where $U'$ is an open subset of $X$
    containing $p$.

    Let $H$ be an ample effective divisor such that
    $\Supp (H) = \partial_X X_0$. There exists an integer $n \geq 1$ such that $D +nH \geq 0$. Pick $P$
    general in $\Gamma (X, \OO_X (nH)) \subset \OO_X (X_0)$, then $\div P = Z_P - n H$ with $Z_P \geq 0$ and $p \not
    \in \supp Z_P$ because we chose $P$ general and $\left| nH \right|$ is basepoint free, in particular $P$ restricts to a
    regular function over $X_0$. Set $\psi := \phi / P$, one has

    \begin{equation}
    \div \left( {\psi}_{|U} \right) = D_{|U} + nH_{|U} - Z_{P|U}. \end{equation}

    Set $U = U' \setminus \supp Z_P$, then $\div (\psi)_{|U'} \geq 0$, i.e $\psi \in \OO_X (U)$ and we are done.
  \end{proof}

  \begin{cor}\label{CorValuationDefinieSurSectionLocale}
    If $D$ is a divisor such that the negative part (if any) of $D$ is at infinity and $v$ is a valuation on $\k[X_0]$,
    then for all small enough affine open subsets $U \subset X$ containing $c_{X}(v)$,
    \begin{equation}
      \Gamma(U, \OO_{X}(-D)) \subset \OO_X(X_0)_{\p_{v_X}}
    \end{equation}
    and $v_X$ can be extended to $\Gamma(U, \OO_X(-D))$.
  \end{cor}

  \begin{proof}
    If $U$ is small enough, then $\Gamma(U, \OO_X(-D))$ is the $\OO_X (U)$-module
      generated by $\phi$ where $\phi$ is a local equation of $D$. Now, by Lemma
      \ref{LemmeGoodFormOfLocalEquation}, $\phi$ is of the form $\phi = P \cdot \psi$ where $P \in \OO_X
      (X_0)$ and $\psi \in \OO_X (U)$. By definition we have $\OO_X(X_0) \subset \OO_X(X_0)_{\p_{v_X}}$ and for all
      affine open neighbourhood $U$ of $c_X (v), \OO_X (U) \subset \OO_X (X_0)_{\p_{v_X}}$ by the proof of Lemma
      \ref{LemmeCentreValuation}.
  \end{proof}

  Let $D$ be divisor of $X$ supported at infinity and let $\phi \in \k(X)$ be a local equation of $D$ at
  $c_{X}(v)$. Then we set
  \begin{equation}
    L_{v,X}(D) := v_{X}(\phi).
  \end{equation}
  This is well defined because by Corollary \ref{CorValuationDefinieSurSectionLocale} because by definition there exists
  an open affine neighbourhood $U$ of $c_X (v)$ such that $\phi \in \Gamma (U, \OO_X (-D))$.
  This does not depend on the choice of the local equation because if $\psi$ is another local equation of $D$, then
  $\frac{\phi}{\psi}$ is a regular invertible function near $c_{X}(v)$ and $v_{X}(\phi / \psi) = 0$.

  \begin{lemme}\label{LemmeValuationJamaisInfinieSurLesDiviseurs}
    Let $v$ be a valuation over $\k[X_0]$ and let $X$ be a completion of $X_0$, then for all $D \in \DivInf
    (X)_\R,\phantom{.} L_{v,X} (D) < \infty$.
  \end{lemme}

  \begin{proof}
    It suffices to show Lemma \ref{LemmeValuationJamaisInfinieSurLesDiviseurs} for $D$ an integral divisor
    supported at infinity in $X$. We can apply corollary \ref{CorValuationDefinieSurSectionLocale} to $D$ and $-D$,
    therefore if $\phi$ is a local equation of $D$, we have that both $\iota_X^* (\phi)$ and $\iota_X^*(1/\phi)$
    belong to $A_{\p_v}$ and this means that $v_X(\phi) < \infty$.
  \end{proof}

  \begin{rmq}
    We can in fact define $L_{v,X}$ at any divisor $D$ on $X$ such that the negative part of $D$ is supported
    at infinity but it could happen that $L_{v,X}(D)$ is infinite. For example, let $X_0 = \A^2, X = \P^2$. Let
    $v$ be the curve valuation centered at $[1 :0 :0]$ associated to the curve $y = 0$, then
    \begin{equation}
      L_{v, \P^2} (\left\{ Y= 0  \right\} - \left\{ Z = 0 \right\}) = v (Y/ Z) = + \infty.
      \label{<+label+>}
    \end{equation}
  \end{rmq}

  \begin{ex}
    If $X$ is a completion of $X_0$, let $E$ be a prime divisor at infinity. Let $D \in \DivInf (X)$. Recall that we
    have defined in Section \ref{SubSecCompletions} that $\ord_E(D)$ is the weight of $D$ along $E$, then
    \begin{equation}
      L_{\ord_E} (D) = \ord_E (D).
      \label{EqValuationDisivorielle}
    \end{equation}
    Indeed, at the generic point of $E$, a local equation of $D$ is $z^{\ord_E (D)} \phi$ where $z$ is a local equation
    of $E$ and $\phi$ is regular not divisible by $z$.
  \end{ex}

  \begin{prop}\label{PropValuationForCartierDivisorOverOneCompletion}
    If $v$ is a valuation over $\k[X_0]$, and $X$ is a completion of $X_0$ then
    \begin{enumerate}
      \item $L_{v, X}(0_{\DivInf (X)}) = 0$.
      \item For any $D, D' \in \DivInf(X), L_{v,X}(D + D') = L_{v,X}(D) + L_{v,X}(D')$, and $L_{v, X}
        (mD) = m L_{v,X} (D)$ for all $m \in \Z$.
      \item If $D \geq 0$, then $L_{v,X}(D) \geq 0$ and $L_{v,X}(D) > 0 \Leftrightarrow c_{X} (v) \in \supp D$. In
        particular, if $v$ is not centered at infinity then $L_v = 0$.
      \item If $P \in \OO_{X}(X_0)$, then $v_{X}(P) = L_{v,X}(\div P)$.
      \item If $Y$ is another completion of $X_0$ above $X$, and $\pi: Y \rightarrow X$ is the
        morphism of completions over $X_0$, then $L_{v,X}(D) = L_{v,Y} (\pi^* D)$.
    \end{enumerate}
  \end{prop}
  Thus, we can extend $L_{v,X}$ to $\DivInf (X)_\R$ by linearity:
  \begin{equation}
    L_{v, X} : \DivInf (X)_\R \rightarrow \R.
    \label{<+label+>}
  \end{equation}

  \begin{proof}
    The first assertion is trivial as 1 is a local equation of the trivial divisor. The second assertion
    follows from the fact that if $\phi, \psi$ are local equations of $D$ and $D'$ respectively, then $\phi \psi$ is a
    local equation of $D + D'$ and $1/ \phi$ is a local equation of $-D$. For the third one, suppose $D$ is an integral
    divisor. If $D$ is effective
    and $f$ is a local equation at
    $c_{X}(v)$, then $f$ is regular at $p$ and by definition of the center $v(f) \geq 0$, now if $c_{X}(v)$ belongs to
    $\supp D$, then $f$ vanishes at $c_{X} (v)$; thus, $v(f) > 0$. If on the other hand $c_{X}(v) \not \in \supp D$, then
    $f$ is invertible at the center of $v_{X}$ and $v_{X}(f) = 0$. The fourth assertion follows from
    $f$ being a local equation of $\div(f)$ and the fact that $f$ has no pole over $X_0$. Finally, if $f
    \in \k(X)$ is a local equation of $D$ at $c_{X}(v)$, then $\pi^* f$ is a local equation of $\pi^* D$
    at $c_{Y}(v)$ and by Remark \ref{RmqMemeValuationApresEclatement}, one has $v_{X}(f) =
    v_{Y}(\pi^* f)$.
  \end{proof}

  \begin{prop}\label{PropPushForwardOnValuationsIsPullbackOnDivisors}
    Let $f: X_0 \rightarrow X_0$ be a dominant endomorphism of $X_0$. Let $Y, X$ be two completions of $X_0$ such
    that the lift $F: Y \rightarrow X$ of $f$ is regular. Then,
    \begin{equation} F (c_Y (v)) = c_X (f_* v) \text { and } \forall D \in \DivInf(X), L_{f_*v, X} (D) =
    L_{v, Y} (F^* D).
  \end{equation}
  \end{prop}

  \begin{proof}
    Let $p = c_Y (v)$ and $q = F(p)$. Then, $F$ induces a local ring homomorphism
    \[ F^* : \OO_{X, q} \rightarrow \OO_{Y,p} \]
    Now, for any $\phi \in \OO_{X,q}$, there exists $P,Q \in \k[X_0]$ such that $\phi = \frac{P}{Q}$. Therefore,
    \[ F^* \phi = \frac{f^* P}{ f^* Q} \]
    and therefore $f_* v (\phi) = v (F^* \phi) \geq 0$ and $f_*v (\phi) > 0$ if $\phi \in \m_{X,q}$. Thus, $q = c_X (f_* v)$.

    Now, to show the second result we distinguish two cases whether or not $f_* v$ is centered at infinty. Let $D \in
    \DivInf (X)$ and let $\phi$ be a local
      equation of $D$ at the center of $v_{X}$, then
      $F^* \phi$ is a local equation of $F^* D$ at the center of $v_{Y}$. Since $\pi_* v_{Y} = v_{X}$, one has

      \begin{equation}
      v_{Y} (F^* \phi) = v_{X} ((F \circ {\pi}^{-1} )^* \phi) = v_{X} (f^*\phi) = (f_* v)_{X} (\phi) 
    \end{equation}
    and this shows the result. In particular, note that if $f_{*} v$ is not centered at infinity, then the local
    equation $\phi$ of $D$ at $c_X (v)$ is actually invertible and $v_Y (F^* g) = 0$ so that $L_{f_*v, X} = 0$.
  \end{proof}

  \begin{rmq}\label{rmq:f-not-proper}
    In particular, if $f$ is not proper and some valuation $v$ centered at infinity is such that $f_*v$ is centered over
    $X_0$, then $L_{f_* v, X} = 0$.
  \end{rmq}

  \section[Valuations as functions over the set of fractional ideals]{Valuations as real-valued functions over the set
  of fractional ideals co-supported at infinity in $X$}
  An \emph{ideal} of $X$ is a sheaf of ideals of $\OO_{X}$ and a \emph{fractional ideal} is a
  coherent sub-$\OO_{X}$-module of the constant
  sheaf $\k(X)$. Let $\aa$ be a fractional ideal of $X$, we say that $\aa$ is \emph{co-supported} at
  infinity if $\aa_{|X_0} = \OO_{X_0}$. For example, for any divisor $D \in \Div (X)$, $\OO_X (D)$ is a
  fractional ideal of $X$ and if $D \in \DivInf (X)$ then $\OO_X (D)$ is co-supported at infinity.

  \begin{prop}\label{PropGoodFormofLocalGenerators}
    Let $\aa$ be a fractional ideal of $X$ co-supported at infinity and let $p \in X$, for any finite
    system $(f_1, \cdots, f_r)$ of local generators of $\aa$ at $p$ there exists an open neighbourhood
    $U$ of $p$ such that $f_{i|U}$ is of the form
    \begin{equation}
      f_i = F_i g_i
    \end{equation}
    with $F_i \in \OO_X(X_0)$ and $g_i \in \OO_X(U)$.
  \end{prop}

  \begin{proof}

    Pick $U'$ an open neighbourhood containing $p$.
    Since $f_i$ is regular over $X_0$, we have $\div f_i = D^+ - D_1^- - D_2^-$ where $D^+, D_1^-$ and
    $D_2^-$ are effective divisors such that $\supp D_1^-
    \subset \partial_X X_0$ and $ {D_2^-}_{|U'} = 0$. By Lemma \ref{LemmeGoodFormOfLocalEquation} there
    exists an open neighbourhood $U_i \subset U'$ of $p$ such that $(D^+ - D_1^-)_{|U_i} = \div F_i g_i'$ with $F_i \in
    \OO_X(X_0)$ and $g_i' \in \OO_X (U_i)$. Therefore, there exists $g_i '' \in \OO_X (U_i)$ such that
    $f_i = F_i g_i ' g_i ''$. Set $U = \cap U_i$ and $g_i = g_i ' g_i ''$.

  \end{proof}

  \begin{cor}\label{CorValuationDefinieSurSectionLocaleIdeaux}
    Let $\aa$ be a fractional ideal co-supported at infinity and let $v$ be a valuation over $\k[X_0]$, then for
    all affine open neighbourhood of $c_{X} (v), \Gamma(U, \aa) \subset \OO_X(X_0)_{\p_{v_X}}$ and
    $v_X$ is defined over $\Gamma(U, \aa)$.
  \end{cor}

  If $v$ is a valuation over $\k[X_0]$, then we define $L_{v,X}(\aa)$ as
  \begin{equation}
  L_{v,X}(\aa) := \min_{f} v_{X}(f). \end{equation}
  where
  the $f$ runs over the germs of sections of $\aa$ at $c_{X} (v)$. This makes sense by Corollary
  \ref{CorValuationDefinieSurSectionLocaleIdeaux}.

  \begin{prop}\label{PropValuationForCoherentSheafOfIdealsForOneCompletion}
    If $v$ is a valuation over $\k[X_0]$, then
    \begin{enumerate}
      \item $L_{v,X}(\OO_{X}) = 0$.
      \item If $\mathfrak a, \mathfrak b$ are two fractional ideals of $X$ co-supported at infinity, then

        \begin{equation}
        L_{v,X} (\aa \cdot \bb) = L_{v,X}(\aa) + L_{v,X}(\bb) \text{ and } L_{v,X}(\aa + \bb) = \min(L_{v,X}(\aa), L_{v,X}(\bb)) \end{equation}

      \item If $f_1, \cdots, f_r \in \k(X)$ is a set of local generators of $\aa$ at $c_X (v)$, then

        \begin{equation}
        L_{v,X}(\aa) = \min (v_X(f_1), \cdots, v_X(f_r)). \end{equation}

      \item If $D \in \Div(X)$ is a divisor, then $L_{v,X}(D) = L_{v,X}(\OO_{X}(-D))$.
      \item If $Y$ is another completion of $X_0$ above $X$, and $\pi: Y \rightarrow X$ is the
        morphism of completions over $X_0$, then $\tilde \aa := \pi^* \aa \cdot \OO_{Y}$ is
        a fractional ideal over $Y$ and $L_{v,X}(\aa) = L_{v,Y}(\tilde \aa)$.
    \end{enumerate}
  \end{prop}

  \begin{proof}
    The first assertion is trivial since $1$ is a local generator of the trivial sheaf. For Assertion (2), notice that if
    $(f_1, \ldots, f_r)$ are local generators of $\aa$ at $c_X (v)$ and $(g_1, \ldots, g_s)$ local generators of $\bb$
    at $c_X (v)$ then $(f_i g_j)_{i,j}$ is a set of local generators of $\aa \cdot \bb$ at $c_X (v)$ and $(f_1,
    \ldots, f_r, g_1, \ldots, g_s)$ is a set of local generators of $\aa + \bb$ at $c_X (v)$, so Assertion (2) follows
    from Assertion (3). To show Assertion (3), let $f_1, \cdots, f_r$ be local generators of $\aa$ at $c_X (v)$. This
    implies that $\aa_{c_X (v)} = f_1 \OO_{c_X (v)} + f_2 \OO_{c_X (v)} + \cdots + f_r \OO_{c_X (v)}$. Since
    $v$ is nonnegative on $\OO_{c_X (v)}$ by
    definition of the center, the assertion follows.
    For assertion 5, if $f_1, \cdots, f_r$ are local generators of $\aa$, then $\pi^* f_1, \cdots, \pi^* f_r$ are local
    generators of $\tilde \aa$ at $c_{Y}(v)$ and the result follows since $\pi_* v_{Y} = v_{X}$. Assertion (4)
    follows from the fact $\OO_X (-D)$ is locally generated by an equation of $D$ and Assertion (5) follows from the
    fact that if $(f_1, \cdots, f_r)$ are local generators of $\aa$ at $c_X (v)$ then $\left( \pi^* f_1, \cdots, \pi^*
    f_r \right)$ are local generators of $\tilde \aa$ at $c_Y (v)$.
  \end{proof}

  \begin{prop}
    If $v$ is a valuation over $\k[X_0]$ and $\aa$ is a fractional ideal co-supported at infinity, then
    $L_{v,X} (\aa) < \infty$.
  \end{prop}

  \begin{proof}
    Take $f_1, \cdots, f_r$ local generators of $\aa$ at $p$ the center of $v$ on $X$. The proof of
    Lemma \ref{LemmeGoodFormOfLocalEquation} shows that there exists an affine open neighbourhood $U$
    of $p$ such that $f_{i | U} = h_i g_i$ with $h_i \in \k[X_0]$ and $g_i \in \OO_X (U)$ and such that ${f_i}^{-1}$ can be put into the same form. This shows that for all $i$, $v(f_i) < \infty$.
  \end{proof}

  \begin{rmq}\label{RmqValuationDefinieSurIdeaux}
    The same definition would allow one to define $L_{v,X}(\aa)$ for any fractional ideal such that $\aa$ is a sheaf of
    ideals of $X_0$ but we have to allow infinite values. In particular, $L_{v,X} (\aa)$ is defined for any sheaf of
    ideals over $X$.
  \end{rmq}

  \section{Valuations centered at infinity}

  Recall that a valuation $v$ over $\k[X_0]$ is \emph{centered at infinity}, if $v$ does not admit a center on $X_0$. We denote by
  $\Vinf$ the set of valuations over $\k[X_0]$ centered at infinity.

  \begin{lemme}\label{LemmeCaracValuationCentreeAlInfini}
    Let $v$ be valuation over $\k[X_0]$. The following assertions are equivalent.

    \begin{enumerate}
      \item $v$ is centered at infinity.
      \item There exists $P \in \k[X_0]$ such that $v(P) < 0$.
      \item For any completion $X$ of $X_0$ and any effective divisor $H$ in $X$ such that $\supp H =
        \BD$, one has $L_{v, X}(H) >0$.
      \item There exists a completion $X$ of $X_0$ and an effective divisor $H \in X$ with $\supp H =
        \BD$ such that $L_{v, X}(H) > 0$.
    \end{enumerate}
  \end{lemme}

  \begin{proof}
    We will show the following implications $2 \Rightarrow 1 \Rightarrow 3 \Rightarrow 4$. Then, we will show
    that $1 \Rightarrow 2$ and finally that $4 \Rightarrow 2$.

    \paragraph{$2 \Rightarrow 1 \Rightarrow 3 \Rightarrow 4$}  If there exists a regular function $P$ over $X_0$ such
    that $v(P) <0$ then the center of $v$ cannot be a
    point of $X_0$ because $P$ is regular at every point of $X_0$. This shows $2 \Rightarrow 1$, then if $v$ is
    centered at infinity, take a completion $X$ of $X_0$, let $E$ be a prime divisor at infinity in $X$
    such that $c_{X}(v) \in E$. Then, since $H$ is effective and $E \in \supp H$, $L_{v,X}(H) \geq v(E) > 0$ by
    Proposition \ref{PropValuationForCartierDivisorOverOneCompletion} (1). This
    shows $1 \Rightarrow 3$ and $3 \Rightarrow 4$ is clear.

    \paragraph{$1 \Rightarrow 2$} Conversely, suppose $v$ is centered at
    infinity and fix a closed embedding $X_0 \hookrightarrow \A^N$ for some integer $N$. Let $X$ be the Zariski
    closure of $X_0$ in $\P^N$ with homogeneous coordinates $x_0, \cdots, x_N$ such that $\left\{ {x_0 = 0} \right\}$ is
    the hyperplane at infinity. The surface $X$ might not be smooth so it is not necessarily a completion of $X_0$ but it
    still is proper and the center $p$ of $v$ on $X$ belongs to $\left\{ {x_0 = 0} \right\} \cap X$. Let $1 \leq i \leq
    N$ be an integer such that $p$ belongs to the open subset $\left\{ {x_i \neq 0} \right\}$. Then, the rational function
    $P := \frac{x_i}{x_0}$ is a regular function on $X_0$ and $1/P$ vanishes at $p$. Therefore, $v
    (P) < 0$.
    \paragraph{$4 \Rightarrow 1$} Suppose that $v$ is not centered at infinity, i.e the center of $v$
    belongs to $X_0$. Then, for any completion $X$ and for any divisor $D \in \DivInf (X)$,
    one has $L_{v,X}(D) = 0$ by Proposition \ref{PropValuationForCartierDivisorOverOneCompletion} (1) since $c_{X}(v) \not
    \in \supp D$.
  \end{proof}

  This lemma shows that being centered at infinity is a property that can be tested on only one completion
  $X_0$.

  \begin{cor}
    The space $\Vinf$ is an open subset of $\cV$.
  \end{cor}

  \begin{proof}
    We have by Lemma \ref{LemmeCaracValuationCentreeAlInfini} that
    \begin{equation}
      \Vinf = \bigcup_{P \in \k[X_0]} \left\{ v(P) < 0 \right\}.
      \label{<+label+>}
    \end{equation}
    Therefore, it is a union of open subsets.

  \end{proof}

\subsection{Topologies over the set of valuations centered at infinity}
  Let $X$ be a completion of $X_0$.
  Call $\sigma$ the coarsest topology such that the evaluation maps $\phi_f: v \in \Vinf \mapsto v(f)$ are continuous
  for all $f \in \k[X_0]$ and $\tau$ the coarsest topology such that the evaluation maps $\psi_\aa: v \in \Vinf \mapsto L_v
  (\aa)$ are continuous for all fractional ideals $\aa$ of $X$ such that $\aa_{|X_0}$ is a sheaf of ideals over $X_0$.
  Recall that we allow in both cases infinite values.
  \begin{prop}{\cite{jonssonValuationsAsymptoticInvariants2012}} \label{PropTopologiesOverIdealsAndFunctionsAreTheSame}
    These two topologies on $\cV$ are the same.
  \end{prop}

  \begin{proof}
    First if $f \in \k[X_0]$, then $v(f) = L_v ( (f))$ where $(f)$ is the fractional ideal generated by $f$. So
    $\sigma$ is finer than $\tau$. For the other way, Let $H$ be an ample divisor supported at infinity and let $\aa$ be
    a fractional ideal co-supported at infinity. There exists an integer $n >0$ such that
    $\aa \otimes \OO_X (nH)$ and $\OO_X(nH)$ are generated by global sections $(f_1,  \cdots, f_r)$ and
    $(g_1, \cdots, g_s)$ respectively. Notice that for all $i,j$, the rational functions $f_i, g_j$ belong
    to $\OO_X (X_0)$. Now, we have that $L_v (\aa)  = L_v (\aa \otimes \OO_X (nH) \otimes \OO_X (-nH))$, therefore
    \[ L_v (\aa) = \min_{i,j} \left( v \left( \frac{f_i}{g_j} \right) \right) = \min_{i,j} \left( v(f_i) -
      v (g_j) \right)
    \]
      Therefore, $\tau$ is finer than $\sigma$ and the result is shown.

  \end{proof}

  \subsection{Valuations centered at infinity as linear forms over
  $\Cinf$}\label{SecValuationAsLinearFormsOverDivisors}
  \begin{dfn} Let $v$ be a valuation over $\k[X_0]$. Let $D \in \Cinf$ and $X$ be a completion of $X_0$ such that $D$ is defined by $D_X$. We define
    \begin{equation}
    L_v(D) := L_{v,X} (D_{X}).
  \end{equation}
  This does not depend on the choice $X$ and defines a linear map on $\Cinf$ by Proposition
    \ref{PropValuationForCartierDivisorOverOneCompletion} and $L_v (D) < + \infty$ by Lemma
    \ref{LemmeValuationJamaisInfinieSurLesDiviseurs}.
    Notice that $L_v = 0$ if and only if $v$ is not centered at infinity.
  \end{dfn}

  \begin{prop}\label{prop:property_plus}
    If $v$ is a valuation on $\k[X_0]$ centered at infinity then $L_v$ is a linear form $L_v: \Cinf \rightarrow \R$
    and satisfies
    \begin{enumerate}
      \item If $D \geq 0$, then $L_v(D) \geq 0$.
      \item For $D, D' \in \Cinf, L_v(D \wedge D') = \min (L_v(D), L_v(D'))$.
    \end{enumerate}
    We will say that an element of $\Hom (\Cinf, \R)$ that satisfies these 2 properties satisfies property (+).
  \end{prop}

  \begin{proof}
    Assertion 1 follows from Proposition \ref{PropValuationForCartierDivisorOverOneCompletion} (3).  We show the second
    assertion. Take $D, D' \in \Cinf$ and $X$ a completion of
    $X_0$ such that $D, D'$ are defined by their incarnation $D_{X}, D_{X}'$.
    By Claim \ref{ClaimPullBackOfIdealSheafIsMinimumCartierClass} (that we prove in the next section), we know that there exists a completion $Y$
    along with a morphism of completions $\pi: Y \rightarrow X$ such that $D \wedge D'$ is the Cartier class
    determined by some divisor $D_{Y}$ in $Y$ such that $\pi^* (\OO_{X}(-D_{X}) + \OO_{X}(-D_{X} '
    )) \cdot \OO_{Y} = \OO_{Y}(-D_{Y})$. Using Proposition
    \ref{PropValuationForCoherentSheafOfIdealsForOneCompletion}, it follows that

    \begin{align*}
      L_v(D \wedge D') &= L_{v,Y}(D_{Y}) \\
      &= L_{v,Y}(O_{Y}(-D_{Y})) \quad \ref{PropValuationForCoherentSheafOfIdealsForOneCompletion}
      (4) \\
      &= L_{v,X} (\OO_{X}(-D_{X}) + \OO_{X}(- D_{X}')) \quad \ref{ClaimPullBackOfIdealSheafIsMinimumCartierClass}\\
      &= \min(L_{v,X}(\OO_{X}(- D_{X})), L_{v,X}(\OO_{X}(-D_{X}'))) \quad
      \ref{PropValuationForCoherentSheafOfIdealsForOneCompletion} (2) \\
      &= \min (L_v(D), L_v(D')) \quad \ref{PropValuationForCoherentSheafOfIdealsForOneCompletion} (4)
    \end{align*}
  \end{proof}

  For any linear form $L \in \Hom (\Cinf, \R)$ and any dominant endomorphism $f : X_0 \rightarrow X_0$ we can define the
  \emph{pushforward} $f_*L \in \Hom (\Cinf, \R)$ by 
  \begin{equation}
    \forall D \in \Cinf, \quad f_* L (D) := L (f^*D).
    \label{<+label+>}
  \end{equation}

    \begin{prop}\label{PropPushForwardOnValuationsIsPushforwardOnLinearMap}
      Let $v$ be a valuation over $\k[X_0]$ and $f : X_0 \rightarrow X_0$ a dominant endomorphism, then for all $D \in \Cinf,$

      \begin{equation}
        L_{f_* v} (D) = L_v (f^* D) = (f_* L_v )(D)
      \end{equation}

    \end{prop}

    \begin{proof}
      Let $X$ be a completion of $X_0$ where $D$ is defined, then $f$ induces a dominant rational map $f : X
      \rightarrow X$. Let $\pi: Y \rightarrow X$ be a projective birational morphism such that the lift $F: Y
      \rightarrow X$ is regular. Since $f$ is an endomorphism of $X_0$ we can suppose that $\pi$ is the identity over
      $X_0$, hence $Y$ is a completion of $X_0$ and $\pi$ is a morphism of completions. Now, if $\phi$ is a local
      equation of $D$ near the center of $v_{X}$, then
      $F^* \phi$ is a local equation of $F^* D$ near the center of $v_{Y}$. Since $\pi_* v_{Y} = v_{X}$, one has

      \begin{equation}
      v_{Y} (F^* g) = v_{X} ((F \circ {\pi}^{-1} )^* g) = v_{X} (f^*g) = (f_* v)_{X} (g) \end{equation}
    \end{proof}
   We equip $\Hom(\Cinf, \R)$ with the weak-$\star$ topology,
  that is the coarsest topology such that the map $L \in \Hom(\Cinf, \R) \mapsto L(D)$ is continuous
  for all $D \in \Cinf$. We extend $L_v$ to $\Cinf_\R$ by linearity.

  \begin{prop}\label{PropTopologiesOverDivisorsAndFunctionsAreTheSame}
    The map $v \in \Vinf \mapsto L_v \in \Hom(\Cinf, \R)$ is a continuous embedding.
  \end{prop}

  \begin{proof}
    For the injectivity, let $v , w \in \Vinf$ such that $v \neq w$. First, if $w = t v$ with $t >0$, then since
    $L_v \neq 0$, we have $L_v \neq L_w$. Otherwise, there exists a completion $X$ such that $c_X (v) \neq c_X
    (w)$. If the centers are both prime divisors at infinity then it is clear that $L_v \neq L_w$. If $c_X
    (v)$ is a point, let $\tilde E$ be the exceptional divisor above it. Then, by Proposition
    \ref{PropValuationForCartierDivisorOverOneCompletion}, $L_v (\tilde E) >0$, but $L_w (\tilde E) = 0$.

    By definition, to show continuity we have to show that for all $D \in \Cinf$, the map $v \in \Vinf \mapsto L_v (D)$ is
    continuous. Let $X$ be a completion where $D$ is defined, then by Proposition
      \ref{PropValuationForCartierDivisorOverOneCompletion} $L_v (D) = L_v (\OO_X (-D))$ and by Proposition
      \ref{PropTopologiesOverIdealsAndFunctionsAreTheSame} the map $v \in \Vinf \mapsto L_v (\OO_X (-D))$ is continuous.
  \end{proof}

  \begin{prop}\label{PropEvaluationSurDiviseurEnFonctionDeAlpha}
  Let $X$ be a completion of $X_0$ and $p \in X$ a closed point at infinity. Let $v \in \cV_X (p; \m_p)$. If $E$
  is a prime divisor of $X$ at infinity such that $p \in E$, then
  \begin{equation}
    1 \leq L_v (E) \leq \alpha (v)
    \label{<+label+>}
  \end{equation}
\end{prop}

\begin{proof}
  Let $z \in \OO_{X,p}$ be a local equation of $E$, $z$ is irreducible and we have $L_v (E) = v(z)$. We have that $z
  \in \m_p$, therefore $v(z) \geq v(\m_p) = 1$. This shows the first inequality. For the second one, let $v_z$ be the
  curve valuation associated to $z$. It does not define a valuation over $\k [X_0]$ but it defines a valuation over
  $\OO_{X,p}$ by Proposition \ref{PropValuationEnFonctionDeAlpha}, we get
  \begin{equation}
    v(z) = \alpha (v_z \wedge v) \leq \alpha (v)
    \label{<+label+>}
  \end{equation}
\end{proof}

  \subsection{Special look at divisorial valuations centered at infinity}

  \begin{lemme}\label{LemmeValuationDivisorielle}
    Let $X$ be a completion of $X_0$ and let $E$ be a prime divisor at infinity. One has $L_{\ord_E}(E) = 1$ and for
    any prime divisor $F \neq E$ in $X$, $L_{\ord_E} (F) = 0$.

    Furthermore, if $\pi: Y \rightarrow X$ is some blow-up of $X$, and $\pi ' (E)$ the strict
    transform of $E$ by $\pi$, then

    \begin{equation}
    \pi_* \ord_{\pi ' (E)} = \ord_E. \end{equation}

  \end{lemme}

  \begin{proof}
    The first assertion follows from Proposition \ref{PropValuationForCartierDivisorOverOneCompletion} (3).
    We show the second assertion. It suffices to show it when $\pi$ is the blow-up of one point of
    $X$. Let $D = a E + \sum_{F \neq E} \ord_F(D) F$, then $\pi^* D$ is of the form

    \begin{equation}
      \pi^* D = a \pi ' (E) + b \tilde E + \sum_{F \neq E} a_F(D) \pi' (F)
    \end{equation}
    where $\tilde E$ is the exceptional divisor of $\pi$. Therefore $\ord_{\pi ' (E)} (\pi^* (D)) = a  = \ord_E (D)$.
  \end{proof}

  \begin{prop}\label{PropExtensionNaturelleValuationDivisorielle}
    Let $v$ be a divisorial valuation, then $L_v$ can be extended naturally to a continuous linear form $L_v : \Winf
    \rightarrow \R$.
  \end{prop}

  \begin{proof}
    Take $W \in \Winf$. Since $v$ is divisorial, there exists a completion $X$ of $X_0$ that contains a prime divisor $E$ at
    infinity such that $(\iota_{X})_* v = \lambda \ord_E$. We set

    \begin{equation}
    L_v(W) := L_{v,X}(W_{X}) \end{equation}

    This does not depend on the completion $X$. To show this, it suffices to show that we get the same
    result if we blow up one point of $X$. So, let $\pi: Y \rightarrow X$ be the blow up of one point
    of $X_0$ at infinity. Then, by Lemma \ref{LemmeValuationDivisorielle}, $v_{Y} = \lambda
    \ord_{\pi ' (E)}$ and $\ord_{\pi' (E)} (W_{Y}) = \ord_E (\pi_* W_{Y}) = \ord_E (W_{X})$. If $D
    \in \Cinf$, then this is compatible with the previous definition of $L_v(D)$ because if $D$ is defined
    over $X$, there exists a completion $\pi: Y \rightarrow X$ such that the center of $v$ on $Y$
    is a prime divisor at infinity and by Proposition
    \ref{PropValuationForCartierDivisorOverOneCompletion} (5) $L_{v, Y}(\pi^*D) = L_{v, X}(D)$.
  \end{proof}

  \begin{rmq}
    Recall that we have defined in \S \ref{SubSecCompletions} the set $\cD_\infty (X_0)$ as the set of equivalence
    classes of prime divisors at infinity modulo the following equivalence relations : $(X_1, E_1) \sim (X_2, E_2)$
    if $\pi = \iota_2 \circ {\iota_1}^{-1} : X_1 \dashrightarrow X_2$ satisfies $\pi(E_1)= E_2$. Lemma
    \ref{LemmeValuationDivisorielle} shows that it makes sense to define $\ord_E$ for $E \in \cD_\infty (X_0)$ and that
    $\ord_E$ is defined over $\Winf$.
  \end{rmq}

  \begin{prop}\label{PropCaracterisationMinimumParValuationDivisorielle}
    Let $W, W' \in \Winf$, then $W'' = W \wedge W'$ if and only if for any divisorial valuation $E \in \cD_\infty
    (X_0)$,

    \begin{equation}
    \ord_E (W'') = \min (\ord_E(W),  \ord_E (W')). \end{equation}

  \end{prop}

  \begin{proof}
    This is immediate as for any completion $X$,

    \begin{equation}
    W_{X} =  \sum_{E \in \partial_{X} X_0} \ord_E (W) \cdot E . \end{equation}

  \end{proof}

  We can now show that the minimum of two Cartier divisors is still a Cartier divisor.

  \begin{prop}\label{PropMinimumDiviseurParEclatementFaisceauIdeaux}
    Let $X$ be a completion of $X_0$, let $D, D' \in \DivInf (X)$ be two effective divisor and let $\aa$ be the sheaf
    of ideals $\aa = \OO_X (-D) + O_X (- D')$. Consider the blow up of $\aa, \pi :Y \rightarrow X$ and let $\omega : Z
    \rightarrow Y$ be a desingularisation of $Y$. Then, $D \wedge D'$ is
    the Cartier divisor defined by $\omega^* \pi^* \aa$. 
\end{prop}
In the statement of the proposition we have to desingularise because we have made the convention that a completion must
be smooth at infinity.
Notice that $\aa$ is not locally principle only at satellite points, so if we choose $\omega$ minimial, then $\omega
\circ \pi$ is a sequence of blow-ups of satellite points.
  This shows the Claim \ref{ClaimPullBackOfIdealSheafIsMinimumCartierClass}.

  \begin{proof}[Proof of Proposition \ref{PropMinimumDiviseurParEclatementFaisceauIdeaux}]\label{ProofClaim}
     Define the sheaf of ideals $\aa = \OO_{X} (-D) + \OO_{X}
    (-D ')$ and let $\pi: Y \rightarrow X$ be the blow up of $\aa$. There exists a Cartier divisor
    $D_{Y}$ on $Y$ such that $\OO_{Y} (- D_{Y}) = \pi^* \aa
    \cdot \OO_{Y}$. Consider a desingularisation $\omega : Z \rightarrow Y$ and define $\b = \OO_Z (- \omega^* D_Y) =
    \omega^* \pi^* \alpha$. We show that $D_Z := \omega^* D_Y = D \wedge D'$. By Proposition
    \ref{PropCaracterisationMinimumParValuationDivisorielle}, we only need to show that for any divisorial
    valuation $v, L_{v,Z}(D_{Z}) = \min (L_{v,X}(D), L_{v,X}(D'))$, but by
    Proposition \ref{PropValuationForCoherentSheafOfIdealsForOneCompletion} we have the following
    equalities

    \begin{equation}
    L_{v,Y}(D_{Y}) = L_{v,Y}(\bb) = L_{v,X}(\aa) = \min(L_{v,X}(D), L_{v,X}(D')) 
  \end{equation}

  \end{proof}

  \section{Local divisor associated to a valuation}\label{SubSecLocalDivisorValuation}
  Let $X$ be a completion of $X_0$ and let $p \in X$ be a closed point at infinity. Let $v$ be a valuation centered
  at $p$. We know by Section \ref{SecValuationAsLinearFormsOverDivisors} that $v$ induces a linear form $L_v$ on
  $\Cinf_\R$. By restriction, it induces a linear form $L_{v, X, p}$ on $\Cartier(X, p)_\R$. Now by Proposition
  \ref{PropLocalPicardManinSpace}, the pairing
    \begin{equation}
      \Weil(X, p)_\R \times \Cartier(X, p)_\R \rightarrow \R
  \label{<+label+>}
\end{equation}
induced by the intersection product is perfect. Thus, there is a unique $Z_{v, X, p} \in \Weil(X,
p)_\R$ such that
\begin{equation}
  \forall D \in \Cartier(X,p)_\R, \quad Z_{v, X,  p} \cdot D = L_{v, X, p} (D)
  \label{<+label+>}
\end{equation}

\begin{ex}
  If $\tilde E$ is the exceptional divisor above $p$, then $Z_{\ord_{\tilde E}, X, p} = - \tilde E$.
\end{ex}

\begin{prop}\label{PropLocalDivisorOfDivisorialValuationIsCartier}
  For any valuation $v \in \cV_X (p)$, we have $Z_{v, X, p} \in \Cartier(X,p)$ if and only if $v$ is
  divisorial. Furthermore, $Z_{v, X, p}$ is defined over any completion such that the center of $v$ is a prime
  divisor at infinity. Furthermore, for any $E \in \cD(X, p), Z_{\ord_E, X, p} \in \Cartier(X, p)_\Q$.
\end{prop}

\begin{proof}
  Let $E \in \cD_{X,p}$,  for every $W
  \in \Weil(X,p), \ord_E (W) = \ord_E (W_Y)$ where $Y$ is a completion exceptional above $p$ by
  Proposition \ref{PropExtensionNaturelleValuationDivisorielle}. Let $E, E_1, \cdots,
  E_r$ be the component of $\partial_Y X_0$ that are exceptional above $p$. The intersection form is non degenerate on
  \begin{equation}
    V:= \Q E \oplus \left(\bigoplus_i \Q E_i\right).
  \end{equation}
  Let $L$ be the restriction of $\ord_E$ to $V$, by duality there exists a unique $Z
  \in V$ such that for all $W \in V, W \cdot Z = L (W) = \ord_E (W)$. This implies that $Z = Z_{\ord_E, X, p}$.
  Conversely, if $v$ is a valuation such that $Z_{v, X, p} \in \Cartier(X, p)$ then let $Y$ be a completion
  where $Z_{v, X, p}$ is defined. If $c_Y (v)$ is a point at infinity, then let $\tilde E$ be the exceptional
  divisor above $c_{Y}(v)$. Then, we must have $Z_{v, X, p} \cdot \tilde E > 0$ but it is equal to 0, this is a
  contradiction.
\end{proof}

\begin{prop}
  The embedding $\cV_X (p; \m_p) \hookrightarrow \Weil(X,p)_\R$ is continuous with respect to the weak topology.
\end{prop}

\begin{proof}
  This is a direct consequence of Proposition \ref{PropTopologiesOverDivisorsAndFunctionsAreTheSame} and Proposition
  \ref{PropConvergenceSurFonctionRegulierePareilQueConvergenceLocale}.
\end{proof}
Thus, for all completion $\pi: Y \rightarrow X$, for all $E \in \Gamma_\pi$, we can consider $Z_{\ord_E, X, p}$ as
an element of $\DivInf(Y)_\R$.

\begin{prop}\label{PropIncarnationLocalDivisorOfValuation}
  Let $\pi: (Y, \Exc (\pi)) \rightarrow (X, p)$ be a completion exceptional above $p$. Let $v$ be a valuation such
  that $c_X (v) = p$. Suppose that $c_Y (v)$ is a point at infinity. Consider $\cV_X (p; \m_p)$ with its generic
  multiplicity function $b$.
  \begin{enumerate}
    \item If $c_Y (v) \in E$ is a free point with $E \in \Gamma_\pi$, then the incarnation of $Z_{v, X, p}$ in
      $Y$ is
      \begin{equation}
        (Z_{v, X, p})_Y = L_v (E) Z_{\ord_E, X, p}
        \label{<+label+>}
      \end{equation}
      Moreover if $v \in \cV_X (p; \m_p)$, then $L_v (E) = \frac{1}{b(E)}$.
    \item If $c_Y (v) = E \cap F$ is a satellite point with $E,F \in \Gamma_\pi$, then
      \begin{equation}
        (Z_{v, X, p})_Y = L_v (E) Z_{\ord_E, v, p} + L_v (F) Z_{\ord_F, X, p}
        \label{<+label+>}
      \end{equation}
      Moreover if $v \in \cV_X (p; \m_p)$, then $L_v (E) b(E) + L_v (F) b(F) = 1$.
  \end{enumerate}
  Furthermore, if $q \neq c_Y (v)$ and $\tau : Z \rightarrow Y$ is the blow up of $q$ then
  \begin{equation}
    (Z_{v, X, p})_Z = \tau^* (Z_{v, X, p})_Y
    \label{<+label+>}
  \end{equation}
\end{prop}

\begin{proof}
  For any prime divisor $E$ at infinity of $Y$, $L_v (E) > 0 \Leftrightarrow c_Y (v) \in E$ by item
  (3) of Proposition \ref{PropValuationForCartierDivisorOverOneCompletion}.
  Therefore, if $c_Y (v) \in E$ is a free point with $E \in \Gamma_\pi$, then for $F \in \Gamma_\pi, L_v (F) \neq 0
  \Leftrightarrow F = E$, hence
  \begin{equation}
    (L_{v})_{|\DivInf (Y)_\R} = (L_v (E))
    (L_{\ord_E})_{|\DivInf(Y)_\R}.
  \end{equation}
  by definition (see Equation \eqref{EqValuationDisivorielle}). This shows the result if $c_Y(v)$ is a free point.
  Now, if $c_Y(v) = E \cap F$ is a satellite point with $E,F \in \Gamma_\pi$, then for all prime divisors $F'$ of $Y$ at
  infinity $L_v (F ' ) > 0 \Leftrightarrow F' = E$ or $F' = F$. We therefore have
  \begin{equation}
    (L_v)_{|\DivInf(Y)_\R} = (L_v \cdot E) (L_{\ord_E})_{|\DivInf(Y)_\R} + (L_v \cdot F)
    (L_{\ord_F})_{|\DivInf(Y)_\R}.
    \label{<+label+>}
  \end{equation}
  This shows the result in the satellite case.

  If $v \in \cV_X (p; \m_p)$. Let $\tau : Z \rightarrow X$ be the
  blow up of $p$. We know then that $L_v (\tilde E) = 1$ where $\tilde E$ is the exceptional divisor above $p$ by Proposition
  \ref{PropValuativeTreeIsRelativeValuativeTreeWRTExceptionalDivisor}. Let $b_{\tilde E}$ be the generic multiplicity
  function of the tree $\cV_Z (\tilde E; \tilde E)$. We have for every prime divisor $F$ exceptional above $p$ that
  $\ord_F (\tilde E) = b_{\tilde E} (F)$ again by Proposition
  \ref{PropValuativeTreeIsRelativeValuativeTreeWRTExceptionalDivisor}.
  In the free point case, we get $1 = L_v (\tilde E) = L_v (b_{\tilde E} (E) E)$ by Proposition
  \ref{PropValuationForCartierDivisorOverOneCompletion} (3) and (5). In the satellite point case, we get
  \begin{equation}
    1 = L_v (\tilde E) = L_v (b_{\tilde E} (E) E + b_{\tilde E} (F) F)
    \label{<+label+>}
  \end{equation}
  again by Proposition \ref{PropValuationForCartierDivisorOverOneCompletion} (3) and (5).

  For the last assertion, if $\tilde F$ is the exceptional divisor above $q$, we have
  \begin{equation}
    (Z_{v, X, p})_Z = \tau^* (Z_{v, X, p})_Y - (Z_{v, X, p} \cdot \tilde F) \tilde F.
    \label{<+label+>}
  \end{equation}
  Since $c_Z (v) \not \in \tilde F$, we have $L_v (\tilde F) = 0$ by Proposition
  \ref{PropValuationForCartierDivisorOverOneCompletion} (3).
\end{proof}

From now on let $b$ be the generic multiplicity function of $\cV_X (p; \m_p)$ and for any prime divisor $E \in
\cD_{X, p} = \Gamma$, set $v_E = \frac{1}{b(E)} \ord_E$.

\begin{prop}\label{PropLocalDivisorOfValuationAfterBlowUp}
  Let $\pi: (Y, \Exc(\pi)) \rightarrow (X, p)$ be a completion exceptional above $p$. Let $q
  \in \Exc(\pi)$ be a closed point. Let $\tau : Z \rightarrow Y$ be the blow up of $q$ and let $\tilde E$ be the
  exceptional divisor above $q$.
  \begin{enumerate}
    \item If $q \in E$ is a free point with $E \in \Gamma_\pi$, then
      \begin{equation}
        Z_{v_{\tilde E}, X, p} = \tau^* (Z_{v_E, X, p }) - \frac{1}{b(\tilde E)} \tilde E \in \DivInf(Z)_\Q
        \label{<+label+>}
      \end{equation}
    \item If $q = E \cap F$ is a satellite point with $E,F \in \Gamma_\pi$, then
      \begin{equation}
        Z_{v_{\tilde E}, X, p} = \frac{b(E)}{b(E) + b(F)} \tau^* Z_{v_E, X, p } + \frac{b (F)}{b(E) + b(F)} \tau^*
        Z_{v_F , X, p} - \frac{1}{b(\tilde E)} \tilde E \in \DivInf(Z)_\Q
        \label{<+label+>}
      \end{equation}
  \end{enumerate}
\end{prop}

\begin{proof}
  If $q \in E$ is a free point with $E \in \Gamma_\pi$, we have by Proposition
  \ref{PropIncarnationLocalDivisorOfValuation} that the incarnation of $Z_{\ord_{\tilde E}, X, p}$ in $Y$ is
  \begin{equation}
    \tau_* (Z_{\ord_{\tilde E}, X, p}) = Z_{\ord_E, X, p}
    \label{<+label+>}
  \end{equation}
  because $\ord_{\tilde E} (E) = 1$. Therefore
  \begin{equation}
    Z_{\ord_{\tilde E}, X, p} \tau^* Z_{\ord_E, X, p} + \lambda \tilde E
    \label{<+label+>}
  \end{equation}
  with $\lambda \in \R$. Since $Z_{\ord_{\tilde E}, X, p} \cdot \tilde E = 1$, we get $\lambda = -1$. Now, by
  the definition of the generic multiplicity, we have $b(\tilde E) = b(E)$. Therefore,
  \begin{equation}
    Z_{v_{\tilde E}, X, p} = \tau^* Z_{v_E, X, p} - \frac{1}{b(\tilde E)} \tilde E
    \label{<+label+>}
  \end{equation}

  If $q = E \cap F$ is a satellite point with $E,F \in \Gamma_\pi$, then $b(\tilde E) = b(E) + b(F)$. Note
  that $\ord_{\tilde E} (E) = \ord_{\tilde E} (F) = 1$. We have by Proposition \ref{PropIncarnationLocalDivisorOfValuation}
  \begin{equation}
    \tau_* Z_{\ord_{\tilde E}, X, p} =  Z_{\ord_E, X, p} + Z_{\ord_F, X, p}
    \label{<+label+>}
  \end{equation}
  and since $\ord_{\tilde E} (\tilde E) = 1$, we get
  \begin{equation}
    Z_{\ord_{\tilde E}, X, p} = \tau^* Z_{\ord_E, X, p} + \tau^* Z_{\ord_F, X, p} - \tilde E.
    \label{<+label+>}
  \end{equation}
  Therefore,
  \begin{equation}
    Z_{v_{\tilde E}, X, p } = \frac{b(E)}{b(E) + b(F)} \tau^* Z_{\ord_E, X, p} + \frac{b(F)}{b(E) + b(F)} \tau^*
  Z_{\ord_F, X, p} - \frac{1}{b(\tilde E))} \tilde E.
    \label{<+label+>}
  \end{equation}
\end{proof}

\begin{thm}\label{ThmAutoIntersectionIsSkewness}
  Let $v, v' \in \cV_X(p; \m_p)$, then
  \begin{equation}
    Z_{v, X, p} \cdot Z_{v', X, p} = - \alpha(v \wedge v').
    \label{<+label+>}
  \end{equation}
  If $v, v' \in \cV_X (p;E)$, then also 
  \begin{equation}
    Z_{v, X, p} \cdot Z_{v', X, p} = - \alpha_E(v \wedge v').
    \label{<+label+>}
  \end{equation}
\end{thm}

\begin{proof}
  We first prove the result for $\cV_X (p; \m_p)$.
  We show by induction the
  \begin{claim}
    For every completion $\pi : (Y, \Exc(\pi)) \rightarrow (X, p)$ exceptional above $p$, for all $E \in
    \Gamma_\pi$, for all $v \in \cV_X (p; \m_p)$,
    \begin{equation}
      Z_{v_E, X, p} \cdot Z_{v, X, p} = - \alpha (v_E \wedge v )
      \label{<+label+>}
    \end{equation}
  \end{claim}

  First if $\pi: Y \rightarrow X$ is the
  blow up of $p$ with exceptional divisor $\tilde E$. Recall that $\pi_* \ord_{\tilde E} = v_{\m_p}$ then
  $Z_{\ord_{\tilde E}, X, p} = - E$
  and
  \begin{equation}
    Z_{\ord_{\tilde E}, X, p} \cdot Z_{v , X, p} = Z_{v, X, p} \cdot (- \tilde E) = \cdot L_{v} (- \tilde E).
  \end{equation}
  By definition, $v (\m_p) =1$ and $\pi^* \m_p = \OO_Y (- \tilde E)$. Therefore, by Proposition
  \ref{PropValuationForCoherentSheafOfIdealsForOneCompletion}, we get $Z_{\ord_{\tilde E}, X, p} \cdot Z_{v, X, p}
  = -1 = - \alpha (v_{\m_p} \wedge v)$.

  Suppose
  that $\pi: (Y, \Exc(\pi)) \rightarrow (X,p)$ is a completion exceptional above
  $p$ for which the claim holds. Let $q \in Y$ be a closed point at infinity, let $\tau : Z \rightarrow Y$ be the
  blow up of $q$ and let $\tilde E$ be the exceptional divisor. Let $v \in \cV_X (p; \m_p)$, we show that $Z_{v ,
  X, p} \cdot Z_{v_{\tilde E}, X, p} = - \alpha(v \wedge v_{\tilde E})$. We
  divide the proof in 2 different cases.
  \paragraph{Case 1: $q \in E$ is a free point with $E \in \Gamma_\pi$} In that case $v_{\tilde E} >
  v_{E}$ by Proposition \ref{PropOrderRelationAfterOneBlowUp}. We also have $b(\tilde E) = b(E)$ and $Z_{v_{\tilde E},
  X, p} = Z_{v_E, X, p} - \frac{1}{b(\tilde E)} \tilde E$ by Proposition
  \ref{PropLocalDivisorOfValuationAfterBlowUp}. If $c_Y (v) \neq (q)$ (this includes the case where $c_Y
    (v)$ is a prime divisor at infinity. Then, $v \wedge v_{\tilde E} = v \wedge v_E$. We have by Proposition
    \ref{PropLocalDivisorOfValuationAfterBlowUp} that $Z_{v_{\tilde E}, X, p} = \tau^* (Z_{v_E, X, p}) -
    \frac{1}{b(\tilde E)} \tilde E$. Since $Z_{v, X, p} \cdot \tilde E = 0$, we get
    \begin{equation}
      Z_{v, X, p } \cdot Z_{v_{\tilde E}, X, p} = Z_{v, X, p } \cdot Z_{v_E, X, p}.
      \label{<+label+>}
    \end{equation} This is equal to $- \alpha(v \wedge v_E)$ by induction and therefore it is equal to $ - \alpha
    (v \wedge v_{\tilde E})$.

    If $c_Y (v) = q$, then $c_Z (v) \in \tilde E$. We either have $v_{\tilde E} \leq v$ or $v_E < v \wedge
    v_{\tilde E} < v_{\tilde E}$.
    \begin{enumerate}
      \item If $v \geq v_{\tilde E}$, then $v \wedge v_{\tilde E} = v_{\tilde E}$ and $c_Z ( v)$ is either
        $\tilde E$ or a free point on $\tilde E$. In both cases by Proposition
        \ref{PropIncarnationLocalDivisorOfValuation}, the incarnation of $Z_{v, X, p}$ in $Z$ is $Z_{v_{\tilde E},
        X, p}$. Therefore
        \begin{equation}
          Z_{v, X, p} \cdot Z_{v_{\tilde E, X, p}} = (Z_{v_{\tilde E}, X, p})^2 = (Z_{v_E, X, p})^2 -
          \frac{1}{b(\tilde E)^2}.
          \label{<+label+>}
        \end{equation}
        By induction $(Z_{v_E, X, p})^2 = - \alpha (v_E)$ and
        $\alpha(v_{\tilde E}) = \alpha (v_E) + \frac{1}{b(\tilde E)^2}$ by Proposition \ref{PropSkewnessBlowUp}, so the claim is shown in that case.
      \item \label{ItemMonomial} If $v_E < v \wedge v_{\tilde E} < v_{\tilde E}$. Then, $v \wedge v_E$ is a
        monomial valuation centered
        at $E \cap \tilde E$ (we still denote by $E$ the strict transform of $E$ in $Z$). Therefore, by Proposition
        \ref{PropMonomialValuationIsSegment} there exists $s, t > 0$ such that $s b(E) + t b(\tilde E) = 1$ and $v
        \wedge v_{\tilde E} = v_{s,t}$ is the monomial valuation with weight $s,t$ with respect to local coordinates
        associated to $E$ and $\tilde E$ respectively. By Proposition \ref{PropIncarnationLocalDivisorOfValuation}, we
        have
        \begin{equation}
          (Z_{v , X, p })_Z = s Z_{\ord_E, X, p} + t Z_{\ord_{\tilde E}, X, p} = s b_E Z_{v_E, X, p} + t
          b_{\tilde E} Z_{v_{\tilde E}, X, p}.
          \label{<+label+>}
        \end{equation}
        Therefore,
        \begin{equation}
          Z_{v, X, p} \cdot Z_{v_{\tilde E}, X, p} = s b(E) Z_{v_E, X, p} \cdot Z_{v_{\tilde E}, X, p} + t
          b(\tilde E) (Z_{v_{\tilde E}, X, p })^2.
          \label{<+label+>}
        \end{equation}
        By induction and the previous case this is equal to $- b(E) (s \alpha (v_E) + t\alpha (v_{\tilde E}))$. By
        Proposition \ref{PropSkewnessBlowUp}, we have $\alpha (v_{\tilde E}) = \alpha (v_E) + \frac{1}{b (E)^2}$.
        Therefore, we get
        \begin{equation}
          - b (E) \left( s \alpha (v_E) + t \alpha (v_{\tilde E}) \right) = - \alpha (v_E) - \frac{t}{b(E)}
          \label{<+label+>}
        \end{equation}
        and this is equal to $- \alpha (\pi_* v_{s,t})$ by Proposition \ref{PropMonomialValuationIsSegment}.
    \end{enumerate}
  \paragraph{Case 2: $q = E_1 \cap E_2$ is a satellite point} We can suppose without loss of generality that
  $v_{E_1} < v_{E_2}$. In that case we get $v_{E_1} < v_{\tilde E} < v_{E_2}, b(\tilde E) = b(E_1) + b(E_2)$ and
  \begin{equation}
    Z_{v_{\tilde E}, X, p} = \frac{b(E_1)}{b(E_1) + b(E_2)} Z_{v_{E_1}, X, p} + \frac{b(E_2)}{b(E_1) + b(E_2)}
    Z_{v_{E_2},X, p} - \frac{1}{b(\tilde E)} \tilde E
    \label{<+label+>}
  \end{equation}
  by Proposition \ref{PropLocalDivisorOfValuationAfterBlowUp}.

  If $c_{Y} (v) \neq q$, then $v \wedge v_{E_2} \leq v_{E_1}$ or $v \geq v_{E_2}$ and we get
  \begin{equation}
    Z_{v , X, p} \cdot Z_{v_{\tilde E}, X, p} = \frac{b(E_1)}{b(E_1) + b(E_2)} (Z_{v, X, p} \cdot
    Z_{v_{E_1}, X, p}) + \frac{b(E_2)}{b(E_1) + b(E_2)} (Z_{v, X, p} \cdot Z_{v_{E_2}, X, p}).
    \label{EqCalcul}
  \end{equation}
  By induction, this is equal to $-\frac{b(E_1)}{b(E_1) + b(E_2)} \alpha (v_{E_1} \wedge v) -
  \frac{b(E_2)}{b(E_1) + b(E_2)} \alpha (v_{E_2} \wedge v)$.

  If $v \wedge v_{E_2} \leq v_{E_1}$, then
  $v \wedge v_{E_2} = v \wedge v_{\tilde E} = v \wedge v_{E_1}$ and the quantity in Equation \eqref{EqCalcul} is
  equal to $- \alpha(v \wedge v_{\tilde E}).$

  If $v \geq v_{E_2}$, then $v > v_{\tilde E}$ and $v \wedge v_{\tilde E} = v_{\tilde E}$.
  In that case $v \wedge v_{E_1} = v_{E_1}$ and $v \wedge v_{E_2} = v_{E_2}$. Therefore, the quantity in Equation
  \eqref{EqCalcul} is equal to
  \begin{equation}
    - \frac{b(E_1)}{b(E_1) + b(E_2)} \alpha (v_{E_1}) - \frac{b(E_2)}{b(E_1) + b(E_2)} \alpha (v_{E_2}).
    \label{<+label+>}
  \end{equation}
  By Proposition \ref{PropSkewnessBlowUp}, $\alpha (v_{E_2}) = \alpha(v_{E_1}) + \frac{1}{b(E_1) b(E_2)}$, so we get
  \begin{equation}
    Z_{v, X, p} \cdot Z_{v_{\tilde E}, X, p} = - \alpha (v_{E_1}) - \frac{1}{b(E_1) (b(E_1) + b(E_2))} = - \alpha
    (v_{E_1}) - \frac{1}{b(E_1) b(\tilde E)}
    \label{<+label+>}
  \end{equation}
  and this is equal to $- \alpha (v_{\tilde E})$ again by Proposition \ref{PropSkewnessBlowUp}.

  If $c_Y (v) = q$, then $c_Z (v) \in \tilde E$. We have that $v_{E_1} < v \wedge v_{\tilde E} <
  v_{E_2}$. Therefore either $v = v_{\tilde E}$ or $c_Z (v) \in \tilde E$ is a point and $v \wedge
  v_{\tilde E}$ is a monomial valuation centered at $E_1 \cap \tilde E$ or $E_2 \cap \tilde E$. We show again the claim
  by induction in an analogous way as in Case 1. We have thus shown the claim by induction.

  To show the Proposition, let $v, v ' \in \cV_X (p; \m_p)$. If $v \neq v'$, then there exists a completion $\pi :
  (Y, \Exc(\pi)) \rightarrow (X, p)$ exceptional above $p$ such that $c_Y (v) \neq c_Y (v')$. Then, we have
  that
  \begin{equation}
    Z_{v, X, p} \cdot Z_{v ', X, p} = (Z_{v, X,p })_Y \cdot (Z_{v', X, p})_{Y}
    \label{<+label+>}
  \end{equation}
  If $v '$ is infinitely singular or a curve valuation, we can suppose that $c_Y (v ')$ is a free point lying over a
  unique prime divisor $E$ at infinity. Then, $v ' > v_E$ and $v ' \wedge v = v ' \wedge v_E$. Furthermore, the
  incarnation of $Z_{v, X, p}$ in $Y$ is exactly $Z_{v_E, X, p}$ by Proposition
  \ref{PropIncarnationLocalDivisorOfValuation}. Therefore,
  \begin{equation}
    Z_{v, X, p} \cdot Z_{v ', X, p} = Z_{v, X, p} \cdot Z_{v_E, X, p}.
    \label{<+label+>}
  \end{equation}
  This is equal to $- \alpha (v \wedge v_E) = - \alpha (v \wedge v ')$ by the Claim.

  If $v '$ is irrational, then we can suppose that $c_Y (v ' ) = E_1 \cap E_2$ for $E_1, E_2$ two prime divisors at
  infinity.  Suppose without loss
  of generality that $v_{E_1} < v_{E_2}$. By Proposition \ref{PropMonomialValuationIsSegment}, we have that $v ' =
  \pi_* v_{s,t}$ for some $s,t > 0$
  such that $s b(E_1) + t b(E_2) = 1$ and $\alpha (v ' ) = \alpha (v_{E_1}) + \frac{t}{b(E_1)}$. Furthermore, by
  Proposition \ref{PropIncarnationLocalDivisorOfValuation}, the incarnation of $Z_{v ', X, p}$ in $Y$ is
  \begin{equation}
    (Z_{v ', X, p})_Y = s b(E_1) Z_{v_{E_1}, X, p} + t b(E_2) Z_{v_{E_2}, X, p}.
    \label{<+label+>}
  \end{equation}
  And we have
  \begin{equation}
    Z_{v, X, p} \cdot Z_{v ', X, p } = s {b(E_1)} (Z_{v, X, p} \cdot Z_{v_{E_1}, X,
    p}) + t {b(E_2)} ( Z_{v, X, p} \cdot Z_{v_{E_2}, X, p}).
    \label{EqCalcul2}
  \end{equation}

  Either $v \wedge v ' = v \wedge v_{E_1}$ or $v \wedge v ' = v '$.
  If $v \wedge v ' = v \wedge v_{E_1}$, then we also have $v \wedge v_{E_2} = v \wedge v_{E_1}$. The quantity in
  Equation \eqref{EqCalcul2} is then equal to
  \begin{equation}
    - s {b(E_1)} \alpha (v \wedge v_{E_1}) - t {b(E_2)} \alpha (v \wedge v_{E_2}) = \alpha (v \wedge v_{E_1} ) = -
    \alpha( v \wedge v ').
    \label{<+label+>}
  \end{equation}

    If $v \wedge v ' = v '$, then $v \wedge v_{E_1} = v_{E_1}$ and $v \wedge v_{E_2} = v_{E_2}$. The quantity
    in Equation \eqref{EqCalcul2} is then equal to
    \begin{equation}
      - s {b(E_1)} \alpha(v_{E_1}) - t {b(E_2)} \alpha (v_{E_2}) = - \alpha (v_{E_1}) - \frac{t}{b(E_1)} = - \alpha
      (v ').
      \label{<+label+>}
    \end{equation}
    To get the last two equalities we use Proposition \ref{PropSkewnessBlowUp} and
    \ref{PropMonomialValuationIsSegment}.

    Finally, if $v = v '$, we need to show that $(Z_{v, X, p})^2 = - \alpha (v)$. We know the result if $v$ is
    divisorial. We use infinitely near sequence to conclude in general. If $v$ is infinitely singular or a curve
    valuation. Let $(X_n, p_n)$ be the sequence of infinitely near points associated to $v$. The infinitely near
    sequence of $v$ (Proposition \ref{PropInfinitelyNearSequence}) is the subsequence $v_n = \frac{1}{b(E_n)}
    \ord_{E_n}$ where $p_n$ is a free point lying over a
    unique prime divisor $E_n$ at infinity. We have that $\alpha (v_n ) \rightarrow \alpha (v)$ and the incarnation of
    $Z_{v, X, p}$ in $X_n$ is $Z_{v_n, X, p}$. Therefore,
    \begin{equation}
      (Z_{v, X, p})^2 = \lim_n (Z_{v_n, X, p })^2 = - \lim_n \alpha (v_n) = - \alpha (v)
      \label{ }
    \end{equation}

    If $v$ is irrational, then let $ (X_n , p_n)$ be the sequence of infinitely near points associated to $v$. For
    every $n$ large enough, $p_n = E_n \cap F_n$ for $E_n, F_n$ two prime divisors at infinity. Suppose that for all $n,
    v_{E_n} < v_{F_n}$. Then, we have $v_{E_n} < v < v_{F_n}$, $\alpha (v_{E_n}) \rightarrow \alpha (v),
    \alpha (v_{F_n}) \rightarrow \alpha (v)$ and $b(E_n) \rightarrow + \infty, b(F_n) \rightarrow + \infty$. We have
    by Proposition \ref{PropIncarnationLocalDivisorOfValuation} that
    the incarnation of $Z_{v, X, p}$ in $X_n$ is
    \begin{equation}
      s_n b(E_n) Z_{v_{E_n}, X, p} + t_n b(F_n) Z_{v_{F_n}, X, p}
      \label{<+label+>}
    \end{equation}
    for some $s_n, t_n > 0$ such that $s_n b(E_n) + t_n b(F_n) = 1$. We have
    \begin{align}
      (Z_{v, X, p})^2 &= \lim_n (s_n b(E_n)Z_{v_n, X, p} + t_n b(F_n) Z_{v_{F_n, X, p}})^2 \\
      &= \lim_n - s_n^2 b(E_n)^2 \alpha (v_{E_n}) - 2s_n t_n b(E_n) b(F_n) \alpha (v_{E_n}) - t_n^2 b(F_n)^2 \alpha (v_{F_n})
      \label{<+label+>}
    \end{align}
    Therefore we get
    \begin{equation}
      \lim_n - \alpha (v_{E_n}) \leq (Z_{v, X, p})^2 \leq \lim_n - \alpha (v_{F_n}).
      \label{<+label+>}
    \end{equation}
    Hence $(Z_{v, X, p})^2 = - \alpha (v)$.

    Finally the result follows from $\cV_X (p;E)$ using Proposition \ref{PropRelationSkewness} and Proposition
    \ref{PropCompatibilityOrdre}.
\end{proof}

\begin{cor}\label{CorLocalDivisorOfIrrationalValuationIsIrrational}
  If $v \in \cV_X (p;\m_p)$, then $Z_{v, X, p} \not \in \Weil(X, p)_\Q$ if and only if $v$ is irrational.
\end{cor}
\begin{proof}
  If $v$ is divisorial, let $E \in \cD_{X, p}$ such that $v$ is equivalent to $\ord_E$. Then,
  \begin{equation}
    Z_{v, X, p} =
    \frac{1}{b(E)} Z_{\ord_E, X, p} \in \Weil(X, p)_\Q
  \end{equation}
  by Proposition
  \ref{PropLocalDivisorOfDivisorialValuationIsCartier}. If $v$ is infinitely singular or a curve valuation, let
  $\mu$ be any divisorial valuation. We have that $\mu \wedge v$ must be a divisorial valuation, therefore by Theorem
  \ref{ThmAutoIntersectionIsSkewness} we have
  \begin{equation}
    Z_\mu \cdot Z_v = - \alpha (v \wedge \mu) \in \Q.
    \label{<+label+>}
  \end{equation}
  Hence $Z_{v, X, p} \in \Weil(X, p)_\Q$.

  If $v$ is irrational, then for all $\mu \geq v$ divisorial we have $\alpha (\mu \wedge v) = \alpha (v) \in \R
  \setminus \Q$. Therefore, $Z_{v, X, p} \not \in \Weil(X, p)_\Q$.
\end{proof}

\begin{prop}\label{PropStronConvergenceForLocalDivisor}
  Let $X$ be a completion, let $p \in X$ be a closed point at infinity. If $(v_n)$ is a sequence of $\cV_X
  (p; \m_p)$ such that $\alpha(v_n) < + \infty$ for all $n$ and $v \in \cV_{X} (p; \m_p)$, then $v_n \rightarrow
  v$ for the strong topology if and only if $Z_{v_n, X, p} \rightarrow Z_{v, X, p}$ for the strong topology of $\L2$.
\end{prop}

\begin{proof}
  This all comes from Theorem \ref{ThmAutoIntersectionIsSkewness} as
  \begin{align}
    \left| \left( Z_{v, X, p} - Z_{v_n, X, p} \right)^2 \right| &= \left| -\alpha(v) + 2 \alpha(v \wedge v_n) -
      \alpha(v_n) \right| \\
    &= \left| \alpha(v) - \alpha(v \wedge v_n) + \alpha(v_n) - \alpha(v \wedge v_n) \right|.
    \label{<+label+>}
  \end{align}
\end{proof}

  \chapter{From linear forms to valuations}\label{ChapterLinearFormsToValuations}
  Suppose now that we have an element $L$ of $\Hom (\Cinf, \R)$ satisfying Property (+) from Proposition
  \ref{prop:property_plus}, we want to construct a
  valuation $v_L: \k[X_0] \rightarrow \R \cup \{\infty\}$ centered at infinity such that $v_{f_* L} = f_* v_L$.

  \section{Construction of $v_L$}
  First we extend $L$ to $\Sinf$ (see Definition \ref{DefSInf}) by setting

  \begin{equation}
  \text{If } D = \bigvee_i D_i \text { with $D_i \in \Cinf$}, \quad L(D) := \sup_i L(D_i). \end{equation}

  The reason why is because $\Cinf$ defines a lattice in the space of continuous functions $\sC^0 (\cV_\infty, \R)$ and
  $\Cinf_\Q$ is dense with respect to the supremum norm over $\cV_\infty$. In the setting of Berkovich spaces, they would be refered as \emph{model
  functions}, see e.g \cite{boucksomGlobalPluripotentialTheory2022}. With this result in mind, one can show that
  $S_\infty$ is exactly the space of lower semicontinuous functions over $\cV_\infty$.

  \begin{prop}
    This definition does not depend on the representation of $D$ as a supremum $D = \bigvee_i D_i$ with $D_i \in \Cinf$.
  \end{prop}

  \begin{proof}
    If $D = \bigvee_{i \in I} D_i = \bigvee_{j \in J} D_j'$. Let $j \in J$ be an index and $X$ a completion such
    that $D_j '$ is defined on $X$. Let $\epsilon >0$ and let $H$ be an effective divisor such that $\Supp
    (H) = \partial_X X_0$. There exists an index $i \in I$ such that $D_i + \epsilon H \geq D_j '$, since
    otherwise we would get $D + \epsilon H \leq D_j' \leq D$. Therefore we have by property $(+)$ item (1)

    \begin{equation}
      L (D_j') \leq L(D_i) + \epsilon L (H) \leq \sup_k L (D_k) + \epsilon L (H).
    \end{equation}

    Letting $\epsilon$ go to 0, we get $\sup_j L (D_j') \leq \sup_k L (D_k)$ and the result holds by symmetry.
  \end{proof}

  \begin{prop}\label{PropThetaDefinieSurSInf}
    We have the following properties: for $D, D' \in \Sinf$
    \begin{enumerate}
      \item $L(D + D') = L(D) + L(D')$.
      \item $L(D \wedge D') = \min(L(D), L(D'))$.
      \item If $D \geq 0$, then $L (D) \geq 0$.
    \end{enumerate}
  \end{prop}

  \begin{proof}
    For (1), write
    \begin{align*}
      L (D + D') &= \sup_{(i,j) \in I \times J} L( D_i + D'_j) \\
      &= \sup_{i \in I} L (D_i) + \sup_{j \in J} L(D'_j) = L(D) + L(D')
    \end{align*}

    For (2), let $D = \bigvee_i D_i$ and $D' = \bigvee_j D_j '$ be two elements of $\Sinf$. Then,
    \begin{equation}
      D \wedge D'
      = \bigvee_{i,j} D_i \wedge D_j '
    \end{equation}
    and

    \begin{align}
      L ( D \wedge D') &= \sup_{i,j} \min (L (D_i), L (D_j ')) \\
      &= \min (\sup_i L (D_i), \sup_j L (D_j')) \\
      &= \min (L (D), L (D')).
    \end{align}

    For (3), if $D = 0$, then $L (D) = 0$. Otherwise, $D >0$ and there exists a Cartier divisor $D_i$
    defined in some completion $X$ of $X_0$ such that $D_{X} \geq D_i \geq 0$ and therefore

    \begin{equation}
    L (D) \geq L (D_i) \geq 0. \end{equation}

  \end{proof}

  Recall the notations of Section \ref{SectionCompletions}. Define

  \begin{equation}
  w(P) := (\div_{\infty, X} (P))_{X}. \end{equation}

  \begin{prop}
    For $P \in \k[X_0]$, $w(P)$ defines an element of $\Winf$, moreover if one identifies for any completion $X$
    the divisor $\div_{\infty, X} (P) \in \DivInf (X)$ with its image in $\Cinf$, then

    \begin{equation}
    w (P) = \bigvee_{X} \div_{\infty, X} (P). \end{equation}

    Thus, $w(P)$ defines an element of $\Sinf$.

  \end{prop}

  \begin{proof}
    To prove both assertions it suffices to show that if $X$ is a completion of $X_0$ and $Y$ is the
    blow up of some point at infinity, then $\pi_* \div_{\infty, Y} (P) = \div_{\infty, X}
    (P)$ and $ \pi^* \div_{\infty, X} (P) \leq \div_{\infty, Y} (P)$. Let $\tilde E$ be the
    exceptional divisor of $\pi$ and let $E_1, \ldots, E_r$ be the prime divisors in $\partial_{X} X_0$.
    Since $P$ is regular over $X_0$, $\div_{X} (P)$ is of the form
    \begin{equation}
    \div_{X} (P) = D + \sum_{i = 1}^r a_i E_i
  \end{equation}
    where $D$ is an effective divisor such that no irreducible component of its support is one of
    the $E_i$'s; by definition $\div_{\infty, X} (P) = \sum_{i=1}^r a_i E_i$. Then,
    $\div_{Y} (P)$ is of the form

    \begin{equation}
    \div_{Y} (P) = \div_Y ( P \circ \pi) = \pi^* \div_{X} (P) =  \pi ' (D) + b \tilde E + \sum_{i=1}^r a_i \pi '
  (E_i) \end{equation}
  for some $b \in \Z$.
    So $\div_{\infty, Y} (P) = b \tilde E + \sum_{i=1}^r a_i \pi ' (E_i)$ and we get
    $\pi_* (\div_{\infty, Y} (P)) = \div_{\infty, X} (P)$ as $\pi_* (\tilde E) = 0$, This shows
    that $w(P)$ is an element of $\Winf$.

    To show that $\pi^* \div_{\infty, X} (P) \leq \div_{\infty, Y} (P)$ we have to be more precise
    about the coefficient $b$. We can write $b = c + d$, where $\pi^* D = \pi ' (D) + d \tilde E$ and $\pi^*
    \div_{\infty, X} (P) = c \tilde E + \sum_i a_i \pi ' (E_i)$. Since, $D$ is effective, we have $d
    \geq 0$ and the result follows.
  \end{proof}

  \begin{dfn}\label{dfn:v_L}
    We define
  \begin{equation}
    v_L (P) := L (w(P)).
  \end{equation}
  \end{dfn}

  \begin{rmq}
    The class $w(P)$ is not in general a Cartier class. Indeed, take $X_0 = \A^2, X = \P^2$ with homogeneous
    coordinates $[x:y:z]$ such that $\left\{ {z=0} \right\}$ is the line at infinity. Consider $P = y/z \in \k(\P^2)$.
    Define a sequence of blow ups $X_{i}$ by $X_0 = \P^2,  E_0 = \{z=0\}$ and $\pi_{i+1} : X_{i+1} \rightarrow
    X_i$  the blow up
    of the intersection point of the strict transform of $\left\{ {y=0} \right\}$ in $X_i$ and $ E_i$, where $ E_i$ is the
    exceptional divisor in $X_i$. Let $C_y$ be the strict transform of $\left\{ {y=0} \right\}$ in any the $X_i$. We still denote
    by $E_i$ its strict transform in every $X_j, j \geq i$. Then,
    \begin{align*}
      \div_{\P^2} (P) &= C_y - E_0 \\
      \div_{X_1} (P) &= C_y - E_0 \\
      \div_{X_2} (P) &= C_y +E_2 - E_0 \\
      \div_{X_3} (P) &= C_y +2 E_3 + E_2 - E_0 \\
    \end{align*}
    and by induction, we get for all $k \geq 2$

    \begin{equation}
    \div_{X_k} (P) = C_y + \sum_{j=2}^k (j-1) E_j - E_0. \end{equation}

    Therefore, for all $k \geq 2$
    \begin{align*}
      \pi_{k+1}^* \div_{\infty, X_k} (P) &= (k-1) E_{k+1} + \sum_{j=2}^k (j-1) E_j - E_0 \\
      &\neq k E_{k+1} + \sum_{j=2}^k (j-1) E_j - E_0 = \div_{\infty, X_{k+1}} (P).
    \end{align*}

    Thus, $w(P)$ is not a Cartier class. This construction works whenever $P = 0$ is non-empty. We have thus proven the
    following lemma.
  \end{rmq}
  \begin{lemme}\label{lemme:class-w-is-Cartier-iff-invertible}
    The class $w(P)$ is Cartier if and only if $P$ is an invertible regular function over $X_0$.
  \end{lemme}
  \begin{proof}
    If $P$ is an invertible regular function, then for any completion $X$ of $X_0$ one has 
    \begin{equation}
      \div_X (P) = \div_{\infty, X} (P).
      \label{<+label+>}
    \end{equation}
    So that $w(P)$ is the Cartier class defined by $\div_X (P)$ for any completion $X$ of $X_0$.

    If $P$ is not invertible, suppose $w(P)$ is Cartier and let $X$ be a completion such that $D:= w (P)$ is defined in
    $X$. Now, the divisor $\div_X (P)$ is of the form 
    \begin{equation}
      \div_X (P) = Z + \div_{\infty, X} (P)
      \label{<+label+>}
    \end{equation}
    with $Z \neq 0$ an effective divisor supported on $X_0$. Let $C$ be one of the irreducible components of $Z$, we can
    assume up to replacing $X$ by doing finitely many blow ups that all the intersection points of $\BD \cap C$ are free
    points and that $C + \BD$ is a simple normal crossing divisor. Pick, one of the intersection points and call it $p$
    and let $E$ be the irreducible component of $\BD$ containing $p$. Now, consider the sequence of centers associated
    to the curve valuations $v_{C, p}$, i.e start by blowing up $p$ and keep blowing up the strict transform of $C$ with
    the exceptional divisor. We get a sequence of completions $(X_m)$ such that $X_1 = X, E_1 = E$ and $\pi_m :
    X_{m+1} \rightarrow X_m$ is the blow-up of $\tilde C \cap E_m$ where $E_m$ is the exceptional divisor in $X_m$.
    Since $D$ is Cartier and defined over $X$ we have that
    \begin{equation}
      \ord_{E_m} (D) = \ord_E (D), \quad \forall m \geq 1.
      \label{<+label+>}
    \end{equation}
    But one can check by induction on $m$ that 
    \begin{equation}
      \ord_{E_m} (\div_{\infty, X} (P)) = \ord_E (\div_{\infty, X} (P)) + (m-1)
      \label{<+label+>}
    \end{equation}
    and this is a contradiction with $D = w (P)$.
  \end{proof}

  \section{Proofs}
  We show that the function $v_L$ constructed in Definition \ref{dfn:v_L} is a valuation centered at infinity and satisfies $f_* v_L = v_{f_* L}$.

  \begin{prop}\label{PropVThetaCentreeAlInfini}
    The function $v_L$ is a valuation on $\k[X_0]$ centered at infinity.
  \end{prop}

  \begin{proof}

    We first show that $v_L$ is in fact a valuation
    \begin{enumerate}
      \item For any $\lambda \in \k^*$ and for any completion $X$ of $X_0$, $\div_{X} (\lambda) =0$ so $v_L
        (\lambda) = 0$.
      \item If $f,g \in \k[X_0]$, then $\div_{X} (fg) = \div_{X} (f) + \div_{X}(g)$. So, $w(fg) =
        w(f) + w(g)$ and by Proposition \ref{PropThetaDefinieSurSInf}
        $v_L(fg) = v_L(f) + v_L(g)$.
      \item Let $f,g \in \k[X_0], f \neq -g$, then $\div_{X} (f+g) \geq  \div_{X} (f) \wedge \div_{X} (g)$,
        therefore
        \begin{equation}
          w(f + g) \geq w(f) \wedge w (g)
        \end{equation}
        and by Proposition \ref{PropThetaDefinieSurSInf} $v_L (f+g)
        \geq \min( v_L (f), v_L (g))$.
    \end{enumerate}
    The fact that $v_L$ is centered at infinity will follow from Chapter \ref{ChapterBijection}.
    \end{proof}

  In Chapter \ref{ChapterValuationsAsLinearForms}, we have constructed a map
    \begin{equation}
    L: \Vinf \rightarrow \Hom (\Cinf ,\R)_{(+)};
  \end{equation}
  here, we have constructed a map
  \begin{equation}
    v: \Hom(\Cinf, \R)_{(+)} \rightarrow \Vinf
  \end{equation}
  where $\Hom (\Cinf, \R)_{(+)}$ are the linear forms over $\Cinf$ that satisfy
  property (+).
  We shall prove that they are mutual inverse in Chapter \ref{ChapterBijection}. Using this result we show

  \begin{prop}
  Let $f$ be a dominant endomorphism of $X_0$. If $L \in \Hom (\Cinf, \R)_{(+)}$, then $f_* L \in \Hom (\Cinf,
  \R)_{(+)}$ and $v_{f_* L} = f_* v_L$.
  \end{prop}
  \begin{proof}
  Let $L \in \Hom (\Cinf, \R)_{(+)}$, then there exists a unique valuation $v \in \Vinf$ such that $L = L_v$. Then, we
  have by Proposition \ref{PropPushForwardOnValuationsIsPullbackOnDivisors} that
  \begin{equation}
    f_* L = f_* L_v = L_{f_*v}.
    \label{<+label+>}
  \end{equation}
  Therefore, $f_* L \in \Hom (\Cinf, \R)_{(+)}$ and if $w \in \Vinf$ such that $f_* L = L_w$ it is clear that $w =
  f_*v$.
  \end{proof}

  \begin{rmq}
    If $P \in \O (X_0)$, it is not true that $w(f^*P) = \bigvee f^* \div_{\infty, X} (P)$. Indeed, the problem is that
    $f$ might not be proper. For example, take $f(x,y) = (x,xy)$ with $X_0 = \A^2$. In $\P^2$, blow up $[0:1:0]$, let
    $E$ be the exceptional divisor and blow up again the intersection point of $E$ and the strict transform of
    $\left\{ X = 0 \right\}$. Let $V$ be the completion obtained after the two blow ups and call $E_1$ the exceptional
    divisor. The lift $f : V \rightarrow \P^2$ is regular and $f_* E_1 = \left\{ X = 0 \right\}$. Thus, we
    have that for all $D \in \Cinf, \ord_{E_1} f^* D = 0$. Now take $P = x$, then $f^* P  =P$ and
    \begin{equation}
      \div_V (P) = \left\{ X = 0 \right\} + E_1 - \left\{ Z = 0 \right\}.
      \label{<+label+>}
    \end{equation}
    Thus, $w (f^*P) \neq \bigvee f^* \div_{\infty, X} (P)$. However it is true in general that
    \begin{equation}
      w(f^* P) \geq \bigvee f^* \div_{\infty, X} (P)
      \label{<+label+>}
    \end{equation}
  \end{rmq}

                \chapter{Proof that $v$ and $L$ are mutual inverses}\label{ChapterBijection}

                Set $\cM := \Hom(\Cinf, \R)_{(+)}$. In Chapters \ref{ChapterValuationsAsLinearForms} and
                \ref{ChapterLinearFormsToValuations} , we have
                defined $L: v \in \Vinf \mapsto L_v \in \cM$ and $v: L \in \cM
                \mapsto v_L \in \Vinf$. The goal is to show that these two maps are inverse of each other.

                \section{First step, $v \circ L = \id_{\Vinf}$}

                \begin{prop}
                  For all valuation $v \in \Vinf$ and for all $P \in \OO_X (X_0), v(P) = L_v(w(P))$.
                \end{prop}

                \begin{proof}
                  Let $X$ be a completion of $X_0$. We have seen that $\div_{\infty, X} (P)  = \div_{X} (P) - D$ where
                  $D$ is an effective divisor not supported in $\partial_{X} X_0$. Therefore,

                  \begin{equation}
                  L_{v,X}(\div_{\infty, X} (P)) = v(P) - L_{v, X}(D) \leq v (P) \end{equation}

                  Taking the supremum over $X$, we get $L_v( w(P)) \leq v(P)$.

                  To show the other inequality, take a valuation $v$ centered at infinity and let $X$ be a completion
                  of $X_0$. Up to further blow ups
                  of point at infinity, we can suppose that $D := \div_{X} (P)$ is a divisor in $X$ with
                  simple normal crossing on $\partial_{X} X_0$. Let $E_1, \cdots, E_r$ be the prime divisors at infinity of
                  $X$. Then, $D$ is of the form
                  \begin{equation}
                  D = \sum_{i=1}^r a_i E_i + \sum_{j \in J} b_j F_j
                \end{equation}
                  for some prime divisors $F_j$ not supported at infinity. Let $p$ be the center of $v$ on $X$, there are two
                  cases.
                  \begin{enumerate}
                    \item For all $j \in J, p \not \in F_j$, in that case for all $j \in J, L_{v,X}(F_j) = 0$ and $v(P) = L_{v,X}
                      (\div_{\infty, X} (P))$. Therefore, $v(P) \leq L_v(w(P))$ and they are equal.
                    \item There exist a unique $j \in J$ and a unique $i$ such that $p = E_i \cap F_j$. The uniqueness comes
                      from the fact that $D$ is a divisor with simple normal crossing. We denote them respectively by $E$
                      and $F$. Then, we construct a sequence of blow
                      up of points $\pi_i: \overline{X_{i+1}} \rightarrow \overline{X_i}$ such that $\pi_i$ is the blow-up of
                      the center of $v$ in $X_i$ and $X_0 = X$. We still denote by $F$ the strict transform of $F$
                      in any of these blow-ups. There are two possibilities:
                      \begin{enumerate}
                        \item Either there exists a number $k$ such that the center of $v$ in $X_k$ does not belong to
                          $F$ (This includes the case where $v$ is divisorial, in that case the center becomes a prime
                          divisor and there are no more blow-ups to be done). In that case, we are back in case 1 and $v (P)
                          = v_{X_k} (\div_{\infty, X_k} (P)) \leq L_v(w(P))$ and we get the desired equality.
                        \item \label{CasValuationDeCourbe} Or for all $k \geq 0$, the center of $v$ in $X_k$ belongs to
                          $F$, in that case $v$ is the curve valuation associated to $F$ at $p$ and $v(P) = + \infty$. We
                          show that $v_{X_k}(\div_{\infty, X_k} (P)) \rightarrow + \infty$ using the following result.
                      \end{enumerate}

                      \begin{lemme}
                        In case 2.(b), set $E_0 = E$ and for $k \geq 1$, $\tilde E_k$ the exceptional divisor in $X_{k}$ above
                        $c_{X_{k-1}}(v)$, then $L_{v,X_0}(E) = L_{v, X_k}(E_k)$ for all $k$ and the divisor
                        $\div_{X_k} (P)$ is of the form
                        \begin{equation}
                        \div_{X_{k}} (P) = (a + kb) \tilde E_k + b F + D_k'
                      \end{equation}
                        where $a = \ord_E(P) > 0$, $b = \ord_F(P) > 0$ and $c_{X_{k+1}} (v)$ does not belong to the
                        support of $D_k'$.
                      \end{lemme}

                      \begin{proof}
                        First, since we are in case \ref{CasValuationDeCourbe} and we have supposed that $\supp
                        \div_X (P)$ is with simple normal crossings, we have that for all $k \geq 0$ the center
                        of $v$ in $X_k$ is the intersection point $p_k := \tilde E_k \cap F$.

                        We proceed by induction on $k$. If $k=0$ then the result is true as $X_0 = X$ and $c_{X}(v) =
                        E \cap F$. Suppose the result true for a given index $k \geq 0$, then when we blow up $p_k$,
                        $p_{k+1}$ is the intersection point of $\tilde E_{k+1}$ and $F$ so it does not belong to $\pi_k '
                        (\tilde E_k)$ therefore $L_{v, X_{k+1}} (\pi_k ' (\tilde E_k)) = 0$. By
                        induction we have $v_{X_k}(\tilde E_k) = L_{v, X_0}(E)$, and we know that
                        \begin{equation}
                          L_{v, X_k}(\tilde E_k)  =
                          L_{v, X_{k+1}}(\pi_k^* \tilde E_k) = L_{v, X_{k+1}}(\pi_k ' (\tilde E_k) + \tilde E_{k+1}) =
                          L_{v, X_{k+1}} (\tilde E_{k+1})
                        \end{equation}
                        so this shows the first assertion. Now, by induction $\div_{X_k}
                        (P)$ is of the form

                        \begin{equation}
                        \div_{X_k} (P) = (a + kb) \tilde E_k + b F + D_k' \end{equation}

                        Now, since $p_k = \tilde E_k \cap F$ and $p_k \not \in \supp D_k '$, one has

                        \begin{equation}
                          \div_{X_{k+1} (P)} = \pi_{k}^* \div_{X_k} (P) = (a + (k+1)b) \tilde E_{k+1} + b F +
                        (a + kb) \pi_k ' (\tilde E_k) + \pi_k ' (D_k ').
                      \end{equation}

                        Since $p_{k+1} \not \in \pi_k ' (\tilde E_k)$, the support of the divisor
                        $D_{k+1} ' := \pi_k ' (D_k') + (a + kb) \pi_k ' (\tilde E_k)$ does not contain $p_{k+1}$ and we
                        are done.
                      \end{proof}
                      Using this lemma we see that

                      \begin{equation}
                        L_{v, X_k} (\div_{\infty, X_k} (P)) = (a + kb) L_{v, X_0}(E) \xrightarrow[k \rightarrow
                      \infty]{} + \infty \end{equation}

                      Therefore $L_v(w(P)) = + \infty$ and since $v(P) \geq L_v(w(P))$ we have that $v(P) =  + \infty$
                  \end{enumerate}

                \end{proof}

                \section{Second step, $L \circ v = \id_{\cM}$}

                To show that $L \circ v = \id_{\cM}$ we need some technical lemmas.

                \subsection{The center of $L$}

                \begin{prop}\label{PropTwoPositiveDivisorsMustIntersect}
                  Let $L \in \cM$ and $X$ be a completion of $X_0$. If there exists two divisors $E, E'$ at infinity
                  in $X$ such that $L (E), L (E') >0$, then $E$ and $E'$ must intersect.
                \end{prop}

                \begin{proof}
                  Suppose that $E$ and $E'$ do not intersect, then the sheaf of ideals $\aa = \OO_X
                  (-E) \oplus \OO_X (- E')$ is trivial, $\aa = \OO_X$. From Proposition
                  \ref{PropMinimumDiviseurParEclatementFaisceauIdeaux}, we get $E \wedge E' = 0$. Thus
                  $L(E \wedge E') = 0$. But $L
                  (E \wedge E') = \min (L (E), L (E')) >0$ and this is a contradiction.
                \end{proof}

                \begin{cor}\label{LemmeSiDeuxDiviseursAlorsIntersectionEstLeCentre}
                  Let $X$ be a completion of $X_0$, suppose there exists two prime divisors at infinity $E,F$ such that $L(E),
                  L (F) >0$. Then, let $\tilde E$ be the exceptional divisor above $p = E \cap F$, one has $L (\tilde E)
                  >0$.
                \end{cor}

                \begin{proof}
                  Let $\pi: Y \rightarrow X$ be the blow up of $p$ and suppose that $L(\tilde E) = 0$.
                  Since $\pi^* E =
                  \pi ' (E) + \tilde E $ and $\pi^* F = \pi ' (F) + \tilde E$, one has $L (\pi ' (E)) > 0$ and $L (\pi '
                  (F)) > 0$ but the two divisors no longer meet and this is a contradiction.
                \end{proof}

                \begin{prop}\label{PropDefinitionCentrePourTheta}
                  Let $X$ be a completion of $X_0$, there are two possibilities
                  \begin{enumerate}
                    \item There exist a unique closed point $p$ in $X$ at infinity such that if $\tilde E$ is the
                      exceptional divisor above $p$, one has $L (\tilde E) > 0$. We call this point the
                      \emph{center} of $L$ in $X$.
                    \item If no point satisfy this property, then there exists a unique divisor at infinity $E$ in $X$
                      such that $L (E) >0$. In that case we call $E$ the \emph{center} of $L$ in $X$.
                  \end{enumerate}
                  and we have the following properties
                  \begin{enumerate}[label=(\alph*)]
                    \item Let $E$ be a prime divisor at infinity in $X$. If the center of $L$ on $X$ is a point $p$, then $p
                      \in E \Leftrightarrow L (E) > 0$.
                    \item If $Y$ is a completion of $X_0$ above $X$, then the center of $L$ in $Y$ belongs to
                      the inverse image of the center of $X$.
                  \end{enumerate}
                \end{prop}

                \begin{proof}
                  Suppose there are two points $p_1, p_2$ satisfying this property on $X$. Let $\pi_i$ be the blow up of $p_i$ in
                  $X$, we have commutative diagram

                  \begin{center}
                    \begin{tikzcd}
                      & Y  \ar[dl, "\tau_1"'] \ar[dr, "\tau_2"]& \\
                      X_1 \ar[dr, "\pi_1"] & & X_2 \ar[dl, "\pi_2"'] \\
                      & X &
                    \end{tikzcd}
                  \end{center}
                  where on the left side we first blow up $p_1$ then we blow up the strict transform of $p_2$ and the other way around
                  on the right. Now let $\tilde E_1, \tilde E_2$ be the exceptional divisors above $p_1$ and $p_2$ respectively in
                  $X_1$ and in $X_2$ and suppose that $L (\tilde E_1) , L (\tilde E_2) > 0$. Then, since $p_1$ does not
                  belong to $\tilde E_2$ and $p_2$ does not belong to $\tilde E_1$, we have that $L (\tilde E_1) = L
                  (\tau_1^* \tilde E_1) = L (\tau_1 ' (\tilde E_1)) > 0$ and $L (\tau_2 ' (\tilde E_2)) >0$. But in $Y$
                  the prime divisors $\tau_1 ' (\tilde  E_1)$ and $\tau_2 ' (\tilde E_2)$ do not intersect and that contradicts
                  Proposition \ref{PropTwoPositiveDivisorsMustIntersect}.

                  Now, if $E,F$ are two divisors at infinity such that $L (E), L (F) >0$, Lemma
                  \ref{LemmeSiDeuxDiviseursAlorsIntersectionEstLeCentre} shows that $E \cap F$ must be the center of
                  $L$ on $X$. Hence if no point of $X$ is the center of $L$ there is only one prime divisor at
                  infinity $E$ such that $L (E) >0$.

                  To show assertion (a), suppose that the center of $L$ on $X$ is a point $p$ and let $\pi$ be the blow up
                  of $p$. If $p \in E$, then $\pi^* (E) = \pi' (E) + \tilde E$ and $L (E) = L (\pi^* E) \geq L
                  (\tilde E) >0$. If $L (E) >0$ then $p$ must belong to $E$ otherwise $\tilde E$ and $E$ would not intersect
                  and this contradicts Proposition \ref{PropTwoPositiveDivisorsMustIntersect}.

                  We now show assertion (b), we only need to show it for the blow up of a point $\pi : Y \rightarrow X$.  Suppose
                  first that the center of $L$ on $X$ is a (closed) point $p$.
                  If we blow up another point than $p$, then it is clear that the center of $L$ on $Y$ is the point
                  ${\pi}^{-1} p$ as the order of the blow ups does not matter in that case.

                  Suppose now that we blow up $p$, then the exceptional divisor $\tilde E$ verifies $L( \tilde E) >0$, if the
                  center of $L$ on $Y$ is a prime
                  divisor then it must be $\tilde E$. If it is a point then it must belong to $\tilde E$ by assertion (a).

                  If the center of $L$ on $X$ is a prime divisor $E$, then for any
                  blow up $\pi: Y \rightarrow X$ of a point of $X$, we show that the center of $L$ on $Y$ is $\pi'
                  (E)$. The exceptional divisor $\tilde E$ verifies $L (\tilde E ) = 0$ and $\pi ' (E)$
                  is the only prime divisor of $Y$ such that $L (\pi' (E)) > 0$. Thus, if the center of $L$ on $Y$
                  is not a point, it must be $\pi' (E)$. If the center of $L$ on $Y$ is a
                  point $q$, then it must belong to $\pi' (E)$ by assertion (a). If $q$ is not the intersection point $\pi' (E)
                  \cap \tilde E$, then it is the strict transform of a point $p \in E$ and in that case $p$ was the center of
                  $L$ in $X$ this is a contradiction. If $q = \tilde E \cap \pi ' (E)$, then $L (\tilde E) > 0$ by
                  assertion (a) and this is also a contradiction. Therefore, the center of $L$ on $Y$ cannot be a point, it
                  is $\pi' (E)$.
                \end{proof}

                \subsection{End of the proof}

                We say that $L$ is \emph{divisorial} if there exists a completion $X$ of $X_0$ such that the
                center of $L$ on $X$ is a prime divisor at infinity.

                \begin{prop}\label{PropDivisorialValuationsAreDivisorialFunctions}
                  The map $v$ sends divisorial valuations to divisorial elements of $\cM$ and the map $L$ sends
                  divisorial functions to divisorial valuations.
                \end{prop}

                \begin{proof}
                  The fact that divisorial valuations induce divisorial functions on Cartier divisors is clear.
                  Suppose that $L$ is a divisorial function and let $X$ be a completion such that the center
                  of $L$ in $X$ is a prime divisor $E$ at infinity. Then, for all completion $\pi : Y
                  \rightarrow X$ above $X$, the center of $L$ on $Y$ is the strict transform of $E$ by
                  Proposition \ref{PropDefinitionCentrePourTheta} and $L (E) = L (\pi ' (E))$. Therefore,
                  let $v$ be the divisorial valuation on $\k[X_0]$ such that $v_{X} = \ord_E$ and let $P \in
                  \OO_{X_0}(X_0)$, then for all completion $Y$ above $X$, we have by Proposition
                  \ref{PropDefinitionCentrePourTheta}

                  \begin{equation}
                    L (\div_{\infty, Y} (P)) = L (\pi ' (E)) \ord_E (\div_{Y} (P)) = L
                  (E) v (P). \end{equation}

                  Therefore $v_L (P) = L (E) v(P)$ and it is a divisorial valuation.
                \end{proof}

                \begin{prop}
                  One has $L \circ v = \id_{\cM}$.
                \end{prop}

                \begin{proof}
                  We can assume that $L$ and $v_L$ are not divisorial.
                  Let $X$ be a completion of $X_0$, we will show first that if $H \in \DivInf (X)$ is an
                  effective divisor such that $\left| H \right|$ is base point free and $\supp H = \partial_{X} X_0$, then
                  $v_{L} (H) = L (H)$.  Pick $f$ generic in $H^0(X, \OO_X(H))$. We have that $\div f =
                  Z_f - H$ with $Z_f$ effective, $\supp Z_f$ does not contain any divisor at infinity and the
                  center of $v_L$ and the center of $L$ do not belong to $\supp Z_f$.
                  Thus, $f$ defines a regular function over $X_0$, $1/f$ is a local equation of $H$ at the center of
                  $v_L$ and we have

                  \begin{equation}
                  v_L (f) = \sup_Y L (\div_{\infty, Y} (f)) \end{equation}

                  Now, by our assumptions on $f$ we have

                  \begin{lemme}
                    For all $Y$ above $X$, $\div_Y (f)$ is of the form $Z_{f, Y} + \div_{\infty, Y} (f)$
                    where $Z_{f, Y}$ is effective, supported on $X_0$ and $\supp Z_{f, Y}$ does not contain the
                    center of $L$. Furthermore, we have $L (\div_{\infty, Y} (f)) = L
                    (\div_{\infty, X}(f))$.
                  \end{lemme}

                  \begin{proof}
                    This is true for $Y = X$. We proceed by induction. Let $Y$ be a completion above $Y$
                    where the lemma is true and let $\pi: Y_1 \rightarrow Y$ be a blow up of $Y$ at a point
                    $p$. If $p$ is not the center of $L$ then the lemma is clearly true over $Y_1$, if $p$ is
                    the center of $L$ over $Y$ then since $p$ does not belong to $\supp Z_{f, Y}$ we have

                    \begin{equation}
                    \div_{f, Y_1} = \pi' (Z_{f, Y}) + \pi^* (\div_{\infty, Y} (f)) \end{equation}
                    and the lemma is
                    true since $Z_{f, Y_1} = \pi ' (Z_{f, Y})$ and $\div_{\infty, Y_1} (f) = \pi^*
                    (\div_{\infty, Y} (f))$.

                  \end{proof}

                  Using this lemma we conclude that $v_L (f) = L (\div_{\infty, X} (f)) = - L
                  (H)$. Therefore,

                  \begin{equation}
                    v_L (H) = v_L (1/f) = L (H).
                  \end{equation}

                  Now take any divisor $D \in \DivInf (X)$. There exists an integer $n \geq 1$ such that $D + nH$
                  is effective and $ \left| D+nH \right|$ is base-point free. Therefore,

                  \begin{equation}
                    v_L (D) = v_L(D+nH) - v_L (nH) = L (D + nH) - L (nH) = L (D).
                  \end{equation}

                  \end{proof}


%% file: chapitre2.tex
\part{Eigenvaluations and dynamics at infinity}
We apply all the results from the first part to prove the results stated in the introduction. If $X$ is a completion of
a normal affine surface $X_0$ and $f: X_0 \rightarrow X_0$, then $f$ induces a
rational self-transformation $f_X: X \dashrightarrow X$. Because $f_X$ has
indeterminacy points in general, $(X, f_X)$ does not define a dynamical system
because we cannot iterate $f_X$. The main strategy is to construct a completion
$X$ over which $f_X$ admits a "special" fixed point at infinity. The key is then to understand the local dynamics of $f_X$ at this fixed point. The plan is as follows:

The pushforward operator $f_*$ on the Picard-Manin space of $X_0$ has an
eigenvector $\theta_*$ for the eigenvalue $\lambda_1$. By the bijection between
linear forms over divisors at infinity and valuations centered at infinity,
this defines a valuation $v_*$ such that $f_* v_* = \lambda_1 v_*$. This is done in \S \ref{ChapterDynamicsQuasiAlbTrivial}.

The next step is to study the local dynamics of $f_*$ at $v_*$. We show that it
is an attracting fixed point in $\hat \Vinf$, this uses the local tree
structure of $\Vinf$. The attractingness of $v_*$ allows one to construct the
desired fixed point at infinity, except in one case. This is the content of \S
\ref{ChapterLocalNormalForms}.

In \S \ref{ChapterExamples}, we deal with examples of affine surfaces. In
particular, we exhibit an example of an affine surface with an elliptic curve
at infinity and an endomorphism that acts by a translation of infinite order on
this elliptic curve, thus there cannot be a fixed point at infinity.

Finally, \S \ref{ChapterAutomorphisms} deals with automorphisms. Using
Gizatullin's work, we show that if $X_0$ admits an automorphism $f$ with
$\lambda_1(f) >1$, then either $\Vinf$ is a tree and $\lambda_1(f) \in \Z$ or $\Vinf$ is homotopically equivalent to
a circle and $\lambda_1(f) \not \in \Z$. The circle case is very particular.
The action of $\Aut(X_0)$ on the circle is given by elements of the Thompson
group. The main example of such surfaces are the cubic surfaces of Markov type
where we make the explicit computations of the action on the circle.

\chapter{General case} \label{ChapterGeneralCase}
In this chapter, we show Theorem~\ref{BigThmDynamicalDegreesEng} when either the condition
$\k[X_0]^\times=\k^\times$ or $\Pic^0(X_0)=0$ is not satisfied.
We rely on the universal property of the quasi-Albanese variety
(see \cite{serreExposesSeminaires195019992001}), as well as on the geometric properties of
subvarieties of quasi-abelian
varieties (see \cite{abramovichSubvarietiesSemiabelianVarieties1994}).

\section{Quasi-Albanese variety and morphism}
Let $G$ be an algebraic group over $\k$ with $\k$ algebraically closed. We say that $G$ is
a \emph{quasi-abelian} variety if there exists an algebraic torus $T = \G_m^r$, an abelian variety
$A$, and an exact sequence of $\k$-algebraic groups
\begin{equation}
  0 \rightarrow T \rightarrow G \rightarrow A \rightarrow 0.
  \label{Eq:Exact_Sequence_Quasi_Abelian}
\end{equation}

\begin{thm}[see~\cite{serreExposesSeminaires195019992001}, Théorème 7]\label{thm:quasi-albanese_morphism}
  Let $X$ be a variety over $\k$, then there exists a quasi-abelian variety $G$ and a morphism $q : X
  \rightarrow G$ such that for any quasi-abelian variety $G'$ and any morphism $\phi: X \rightarrow
  G'$ there exists a unique morphism $g: G \rightarrow G'$ and a unique $b\in G'$ such
  that $$\phi = g \circ q.$$ Moreover, $g$ is the composition of a homomorphism $L_g: G\to
  G'$ of algebraic groups and a translation $T_g : G'\to G'$ by some element $b\in G'$.
\end{thm}

Such a $G$ is unique up to (a unique) isomorphism. It is called the \emph{quasi-Albanese
variety} of $X$ and it will be denoted by $\QAlb(X)$; the universal morphism $q: X\to
\QAlb(X)$ is ``the'' \emph{quasi-Albanese morphism} (it is unique up to post-composition
with an isomorphism of $G$). Of course if $X$ is projective, then $\QAlb(X)$ is the classical Albanese
variety of $X$. If a quasiprojective variety $U$ has a trivial quasi-Albanese variety, then this has some geometric
consequences. We define $\Pic^0(U)$ as $\Pic^0 (X)$ for any projective variety containing $U$ as a dense open subset. The
variety $\Pic^0 (X)$ is a abelian variety which is the dual for the Albanese variety of $X$, it is a birational
invariant so it does not depend on the choice of $X$. It is also the kernel of the map $\Pic(X) \rightarrow N^1
(X)$ where $N^1(X)$ is the group of Cartier divisors modulo numerical equivalence.

\begin{prop}\label{prop:carac-trivial-quasi-albanese}
  Let $X_0$ be an affine variety. Then  $\k [X_0]^\times = \k^\times$
  and $\Pic^0 (X_0) = 0$ if and only if $\QAlb(X_0) = 0$.
\end{prop}
\begin{proof} Let $G = \QAlb(X_0)$ and $q: X_0\to G$ be a quasi-Albanese morphism. Let
  \begin{equation}
    0 \rightarrow T \rightarrow G \xrightarrow{\pi} A \rightarrow 0.
  \end{equation}
  be an exact sequence, as in Equation~\eqref{Eq:Exact_Sequence_Quasi_Abelian}. Let $X$ be a completion of $X_0$ such
  that $\pi \circ q$ extends to a regular map $\pi \circ q: X \rightarrow A$.

  Assume $\k [X_0]^\times = \k^\times$
  and $\Pic^0 (X_0) = 0$.
  Then, $\pi \circ q (X_0)$ is a point in $A$ as the Albanese variety of $X$ is trivial, and composing $q$ with a
  translation of $G$, we can
  assume that this point is the
  neutral element of $A$. Then, $q(X_0) \subset T$, so $q$ is a regular map from
  $X_0$ to an algebraic torus, and $\k [X_0]^\times = \k^\times$ implies that $q(X_0)$
  is a point. This shows that $\QAlb(X_0)$ is a point.

  Now, suppose that $\k[X_0]^\times \neq \k^\times$, then any non-constant invertible function
  $X_0\to \k^\times$ provides a dominant morphism to a $1$-dimensional torus, so $\dim(\QAlb(X_0))\geq 1$ by the universal
  property. And if $\Pic^0(X_0)\neq 0$, the Albanese morphism of any completion of $X_0$ also shows that
  $\dim(\QAlb(X_0))\geq 1$. This concludes the proof.
\end{proof}

In the following, we show that if $\car \k = 0$ and $X_0$ is an irreducible normal affine surface with non-trivial quasi-Albanese
variety and $f$ is a dominant endomorphism of $X_0$, then $\lambda_1 (f)$ is a quadratic integer. In positive
characteristic, we show that the results hold for separable endomorphism of $X_0$. See
Theorem~\ref{thm:dynamical_degree_quasi_albanese} below.

\section{Logarithmic Kodaira dimension}
Let $V$ be a smooth algebraic surface, let $\overline V$ be a good completion of $V$ and $\overline D = \overline
V \setminus V$, it is a simple normal crossing divisor. We write $\overline K$ for a canonical divisor of $\overline V$.
Following
\cite{iitakaLogarithmicKodairaDimension1977} we
define the following invariant
\begin{equation}
  \overline \kappa (V) = \kappa (\overline V, \overline K + \overline D).
  \label{<+label+>}
\end{equation}
Where for a line bundle $L$ over $\overline V$,
\begin{equation}
  \kappa (\overline V, L) = \limsup_{k \rightarrow + \infty} \frac{n!}{k^n} \dim H^0 (\overline V, L).
  \label{<+label+>}
\end{equation}
It is called the \emph{logarithmic Kodaira
dimension} of $V$ (see \cite{iitakaLogarithmicKodairaDimension1977}). We have the following
characterization of the algebraic torus of dimension 2. If $V$ is projective, then the log Kodaira
dimension is nothing but the classical Kodaira dimension of $V$. We also write $q(V)$ for the dimension of
$\QAlb(V)$.

\begin{thm}[Theorem 2 of \cite{iitakaNumericalCriterionQuasiabelian1979}]\label{ThmCaracAlgebraicTorusChar0}
  Suppose $\car \k = 0$. Let $V$ be a normal affine surface, then $V \simeq \G_m^2$ if and only if $\overline \kappa
  (V) = 0$ and $\overline q (V) = 2$.
\end{thm}

\begin{thm}[Theorem 3.1 of \cite{kojimaOpenSurfacesLogarithmic2001}]\label{ThmCaracAlgebraicTorusAnyChar}
  Let $X_0 = \spec A$ be a normal affine surface over an algebraically closed field $\k$ with $\overline \kappa
  (X_0) = 0$ then $\rank \k [X_0]^\times / \k^\times \leq 2$ with equality if and only if $X_0 \simeq \G_m^2$.
\end{thm}

\begin{thm}\label{ThmAffineRuled}
  Let $X_0$ be a normal affine surface, then $\overline \kappa(X_0) = - \infty$ if and only if a desingularisation of
  $X_0$ admits an open subset isomorphic to $\A^1_\k \times C$ where $C$ is a curve.
\end{thm}
\begin{proof}
  This follows from Theorem 1 of \cite{miyanishiAffineruledIrrationalSurfaces1982} if $X_0$ is irrational and from
  Theorem 1 of \cite{russellAffineruledRationalSurfaces1981} if $X_0$ is rational as for any completion $X$ of $X_0$ the
  boundary $\BD$ is connected by \cite{goodmanAffineOpenSubsets1969}, so the boundary of any completion of $Y_0$ must
  also be connected where $Y_0$ is a desingularisation of $X_0$.
\end{proof}

\begin{prop}[Theorem 1.10 of \cite{kojimaOpenSurfacesLogarithmic2001}]\label{PropVanishingLogKodairaDimensionImpliesRational}
  Let $X_0$ be a normal affine surface with $\overline \kappa (X_0) = 0$, then $X_0$ is rational.
\end{prop}

\begin{lemme}[\cite{iitakaLogarithmicKodairaDimension1977} Proposition 1]
  \label{LemmeIsomorphismeDifferentielleLogarithmique}
  If $V$ is a quasiprojective variety and $f : V \rightarrow V$ is a separable endomorphism. Suppose that there
  exists smooth completions $V_1, V_2$ of $V$ such that $f$ lifts to a regular map $f : V_1 \rightarrow V_2$ then $f$
  induces a linear injective homomorphism
  \begin{equation}
    f^* : H^0 (\overline V_2, m (\overline K + \overline D_2)) \rightarrow H^0 (\overline V_1, m (\overline K + \overline D_1))
    \label{<+label+>}
  \end{equation}
  for all $m \geq 1$.
\end{lemme}
\begin{proof}
  First notice that since $f$ is separable and dominant, the linear operator $H^0 (V, mK_V) \rightarrow H^0 (V,mK_V)$ is
  injective and therefore an isomorphism.
  Let $q \in \overline D_2$ and $p \in \overline D_1$ such that $f(p) = q$. Let $w_1, \dots, w_r$ be regular parameters at
  $q$ such that $w_1 \cdots w_r = 0$ is a local equation of $D_2$ at $q$ and let $z_1, \dots, z_s$ be local parameters
  at $p$ such that $z_1 \cdots z_s = 0$ is a local equation of $D_1$ at $p$. Then, for every $i = 1, \dots, r$, we must
  have
  \begin{equation}
    f^* w_i = \epsilon_i \prod_{j=1}^s z_j^{n_{ij}}.
    \label{eq:<+label+>}
  \end{equation}
  Now, since $f$ is separable, $f^*w_i$ cannot be zero. Thus, none of the integers $n_{ij}$ can be divisible by $\car
  K$.
  Therefore we get
  \begin{equation}
    f^* \frac{d w_i}{w_i} = \sum_j n_{ij} \frac{d z_j}{z_j} + \frac{d \epsilon}{\epsilon}
    \label{eq:<+label+>}
  \end{equation}
  with none of the $n_{ij}$ divisible by $p$. Thus, we see that if $\omega$ is a volume form over $V$ with logarithmic
  poles along $\overline D_2$, then $f^* \omega$ is a volume form over $V$ with logarithmic poles along $\overline D_1$.
\end{proof}
\begin{cor}\label{CorIsomorphismeDifferentielleLogarithmique}
  If $X_0$ is a normal affine surface and $f$ is a dominant separable endomorphism of $X_0$, then for any
  completion $X$ of $X_0$, $f$ induces a linear isomorphism
  \begin{equation}
    f^* : H^0 (X, m(K_X + D)) \rightarrow H^0 (X, m(K_X +D)).
    \label{eq:<+label+>}
  \end{equation}
\end{cor}
\begin{proof}
  Let $\pi: Y \rightarrow X$ be a completion of $X_0$ above $X$ such that $f : X \dashrightarrow X$ lifts to a regular
  map $f : Y \rightarrow X$. Then, by Lemma \ref{LemmeIsomorphismeDifferentielleLogarithmique} we have an injective
  linear homomorphism
  \begin{equation}
    f^* : H^0 (X, m( K_X + D_X)) \rightarrow H^0 (Y, m(K_Y + D_Y)).
    \label{eq:<+label+>}
  \end{equation}
  But since $\pi : Y \rightarrow X$ is a birational morphism which is an isomorphism above $X_0$, it induces a linear
  isomorphism $H^0 (X, m( K_X + D_X)) \xrightarrow{\sim} H^0 (Y, m(K_Y + D_Y))$ and we get the result.
\end{proof}

We define the \emph{log Kodaira Iitaka} fibration
\begin{equation}
  \Phi_m : V \dashrightarrow \P \left( H^0, \overline V, m(\overline K + \overline D) \right).
\end{equation}
By Corollary \ref{CorIsomorphismeDifferentielleLogarithmique}, if $V$ is a quasiprojective surface, every dominant
separable endomorphism of $V$ must preserve the log Kodaira Iitaka fibration for $m \gg 1$. We write $\End_{sep}(V)$ for
the set of dominant separable endomorphism of $V$.
We say that $V$ is of \emph{log general type} if $\overline \kappa (V) = \dim V$.

\begin{cor}[\cite{iitakaLogarithmicKodairaDimension1977} Proposition 2 and Corollary
  p.5]\label{CorLogGeneralType}
  If $X_0$ is a quasiprojective surface of log general type over any algebraically closed field, then $\Aut X_0$ is a finite
  group and $\End_{\text{sep}}(X_0) = \Aut (X_0)$.
\end{cor}

\section{Dynamical degree in presence of an invariant fibration}

\begin{prop}[Stein Factorization]\label{PropSteinFactorization}
  Let $X$, $S$ be projective varieties and let $f: X \dashrightarrow X$ be a rational transformation. Suppose that
  there exist morphisms $\phi : X \rightarrow S$ and $g : S \rightarrow S$ such that the following diagram commutes,
  \begin{center}
    \begin{tikzcd}
      X \ar[r, dashed, "f"] \ar[d, "\phi"] & X \ar[d, "\phi"] \\
      S \ar[r, "g"] & S
    \end{tikzcd}
  \end{center}
  Then there exists a variety $\tilde S$ and morphisms $\psi: X \rightarrow \tilde S$, $\pi: \tilde S
  \rightarrow S$ such that
  \begin{itemize}
    \item $\phi = \pi \circ \psi$,
    \item $\pi$ is finite and $\psi$ has connected fibers
    \item there exists a    rational transformation $\tilde g : \tilde S \dasharrow \tilde S$ such that the diagram
      \begin{center}
        \begin{tikzcd}
          X \ar[r, dashed, "f"] \ar[d, "\psi"] & X \ar[d, "\psi"] \\
          \tilde S \ar[r, dashed, "\tilde g"] \ar[d, "\pi"] & \tilde S \ar[d, "\pi"] \\
          S \ar[r, "g"] & S
        \end{tikzcd}
      \end{center}
      commutes.
  \end{itemize}
\end{prop}

\begin{proof}
  The existence of $\tilde S$ along with $\pi$ and $\psi$ is due to Stein Factorization theorem: It is
  known that one can take $\tilde S = \spec_S \phi_* \OO_X$ where $\spec_S$ is the
  relative $\spec$; that is for every affine open subset $U$ of
  $S$, one has
  \begin{equation}
  {\pi}^{-1} (U) \simeq \spec \OO_X ({\phi}^{-1} (U)). \end{equation}
  Now to construct $\tilde g$, take  affine open subsets $U$ and $V$ of $S$ such that $U \subset {g}^{-1} (V)$.
  Suppose also that ${\phi}^{-1} (U)$ and ${\phi}^{-1} (V)$ do not contain any  indeterminacy of $f$. To
  construct
  \begin{equation}
    \tilde g_{| {\pi}^{-1} (U)} : {\pi}^{-1} (U) \rightarrow {\pi}^{-1} (V)
  \end{equation}
  we use
  the map $f^*: \OO_X ({\phi}^{-1} (V)) \rightarrow \OO_X ({\phi}^{-1} (U))$ induced by $f$; this is well defined since ${\phi}^{-1} (U) \subset {f}^{-1} ({\phi}^{-1} (V))$. It is clear
  that $\psi \circ f = \tilde g \circ \psi$.
\end{proof}

\begin{prop}\label{PropDynamicalDegreeFibrationOverCurve}
  Let $S$ be a quasiprojective surface and $f$ be a dominant endomorphism of $S$. Suppose there exists
  a quasiprojective curve $C$ with a dominant morphism $\pi: S \rightarrow C$ and an endomorphism $g:
  C \rightarrow C$ such that $\pi \circ f = g \circ \pi$. Then, the first dynamical degree of $f$ is
  an integer.
\end{prop}

\begin{proof}
  Let $X$ be a completion of $S$; $f$ extends to a rational transformation of $X$. We can also
  suppose that $C$ is a projective curve, and then we apply Theorem \ref{PropSteinFactorization} to
  suppose also that $\pi$ has connected fibers.

  Let $P$ be a general point of $C$ and $H$ an ample divisor of
  $X$. We have by \cite{dinhComparisonDynamicalDegrees2011, truongRelativeDynamicalDegrees2015} that
  \begin{equation}
  \lambda_1 (f) = \max \left( \lambda_1 (g), \lambda_1 (f_{|\pi})) \right)
\end{equation}
where $\lambda_1(g)$ is the integer given by the topological degree of $g$ and
\begin{equation}
  \lambda_1 (f_{|\pi}) := \lim_n \left( H \cdot (f^n)_*{\pi}^{-1} (P) \right)^{1/n}.
  \label{ }
\end{equation}
Since $C$ is a curve and $\pi$ is dominant we have that
$\pi$ is flat (\cite{hartshorneAlgebraicGeometry1977} Proposition III.9.7) so for any point $P \in C$,
\begin{itemize}
  \item ${\pi}^{-1} (P)$ is an
    irreducible curve $C_P$ and the topological degree of $f: C_P \rightarrow C_{g(P)}$ is an
    integer $d$ that does not depend on $P$
  \item  $d \cdot d_{\text{top}} (g) = \lambda_2 (f)$.
\end{itemize}
Indeed, consider the following $0$-cycle in $S \times S$:
\begin{equation}
  \alpha (P) = (\pi_1^* C_P) \cdot (\pi_2^* H) \cdot \Gamma_f
  \label{<+label+>}
\end{equation}
where $\pi_1, \pi_2 : S \times S \rightarrow S$ are the two projections and $\Gamma_f$ is the graph of
$f$. The degree of $\alpha (P)$ is
\begin{equation}
  \deg \alpha(P) = (H \cdot C_{g(P)}) \cdot \deg (f : C_P \rightarrow C_{g(p)}).
\end{equation}
Now, since $C$ is a curve the morphism $\pi \circ \pi_1 : S \times S \rightarrow C$ is flat, therefore
$\deg (\alpha (P))$ does not depend on $P$ (\cite{fultonIntersectionTheory1998} \S 20.3) and since
$\pi$ is flat, the intersection number $(H \cdot C_P)$ does not depend on $P$ either. Therefore, $\deg
(f : C_P \rightarrow C_{g(P)})$ is an integer $d$ independent of $P$. Hence, we infer
\begin{equation}
  \lambda_1 (f_{|\pi}) = \lim_n \left( H \cdot (f^n)_* {\pi}^{-1} P \right) = d \cdot \lim_n \left( H \cdot {\pi}^{-1} P
  \right)^{1/n} = d
\end{equation}
and we get that $\lambda_1(f)$ is the integer $\max(d, \lambda_1(g))$.
                \end{proof}

                \section{Dynamical degree when the quasi-Albanese variety is non-trivial}
                \label{SubsecDynDegreeNonTrivialQAlb}

                The goal of this section is to show the following proposition.

                \begin{thm}\label{thm:dynamical_degree_quasi_albanese}
                  Let $X_0$ be an irreducible normal affine surface and $f$ a separable dominant endomorphism of $X_0$. Suppose
                  that $\QAlb (X_0)$ is non-trivial, then $\lambda_1 (f)$ is an algebraic integer of degree $\leq 2$.
                  Furthermore, if $\lambda_1 (f)$ is not an integer or if $f$ is a loxodromic automorphism, then $X_0
                  \simeq \G_m^2$.
                \end{thm}

                \begin{proof}
                  Set $Q_0=\QAlb(X_0)$ and let $q : X_0\to Q_0$ be a quasi-Albanese morphism.
                  Let $V = \overline{q (X_0)}$ be the closure of the image of $X_0$ in $Q_0$. By the universal
                  property, there exists an endomorphism $g$ of $Q_0$ such that $q \circ f = g \circ q$ and $g$ must
                  preserve $V$. We cannot have $\dim V = 0$, otherwise $Q_0 = 0$, therefore $\dim V =1,2$.

                  If $\dim V = 1$, then $f$ must preserve the curve fibration $q :X_0 \rightarrow V$. By Proposition
                  \ref{PropDynamicalDegreeFibrationOverCurve}, $\lambda_1(f)$ is an integer. If $f$ is an automorphism
                  then it cannot be loxodromic because loxodromic automorphisms cannot preserve any fibration.

                  If $\dim V =2$, we look at $\overline \kappa (X_0)$. If $\overline \kappa (X_0) = 2$, then $X_0$ is of
                  log general type and by  Corollary \ref{CorLogGeneralType}, $f$ is of finite order and $\lambda_1
                  (f) = 1$. If $\overline \kappa (X_0) =1$, then $f$ preserve the curve fibration given by the log Kodaira
                  Iitaka fibration and by Proposition \ref{PropDynamicalDegreeFibrationOverCurve}, $\lambda_1 (f)$ is an
                  integer. If $\overline \kappa (X_0) = - \infty$, then by Theorem \ref{ThmAffineRuled}, a
                  desingularisation of $X_0$ admits an open subset of the form $\A^1 \times C$ where $C$ is a curve. Since
                  $\QAlb (\A^1) = 0$, we cannot have $\dim V = 2$ and this is a contradiction. Finally, if $\overline
                  \kappa (X_0) = 0$, then by Proposition \ref{PropVanishingLogKodairaDimensionImpliesRational} $X_0$ must
                  be rational and $\QAlb (X_0)$ is an algebraic torus $\G_m^s$. Let $r$ be the rank of $\k [X_0]^\times /
                  \k^\times$. If $r = 0$, then $X_0$ does not admit any nontrivial morphism into an algebraic torus and
                  therefore $\QAlb (X_0) =0$ this is a contradiction. If $r = 1$, then $X_0$ admits a curve fibration
                  $X_0 \rightarrow \G_m$ that is $f$-invariant and $\lambda_1 (f)$ is an integer by Proposition
                  \ref{PropDynamicalDegreeFibrationOverCurve}. Finally, we must have $r =2$ and by Theorem
                  \ref{ThmCaracAlgebraicTorusAnyChar}, we have $X_0 = \G_m^2$ and $\lambda_1 (f)$ is either an integer or
                  a quadratic integer. This is the only case where $\lambda_1 (f)$ is not an integer or where $f$ can be a
                  loxodromic automorphism.
                \end{proof}

                \begin{proof}[Proof of Theorem \ref{BigThmClassificationQAlbNonTrivial}]
                  Suppose $\car \k = 0$, following the proof of Theorem \ref{thm:dynamical_degree_quasi_albanese},
                  we have that if $X_0$ is not of log general type or if there is no curve fibration $X_0 \rightarrow C$
                  invariant by every endomorphism of $X_0$, then we must have $\overline \kappa (X_0) = 0$ and
                  $q(X_0)=  \dim \QAlb(X_0) = 2$. By Theorem \ref{ThmCaracAlgebraicTorusChar0} we must have that
                  $X_0 \simeq \G_m^2$.

                  In positive characteristic the theorem also holds if we restrict the statement to separable
                  endomorphisms.
                \end{proof}

                \chapter{Dynamics when $\k[X_0]^\times = \k^\times$ and $\Pic^0 (X_0) = 0$}\label{ChapterDynamicsQuasiAlbTrivial}
                In this chapter, we will prove Theorem \ref{BigThmExistenceValuationPropreEng} and derive Theorems
                \ref{BigThmDynamicalDegreesEng} and \ref{BigThmDynamicalSpectrum}. The two hypothesis allows one to
                describe the Picard-Manin space of $X_0$ more precisely. In particular, we show that $\Vinf$ embeds into
                $\Winf_\R$ and $\Vinf '$ embeds into $\L2$.

                \section{The structure of the Picard-Manin space of $X_0$}
                From \S \ref{SecPicardManin} we have linear maps
                \begin{equation}
                  \tau: \Cinf_\R \rightarrow \cNS_\R, \quad
                  \tau: \Winf_\R \rightarrow \wNS_\R.
                \end{equation}
                For this section we suppose that $X_0$ is a normal affine surface over an algebraically closed
                field $\k$ such that
                \begin{enumerate}
                  \item $\k[X_0]^\times = \k^\times$;
                  \item For all completion $X$ of $X_0, \Pic^0 (X) = 0$.
                \end{enumerate}
                It suffices to test the second condition on one completion of $X_0$ as the Albanese variety of a
                projective variety is a birational invariant. We will make an abuse of notations and write $\Pic^0 (X_0)
                = 0$ for the second hypothesis.

                If these two conditions are satisfied, the finite dimensional subspace $\DivInf (X)$ embeds into
                $\NS(X)$. Indeed, consider the composition
                \begin{equation}
                  \DivInf(X) \rightarrow \Pic (X) \rightarrow \NS(X),
                  \label{<+label+>}
                \end{equation}
                the first map is injective since $\k[X_0]^\times = \k^\times$ and the second is an isomorphism because
                $\Pic^0 (X) = 0$. Therefore the maps $\tau$ are injective and we have the
                orthogonal decomposition
                \begin{equation}
                  \wNS_\R = \Winf_\R \obot \phantom{.}V
                \end{equation}
                where $V$ is a finite-dimensional vector space(this decomposition also holds over $\Q$); in fact let $X$ be a
                completion of $X_0$, then $V$ is the orthogonal of $\DivInf(X)$ in $\NS (X)$.

                \subsection{The intersection form at infinity}

                \begin{prop}\label{PropIntersectionFormNonDegenerateAtInfinity}
                  Let $X$ be a completion of $X_0$, then

                  \begin{itemize}
                    \item $\DivInf (X)_\A$ embeds into $\NS(X)_\A$ and the intersection form is non degenerate on
                      $\DivInf(X)_\A$.
                    \item The perfect pairing $\cNS_\R \times \wNS_\R \rightarrow \R$ induces a pairing
                      \begin{equation}
                        \Cinf_\R \times \Winf_\R \rightarrow \R
                      \end{equation}
                      that is also perfect.
                    \item $\Winf_\R$ is isomorphic, as a topological vector space, to $\Cinf_\R^*$ endowed with the
                      weak-$*$ topology.
                  \end{itemize}
                \end{prop}

                \begin{proof}
                  Everything follows from Propositions \ref{PropIntersectionFormNonDegenerateAtInfinityGeneralForm} and
                  \ref{PropPerfectPairing} and that $\tau: \DivInf (X) \hookrightarrow \NS (X)$ is injective.
                \end{proof}

                \begin{cor}\label{CorLPlusClosedSubset}
                  The subspace $\LPLus$ is a closed subspace of $\Winf_\R$ with the weak-$\star$ topology.
                \end{cor}

                \begin{proof}
                  All the conditions that  elements of $\LPLus$ have to satisfy are closed conditions. Indeed, we have
                  \begin{equation}
                    \LPLus = C_1 \cap C_2
                    \label{<+label+>}
                  \end{equation}
                  where
                  \begin{align}
                    C_1 &= \bigcap_{D \geq 0} \left\{ L(D) \geq 0 \right\}\\
                    C_2 &= \bigcap_{D, D' \in \Cinf} \left\{ L (D \wedge D') = \min (L(D), L(D')) \right\}.
                    \label{<+label+>}
                  \end{align}
                \end{proof}

                \subsection{A continuous embedding of $\Vinf$ into $\Winf_\R$}

                From Proposition \ref{PropIntersectionFormNonDegenerateAtInfinity}, we get the immediate corollary.

                \begin{cor}
                  For any valuation $v$ centered at infinity, there exists a unique $Z_v \in \Winf_\R$ such that for
                  all $D \in \Cinf_\R, L_v (D) = Z_v \cdot D$.

                \end{cor}

                \begin{cor}\label{CorEbeddingIntoWinf}
                  A valuation $v$ is divisorial if and only if $Z_v$ belongs to $\Cinf_\R$. In particular, for any prime
                  divisor $E$ at infinity, $Z_{\ord_E} \in \Cinf_\Q$. The embedding
                  \begin{equation}
                    v \in \cV_\infty \mapsto Z_v \in \Winf_\R
                  \end{equation}
                  is a continuous map for the weak topology.
                \end{cor}

                \begin{proof}
                  If $v$ is divisorial, then there exists a completion $X$ such that the center of $v$ is a prime
                  divisor $E$ at infinity. For every $W \in \Winf, L_{\ord_E} (W) = L_{\ord_E, X} (W_X)$, by
                  Proposition \ref{PropExtensionNaturelleValuationDivisorielle}. By non-degeneracy of the
                  intersection pairing on $\DivInf (X)_\Q$, there exists $Z \in \DivInf (X)_\Q$ such that
                  for all $D \in \DivInf(X)_\Q, L_{\ord_E, X}(D) = Z \cdot D$. It follows that $Z_{\ord_E}$ is the Cartier class
                  defined by $Z$, hence it is an element of $\Cinf_\Q$.

                  Conversely, if $Z_v \in \Cinf_\R$, let $X$ be a completion where $Z_v$ is
                  defined. The center of $v$ over $X$ cannot be a closed point $p$; otherwise let $\tilde E$ be the
                  exceptional divisor above $p$, we would have $L_v (\tilde E) >0$, but $Z_v \cdot \tilde E = 0$.

                  Now to show the continuity of the map of the Corollary, it suffices by Proposition
                  \ref{PropIntersectionFormNonDegenerateAtInfinity} to show that for any $D \in
                  \Cinf_\R$, the map $v \in \Vinf \mapsto Z_v \cdot D$ is continuous. It actually suffices to show
                  this for $D \in \Cinf$ and this follows immediately from
                  $Z_v \cdot D = L_v (D)$ and Proposition \ref{PropTopologiesOverDivisorsAndFunctionsAreTheSame}.
                \end{proof}

                \begin{prop}\label{PropIncarnationDivisorOfValuation}
                  Let $v$ be a valuation centered at infinity and $X$ a completion of $X_0$ such that $c_X
                  (v)\in E$ is a free point. Then, the incarnation
                  of $Z_v$ in $X$ is

                  \begin{equation}
                  Z_{v, X} = (Z_v \cdot E) Z_{\ord_E}. \end{equation}

                  If $c_X (v) = E \cap F$ is a satellite point, then
                  \begin{equation}
                    Z_{v, X} = (Z_v \cdot E) Z_{\ord_E} + (Z_v \cdot F) Z_{\ord_F}.
                    \label{<+label+>}
                  \end{equation}

                  Furthermore, if $\pi: Y \rightarrow X$ is the blow up of a point at infinity $p \neq c_X
                  (v)$, then
                  \begin{equation}
                    Z_{v, Y} = \pi^* Z_{v, X}.
                    \label{EqMemeIncarnation}
                  \end{equation}
                \end{prop}

                \begin{proof}
                  If $c_X (v) \in E$ is a free point. For any $D \in \DivInf(X)_\R$, one has $D = \sum_F L_{\ord_F} (D) F$,
                  therefore by Proposition \ref{PropValuationForCartierDivisorOverOneCompletion} (2) and (3) $L_v (D) =
                  L_{\ord_E} (D) L_v (E) $ . Since $(Z_v \cdot E) = L_v (E)$, we get the result. The proof is similar
                  for the case $c_X (v) = E \cap F$.

                  For the last assertion, if $\tilde E$ is the exceptional divisor of $\pi : Y \rightarrow X$,
                  then by definition
                  \begin{equation}
                    Z_{v, Y} = \pi^* Z_{v, X} - (Z_v \cdot \tilde E) \tilde E
                    \label{<+label+>}
                  \end{equation}
                  However, since $c_X (v) \neq p$, we have that $c_Y (v) \not \in \tilde E$ and therefore $Z_v
                  \cdot \tilde E = 0$ by Proposition \ref{PropValuationForCartierDivisorOverOneCompletion}.
                \end{proof}

                Recall that in \S\ref{SubSecLocalDivisorValuation}, we have defined for a point $p$ at infinity in a completion
                $X$ the local divisor $Z_{v, X, p}$ for every valuation $v$ centered at $p$. The divisor is defined by duality
                via the following property
                \begin{equation}
                  \forall D \in \Cartier(X, p)_\R, \quad L_v (D) = Z_{v, p, X} \cdot D.
                  \label{<+label+>}
                \end{equation}

                \begin{cor}\label{CorDivisorOfValuationWithLocalDivisorOfValuation}
                  Let $X$ be a completion of $X_0$ and let $v$ be a valuation centered at infinity.
                  \begin{itemize}
                    \item If $p := c_X (v) \in E$ is a free point, then
                      \begin{equation}
                        Z_v = (Z_v \cdot E) Z_{\ord_E} + Z_{v, X, p}
                        \label{<+label+>}
                      \end{equation}
                    \item If  $p := c_X(v) = E \cap F$ is a satellite point, then
                      \begin{equation}
                        Z_v = (Z_v \cdot E) Z_{\ord_E} + (Z_{v} \cdot Z_{\ord_F}) Z_{\ord_F} + Z_{v, X, p}
                        \label{<+label+>}
                      \end{equation}
                  \end{itemize}

                  In particular, $Z_v \in \L2$ if and only if $v$ is quasimonomial or there exists a completion $X$
                  and a closed point $p \in X$ at infinity such that $c_X (v) = p$ and $\alpha (\tilde v ) < +
                  \infty$ where $\tilde v$ is the valuation equivalent to $v$ such that $\tilde v \in \cV_X (p;
                  \m_p)$.
                \end{cor}

                \begin{proof}
                  We have that
                  \begin{equation}
                    Z_v = Z_{v, X} + Z'
                    \label{ }
                  \end{equation}
                  where $Z' \in \Winf$ is exceptional above $X$. Now, for every divisor $D$ exceptional above $X$, we have
                  \begin{equation}
                    L_v (D) = Z_v \cdot D = Z' \cdot D.
                    \label{<+label+>}
                  \end{equation}
                  If $D$ is exceptional above a point $q \neq p$, then $L_v (D) = 0$ by Proposition
                  \ref{PropValuationForCartierDivisorOverOneCompletion} as $q \neq c_X (v)$. Therefore, we get that $Z' = Z_{v,
                  X, p}$.

                  Now, we have $Z_v \in \L2 \Leftrightarrow (Z_v)^2 < - \infty$. Replace $v$ by the equivalent valuation such
                  that $v \in \cV_X (p; \m_p)$, then by Theorem \ref{ThmAutoIntersectionIsSkewness} $(Z_{v, X, p})^2 = -
                  \alpha (v)$ and therefore
                  \begin{equation}
                    (Z_v)^2 = (Z_{v, X})^2 - \alpha (v).
                    \label{<+label+>}
                  \end{equation}
                  This shows the result.
                \end{proof}

                \begin{cor}\label{CorDivisorOfValuationIsRational}
                  Let $v \in \Vinf$, then up to normalisation $Z_v \in \Winf_\Q$ if and only $v$ is not irrational.
                \end{cor}
                \begin{proof}
                  First, if $v$ is divisorial, the result follows from Corollary \ref{CorEbeddingIntoWinf}. Then, if $v$ is
                  infinitely singular or a curve valuation. Then, there exists a completion $X$ such that $c_X (v)$ is a free
                  point $p \in E$. Then, replace $v$ by its equivalent valuation such that $v \in \cV_X (p; \m_p)$. Let
                  $(z,w)$ be local coordinates at $p$ such that $z = 0$ is a local equation of $E$. Then, $Z_v (E) = v (z) =
                  \alpha (v \wedge v_z) \in \Q$ because $v \wedge v_z$ has to be a divisorial valuation. Therefore, by Corollary
                  \ref{CorLocalDivisorOfIrrationalValuationIsIrrational} and Proposition \ref{PropIncarnationDivisorOfValuation}, we
                  get that $Z_{v} \in \Winf_\Q$.

                  Finally, if $v$ is irrational then let $X$ be a completion such that $c_X (v) = E \cap F$ is a satellite
                  point. Then, $Z_{v, X} = s Z_{\ord_E} + t Z_{\ord_F}$ with $s/t \not \in \Q$ by Proposition
                  \ref{PropIncarnationDivisorOfValuation}. It is clear that no multiple of $Z_{v, X}$ can be in
                  $\DivInf(X)_\Q$.
                \end{proof}

                \begin{cor}\label{CorStrongConvergenceInL2}
                  Let $\cV_\infty '$ be the subspace of $\cV_\infty$ consisting of $v \in \Vinf$ such that $Z_v \in \L2$, then
                  \begin{equation}
                    \cV_\infty ' \hookrightarrow \L2
                    \label{<+label+>}
                  \end{equation}
                  is a continuous embedding for the strong topology. Furthermore, it is a homeomorphism onto its image.
                \end{cor}
                \begin{proof}
                  Let $X$ be a completion of $X_0$. Let $v_n$ be a sequence of $\cV_\infty '$ converging towards $v \in
                  \cV_\infty '$ for the strong topology. We treat two cases, whether $v$ is associated to a prime divisor of $X$
                  or $v$ is centered at a closed point $p \in X$ at infinity.

                  If $v$ is centered at a closed point $p$ at infinity, then since $v_n$ converges strongly towards $v$ then it
                  converges also weakly, therefore for $n$ big enough, $v_n$ is centered at $p$ by Proposition
                  \ref{PropConvergenceDesCentres}. We can replace each $v_n$ and
                  $v$ by their representative such that $v_n, v \in \cV_X (p; \m_p)$. Then
                  \begin{itemize}
                    \item If $p \in E$ is a free point,
                      \begin{equation}
                        Z_{v_n} = (Z_{v_n} \cdot E) Z_{\ord_E} + Z_{v_n, X, p}
                        \label{<+label+>}
                      \end{equation}
                    \item If $p = E \cap F$ is a satellite point, then
                      \begin{equation}
                        Z_{v_n} = (Z_{v_n} \cdot E) Z_{\ord_E} + (Z_{v_n} \cdot F) Z_{\ord_F} + Z_{v_n, X, p}
                        \label{<+label+>}
                      \end{equation}
                  \end{itemize}
                  and we have similar formulas for $Z_v$. Now the incarnation of $Z_{v_n}$ in $X$ converges
                  towards the incarnation of $Z_{v}$ in $X$ in both the free and the satellite case by weak convergence. Let
                  $\left||{\cdot} \right||$ be any norm over $\NS(X)_\R$, then
                  \begin{equation}
                    \left||{Z_v - Z_{v_n}} \right||^2_{\L2} \asymp \left||{Z_{v, X} - Z_{v_n, X}} \right||^2 - \left( Z_{v,X, p} - Z_{v_n, X, p}
                    \right)^2
                    \label{<+label+>}
                  \end{equation}
                  where $f \asymp g$ means that there exists constants $A, B > 0$ such that $A g \leq f \leq B g$. By Proposition
                  \ref{PropStronConvergenceForLocalDivisor}, we have that $\left||{Z_v - Z_{v_n}} \right||^2_{\L2} \rightarrow 0$.

                  If $v \simeq \ord_E$ for some prime divisor $E$ at infinity in $X$, then for all $n$ large enough, $c_X
                  (v_n) \in E$. We can suppose that $v = \ord_E$ and for all n $v_n (E) > 0$, i.e $v, v_n \in \cV_X
                  (E)$ and $Z_{v_n} \cdot E \rightarrow 1$ as $n \rightarrow \infty$. We show that
                  \begin{equation}
                    \frac{Z_{v_n}}{Z_{v_n} \cdot E} \xrightarrow[n \rightarrow +\infty]{} Z_{\ord_E}
                    \label{<+label+>}
                  \end{equation}
                  in $\L2$. We can replace $v_n$ by its equivalent valuation
                  such that $v_n \in \cV_X (p_n, \m_{p_n})$ where $p_n = c_{X} (v_n)$. Then, we have
                  that $Z_{v_n ,X} / Z_{v_n} \cdot E$ converges towards $Z_{\ord_E}$ in $\NS(X)_\R$ by weak convergence. It
                  suffices to show
                  \begin{equation}
                    \frac{(Z_{v_n, X, p})^2}{ (Z_{v_n} \cdot E)^2} \rightarrow 0
                  \end{equation}
                  but this is equal to
                  \begin{equation}
                    - \frac{\alpha_{\m_{p_n}} (v_n)}{ v_n (E)^2} = - \frac{\alpha_E (v_n )}{v(E)^2} \xrightarrow[n \rightarrow + \infty]{} 0
                  \end{equation}
                  by Theorem \ref{ThmAutoIntersectionIsSkewness} and Proposition \ref{PropRelationSkewness} so we are done.

                  Finally, to show the homeomorphism, we have to show that if $Z_{v_n} \rightarrow Z_{v}$ in $\L2$, then $v_n$
                  converges strongly towards $v$. Let $X$ be a completion of $X_0$. Suppose first that $c_X (v)$ is a point at
                  infinity. Let $\tilde E$ be the exceptional
                  divisor above $c_X (v)$, we have $Z_v \cdot \tilde E > 0$, therefore for all n large enough $Z_{v_n} \cdot
                  \tilde E > 0$ and $c_{X} (v_n ) = c_{X} (v) =: p$. Now, we can suppose that $v_n, v \in \cV_X (p; \m_p)$,
                  it suffices to show that $v_n \rightarrow v$ for the strong topology of $\cV_X (p; \m_p)$ and this is a direct
                  consequence of Proposition \ref{PropStronConvergenceForLocalDivisor}.

                  If $c_X (v) =E$ a prime divisor at infinity, then for all $n$ large enough, $Z_{v_n} \cdot E > 0$. Suppose
                  that $v = \ord_E$ and $v_n \in \cV_X (E)$. We have that $Z_{v_n, X} / Z_v \cdot E \rightarrow Z_{\ord_E}$
                  in $\NS (X)_\R$. We need to show that $\alpha_E (\frac{v_n}{v_n (E)}) \rightarrow
                  0$. We can suppose that $v_n \in \cV_X (p_n, \m_{p_n})$ where $p_n = c_X (v_n)$, then by Proposition
                  \ref{PropRelationSkewness},
                  \begin{equation}
                    \alpha_E \left( \frac{v_n}{v_n (E)} \right) = \frac{\alpha_{\m_{p_n}} (v_n)}{v_n (E)^2}.
                    \label{<+label+>}
                  \end{equation}
                  Thus, by Proposition \ref{PropRelationSkewness} and Theorem \ref{ThmAutoIntersectionIsSkewness}
                  \begin{equation}
                    \alpha_E \left( \frac{v_n}{v_n (E)} \right) = \left| \frac{Z_{v_n, X, p_n}^2}{(Z_{v_n} \cdot E)^2} \right|
                    \xrightarrow[n \rightarrow + \infty]{} 0.
                    \label{<+label+>}
                  \end{equation}

                \end{proof}

                \begin{cor}\label{CorValuationDeCourbeAutoIntersectionInfinie}
                  If $v$ is a curve valuation, then $Z_v$ is a Weil class satisfying $Z_v^2 = - \infty$.
                \end{cor}

                \begin{proof}
                  Let $X$ be a completion of $X_0$, let $p = c_X (v)$ and replace $v$ by the valuation equivalent to $v$ such
                  that $v \in \cV_X (p; \m_p)$. We have by Corollary \ref{CorDivisorOfValuationWithLocalDivisorOfValuation}
                  that \begin{equation}
                    Z_v = Z_{v, X} + Z_{v, X, p}.
                    \label{<+label+>}
                  \end{equation}
                  Therefore, by Theorem \ref{ThmAutoIntersectionIsSkewness}
                  \begin{equation}
                    (Z_v)^2 = Z_{v, X}^2 + (Z_{v, X, p})^2 = Z_{v, X}^2 - \alpha (v) = - \infty
                    \label{<+label+>}
                  \end{equation}
                  because $\alpha (v) = - \infty$ for any curve valuation $v$ (see \cite{favreValuativeTree2004} Lemma 3.32).
                \end{proof}

                \subsection{Positivity of $Z_v$}\label{subsec:positivity-divisor-valuations}
                We give here a criterion to determine whether $Z_v$ is nef.
                \begin{thm}\label{thm:nef-valuation}
                  Let $v \in \Vinf$, then the following are equivalent.
                  \begin{enumerate}
                    \item \label{item:nef} $Z_v$ is nef.
                    \item \label{item:nonnegative-self-intersection} $Z_v^2 \geq 0$.
                    \item \label{item:divisor-effective} $Z_v \geq 0$ as an element of $\Winf_\R$.
                  \end{enumerate}
                  Furthermore, if $v$ is \emph{not} divisorial and $Z_v^2 \geq 0$, then for every completion $X$ of
                  $X_0$, $Z_{v,x}^2 > 0$ and $\Supp Z_{v,x} = \BD$.                 
                \end{thm}
                \begin{proof}
                  It is clear that \ref{item:divisor-effective} $\Rightarrow$ \ref{item:nef} $\Rightarrow$
                  \ref{item:nonnegative-self-intersection}. So we show \ref{item:nonnegative-self-intersection}
                  $\Rightarrow$ \ref{item:divisor-effective}. Assume first that $v = \ord_E$ is divisorial and let $X$
                  be a completion such that $E \subset \BD$. Assume first that $Z_{\ord_E}^2 > 0$. Then, for every
                  $F \neq E$ at infinity we have $Z_{\ord_E}\cdot F = 0$, by the Hodge index theorem this implies that
                  the intersection form in negative definite on $\Vect_\R (F; F \neq E)$. We can assume that
                  $Z_{\ord_E}^2 = 1$. Write
                  \begin{equation}
                    Z_{\ord_E} = E + \sum_{F \neq E} a_F F.
                    \label{<+label+>}
                  \end{equation}
                  Start with the following lemma.
                  \begin{lemme}\label{lemme:existence-effectif-exceptionnel}
                    There exists a divisor $D = \sum_{F \neq E} b_F F$
                  with $b_F > 0$ and such that for every $F \neq E, D \cdot F < 0$.
                  \end{lemme}
                  \begin{proof}
                    Consider the $\R$-vector space $V:= \Vect_R (F; F \neq E)$. The intersection form induces a negative
                    definite form on $V$ such that for $F \neq G, F \cdot G \geq 0$. Let $(\hat F, F \neq E)$ be the
                    dual basis, by Lemma 1.8 of \cite{gignacLocalDynamicsNoninvertible2021}, we have for all $F,G \neq
                    E, \hat F \cdot \hat G \leq 0$ and it is $<0$ if and only if $F$ and $G$ belong to the same
                    connected component of $\BD \setminus E$. In particular, the divisors $\hat F$ are anti-effective
                    and supported over $\BD \setminus E$. We define
                    \begin{equation}
                      D = - \sum_F \hat F.
                      \label{<+label+>}
                    \end{equation}
                    The divisor $D$ is effective and satisfies $\Supp D = \BD \setminus E$ and for every $F \neq E, D
                    \cdot F = -1$.
                  \end{proof}
                  Now, suppose that there exists $F \neq E$ such that $a_F \leq 0$ and such that for every $G \neq
                  E$, $\frac{a_F}{b_F} \leq \frac{a_G}{b_G}$. Then we have
                  \begin{equation}
                    Z_{\ord_E} \cdot F = 0 \geq \sum_{G \neq E} a_G G \cdot F = \sum_{G \neq E} \frac{a_G}{b_G} \left(
                    b_G G \cdot F \right)  \geq \frac{a_F}{b_F} D \cdot F \geq 0.
                    \label{<+label+>}
                  \end{equation}
                  This implies that $a_F = 0$ and also that for every $G \neq F$ such that $G \cap F \neq \emptyset$ we
                  have $a_G = 0$. But this is absurd because repeating this process we get that $a_G = 0$ for every
                  divisor $G$ in the same connected component of $F$ in $\BD \setminus E$. Then we can replace $F$ by
                  one of the divisor in the same connected component of $\BD \setminus E$ such that $F \cap E \neq
                  \emptyset$ and we would have that $Z_{\ord_E} \cdot F > 0$. So we see that $\supp Z_{\ord_E} = \BD$.

                  Now suppose that $Z_{\ord_E}^2 = 0$. Since $Z_{\ord_E} \cdot E = 1$ there must exists $F \neq E$ such
                  that $F \cap E \neq \emptyset $ and such that $Z_{\ord_E} \cdot Z_{\ord_F} > 0$. Consider 
                  for $s> 0$ the divisor
                  \begin{equation}
                    Z_s = Z_{\ord_E} + s Z_{\ord_F}.
                    \label{<+label+>}
                  \end{equation}
                  We have
                  \begin{equation}
                    Z_s^2 = 2s Z_{\ord_E}\cdot Z_{\ord_F} + s^2 Z_{\ord_F}^2.
                    \label{<+label+>}
                  \end{equation}
                  So that for $0 < s \ll 1$ we have
                  \begin{equation}
                    Z_s^2 >0, Z_{\ord_E}\cdot Z_s > 0, \text{ and } Z_{\ord_F} \cdot Z_s > 0.
                    \label{<+label+>}
                  \end{equation}
                  Then for every divisor $G \neq E,F$ we have that $G \cdot Z_s =
                  0$. By the Hodge index theorem, this implies that the intersection form is negative definite over
                  $\Vect_\R (G; G \neq E,F)$ and by a similar proof as in the previous case we have that $Z_s \geq 0$
                  and actually that $\Supp Z_s = \BD$.
                  Since $Z_s \xrightarrow[s \rightarrow 0]{} Z_{\ord_E}$ we have the result. What we have shown is the
                  following corollary.
                  \begin{cor}\label{cor:support-everywhere}
                    Let $X$ be a completion of $X_0$ and $E,F$ are prime divisors at infinity.
                    \begin{enumerate}
                      \item If $Z_{\ord_E}^2 > 0$, then $Z_{\ord_E} \geq 0$ and $\Supp Z_{\ord_E} = \BD$.
                      \item Let $s > 0$ and define $Z_s = Z_{\ord_E} + s Z_{\ord_F}$ such that $Z_s^2, Z_{\ord_E} \cdot
                        Z_{s}, Z_{\ord_F} \cdot Z_s > 0$, then $Z_s \geq 0$ and $\supp Z_s = \BD$.
                    \end{enumerate}
                  \end{cor}

                  Now we treat the general case. If $v$ is infinitely singular or a curve valuation then we can fix a
                  completion $X$ where $c_X (v) = p \in E$ is a free point. Normalise $v$ such that $Z_v \cdot E = 1$
                  and approximate $v$ by its infinitely near sequence $v_n$. We have that $v_n < v$ and therefore
                  $\alpha_E (v_n) < \alpha_E (v)$. And we have
                  \begin{equation}
                    Z_{v_n}^2 = Z_{\ord_E}^2 - \alpha_E (v_n) > Z_{\ord_E}^2 - \alpha_E (v) = Z_v^2.
                    \label{<+label+>}
                  \end{equation}
                  Therefore if $Z_v^2 \geq 0$, we have $Z_{v_n}^2 > 0$ so that $Z_{v_n} \geq 0$ and since $Z_{v_n}
                  \rightarrow Z_v$ for the weak topology we have that $Z_v \geq 0$. Notice that $Z_{v_n} = Z_{v,X_n}$ so
                  that $\Supp Z_{v,X_n} = \partial_{X_n} X_0$ by Corollary \ref{cor:support-everywhere}.

                  Finally if $v$ is irrational, then let $X$ be a completion such that $c_X (v) = p \in E \cap F$ is a
                  satellite point. Consider the normalisation such that $L_v (E) = 1$. Then $v$ is the monomial
                  valuation $v_{1,s_0}$ centered at $p$ with $s_0 > 0$ irrational. By Corollary
                  \ref{CorDivisorOfValuationWithLocalDivisorOfValuation} we have
                  \begin{equation}
                    Z_{v_{1,s}}^2 = Z_{\ord_E}^2 + 2s Z_{\ord_E} \cdot Z_{\ord_F} + s^2 Z_{\ord_F}^2 - s.
                    \label{<+label+>}
                  \end{equation}
                  This is a polynomial function $\phi(s)$ of degree $\leq 2$ and we have $\phi (s_0) \geq 0$. If $\phi
                (s_0) > 0$, then for $s$ rational sufficiently close to $s_0$ we have $\phi(s) = Z_{v,s}^2 > 0$ so
                  that $Z_{v_{1,s}} \geq 0$ and we get $Z_{v_{1,s_0}} \geq 0$ by letting $s \rightarrow s_0$. Suppose
                  now that $\phi (s_0) = 0$. If $\phi$ is constant equal to $0$ then we get the result by approximation. If
                  $\phi$ is of degree 1, then $\phi (s) > 0$ for every $s > s_0$ sufficiently close to $s_0$ or for
                  every $s < s_0$ depending on the monotonicity of $\phi$. Finally if $\phi$ is of degree 2 $s_0$ is of
                  the root of a degree 2 polynomial so that $\phi (s) > 0$ for $s > s_0$ sufficiently close or
                  $s < s_0$ depending on the sign of the discriminant of $\phi$ and we conclude by approximation. Unless
                  $s_0$ is a double zero of $\phi$. But then we would get $\phi ' (s_0) =0 $ which implies
                  \begin{equation}
                    2 s_0 Z_{\ord_F}^2 + 2 Z_{\ord_E}\cdot Z_{\ord_F} - 1 = 0.
                    \label{<+label+>}
                  \end{equation}
                  By Corollary \ref{CorEbeddingIntoWinf} this would imply that $s_0 \in \Q$ which is absurd.

                  We now show the last statement. Let
                  $X \leftarrow X_1 \leftarrow X_2 \leftarrow \cdots \leftarrow X_n \leftarrow \cdots$ be the sequence
                  consisting of blowing up the center of $v$. For every $n \geq 1$, we have that $c_{X_n} (v) = E_n \cap
                  F_n$ and $Z_{v, x_n}$ is of the form
                  \begin{equation}
                    Z_{v,x_n} = \alpha_n Z_{\ord_{E_n}} + \beta_n Z_{\ord_{F_n}}
                    \label{eq:<+label+>}
                  \end{equation}
                  with $\alpha_n, \beta_n > 0$.
                  If $\phi(s) > 0$, then we can assume that for every $n \geq 1, Z_{\ord_{E_n}^2},
                  Z_{\ord_{F_n}}^2 > 0$ and then $\Supp Z_{v,x_n} = \partial_{X_n} X_0$ since $Z_{v, X_n} \geq
                  Z_{\ord_{E_n}}$.

                  If $\phi(s) = 0$. Suppose $\phi$ is constant, then we can assume that for every $n\geq 1$,
                  $Z_{\ord_{E_n}}^2 = Z_{\ord_{F_n}}^2 = 0$. Since $Z_v$ is nef and $Z_v^2 = 0$, we have that
                  $Z_{v,X_n}^2 > 0$. This implies that $Z_{\ord_{E_n}} \cdot Z_{\ord_{F_n}} > 0$ and we therefore get
                    \begin{equation}
                      Z_{v,X_n}^2, Z_{\ord_{E_n}} \cdot Z_{v,X_n}, Z_{\ord_{F_n}} \cdot Z_{v,X_n} > 0.
                      \label{eq:<+label+>}
                    \end{equation}
                    By Corollary \ref{cor:support-everywhere}, we have that $Z_{v, X_n} \geq 0$ and $\Supp Z_{v,X_n} =
                    \partial_{X_n} X_0$.

                    Otherwise $\phi(s) = 0$ with $\phi$ of degree 1 or 2 and we can assume after a finite number of
                    blowups that there exists a completion $X$ such that $c_X(v) = E \cap F$ with $Z_{\ord_E}^2 >0,
                    Z_{\ord_F}^2 > 0$ and $v = v_{1,s}$. In particular, we have that $\phi(s) = 0$ and that
                    $\phi'(s) < 0.$ Since $Z_v^2 = 0$ and $Z_v$ is nef, we have that for any
                    completion $Y$ of $X_0$, $Z_{v,Y}^2 > 0$. To apply Corollary \ref{cor:support-everywhere} we need to
                    show the following. For any $t$ sufficiently close to $s$, we have
                    \begin{equation}
                      \psi(t) := Z_{v_{1,t}} \cdot Z_{v_{1,s}} >0.
                      \label{eq:<+label+>}
                    \end{equation}
                    Suppose first $t <s$ and write $t = s - \alpha$ for some $\alpha > 0$. Then
                    \begin{equation}
                      \psi(t) = Z_{\ord_E}^2 + (s+t) Z_{\ord_E} \cdot Z_{\ord_F} + st Z_{\ord_F}^2 - t.
                      \label{eq:<+label+>}
                    \end{equation}
                    Replacing $t = s-\alpha$ and using $\phi (s) = \psi(s) = 0$ we get
                    \begin{equation}
                      \psi (t) = - \alpha \left( s Z_{\ord_F}^2 + Z_{\ord_E} \cdot Z_{\ord_F} - 1 \right)
                      \label{eq:<+label+>}
                    \end{equation}
                    Since $\phi'(s) =  2s Z_{\ord_F}^2 + 2Z_{\ord_E} \cdot Z_{\ord_F} - 1 < 0$ we get that
                    \begin{equation}
                      \psi(t) = -\alpha (\frac{1}{2} \phi'(s) - \frac{1}{2}) > 0.
                      \label{eq:<+label+>}
                    \end{equation}

                    If now $t > s$, we get
                    \begin{equation}
                      \psi (t) =  Z_{\ord_E}^2 + (s+t) Z_{\ord_E} \cdot Z_{\ord_F} + st Z_{\ord_F}^2 - s.
                      \label{eq:<+label+>}
                    \end{equation}
                    Notice that $\psi(t) \geq 0$ because $Z_v = Z_{v_{1,s}} \geq 0$. And $\psi(t)$ is either an affine
                    map or a constant one. If it is affine then we must have $\psi(t) > 0$ for every $t > s$ because
                    $\psi(s) = 0$ and $\psi(t) \geq 0$. Otherwise $\psi$ is constant but this would imply
                    \begin{equation}
                      Z_{\ord_E} \cdot Z_{\ord_F} + s Z_{\ord_F}^2 = 0.
                      \label{eq:<+label+>}
                    \end{equation}
                    Thus, $s$ would be rational by Corollary \ref{CorEbeddingIntoWinf} and this is a contradiction.
                \end{proof}

                \section{Endomorphisms}
                \begin{prop}\label{PropPushForwardOfCenter}
                  Let $f$ be an endomorphism of $X_0$ and let $X, Y$ be completions of $X_0$ such that the lift
                  $F: X \rightarrow Y$ of $f$ is regular. Let $p \in X$ be a closed point and $q := F(p) \in
                  Y$. Then,
                  \begin{itemize}
                    \item $f_* \cV_X (p) \subset \cV_Y (q)$.
                    \item $f_*$ preserves the set of divisorial, irrational and infinitely singular valuations.
                    \item If $v_C$ is a curve valuation centered at infinity and such that $f_* v_C$ is still centered at
                      infinity, then $f_* v_C$ is also a curve valuation.

                  \end{itemize}

                \end{prop}

                \begin{proof}
                  The map $F$ induces a local ring homomorphism $F^*: \hat{\OO_Y (q)} \rightarrow \hat{\OO_X (p)}$.
                  Let $v$ be a valuation centered at $p$. For $\phi \in \OO_Y (q), f_* v (\phi) = v
                  (F^* \phi) \geq 0$ and for $\psi \in \m_{Y,q}, f_* v (\psi) = v (F^* \psi) >0$. Therefore
                  $f_* v$ is centered at $q$. The fact that $f_*$ preserves the type of valuations is shown in
                  Proposition \ref{PropPreservationOfType}. It only remains to show the statement for curve valuations.
                  Let $p = c_X (v_C)$ and $q = c_Y (f_* v_C)$, here we assume that $q \not \in X_0$. We have that
                  $F(p) = q$. By Proposition \ref{PropPreservationOfType} $f_* v_C$ is not a curve valuation only if
                  it is contracted by $F$. But the only germ of holomorphic curve at $p$ that can be contracted by $F$
                  is the germ of a prime divisor $E$ at infinity on which $p$ lies, and the curve valuation associated
                  to $E$ does not define a valuation on $\k[X_0]$. So, $f_* v_C$ is a curve valuation.
                \end{proof}

                \begin{ex}
                  It might happen that $f_* v$ is not centered at infinity even though $v$ is; if this is the case then $f$ is not proper.
                  For example, let $X_0 = \A^2$ with affine coordinates $(x,y)$ and consider the completion $\P^2$ with homogeneous
                  coordinates $[X:Y:Z]$. We have the relation $x = X/ Z, y = Y/ Z$. Consider the chart $X \neq 0$ with affine
                  coordinates $y' = Y /X$ and $z ' = Z / X$.
                  Define $v_t$ to be the monomial valuation centered at $[1:0:0]$ such that $v_t (y') = 1$ and
                  $v_t(z') = t$ with $t >0$. Let $P = \sum_{i,j} a_{ij} x^i y^j \in \k [ x,y]$, we have that $v_t (P ) = \min \left\{ j +
                  (j-i)t | a_{ij} \neq 0 \right\}$.
                  Now take the map $f: (x,y)
                  \in \A^2 \mapsto (xy, y)$, $f$ contracts the curve $\left\{ y=0 \right\}$ to the point $(0,0)$ in $\A^2$, hence
                  it is not proper. For
                  any polynomial $P = \sum_{i,j} a_{ij} x^i y^j, f^* P = \sum_{i,j} a_{ij} x^i y^{i+j}$. We get

                  \begin{equation}
                    v_{1,t} (f^* P)  = \min_{i,j} \left\{ i + j (t+1) | a_{ij} \neq 0  \right\}.
                  \end{equation}
                  The center of $f_* v_t$ is $[0:0:1]$ and $f_* v_t$ is the monomial valuation centered at $[0:0:1]$
                  such that $v_t(x) = 1, v_t (y) = t+1$.
                \end{ex}

                \begin{lemme}[Proposition 3.2 of \cite{favreEigenvaluations2007}] \label{LemmeConditionRegularité}
                  Let $f: X_0 \rightarrow X_0$ be a dominant endomorphism and let $X, Y$ be completions of $X_0$.
                  Let $F: X \dashrightarrow Y$ be the lift of $f$, let $p$ be a closed
                  point of $X$ at infinity and $\cV_{X} (p)$ be the set of valuations on $\k[X_0]$ centered at
                  $p$. Then, $F$ is defined at $p$ if and only if $f_* \cV_{X} (p)$ does not contain any divisorial
                  valuation associated to a prime divisor (not necessarily at infinity) of $Y$. If $F$ is defined at
                  $p$, then $F(p)$ is the unique point $q$ such that $f_* \cV_{X} (p) \subset \cV_{Y} (q)$.
                \end{lemme}

                \begin{proof}
                  If $\hat f$ is defined at $p$, then let $q = \hat f (p)$, we have that $f_* \cV_X (p) \subset
                  \cV_Y(q)$ by Proposition \ref{PropPushForwardOfCenter}.

                  Conversely, If $p$ is an indeterminacy point of $\hat f$. Let $\pi: Z \rightarrow X$ be a completion
                  above $X$ such that the lift $F: Z \rightarrow Y$ is regular. Then, $F ({\pi}^{-1} (p))$ contains a
                  prime divisor $E'$ of $Y$. Let $E$ be a prime divisor at infinity in $Z$ above $p$ such that
                  $F(E) = E'$,
                  then $F_* \ord_E = f_* (\pi_* \ord_E)  = \lambda \ord_{E'}$ for some constant $\lambda >0$ and $
                  \ord_{E'} \in f_*
                  \cV_{X} (p)$.
                \end{proof}

                \begin{prop}\label{PropPushForwardOnDivisorIsPushForwardOnValuation}
                  Let $v$ be a valuation over $\k[X_0]$ and let $f: X_0 \rightarrow X_0$ be a dominant endomorphism, then
                  \begin{itemize}
                    \item $f_* Z_v = Z_{f_* v} \mod \Cinf^\perp$.
                    \item If $f$ is proper then $f_*$ preserves $\Winf$ and $f_* Z_v = Z_{f_* v}$.
                  \end{itemize}
                \end{prop}

                \begin{proof}
                  Indeed, let $D \in \Cinf$, then
                  \begin{equation}
                  f_* Z_v \cdot D = Z_v \cdot f^* D = L_v (f^*D) = L_{f_* v} (D) = Z_{f_* v} \cdot D. \end{equation}

                  Therefore, we get that $Z_{f_* v} - f_* Z_v $ belongs to $\Cinf^\perp$. If $f$ is proper, then
                  $\Winf$ is $f_*$-stable and $f_* Z_v \in \Winf$, thus $Z_{f_* v} = f_* Z_v$.
                \end{proof}

                \begin{ex}\label{ExEllipticCurveAtInfinity}
                  Let $E$ be a general $(2,2)$-curve in $\P^1 \times \P^1$ defined by the equation
                  \begin{equation}
                    y^2 - P(x) y + Q(x) = 0
                    \label{<+label+>}
                  \end{equation}
                  where $P(x), Q(x)$ are rational functions of degree 2.
                  If $P$ and $Q$ are general, then $E$ is smooth and irreducible and it is an elliptic curve.
                  Let $X = \P^1
                  \times \P^1$ and $X_0 = X \setminus E$. We have $\Pic^0 (X_0) = 0$ because it is a rational surface and
                  $\k[X_0]^\times = \k^\times$ because $X \setminus X_0$ consists of a single irreducible curve. We have $Z_{\ord_E} =
                  \frac{1}{8} E$. Consider the projection $\pr_1 : X \rightarrow \P^1$ to the first coordinates.
                  Each fiber of $\pr_1$ is isomorphic to $\P^1$ and generically it has two intersection points with
                  $E$. Let $x_0, x_1, x_2, x_3$
                  be the four roots of the discriminant $\Delta = P(x)^2 - 4 Q(x)$. Then, $\pr_1^{-1} (x_i)$ has only one
                  intersection point with $E$. Consider the following selfmap of $X_0$
                  \begin{equation}
                    f(x,y) = \left(x, \frac{y^2 - Q(x)}{2y - P(x)}\right).
                    \label{}
                  \end{equation}
                  It preserves the fibers of $\pr_x$ and it
                  acts as $z \mapsto z^2$ in each fiber where the points $0$ and $\infty$ of $\P^1$ are the
                  intersection points of
                  the fiber with $E$. See Figure \ref{fig:elliptic_fibration}. There are exactly 4 indeterminacy points on $X$,
                  they are the points $(x_i, y_i)$ where $x_i$ is one of the roots of $\Delta$ and $y_i \in \P^1$ is such that
                  $(x_i, y_i) \in E$.

                  \begin{figure}[h]
                    \centering
                    \includegraphics[scale=0.8]{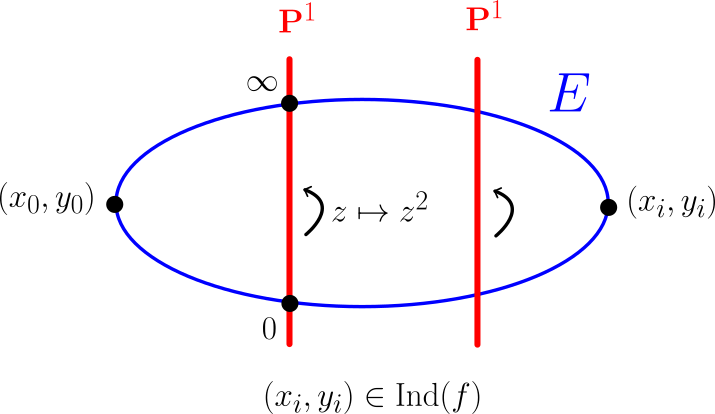}
                    \caption{The endomorphism $f$ on $X_0$}
                    \label{fig:elliptic_fibration}
                  \end{figure}

                  Let $C_0 = \{x_0\} \times \P^1$. Then, $\Cinf^\perp = \R \cdot (4C_0 - E)$ because $C_0 \cdot E = 2, E^2 = 8$
                  and $\dim \NS (X)_\R = 2$.

                  The endomorphism $f$ is not proper, indeed we have in $\NS (X), f_* E = E + 4 C_0$. Since $f^*E$
                  is of the form $f^* E =2 E + \ldots$, we have $f_*
                  \ord_E = 2 \ord_E$. And we get
                  \begin{align}
                    f_* Z_{\ord_E} = \frac{1}{8} E + \frac{1}{2} C_0 \\
                    &= \frac{1}{8} E + \frac{1}{8}(4 C_0 - E) + \frac{1}{8} E \\
                    &= 2 Z_{\ord_E} + \frac{1}{8} (4 C_0 - E)
                    \label{<+label+>}
                  \end{align}
                \end{ex}

                \section{Existence of Eigenvaluations}
                Recall from Theorem \ref{ThmEigenclasses} that there exists unique nef classes $\theta^*, \theta_* \in \L2$ up to
                normalization such that $f^* \theta^* = \lambda_1 \theta^*$ and $f_* \theta_* = \lambda_1 \theta^*$.

                \begin{prop}\label{PropThetaEffectif}
                  If $\k[X_0]^\times = \k^\times$ and $\Pic^0 (X_0) = 0$, then $\theta^* \in \Winf_\R \cap \L2$ and is
                  effective.
                \end{prop}
                \begin{proof}
                  This is a restatement of Corollary \ref{cor:eigenclass-at-infinity} using that $\tau: \Winf_\R
                  \rightarrow \wNS_\R$ is now an embedding.
                \end{proof}

                \begin{thm}\label{ThmExistenceEigenvaluationSurfaces}
                  Let $X_0$ be an irreducible normal affine surface such that $\k[X_0]^\times = \k^\times$ and $\Pic^0
                  (X_0) = 0$. Let $f$ be a
                  dominant endomorphism such that $\lambda_1(f)^2 > \lambda_2 (f)$, then there exists a unique
                  valuation $v_*$ centered at infinity up to equivalence satisfying
                  \begin{align}
                    \forall P \in \k[X_0], v_* (P) &\leq 0 \label{EqPositivité}\\
                    f_* v_* &= \lambda_1 (f) v_* \label{EqInvariance} \\
                    Z_{v_*}^2 &> - \infty \label{EqAlphaFini}
                  \end{align}
                  Furthermore there exists $ w \in \Cinf^\perp$ such that $\theta_* = w + Z_{v_*}$ and
                  $Z_{v_*}$ is nef. In particular, $v_*$ is not a curve valuation.
                \end{thm}
                We call $v_*$ the \emph{eigenvaluation} of $f$.
                \begin{proof}
                  By Theorem \ref{ThmEigenclasses}, there exists nef classes $\theta_*, \theta^* \in \L2$ that satisfy
                  \begin{enumerate}
                    \item $f^* \theta^* = \lambda_1 \theta^*$
                    \item $f_* \theta_* = \lambda_1 \theta_*$
                    \item \label{ItemAsymptotic} $\forall \alpha \in \L2, \frac{1}{\lambda_1^n} (f^n)^* \alpha \rightarrow
                      (\theta_* \cdot \alpha) \theta^*$
                  \end{enumerate}
                  Let $X$ be a completion of $X_0$. Write the decomposition $\theta_* = w + Z$ with $w \in
                  \Cinf^\perp$ and $Z \in \Winf_\R \cap \L2$. Let $E$ be a prime divisor at infinity in $X$
                  such that $Z_{\ord_E} \cdot \theta^* >0$, it exists because $\theta^*$ is effective and nef. Then,
                  Item (\ref{ItemAsymptotic}) and the continuity of the intersection product in $\L2$ imply that for
                  all $D \in \Cinf$,
                  \begin{equation}
                    Z_{\ord_E} \cdot \left( \frac{1}{\lambda_1^n} (f^n)^* D \right) \rightarrow (Z_{\ord_E} \cdot
                    \theta^*) (\theta_* \cdot D) = (Z_{\ord_E} \cdot \theta^*) (Z \cdot D)
                    \label{EqConvergenceFaible}
                  \end{equation}
                  Now, set $v_n := \frac{1}{\lambda_1^n} (f^n)_* \ord_E$. Equation \eqref{EqConvergenceFaible}
                  shows that $Z_{v_n}$ converges towards $Z$ in $\Winf$. But, for all $n$, $Z_{v_n}$ belongs to
                  $\LPLus$ which is a closed set of $\Winf$ by Corollary \ref{CorLPlusClosedSubset}. Therefore, $Z
                  \in \LPLus$ and it defines a valuation $v_*$ by Proposition \ref{PropVThetaCentreeAlInfini}. From
                  the relation $f_* \theta_* = \lambda_1 \theta_*$ we get that $f_* v_* = \lambda_1 v_*$.
                  Using the decomposition $\theta_* = w + Z_{v_*}$ we have
                  \begin{equation}
                    0 \leq \theta_*^2 = w^2 + Z_{v_*}^2.
                  \end{equation}
                  Since $w^2 \leq 0$, we get that $Z_{v_*}^2 \geq 0$ so that $Z_{v_*}$ is nef by Theorem
                  \ref{thm:nef-valuation}. By Corollary \ref{CorValuationDeCourbeAutoIntersectionInfinie}, $v_*$ is not
                  a curve valuation.

                  Now to show the uniqueness of $v_*$, if $v$ is another valuation satisfying Equations
                  \eqref{EqPositivité}, \eqref{EqInvariance}, \eqref{EqAlphaFini}, then for all $D \in \Cinf$, Item
                  (\ref{ItemAsymptotic}) implies
                  \begin{equation}
                    Z_v \cdot D = \frac{1}{\lambda_1^n} Z_v \cdot (f^n)^* D \xrightarrow[n \rightarrow
                    \infty]{}
                    (Z_v \cdot \theta^*) (\theta_* \cdot D)
                    \label{<+label+>}
                  \end{equation}
                  Since $v \neq 0$, we get $Z_v \cdot \theta^* > 0$. And then $v = v_*$ up to a scalar factor.
                \end{proof}

                \begin{rmq}
                  It can happen that $f$ admits a curve valuation $\mu$ such that $f_* \mu = \lambda_1 \mu$. For
                  example take the dominant endomorphism of $\C^2$
                  \begin{equation}
                    f(x,y) = \left( x^2, y^3 \right).
                    \label{<+label+>}
                  \end{equation}
                  Then, $\lambda_1 (f) =3, \lambda_2 (f) = 6$. The curves $x = \pm 1$ are invariant by $f$, so they
                  defines curve valuations
                  at infinity centered at $[0:1:0]$ in $\P^2$. The extension of $f$ to $\P^2$ is the rational map
                  \begin{equation}
                    f [X : Y : Z] = \left[ X^2 Z : Y^3 : Z^3 \right]
                    \label{<+label+>}
                  \end{equation}
                  We see that $p = \left[ 0:1:0 \right]$ is a fixed point of $f$. Take the local coordinates $u = X/
                  Y, v = Z / Y$, then we have
                  \begin{equation}
                    f(u,v) = \left( u^2v, v^3 \right)
                    \label{<+label+>}
                  \end{equation}
                  The curve $x = \pm 1$ becomes $u = \pm v$ in these coordinates. We can see that they are both
                  invariant by $f$ and their curve valuations satisfy $f_* \mu = 3 \mu$. Now, if $v_{1,1}$ is the
                  multiplicity valuation at $p$, then we get also that $f_* v_{1,1} = v_{3,3} = 3 v_{1,1}$. Thus, this
                  is the eigenvaluation of $f$ and it is divisorial.
                \end{rmq}

                \chapter{Local normal forms}\label{ChapterLocalNormalForms}
                Let $X_0$ be an affine surface and $f$ a dominant endomorphism over a field $\k$. We show that the
                existence of this eigenvaluation allows one to find a fixed point at infinity and a local
                normal form at this fixed point, in most cases this fixed point will also be attracting.

                \begin{thm}\label{ThmLocalFormOfMapAndDynamicalDegreAreQuadraticInteger}  Let $X_0$ be an
                  irreducible normal affine surface over an algebraically closed field $\k$. Let $f$ be a
                  dominant endomorphism of $X_0$ such that $\lambda_1^2 > \lambda_2$. Suppose that
                  $\QAlb(X_0) = 0$ and let $Y$ be any completion of $X_0$
                  \begin{enumerate}
                    \item If $v_*$ is infinitely singular or irrational, there exists a completion $X$ above $Y$ such that
                      the lift $ f: X \dashrightarrow X$ is defined at $c_X (v_*)$ and $f (c_X (v_*)) = c_X
                      (v_*)$.  If $\k$ is complete, then $f$ defines a contracting germ of holomorphic function at $c_X (v_*)$ with
                      no $f$-invariant germ of curves at $c_X (v_*)$. If furthermore $\car \k = 0$, the germ can also be made
                      rigid. We have the
                      following local normal form:
                      \begin{enumerate}
                        \item If $v_*$ is infinitely singular, $c_X (v_*) \in E$ is a free point and $f$ has the
                          local normal form \eqref{EqPseudoFormeInfSing} or
                          \eqref{EqLocalNormalFormInfinitelySingular} if $\car \k = 0$ and $\k$ complete with $\{x = 0\}$ a local equation
                          of $E$ and $\lambda_1 = a \in \Z_{\geq 2}$.
                        \item If $v_*$ is irrational, $c_X (v_*) = E \cap F$ is a satellite point. The local
                          normal form is pseudomonomial \eqref{EqPseudoFormeMonomiale} with $(x, y)$ associated to
                          $(E,F)$ and $a_{11} a_{22} - a_{12} a_{21} \neq 0$. The dynamical degree $\lambda_1$ is the
                          spectral radius of the matrix $\left( a_{ij} \right)$. It is a Perron number of degree 2; in
                          particular $\lambda_1 \not \in \Z$. If $f$ is tamely ramified, the local normal form can be
                          made monomial \eqref{EqFormeNormaleMonomiale} with the same matrix $A$.
                      \end{enumerate}
                    \item \label{item:divisorial} If $v_*$ is divisorial, then there exists a completion $X$ above $Y$ such that
                      $c_X (v_*)$ is a prime divisor $E$ at infinity. In that case, $E$ is $f$-invariant and $\lambda_1 \in \Z_{\geq 2}$ is
                      such that $f_X^* E = \lambda_1 E + D$ where $D \in \DivInf(X)$ and $E \not \in \supp D$. And we
                      have three cases.
                      \begin{enumerate}
                        \item $E$ is of general type and $f_{|E}$ is not separable, of infinite order and no iterates
                          of $f_{|E}$ admits a fixed point. In particular,
                          $f$ is not tamely ramified.
                        \item \label{item:normal-form-divisorial} Up to replacing $f$ by some iterate, there exists a fixed point $p \in E$ of
                          $f_{|E}$, $p = E \cap E_0$ is a satellite point, $f : X \dashrightarrow X$ is defined at
                          $p$, $f(p) = p$ and $f$ is a germ of regular function at $p$ (not necessarily contracting)
                          with $E$ the only $f$-invariant germ of curves at $p$. The local normal form of $f$ at $p$ is
                          \eqref{EqLocalNormalFormDivisorial} with $(x,y)$ associated to $(E, E_0)$ and $\lambda_1 =
                          a$. If $\car \k = 0$ and $\k$ complete, the germ can be made rigid and if $f$ is tamely ramified, then $p$
                          can be chosen as a noncritical fixed point of $f_{|E}$ (and in particular $c =1$ in the
                          normal form).
                        \item \label{EllipticCase} The curve $E$ is an elliptic curve and $f_{|E}$ is a
                          translation by a non-torsion element.
                      \end{enumerate}
                  \end{enumerate}

                  In particular, the dynamical degree of $f$ is a Perron number of degree $\leq 2$, and if it is not
                  an integer then the eigenvaluation $v_*$ of $f$ is irrational and the normal form is monomial.
                \end{thm}
                This finishes the proof of Theorem \ref{BigThmDynamicalDegreesEng}.
                We will call \ref{EllipticCase} the \emph{elliptic} case, we will show examples of this case in \S
                \ref{sec:example-elliptic-case}.  The rest of this section is devoted
                to the proof of Theorem \ref{ThmLocalFormOfMapAndDynamicalDegreAreQuadraticInteger}, we will prove the
                Theorem page \pageref{ParLocalNormalForm}.

                To illustrate \ref{item:normal-form-divisorial}, take
                $X_0 = \A^2$ and $f(x,y) = (xP(x,y), xQ(x,y))$ with $P,Q$ homogeneous of same degree $d$ without common
                factor. Then, the eigenvaluation of $f$ is $- \deg = \ord_{L_\infty}$ where $L_\infty$ is the line at
                infinity in $\P^2$. The fixed points on the line at infinity may all be contained in the indeterminacy locus of
                $f$. We will do the following procedure, we pick one such fixed point $p$ and blow it up. If $f$ is
                not defined at the strict transform of $p$ then we blow it up again. After a finite number of blow ups
                we will show that the lift $f$ is eventually defined at the strict transform of $p$.

                To prove the theorem we need to understand the dynamics of $f_*$ on the space of valuations.
                \begin{prop}\label{PropCaracterisationConvergenceVersEigenvaluation}
                  Let $v \in \Vinf$ such that $Z_v \in \L2$. If $Z_v \cdot \theta^* >0$, then
                  $\frac{1}{\lambda_1^n}f^n_* v$ strongly converges towards $(Z_v \cdot \theta^*)v_*$.
                \end{prop}

                \begin{proof}
                  This is a direct consequence of Equation \eqref{EqAsymptoticPushForward} and Corollary
                  \ref{CorStrongConvergenceInL2}.
                \end{proof}
                We will use this to show that $f$ admits a fixed point at infinity on some
                completion and that $f$ contracts a divisor at infinity there.

                For the rest of Chapter \ref{ChapterLocalNormalForms}, we suppose that we are in the conditions of Theorem
                \ref{ThmExistenceEigenvaluationSurfaces}.

                \section{Attractingness of $v_*$, the infinitely singular case}
                We show the following
                \begin{prop}\label{PropInfinitelySingularPreparation}
                  If the eigenvaluation $v_*$ is infinitely singular, then there exists a completion $X$ of $X_0$ such that
                  \begin{enumerate}
                    \item $p := c_X (v_*) \in E$ is a free point at infinity.
                    \item $f_* \cV_X (p) \subset \cV_X (p)$;
                    \item $f$ contracts $E$ to $p$.
                    \item Let $f_\bullet : \cV_X (p; \m_p) \rightarrow \cV_X (p; \m_p)$, then for all $v \in \cV_X
                      (p; \m_p), f_\bullet^n v \rightarrow v_*$.
                  \end{enumerate}
                  Furthermore, the set of completions $Y$ above $X$ that satisfy these 4 properties is cofinal in the set of
                  all completions above $X$.

                \end{prop}

                Let $X$ be a completion of $X_0$ such that $c_X (v_*)$ is a
                free point $p_X \in E_X$. Such a completion $X$ exists and there are infinitely
                many of them above $X$ by Proposition \ref{PropSequenceOfInfinitelyNearPoints}. Let $Y$ be a
                completion above $X$ such that $c_Y(v_*)$ on $Y$ is a free point $p_Y \in E_Y$ and such that the diagram
                \begin{center}
                  \begin{tikzcd}
                    & Y \ar[dl, "\pi"'] \ar[dr, "F"] & \\
                    X \ar[rr, dashed, "f"] & & X
                  \end{tikzcd}
                \end{center}
                commutes, where $F$ is regular and $F(p_Y) = p_X$. Let $x,y$ be coordinates at $p_X$ such that $x = 0$ is a local
                equation of $E_X$ and $ z, w$ be coordinates at $p_Y$ such that $z = 0$ is a local equation for $E_Y$. We
                use the notations of Chapter \ref{ChapterValuationTree}. We have that $f_* \cV_Y (p_Y)
                \subset \cV_X (p_X)$ by Lemma \ref{LemmeConditionRegularité}. We define $F_\bullet:
                \cV_Y(p_Y; E_Y) \mapsto \cV_X(p_X, \m_{p_X})$ as follows:
                \begin{equation}
                  \forall v \in \cV_Y (p_Y; E_Y), \quad F_\bullet (v) := \frac{F_* v}{\min\left( v (F^* x), v
                  (F^* y) \right)}.
                  \label{<+label+>}
                \end{equation}
                Similarly, we define
                \begin{equation}
                  \forall v \in \cV_Y (p_Y; E_Y), \quad \pi_\bullet (v) := \frac{\pi_* v}{\min\left( v (\pi^* x), v
                  (\pi^* y) \right)}.
                  \label{<+label+>}
                \end{equation}
                By Proposition \ref{PropCompatibilitéOrdreContractionMap} item (1), the map $\pi_\bullet : \cV_Y (p_Y; E_Y)
                \rightarrow \cV_X (p_X; \m_{p_X})$ is an inclusion of trees and allows one to view $\cV_Y (p_Y;
                E_Y)$ as a subtree of $\cV_X (p_X; \m_{p_X})$.

                See Figure \ref{fig:embedding_trees}. The tree $\cV_X (p_X, \m_{p_X})$ is in black with its root
                $v_{\m_{p_X}}$ in blue, the tree $\cV_Y (p_Y; E_Y)$ is in orange with its root $\ord_{E_Y}$ in red. One
                can see how $\pi_\bullet $ maps homeomorphically $\cV_Y (p_Y; E_Y)$ to a subtree of $\cV_X (p_X,
                \m_{p_X})$.
                \begin{figure}[htb]
                  \centering
                  \includegraphics[scale=0.8]{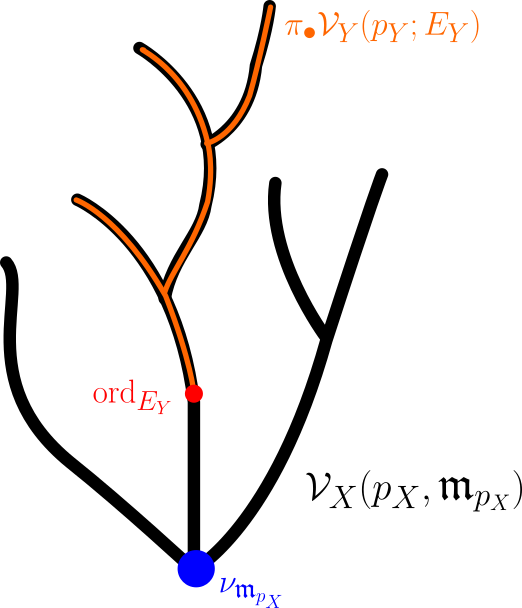}
                  \caption{The embedding $\pi_\bullet$}
                  \label{fig:embedding_trees}
                \end{figure}
                \begin{rmq}\label{RmqOrdreCompatible}
                  Since the orders $\leq_{\m_{p_X}}$ and $\leq_{E_Y}$ are compatible on $\cV_Y (p_Y; E_Y)$
                  and $\pi_\bullet \cV_Y (p_Y; E_Y)$ we will not write $\pi_\bullet$ or $\leq_{E_Y}$ when no
                  confusion is possible to avoid heavy notations.
                \end{rmq}
                By Proposition \ref{PropSkewnessRelationAfterBlowUpContractionMap}, we have the following relation

                \begin{equation}
                  \alpha_{\m_{p_X}} ( \pi_\bullet \mu) = \alpha_{\m_{p_X}} (\pi_\bullet \ord_{ E_Y}) + b(E_Y)^{-2}
                \alpha_{E_Y} (\mu) \end{equation}
                where $b$ is the generic multiplicity function of the tree $\cV_X (p; \m_p)$ and
                $\alpha_{\m_{p_X}}, \alpha_{E_Y}$ are the skewness
                functions defined in Chapter \ref{ChapterValuationTree}. Indeed, with the notation of Proposition
                \ref{PropSkewnessRelationAfterBlowUpContractionMap}, $v_{E_Y} = \pi_\bullet \ord_{E_Y}$.

                \begin{lemme}\label{LemmeFonctionOrderPreservingSurUnBonOuvert}
                  There exists $v \in \cV_Y (p_Y; E_Y)$ such that $v < v_*$ and for all $\mu \geq v$,
                  \begin{equation}
                    \min \left( \mu (F^*x), \mu (F^* y)\right) = b(E_Y) \lambda_1.
                    \label{<+label+>}
                  \end{equation}
                  I.e set $U = \left\{ \mu \geq v \right\}$, we have $F_\bullet = \frac{F_*}{b(E_Y)\lambda_1}$ over $U$.
                  In particular, $F_\bullet$ is order
                  preserving over $U$ and  $F_\bullet ([v, v_*]) \subset \left[ v_{\m_{p_X}}, v_* \right]$.
                \end{lemme}

                \begin{proof}
                  Using Proposition \ref{PropValuationEnFonctionDeAlpha}, we see that the map $ v \mapsto \min (v (f^* x, f^*
                  y))$ is locally constant outside a finite subtree of $\cV_Y (p_Y; E_{p_Y})$. Indeed, one has $f^*x = \prod_i
                  \psi_i$ with $\psi_i$ irreducible and therefore
                  \begin{align}
                    v (f^*x) &= \sum_i v (\psi_i) \\
                    &=  \sum_i \alpha_{E_Y} (v \wedge v_{\psi_i}) m_{E_Y} (\psi_i) \quad \text{by Proposition
                    \ref{PropValuationEnFonctionDeAlpha}}.
                    \label{<+label+>}
                  \end{align}
                  Let $S_x$ be the finite subtree consisting of the segments $[\ord_{E_Y}, v_{\psi_i}]$,
                then the map $\mu \mapsto \mu (f^* x))$ is locally constant outside of $S_x$. Let $S$ be the maximal finite subtree
                of $\cV_Y (p_Y; E_{p_Y})$ such that the evaluation maps on $f^* x, f^* y$ and $z$ are locally constant outside
                of $S$. Since $v_*$ is an infinitely singular valuation it does not belong to $S$ and these three evaluation maps
                are constant on the open connected component $V$ of $\cV_Y (p_Y; E_{p_Y}) \setminus S$ containing $v_*$.
                Since $f_* v_* = \lambda_1 v_*$, this means that $F_* v_* = \lambda_1 \pi_* v_*$. Since the ideal generated by
                $\pi^* x, \pi^*y$ is the ideal generated by $z^{b(E_Y)}$,  we have $f_{\bullet | V} = \frac{f_*}{b(E_Y) \lambda_1}$
                and the map $F_\bullet$ is
                order preserving on $V$. Following Remark \ref{RmqOrdreCompatible}, the two orders $\leq_{\m_{p_X}}$ and
                $\leq_{E_Y}$ agree on $V$. Let $v \in [\ord_{E_Y} , v_*] \cap V$ be a divisorial valuation,
                $F_\bullet$ sends the segment $[v, v_*] \subset \cV_Y (p_Y; E_{Y})$ inside the segment $[v_{\m_{p_X}},
                v_*] \subset \cV_X (p_X; \m_{p_X} )$. Notice that $U := \left\{ \mu \geq v \right\} \subset V$ so the
                valuation $v$ satisfies Lemma \ref{LemmeFonctionOrderPreservingSurUnBonOuvert}.
              \end{proof}

              \begin{prop}[\cite{favreEigenvaluations2007}, Theorem 3.1]\label{PropApplicationAffineSurUnPetitSegment}
                Let $v$ be as in Lemma \ref{LemmeFonctionOrderPreservingSurUnBonOuvert}. For $t \in [\alpha_{E_Y} (v),
                \alpha_{E_Y} (v_*)]$, let $v_t$ be the unique valuation in $[v, v_*]$ such that $\alpha_{E_Y} (v_t) = t$.
                Then, there exists a divisorial valuation $v ' \in [v, v_*]$ such that the map

                \begin{equation}
                t \in [\alpha_{E_Y} (v '), \alpha_{E_Y} (v_*)] \mapsto \alpha_{\m_{p_X}} (F_\bullet v_t) \end{equation}
                is an affine function of $t$
                with nonnegative coefficients.
              \end{prop}

              \begin{proof}
                Let $v_1, v_2 \in \cV_Y (p_Y; E_Y)$ be such that $v < v_1 < v_2 < v_*$. Since $F_\bullet$ is order
                preserving on $U = \left\{ \mu \geq v \right\}$ one has that $F_\bullet$ maps $[ v_1, v_2]$ homeomorphically to
                $[F_\bullet v_1, F_\bullet v_2]$.
                Let $\psi \in \hat{\OO_{X,p_X}}$ be irreducible such that $v_\psi > F_\bullet v_2$, then by Proposition
                \ref{PropValuationEnFonctionDeAlpha},  for all $\mu \in [v_1,
                v_2]$ one has
                \begin{equation}
                  \alpha_{\m_{p_X}} (F_\bullet \mu) = \frac{F_\bullet \mu (\psi)}{ m_{p_X}(\psi)} =
                  \frac{ \mu (f^* \psi)}{m_{p_X}(\psi) b(E_Y)\lambda_1}
                \end{equation}
                Now let $\psi_1, \cdots, \psi_r \in \hat{\OO_{Y, p_Y}}$ be
                irreducible (not necessarily distinct) such that $f^* \psi = \psi_1 \cdots \psi_r$. One has,

                \begin{equation}
                \mu (f^* \psi) = \sum_i \mu (\psi_i) = \sum_i \alpha_{E_Y} (\mu \wedge v_{\psi_i}) m_{E_Y} (\psi_i). \end{equation}

                Take one of the $\psi_i$ and call it $\psi_0$, we shall study the map $\mu \in [v_1, v_2] \mapsto \alpha_{E_Y}
                (\mu \wedge v_{\psi_0})$. Let $\mu_0 = v_2 \wedge v_{\psi_0}$, this map is equal to $\alpha_{E_Y}$ on $[v_1,
                \mu_0]$ and constant
                equal to $\alpha_{E_Y}(\mu_0)$ on $[\mu_0, v_2]$. Therefore, the map $\mu  \in[ v_1, v_2] \mapsto \mu
                (f^* \psi)$
                is a piecewise affine function with nonnegative coefficients of $\alpha_{E_Y} (\mu)$. The points
                on $[v_1, v_2]$ where this map is not smooth are exactly the valuations $v_* \wedge v_{\psi_i}$ and there are
                at most $\lambda_2$ of them by Proposition \ref{PropFiniteNumberOfPreimages}. Therefore the map $\mu \mapsto v
                (f^* \psi)$ is an affine function of $\alpha_{E_Y}$ with nonnegative coefficients on the segment $[\mu', v_*]$
                for any $\mu ' < v_*$ close enough to $v_*$.
              \end{proof}

              As a corollary of the proof, we get  the following proposition.
              \begin{prop}\label{PropLocalisationDeLaFibreValuationDeCourbe}
                Let $v \in \cV_Y (p_Y; E_Y)$ be as in Proposition \ref{PropApplicationAffineSurUnPetitSegment}, let $v_0\in
                [v, v_*]$ and let  $\psi \in \hat \OO_{X, p}$ be irreducible such that $v_\psi > f_\bullet v_0$. Then, for
                all $\phi \in \hat{\OO_{Y, p_Y}}$ such that $f_\bullet v_\phi = v_\psi$, one has two possibilities:
                \begin{enumerate}
                  \item Either $v_\phi > v_0$.
                  \item or $v_0\wedge v_\phi = v_* \wedge v_\phi \leq v$.
                \end{enumerate}
              \end{prop}
              I.e the configuration of Figure \ref{fig:not_possible_configuration} cannot occur.
              \begin{figure}[h]
                \centering
                \includegraphics[scale=0.8]{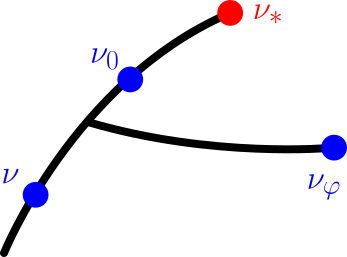}
                \caption{Configuration which is not possible}
                \label{fig:not_possible_configuration}
              \end{figure}

              \begin{proof}
                The map $ \mu \in [v, v_0] \mapsto \alpha_{m_{p_X}} (F_\bullet \mu)$ is a smooth affine function of $\alpha_{E_Y}
                (\mu)$. If (1) and (2) were not satisfied, then we would get $v_\phi \wedge v_* \in [v , v_*]$ and this
                would contradict the smoothness of the map $\mu \in [v, v_*] \mapsto \alpha_{m_{p_X}} (F_\bullet \mu)$
              \end{proof}

              \begin{lemme}\label{LemmeExistenceValuationContractante}
                Let $v$ be as in Proposition \ref{PropApplicationAffineSurUnPetitSegment}. If $\mu \in [v, v_*]$ is
                sufficiently close to $v_*$, then $F_\bullet \mu > \mu $ and $ F_\bullet (\left\{ \mu ' \geq \mu \right\}) \Subset U
                (\overrightarrow v)$ where $\overrightarrow v$ is the tangent vector at $\mu$ defined by $v_*$ and $U
                (\overrightarrow v)$ is its associated open subset.
              \end{lemme}
              We sum up Lemma \ref{LemmeExistenceValuationContractante} in Figure \ref{fig:dynamics_infinite_sing_case}

              \begin{figure}[h]
                \centering
                \includegraphics[scale=0.4]{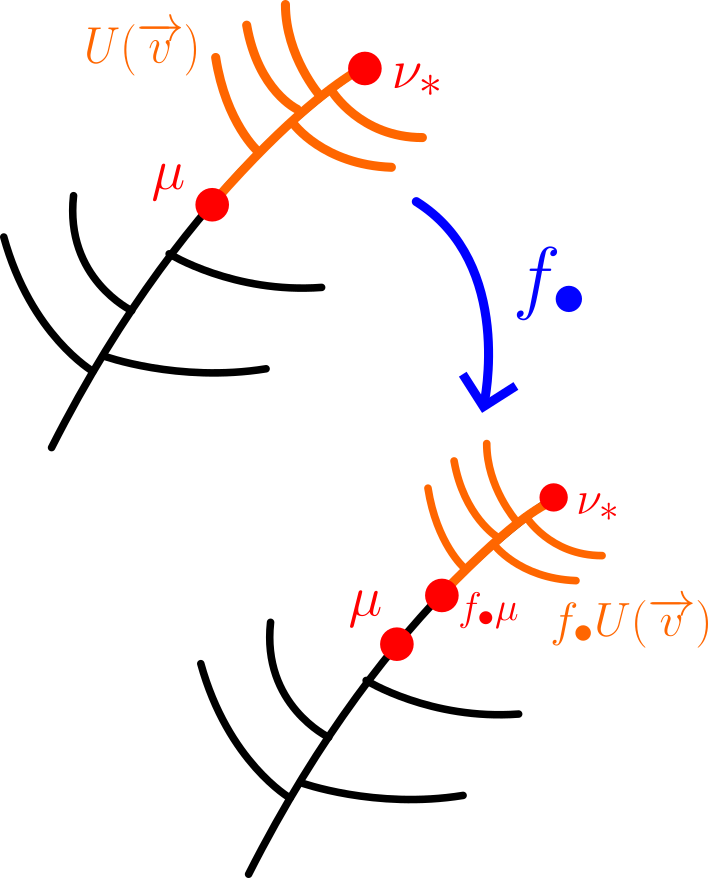}
                \caption{An $f_\bullet$-invariant open subset of $\Vinf$, infinitely singular case}
                \label{fig:dynamics_infinite_sing_case}
              \end{figure}

              \begin{proof}
                Let $U = \left\{ \mu \geq v \right\}$. Recall that $F_\bullet$ is order preserving over $U$. We first notice that
                if every $\mu \in [v, v_*]$ close enough to $v_*$ satisfies $F_\bullet \mu > \mu$, it is clear that $F_\bullet
                \left\{ \mu ' \geq \mu \right\} \Subset U(\overrightarrow v)$. Indeed, let $\mu ' \geq \mu$ and
                set $\mu_0 := \mu ' \wedge v_* \geq \mu$. Then, $F_\bullet \mu ' \geq F_\bullet \mu_0 > \mu_ 0$. In particular,
                $F_\bullet \mu ' \wedge v_* > \mu ' \wedge v_* \geq \mu$.

                Secondly, by Proposition \ref{PropApplicationAffineSurUnPetitSegment}, the map $t \in [\alpha_{E_Y} (v), \alpha_{E_Y}
                (v_*)] \mapsto \alpha_{m_{p_X}} (v_t)$ is affine and we know that it is non decreasing.

                \begin{lemme}\label{LemmeTechniqueFonctionAffine}
                  Let $a : \R \rightarrow \R$ be a non-decreasing non constant affine function that admits a fixed point $t_0$. If
                  there exists $s < t_0$, $a (s) > s$ then the slope of $a$ is $< 1$ and for all $t < t_0$, $a(t) > t$.
                \end{lemme}
                \begin{proof}[Proof of Lemma \ref{LemmeTechniqueFonctionAffine}]
                  We can suppose that $t_0 = 0$ by a linear change of coordinate. Then, $a(t)$ is of the form
                  \begin{equation}
                    a(t) = \alpha t
                    \label{<+label+>}
                  \end{equation}
                  with $\alpha > 0$. Now, if $s < 0$ satisfies $a(s) > s$, this means that $0 < \alpha < 1$ and therefore for all
                  $t < 0$, $a(t) > t$.
                \end{proof}

                We show that there exists $\mu \in [v, v_*]$ such that $F_\bullet \mu > \mu$. If not, then for all
                $\mu \in [v, v_* [, F_\bullet \mu \leq \mu$. Under such an assumption, we show the following
                    \medskip
                    \paragraph{\textbf{Claim}} For all $\mu ' \geq v$ we have $F_\bullet \mu ' \wedge v_* \leq  \mu ' \wedge  v_*$.

                    Suppose that the claim is false and let $\mu '$ be a valuation that contradicts this
                    statement. It is clear that $\mu '$ does not belong to $[v, v_*]$.  Pick $v_0 \in [v , v_*]$ such that
                    $ v \leq  \mu ' \wedge  v_* <  v_0 < F_\bullet \mu ' \wedge
                    v_*$. Let $\phi \in \hat \OO_{Y, p_Y}$ be such that $v_\phi > \mu '$ and let $\psi \in \hat \OO_{X, p}$ be
                    such that $f_{\bullet} v_\phi = v_\psi$. Since $f$ is order preserving we get that  $v_\psi > F_\bullet
                    \mu ' \geq F_\bullet \mu' \wedge v_* > v_0$, therefore $v_\psi > F_\bullet v_0$. But then $\phi$
                    contradicts Proposition \ref{PropLocalisationDeLaFibreValuationDeCourbe} since $v_\phi \wedge v_0 = \mu '
                    \wedge v_0 \in [v, v_0]$. So the claim is shown.

                    Now, pick $\w$ divisorial such that $Z_\w \cdot \theta^* >0$ by Proposition
                    \ref{PropCaracterisationConvergenceVersEigenvaluation} the sequence $\frac{1}{\lambda_1^n}f^n_* \w$ converges
                    towards $ (Z_\w \cdot \theta^*) v_*$. Hence, there exists an integer $N_0 >0$ such that for all $N \geq N_0$,
                    $f_*^N v \in \cV_Y
                    (p_Y) $, replace $\w$ by $f_*^{N_0} \w$ and normalize it such that $\w \in \cV_Y (p_Y, E_Y)$.  We can
                    suppose up to choosing a larger $N_0$ that $\w > v$. In that case $F_\bullet^N \w$ converges towards $v_*$
                    but by the claim, $ \forall N \geq 0, F_\bullet^N \w \wedge v_* \leq \w \wedge v_*$ which is a contradiction.

                    Therefore, there exists a valuation $\mu \in [v, v_*[$ such that $F_\bullet \mu >
                      \mu$.  \end{proof}

                      \begin{prop}\label{PropVecteurTangentInvariantEtToutLeMondeConvergeInfSingularCase}
                        With the notations from Lemma \ref{LemmeExistenceValuationContractante}, we have $F_\bullet
                        (U(\overrightarrow v)) \Subset U(\overrightarrow v)$ and for all $\mu ' \in U(\overrightarrow v)$,
                        \begin{equation}
                          F_\bullet^n \mu' \xrightarrow[n \rightarrow +\infty]{} v_*
                          \label{<+label+>}
                        \end{equation}
                        for the weak topology.
                      \end{prop}

                      \begin{proof}
                        For every $\mu '$ in $U(\overrightarrow v)$, write $\tilde{\mu '} = \mu ' \wedge v_*$. By
                        the proof of Lemma \ref{LemmeExistenceValuationContractante}, $F_\bullet^n (\mu ') \rightarrow v_*$ for
                        the strong topology. Therefore, $F_\bullet^n \mu ' \wedge v_* \geq F_\bullet^n (\tilde{\mu '})
                        \rightarrow v_*$ and $F_\bullet^n \mu'$ converges weakly towards $v_*$ because for all $\phi \in
                        \OO_{Y, p}$ irreducible, we have
                        \begin{equation}
                          F_\bullet^n (\mu ') (\phi) = \alpha_{E_Y} (F_\bullet^n \mu ' \wedge v_\phi) m_{E_Y} (\phi).
                          \label{<+label+>}
                        \end{equation}
                        For $n$ large enough we have $F_\bullet^n \mu ' \wedge v_* \geq v_* \wedge v_{\phi}$, hence
                        $F_\bullet^n \mu ' \wedge v_\phi = v_* \wedge v_\phi$ and
                        \begin{equation}
                          F_\bullet^n (\mu ') (\phi) =  \alpha_{E_Y} (v_* \wedge v_\phi) m_{E_Y} (\phi) = v_* (\phi)
                          \label{<+label+>}
                        \end{equation}
                      \end{proof}

                      \begin{proof}[Proof of Proposition \ref{PropInfinitelySingularPreparation}]
                        Let $v$ be as in Proposition \ref{PropApplicationAffineSurUnPetitSegment}. Let $v_n$ be the
                        infinitely near sequence of $v_*$ (see Proposition \ref{PropInfinitelyNearSequence}). We have
                        for $n$ large enough $v_n \in [v, v_*]$ and $v_n$ satisfies Lemma
                        \ref{LemmeExistenceValuationContractante}. Set $\mu = v_n$ for some $n$ large enough and let
                        $Z$ be a completion such that $c_Z (\mu) = E$ and $c_Z (v_*) =: p \in E$ is a free
                        point. The open subset $U(\overrightarrow v)$ associated to the tangent vector at $\mu$
                        defined by $v_*$ is exactly the image of $\cV_Z (p)$ in $\cV_Y (p_Y; E_Y)$.  By
                        Proposition \ref{PropVecteurTangentInvariantEtToutLeMondeConvergeInfSingularCase}, $F_\bullet
                        U (\overrightarrow v) \Subset U (\overrightarrow v)$, this means that $f_* \cV_Y (p) \subset
                        \cV_Y (p)$. By Lemma \ref{LemmeConditionRegularité}, $f$ is defined at $p, f(p) = p$ and
                        since $F_\bullet \mu > \mu$, we get $f$ contracts $E$ to $p$. We have that for every $\mu \in
                        \cV_Z (p; \m_p), f_\bullet^n \mu \rightarrow v_*$ also by Proposition
                        \ref{PropVecteurTangentInvariantEtToutLeMondeConvergeInfSingularCase}.

                        The statement about cofinalness follows from the fact that the sequence of infinitely near
                        points associated to $v_*$ contains infinitely many free points, so for every completion
                        $X$ of $X_0$, there exists a completion above it where the center of $v_*$ is a free point
                        at infinity.
                      \end{proof}

                      \section{Attractingness of $v_*$, the irrational case}
                      Suppose now that $v_*$ is an irrational valuation. There exists a completion $X$ such that the center of
                      $v_*$ on $X$ and on any completion above $X$ is the intersection of two divisors at
                      infinity $E, F$. We still write $f : X \dashrightarrow
                      X$ for the lift of $f$.

                      Let $Y$ be a completion above $X$ such that $f : Y \dashrightarrow X$ is defined at $c_Y
                      (v_*)$. Then, we must have $f (c_Y (v_*)) = c_X (v_*)$. Write $c_Y (v_*) = E_Y \cap F_Y$ and
                      $c_X (v_*) = E_X \cap F_X$.

                      \begin{prop}\label{PropMonomializationOnlyAtTheCenter}
                        Let $(x,y)$ (resp. (z,w)) be local coordinates at $c_Y (v_*)$ (resp. $c_X (v_*)$) associated to
                        $(E_Y, F_Y)$ (resp. $(E_X, F_X)$), then in these coordinates $f$ is of the form
                        \begin{equation}
                          f(x,y) = \left( x^a y^b \phi, x^c y^d \psi \right)
                          \label{EqFormeMonomiale}
                        \end{equation}
                        with $ad - bd \neq 0$.
                      \end{prop}

                      \begin{proof}
                        It is clear that $f$ is of this quasimonomial form in these coordinates because $f$ is an
                        endomorphism of $X_0$ so no germ of curve in $X_0$ can be sent to $zw = 0$. The only
                        difficulty is to show that $ad -bd \neq 0$. If that was not the case, then there exists $t \in
                        \Q_{>0}$ such that $(a,b) = t (c,d)$. And therefore for all $v$ centered at $c_Y (v_*)$ we
                        should have
                        \begin{equation}
                          \frac{f_* v (z)}{f_* v (w)} = \frac{a v(x) + b v(y)}{c v(x) + d v(y)}= t \in \Q_{>0}
                          \label{<+label+>}
                        \end{equation}
                        But since $f_* v_* = \lambda_1 v_*$ and $v_*$ is irrational we have that $f_* v_* =
                        v_{s,t}$ with $s/t \not \in \Q$ and this is a contradiction.
                      \end{proof}

                      Using this we show

                      \begin{prop}\label{PropGoodCompactificationIrrationalCase}
                        There exists a completion $Y'$ such that
                        \begin{enumerate}
                          \item The lift $\hat f : Y' \rightarrow Y'$ is defined at $p = c_{Y'}(v_*)$;
                          \item $\hat f (p) = p$;
                          \item If $E,F$ are the two divisors at infinity such that $p = E \cap F$, then $\hat f$ contracts at
                            least one of the two divisors and $\hat f^2$ contracts both of them.
                          \item Define $f_\bullet : \cV_{Y'} (p; \m_p) \rightarrow \cV_{Y'} (p; \m_p)$. For all $\mu \in
                            \cV_{Y'}
                            (p; \m_p), f_\bullet^n \mu \rightarrow v_*$ for the weak topology of $\cV_{Y'} (p; \m_p)$.
                        \end{enumerate}
                        Furthermore, If $Z$ is a completion above $Y'$, then (1)-(4) remain true.
                      \end{prop}

                      \begin{proof}
                        Let $Y$ be as in Proposition \ref{PropMonomializationOnlyAtTheCenter}. We still write
                        $f : Y \dashrightarrow X$ for the lift of $f$ and
                        $\pi: Y \rightarrow X$ for the composition of blow ups. Set $p_Y = c_Y (v_*), p_X = c_X
                        (v_*)$. We use the local coordinates from Proposition \ref{PropMonomializationOnlyAtTheCenter}.
                        Consider the tree $\cV_{X} (p_{M}; E_X)$ with its order $<_X$, its skewness function
                        $\alpha_X$ and the generic multiplicity function $b_X$. This tree is rooted in $\ord_{E_X}$ and $F_X$
                        defines the end $v_w$. We have the map $f_\bullet : \cV_{Y}
                        (p_Y; E_Y) \rightarrow \cV_{X} (p_X; E_X)$. The root of $\cV_{Y} (p_Y; E_Y)$ is
                        $\ord_{E_Y}$ and
                        $F_Y$ defines the end $v_y$ in $\cV_{Y} (p_Y; E_Y)$. We
                        have that
                        $\ord_{E_Y} <_Y v_* <_Y v_y$. Using Equation \eqref{EqFormeMonomiale}, we can write
                        \begin{equation}
                          \forall v \in \cV_{Y} (p_Y; E_Y), \quad f_\bullet (v) = \frac{f_* v}{a + b v (y)}.
                          \label{EqFormeDeFbullet}
                        \end{equation}
                        In particular, by Lemma \ref{LemmeActionSurMonomiale} we have $f_\bullet ([\ord_{E_Y}, v_y])
                        \subset [\ord_{E_X}, v_w]$ and
                        \begin{equation}
                          f_\bullet (v_{1,t}) = v_{1, \frac{c+ td}{a + bt}}.
                          \label{<+label+>}
                        \end{equation}
                        If we chose the normalization given by $F_X$, then the root of $\cV_{X} (p_X; F_X)$ is
                        $\ord_{F_X}$ and $E_X$ defines the end $v_z$ and we have $\ord_{F_X} < v_* < v_z$.
                        The map $f_\bullet$ would be of the form
                        \begin{equation}
                          \forall v \in \cV_{X} (p_Y; E_Y) = \frac{f_* v}{c+ dv(y)}
                          \label{<+label+>}
                        \end{equation}
                        and we still have $f_\bullet ([\ord_{F_Y, v_x}]) \subset [\ord_{E_X}, v_w]$.

                        Recall the
                        notations of \S \ref{SecMonomialValuationsAtSatellitepoint}, we have a homeomorphism
                        \begin{equation}
                          N := N_{p_X, F_X} \circ N_{p_X, E_X}^{-1} : \cV_{X} (p_X; E_X) \rightarrow \cV_{X}
                          (p_X; F_X)
                          \label{<+label+>}
                        \end{equation}
                        that is defined as follows
                        \begin{align*}
                          N(v) &= \frac{v}{v(y)} \text{ if } v \neq \ord_{E_X}, v_y \\
                          N(\ord_{E_X}) &= v_x \\
                          N(v_y) &= \ord_{F_X}.
                          \label{<+label+>}
                        \end{align*}
                        In particular, by Lemma \ref{LemmelevelfunctionValuationMonomiale} we have for all $t > 0$
                        \begin{equation}
                          N (v_{1,t}) = v_{1/t, 1}.
                          \label{<+label+>}
                        \end{equation}
                        If $t > 0$ is rational, let $\overrightarrow v$ be a tangent vector at $v_{1,t}$ that is not
                        represented by $\ord_{E_X}$ or $v_y$. The order on $\overline{U(\overrightarrow v)} =
                        U(\overrightarrow v) \cup \left\{ v_{1,t} \right\}$ is compatible with the order on $N \left(
                        \overline{U(\overrightarrow v)} \right)$ by Proposition \ref{PropCompatibilityOrdre}.
                        Comparing with Lemma \ref{LemmeFonctionOrderPreservingSurUnBonOuvert}, we have

                        \begin{lemme}\label{LemmeFonctionOrderPreservingSurUnBonOuvertIrrationalCase}
                          For any rational number $t > 0$, $f_\bullet : \cV_{Y} (p_Y; E_Y) \rightarrow
                          \cV_{X}(p_X; E_X)$ is order preserving
                          on every tangent vector of $v_{1,t}$ which is not represented by $\ord_{E_Y}$ or $v_{y}$
                          (in particular $v \wedge v_y = v_{1,t}$ for such a valuation). Furthermore, for every $v$ in such a
                          connected component of $v_{1,t}$, $f_\bullet v \geq f_\bullet v_{1,t}$ and
                          \begin{equation}
                            (f_\bullet v) \wedge v_{w} = f_\bullet (v \wedge v_y).
                            \label{<+label+>}
                          \end{equation}

                          The analogous statement with $\cV_{X} (p_X; F_X)$ also holds.
                        \end{lemme}
                        \begin{proof}
                          We have that $f_\bullet (v) = \frac{f_*}{a + b v(y)}$ and the evaluation map $v \in
                          \cV_{Y} (p_Y; E_Y) \mapsto v(y)$ is locally constant outside the segment
                          $S = [\ord_{E_Y}, v_y)$ and the connected components of $\cV_{Y} (p_Y; E_Y)
                          \setminus S$ are exactly the tangent vectors of $v_{1,t}$ for $t > 0$ rational that are not
                          represented by $\ord_{E_Y}$ or $v_y$. On each connected component we have
                          \begin{equation}
                            f_\bullet (v) = \frac{f_* v}{a + b t}
                            \label{<+label+>}
                          \end{equation}
                          and therefore $f_\bullet$ is order preserving on each connected component.

                          Now take $v$ in such a connected component of $v_{1,t}$, for every $v_{1,t} < \mu \leq v$,
                          we have $f_\bullet v \geq f_\bullet \mu$, letting $\mu \rightarrow v_{1,t}$ yields
                          $f_\bullet v \geq f_\bullet v_{1,t}$.

                          Now, we can do the same proof with $\cV_{X} (p_X; F_X)$. Denote the two maps
                          \begin{align}
                            f_\bullet^E : \cV_{Y} (p_Y; E_Y) &\rightarrow \cV_{X} (p_X; E_X) \\
                            f_\bullet^F : \cV_{Y} (p_Y; E_Y) &\rightarrow \cV_{X} (p_X; F_X).
                            \label{<+label+>}
                          \end{align}
                          Using the homeomorphism $N$, we have the relation
                          \begin{equation}
                            f_\bullet^F = N \circ f_\bullet^E.
                            \label{<+label+>}
                          \end{equation}
                          And we get that if $v$ represents a tangent vector at $v_{1,t}$ that is not represented by
                          $\ord_{E_Y}$ or $v_y$ (in particular $v_{1,t} = v \wedge v_y$), then
                          \begin{equation}
                            f_\bullet^E (v) \geq f_\bullet^E (v_{1,t}) \text{ and } N (f_\bullet^E (v)) \geq N
                            (f_\bullet^E (v_{1,t})).
                            \label{<+label+>}
                          \end{equation}
                          This implies by Proposition \ref{PropCompatibilityOrdre} that
                          \begin{equation}
                            f_\bullet^E (v) \wedge v_w = f_\bullet^E (v_{1,t}).
                            \label{<+label+>}
                          \end{equation}
                        \end{proof}

                        Suppose without loss of generality that $v_{E_Y}
                        <_X v_{F_Y}$.
                        We have by Proposition \ref{PropCompatibilitéOrdreContractionMap} item (2) that the map
                        $\pi_\bullet : \cV_{Y} (p_Y; E_Y) \rightarrow \cV_{X} (p_X; E_X)$ is an inclusion of
                        trees. Hence, the orders $<_X, <_Y$ are compatible and $\cV_{Y} (p_Y; E_Y)$ is naturally
                        a subtree of $\cV_{X}(p_X; E_X)$ via the map $\pi_\bullet$.
                        Now, both maps $f_\bullet$ and
                        $\pi_\bullet$ send the segment $[\ord{E_Y}, v_{F_Y})$ into the segment $[\ord_{E_X},
                        v_{F_X})$ via a Möbius transformation by Lemma \ref{LemmeActionSurMonomiale}. Indeed, if $v_{1,t} \in \cV_{Y} (p_Y; E_Y)$ is a
                        monomial valuation at $p_Y$, then $f_* v_{1,t} = v_{a+bt, c+td}$ and one has by Lemma
                        \ref{LemmelevelfunctionValuationMonomiale} and
                        Equation \eqref{EqFormeDeFbullet}
                        \begin{equation}
                          \alpha_{M} (f_\bullet v_{1,t}) = \alpha_X \left(v_{1, \frac{c + td}{a + t b}}\right) =
                          \frac{c + \alpha_Y (v_{1,t}) d}{a + \alpha_Y (v_{1,t})b} = M_f (\alpha_Y (v_{1,t}))
                        \end{equation}
                        Where $M_f$ is the Möbius transformation associated to the matrix $\begin{pmatrix} d & c \\ b & a
                        \end{pmatrix}$. We can do the same process with the map $\pi_\bullet$ to get a Möbius transformation
                        represented by a matrix $M_\pi$. Set $M$ to be the Möbius transformation $M_f \circ
                        {M}^{-1}_\pi$. If $t_E := \alpha_X (\pi_\bullet \ord_{E_Y})$ and $t_F := \alpha_X (\pi_\bullet
                        \ord_{F_Y})$, then $M$ sends the segment $[t_E, t_F]$ into $[0; + \infty]$.

                        \begin{lemme}\label{LemmaMatrixLoxodromic}
                          The Möbius map $M$ is loxodromic with an attracting fixed point $t_* = \alpha_X (\pi_\bullet v_*)$
                          and the multiplier of $M$ at $t_*$ is $\leq \sqrt{\frac{\lambda_2}{\lambda_1^2}} <1$.

                          In particular, for every $v_1, v_2 \in \cV_{Y} (p_Y; E_Y)$ close enough to $v_*$ such that
                          $v_1 < v_* < v_2$, $f_\bullet ([v_1, v_2]) \Subset [\pi_\bullet v_1 , \pi_\bullet v_2]$.
                        \end{lemme}

                        \begin{proof}[Proof of Lemma \ref{LemmaMatrixLoxodromic}]
                          Let $\ord_E$ be a divisorial valuation
                          such that $Z_{\ord_E} \cdot \theta^* > 0$, then $f_\bullet^n \ord_E \rightarrow v_*$ by
                          Proposition \ref{PropCaracterisationConvergenceVersEigenvaluation}. So there exists $n_0
                          \geq 0$ such that for all $n \geq n_0, f_\bullet \ord_E$ belongs to $\cV_{Y} (p_Y; E_Y)$.
                          Write $v = f_\bullet^{n_0} \ord_E$ and define $\mu := v \wedge v_y$ and $t = \alpha_X
                          (\pi_\bullet (\mu))$.
                          By Lemma \ref{LemmeFonctionOrderPreservingSurUnBonOuvertIrrationalCase}, we must have for
                          all $k \geq 0$ $f^k_\bullet \mu \in \cV_{Y} (p_Y; E_Y)$ and $f_\bullet^k \mu \rightarrow
                          v_*$ weakly. By Lemma \ref{LemmeWeakStrongTopologiesSameOnSegments}, $f_\bullet^k \mu
                          \rightarrow v_*$ for the strong topology. In terms of skewness, this translates as follows:
                          for all $k \geq 0, M^k (t) \in [t_E, t_F]$ and
                          \begin{equation}
                            M^k (t) \xrightarrow[k \rightarrow \infty]{} t_*.
                            \label{<+label+>}
                          \end{equation}
                          To show that $M$ is a contracting fixed point, we show that there exists $0 < \rho < 1$ such that
                          \begin{equation}
                            \left| M^k (t) - t_* \right| = O (\rho^k)
                            \label{<+label+>}
                          \end{equation}

                          First, since $\theta_* \cdot \theta^* = 1$, up to replacing $\mu$ by some higher iterate
                          $f_\bullet^k \mu$ we can suppose by Corollary \ref{CorStrongConvergenceInL2} that $Z_\mu \cdot
                          \theta^* > 0$ because $f_\bullet^k \mu$ converges towards $v_*$ for the strong topology.

                          By Lemma \ref{LemmelevelfunctionValuationMonomiale} and Proposition
                          \ref{PropPushForwardOnDivisorIsPushForwardOnValuation}, we have
                          \begin{equation}
                            M^k (t) = \alpha_X (f_\bullet^k \mu )= \frac{Z_{f^k_*\mu} \cdot F_X}{Z_{f^k_*\mu} \cdot E_X}
                            = \frac{(f_k)_* Z_\mu \cdot F_X}{(f_k)_* Z_\mu \cdot E_X}.
                            \label{<+label+>}
                          \end{equation}
                          Applying \eqref{EqAsymptoticPushForward} with $\theta_* = Z_{v_*} + w$ with $w \in
                          \Cinf^\perp$. We get
                          \begin{equation}
                            M^k (t) = \frac{(Z_\mu \cdot \theta^* ) (Z_{v_*} \cdot F_X) + O\left(
                            \frac{\lambda_2}{\lambda_1^2} \right)^{k/2}}{(Z_\mu \cdot \theta^* ) (Z_{v_*} \cdot E_X)+ O\left(
                            \frac{\lambda_2}{\lambda_1^2} \right)^{k/2}}.
                            \label{<+label+>}
                          \end{equation}
                          Thus we get
                          \begin{equation}
                            M^k (t) = \frac{Z_{v_*} \cdot F_X}{ Z_{v_*} \cdot E_X} + O (\rho^k) = t_* + O(\rho^k)
                            \label{<+label+>}
                          \end{equation}
                          with $\rho = \sqrt{\frac{\lambda_2}{\lambda_1^2}}$. This finishes the proof.
                        \end{proof}
                        \textbf{End of Proof of Proposition \ref{PropGoodCompactificationIrrationalCase}.--}
                        By Lemma \ref{LemmaMatrixLoxodromic}, pick $v_1, v_2 \in \cV_{Y} (p_Y; E_Y)$ divisorial
                        sufficiently
                        close to $v_*$ such that
                        \begin{equation}
                          \ord_{E_Y} <_Y v_1 <_Y v_* <_Y v_2 <_Y v_y
                        \end{equation}
                        and
                        \begin{equation}
                          f_\bullet ([v_1,
                          v_2]) \subset [\pi_\bullet v_1, \pi_\bullet v_2 ].
                        \end{equation}
                        If $f_\bullet$ is order preserving, then we must have $f_\bullet [v_1,  v_2]
                      \subset ] \pi_\bullet v_1, \pi_\bullet v_2 [ $. If $f_\bullet$ is not order-preserving, it is possible to
                      have $f_\bullet (v_2) = \pi_\bullet v_1$ and $f_\bullet (v_1) \in ] \pi_\bullet v_1, \pi_\bullet v_2
                      [$. In that case, $f_\bullet^2$ is order-preserving and we have $f_\bullet^2 [v_1, v_2] \subset ]
                      \pi_\bullet v_1, \pi_\bullet v_2[$.

                        Let $U_Y = \left\{ v : v_1 < v \wedge v_{F_Y} < v_2 \right\} \subset \cV_{Y} (p_Y;
                        E_Y)$. By Lemma
                        \ref{LemmeFonctionOrderPreservingSurUnBonOuvertIrrationalCase} and the fact that $f_\bullet
                        ([v_1, v_2]) \subset [\pi_\bullet v_1, \pi_\bullet v_2 ]$ we have that $f_\bullet
                        (\pi_\bullet U_Y)$ so we can iterate $f_\bullet$ on $U_Y$. In particular, $f_\bullet^2 (U_Y)
                      \Subset \pi_\bullet (U_Y)$ because $f_\bullet^2 [v_1, v_2] \Subset ]\pi_\bullet v_1,
                      \pi_\bullet v_2[$.

                        \begin{prop}\label{PropVecteurTangentInvariantEtToutLeMondeConvergeInfIrrationalCase}
                          For every $\mu \in U_Y, f_\bullet^n \mu \rightarrow v_*$ for the weak topology.
                        \end{prop}

                        \begin{proof}
                          Let $\mu \in U_Y$ and let $\tilde \mu := \mu \wedge v_2$. We have $f_\bullet^n \tilde \mu \rightarrow
                          v_*$ for the strong topology by Lemma \ref{LemmaMatrixLoxodromic}. By Lemma
                          \ref{LemmeFonctionOrderPreservingSurUnBonOuvertIrrationalCase},
                          we have $f_\bullet^n \mu \wedge v_2 = f_\bullet^n \tilde
                          \mu \wedge v_2$. Therefore for $\phi \in \OO_{Y, p_Y}$ irreducible and for $n$ large enough,
                          $f_\bullet^n \mu \wedge v_\phi = f_\bullet^n \tilde \mu \wedge v_\phi$. Thus,
                          \begin{align}
                            f_\bullet^n \mu (\phi) &= \alpha_Y (f_\bullet^n \mu \wedge v_\phi) m_Y (\phi) \\
                            &= \alpha_Y (f_\bullet^n \tilde \mu \wedge v_\phi) m_Y (\phi) \\
                            &= f_\bullet^n \tilde \mu (\phi) \xrightarrow[n \rightarrow +\infty]{} v_* (\phi).
                            \label{<+label+>}
                          \end{align}
                        \end{proof}

                        Now pick a completion $Y'$ above $Y$ such that for $i=1,2$, the
                        center of $v_i$ is a prime divisor $E_i$ at infinity such that $E_1$ and $E_2$ intersect at a unique
                        point $p$. We have $c_{Y'} (v_*) = p$. The open set $U_Y \subset \cV_{Y} (p_Y; E_Y)$ is the image of
                        $\cV_{Y'} (p; E_1)$ via the inclusion $\cV_X(p; E_1) \hookrightarrow \cV_{Y} (p_Y; E_Y)$. Since
                        $f_\bullet U_Y \subset \pi_\bullet(U_Y)$, this shows that $f_* \cV_X (p)
                        \subset \cV_X (p)$. Therefore by
                        Lemma \ref{LemmeConditionRegularité} the lift $ f : Y' \dashrightarrow Y'$ is defined at $p$,
                        $ f (p) = p$ and since $f_\bullet$ (or $f_\bullet^2$) contracts the segment $[v_1, v_2]$ we have that
                        $f$ contracts $E_1$ and $E_2$ to $p$. We have for every $\mu \in \cV_{Y'}(p; \m_p), f_\bullet^n \mu
                        \rightarrow v_*$ by Proposition
                        \ref{PropVecteurTangentInvariantEtToutLeMondeConvergeInfIrrationalCase}.

                        If $Z$ is a completion above $X'$, then $c_Y (v_*) = F_1 \cap F_2$ where $F_i$ is a prime
                        divisor at infinity because $v_*$ is irrational. The segment $[v_{F_1}, v_{F_2}]$ is a subsegment
                        of $[v_{E_1}, v_{E_2}]$ and the same proof applies. This shows that $Z$ satisfies also
                        Proposition \ref{PropGoodCompactificationIrrationalCase}.
                      \end{proof}

                      \section{Attractingness of $v_*$, the divisorial case}
                      Suppose that $v_*$ is
                      divisorial and let $X$ be a completion such that the center of $v_*$ on $X$
                      is a prime divisor $E$ at infinity. Since $f_* \ord_E = \lambda_1 \ord_E$ we have that $f$ induces a
                      dominant rational selfmap of $E$.

                      \begin{lemme}
                        We have three cases: 
                        \begin{enumerate}
                          \item $E$ is of general type, $f_{|E}$ is not separable and no iterates of $f_{|E}$ admits a
                            fixed point. In particular, $f$ is not tamely ramified.
                          \item There exists an integer $N >0$ such that $f_{|E}^N$ admits a fixed point on
                            $E$, if $f$ is tamely ramified, then the fixed point of $f_{|E}$ can be chosen to be noncritical.
                          \item $E$ is an elliptic curve and $f_{|E}$ is a translation by a non-torsion element of
                            $E$. 
                        \end{enumerate}
                        \end{lemme}

                      \begin{proof}
                        The rational transformation $f$ induces a rational selfmap on $E$. If $E$ is rational, then $E \simeq
                        \P^1$ and $f_{|E}$ is given by a rational fraction $\frac{P(x)}{Q(x)}$ and therefore admits fixed
                        points. It admits a noncritical fixed
                        point if and only if $f_{|E}$ is not a rational fraction of the form $\frac{P_1(x^p)}{Q_1(x^p)}$ where
                        $p = \car \k$ that is if and only if $f_{|E}$ is separable, in particular it holds if $f$ is
                        tamely ramified by Lemma \ref{lemme:tamely-ramified-separable-over-curves}. If $E$ is of general
                        type and we are not in case 1, then either some iterate of
                        $f_{|E}$ admits a fixed point or $f_{|E}$ is separable but then it is of finite order and some
                        iterate of $f$ induces the
                        identity on $E$, this is the case in particular if $f$ is tamely ramified by Lemma
                        \ref{lemme:tamely-ramified-separable-over-curves}. Finally, if $E$ is an
                        elliptic curve, then $f_{|E} = t_{-b} \circ g$ where $g : E
                        \rightarrow E$ is a \emph{homomorphism} of elliptic curves and $t_{-b}$ is the translation by $-b$. We
                        have $f_{|E} (p) = p \Leftrightarrow g(p) - p = b$. Thus $f_{|E}$ admits a fixed point if and
                        only if $g - \id$ is not the trivial homomorphism, i.e $g \neq \id$. Now, by
                        \cite{silvermanArithmeticEllipticCurves2009} III.5, there exists an invariant 1-form $\omega$
                        with no poles or
                        zeros on $E$ such that $g^*
                        \omega = a(g) \omega$ where $a(g) \in \k$. If $a(g) \neq 0$ then every fixed point of $f$ is
                        non-critical,
                        if $a(g) = 0$, then every fixed point of $f_{|E}$ is critical and $f_{|E}$ is inseparable.
                      \end{proof}

                      \begin{rmq}\label{rmq:}
                        We see here that following Remark \ref{rmq:tamely-ramified-at-infinity}, we only use here the
                        fact that $f$ is \emph{tamely ramified at infinity}.
                      \end{rmq}

                      We suppose now that some iterate of $f_{|E}$ admits a fixed point.

                      \begin{prop}\label{PropGoodCompactificationDivisorialCase}
                        There exists a completion $X$ such that
                        \begin{enumerate}
                          \item $c_X (v_*)$ is a prime divisor $E$ at infinity.
                          \item $E$ intersects another prime divisor $E_0$ at infinity and we set $p = E \cap E_0$.
                          \item Up to replacing $f$ by an iterate, $f : X \dashrightarrow X$ is defined at $p$, $f (p) = p$.
                          \item If $f$ is tamely ramified, $p$ is a noncritical fixed point of $f_{|E}$.
                          \item $f$ leaves $E$ invariant and contracts $E_0$ to $p$.
                          \item Define $f_\bullet : \cV_X (p; E) \rightarrow \cV_X (p; E)$, then for all $\mu \in \cV_X
                            (p; E), f_\bullet^n \mu \rightarrow \ord_E$ for the weak topology.
                        \end{enumerate}
                        If $\pi: (Y, \Exc(\pi)) \rightarrow (X, p)$ is a completion exceptional above $p$, then all the
                        item above remain true in $Y$.
                      \end{prop}
                      \begin{proof}

                        Let $X$ be a completion of $X_0$ such that $E \subset X$ and let $p \in E$ be a fixed point of
                        $f_{|E}$ which is non critical if $f$ is tamely ramified. Up to blowing up $p$ once we can
                        suppose that $p$ is a satellite point $p = E \cap E_0$. Define the sequence of completions
                        $(X_n)$ as follows: $X_1 = X, p_1 = p$ and $\pi_n : X_{n+1}
                        \rightarrow X_n$ is the blow up of $X_n$ at $p_n$ and $p_{n+1}$ is the intersection point of
                        the strict transform of $E$ with the exceptional divisor of $\pi_{n+1}$. We still denote by $E$ its
                        strict transform in every $X_n$. For every $n$, we have
                        $f_{|E} (p_n) = p_n$ and if $f : X_n \dashrightarrow X$ is defined at $p_n$, we have $f(p_n) = p$.
                        Let $n_0$ be the minimal $n$ such that $f : X_n \dashrightarrow X$ is defined at $p_n$ and set
                        $Y = X_{n_0}$ and let $\pi : Y \rightarrow X$ be the induced morphism of completions. We still
                        write $p$ for $p_{n_0}$ we have that $p \in Y$ is a satellite point and we write $p = E \cap
                        F_0$.
                        We therefore have a map $f_\bullet : \cV_Y (p, E) \rightarrow \cV_X (p, E)$.
                        Again, the tree $\cV_Y (p, E)$ is a subtree of $\cV_X (p,E)$ via the map $\pi_\bullet$ and they are both
                        rooted at the divisorial valuation $\ord_E$. We write $\alpha_X, \alpha_Y$ for the skewness
                        function of $\cV_X (p; E), \cV_Y (p;E)$ respectively.

                        Let $(x,y), (z,w)$ be local coordinates at $p$ in $Y$ and $X$ respectively such that
                        $(x,y)$ is associated to $(E, F_0)$ and $(z,w)$ to $(E,E_0)$. We can write
                        \begin{equation}
                          f(x,y) = \left( x^a y^b \phi, x^c y^d \psi \right) \text{ and } \pi(x,y) = (xy^{n_0} \phi ',
                          y \psi ')
                          \label{EqFormeMonomialeDivisorialCase}
                        \end{equation}
                        with $\phi, \phi ', \psi, \psi '$ invertible.
                        Since we know that $E$ is not contracted by $f$ we actually have $c = 0$, since $f_* \ord_E =
                        \lambda_1 \ord_E$ we have $a = \lambda_1$ and $d = \deg(f_{|E}: E \rightarrow E)$. In
                        particular, since in $\L2, f_* f^* E = \lambda_2 E$ we must have $\lambda_2 \geq \lambda_1 d$.
                        We can therefore write
                        \begin{equation}
                          f(x,y) = (x^{\lambda_1} y^b \phi, y^d \psi)
                          \label{<+label+>}
                        \end{equation}
                        and thus
                        \begin{equation}
                          \forall v \in \cV_{Y} (p_Y; E), \quad f_\bullet (v) = \frac{f_* v}{ \lambda_1 + b v (y)}.
                          \label{EqFormeDeFbulletDivisorialCase}
                        \end{equation}
                        We have by Lemma \ref{LemmeActionSurMonomiale}
                        \begin{equation}
                          f_\bullet [\ord_E, v_y] \subset [\ord_E, v_w]
                        \end{equation}
                        and the map is given by the following formula
                        \begin{equation}
                          f_\bullet v_{1,s } = v_{1, \frac{sd}{\lambda_1 + sb}}.
                          \label{EqMobiusMapSurValuationMonomiale}
                        \end{equation}
                        As in the irrational case, we can consider the matrix $M_f= \begin{pmatrix}
                          d & 0 \\
                          b & \lambda_1
                        \end{pmatrix}$ and $M_\pi = \begin{pmatrix}
                          1 & 0 \\
                          n_0 & 0
                        \end{pmatrix}$ and study the type of
                        the Möbius transformation induced by $M := {M}^{-1}_\pi \circ M_f$.
                        \begin{lemme}
                          We have
                          \begin{equation}
                            M = \begin{pmatrix}
                              d & 0 \\
                              b - n_0 & \lambda_1
                            \end{pmatrix}
                            \label{<+label+>}
                          \end{equation}
                          and thus $M'(0) = \frac{d}{\lambda_1} \leq \frac{\lambda_2}{\lambda_1^2} < 1$. The matrix $M$
                          is loxodromic and $0$ is an attracting fixed point.
                        \end{lemme}
                        Thus, $\ord_E$ is an attracting fixed point and for every divisorial valuation $v_0 \in
                        [\ord_E, v_y]$ close enough to $\ord_E$, $f_\bullet [\ord_E, v_0] \Subset \pi_\bullet [\ord_E,
                        v_0]$. Fix such a valuation $v_0$.
                        Let $U_Y = \left\{ \mu : \ord_E \leq \mu \wedge v_y < v_0 \right\} \subset
                        \cV_{Y} (p_Y;
                        E)$. Let $\psi \in \hat{\OO_{X, p_X}}$ be irreducible
                        such that
                        $v_{\psi} > f_\bullet ([\ord_E, v_0])$.
                        Let $\psi_1 ,\cdots, \psi_r, \in \hat \OO_{Y,p_Y}$ be irreducible such that
                        $f^* \psi = \psi_1 \cdots \psi_r$. Up to shrinking the segment $[\ord_E, v_0]$ we can
                        suppose that none of the $v_{\psi_i}$ belong to $U_Y$ (See Figure \ref{fig:dynamics_divisorial_case}).
                        If this is the case, then for all $\mu \in U_Y$, set $\tilde \mu = \mu \wedge v_0$, for all
                        $i$
                        \begin{equation}
                          \mu \wedge v_{\psi_i} = \tilde \mu \wedge v_{\psi_i}, \quad \mu \wedge v_{y} = \tilde \mu \wedge
                          v_{y}.
                          \label{EqLocalisationMinimumDeValuationDivisorialCase}
                        \end{equation}
                        Now, for all $\mu \in U_Y$, by Equation \eqref{EqFormeDeFbulletDivisorialCase} and Proposition
                        \ref{PropValuationEnFonctionDeAlpha}
                        \begin{equation}
                          (f_\bullet \mu) (\psi) = \frac{\mu (f^* \psi)}{ \lambda_1 + b \mu (y)} = \frac{\sum_k
                          \alpha_Y (\mu \wedge v_{\psi_k}) m (\psi_k)}{\lambda_1 + b \mu (y)}.
                          \label{<+label+>}
                        \end{equation}
                        By Equation \eqref{EqLocalisationMinimumDeValuationDivisorialCase}, we get
                        \begin{equation}
                          (f_\bullet \mu) (\psi) = (f_\bullet \tilde \mu) (\psi).
                          \label{<+label+>}
                        \end{equation}
                        This means that
                        \begin{equation}
                          \forall \mu \in U_Y, \quad \alpha_X ((f_\bullet \mu) \wedge v_{\psi}) = \alpha_X ( (f_\bullet \tilde
                          \mu) \wedge v_{\psi}).
                          \label{EqLocalisationMinimumDeValuationDivisorialCase2}
                        \end{equation}
                        Since $f_\bullet [\ord_E, v_0] \Subset \pi_\bullet [\ord_E, v_0]$, we have $f_\bullet (U_Y)
                        \Subset \pi_\bullet (U_Y)$. So we can iterate $f_\bullet$ on $U_Y$.
                        \begin{figure}[h]
                          \centering
                          \includegraphics[scale=0.6]{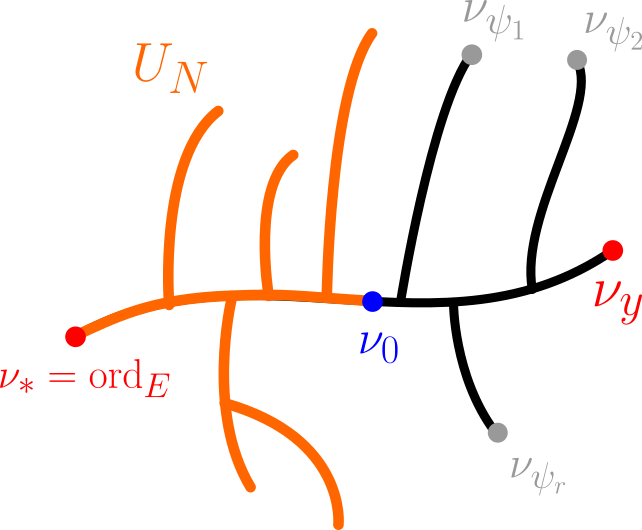}
                          \caption{An $f_\bullet$-invariant open subset of $\Vinf$, divisorial case}
                          \label{fig:dynamics_divisorial_case}
                        \end{figure}
                        \begin{prop}\label{PropVecteurTangentInvariantEtToutLeMondeConvergeInfDivisorialCase}
                          For all $\mu \in U_Y$, $f_\bullet^n \mu \rightarrow \ord_E$ for the weak topology.
                        \end{prop}
                        \begin{proof}
                          The proof is similar to the proof of Proposition
                          \ref{PropVecteurTangentInvariantEtToutLeMondeConvergeInfIrrationalCase}. Let $\mu \in U_Y$ and set
                          $\tilde \mu = \mu \wedge v_0$. Since $\ord_E$ is an attracting fixed point for $f_\bullet$ and
                          $f_\bullet [\ord_E, v_0] \Subset [\ord_E, v_0]$, we have $f_\bullet^n \tilde \mu \rightarrow
                          \ord_E$ for the strong topology. Then, by Equation
                          \eqref{EqLocalisationMinimumDeValuationDivisorialCase2}, $f_\bullet^n \mu \wedge v_0 = f_\bullet^n \tilde
                          \mu$. Let $\phi \in \OO_{Y, p_Y} $ be irreducible such that $\phi$ is not a local equation of $E$, then
                          for $n$ large enough
                          \begin{align}
                            f_\bullet^n \mu (\phi) &= \alpha_E (f_\bullet^n \mu \wedge v_\phi) m_{E} (\phi)  \\
                            &= \alpha_E (f_\bullet^n \tilde \mu \wedge v_{\phi}) m_{E}(\phi) \\
                            &= \alpha_E (f_\bullet^n \tilde \mu) m_{E}(\phi) \xrightarrow[n \rightarrow +
                            \infty]{} 0
                            \label{<+label+>}
                          \end{align}
                        \end{proof}

                        Let $E_0$ be the divisor associated to the divisorial valuation $v_0$ and let $Z$ be a completion
                        such that $c_Z (v_0)$ is the divisor $E_0$ and
                        such that $E_0 \cap E$ is a point $p$. Then, the open subset $U_N$ corresponds to
                        $\cV_Z (p)$ and we have $f_* \cV_Z (p) \subset \cV_Z (p)$. By Lemma
                        \ref{LemmeConditionRegularité}, we have that the lift $\hat f : Z \rightarrow Z$
                        is regular at $p$, $\hat f (p) = p$ and since we know that $f_\bullet v_0 < v_0$ and
                        $f_* \ord_E = \lambda_1 (f) \ord_E$ we have that $\hat f$ contracts $E_0$ at $p$, $E$ is $f$-invariant
                        and for all $\mu \in \cV_Z (p; E), f_\bullet^n \mu \rightarrow v_*$ by Proposition
                        \ref{PropVecteurTangentInvariantEtToutLeMondeConvergeInfDivisorialCase}.

                        If $\pi : (Z ', \Exc (\pi)) \rightarrow (Z, p)$ is a completion exceptional above $p$, then
                        $\Exc(\pi)$ is a tree of rational curves, let $E_0'$ be the irreducible component of $\Exc(\pi)$ that
                        intersect the strict transform of $E$. Then $E_0 ' $ corresponds to a divisorial valuation $v_0 '$
                        such that $\ord_E = v_* < v_0 ' < v_0$ and all the proofs above apply so Proposition
                        \ref{PropGoodCompactificationDivisorialCase} holds also for $Z '$.
                      \end{proof}

                      We finish this section with the following result. When the eigenvaluation is divisorial we have
                      a stronger constraint on the dynamical degrees.

                      \begin{lemme}\label{LemmaDivisorialConditionOnDegrees}
                        When $v_*$ is divisorial, $\lambda_1 \leq \lambda_2$, with equality if and only if $f_{|E} : E
                        \rightarrow E$ has degree 1.
                      \end{lemme}

                      \begin{proof}
                        Let $X$ be a completion such that the center of $v_*$ is a prime divisor $E$ at
                        infinity. Since $f_* v_* = \lambda_1 v_*$, we have that $f^* E = \lambda_1 E + R$
                        where $R$ denotes an effective divisor supported at infinity. Now, we also have $f_* E = dE +
                        R'$. From the equality $f_* \circ f^* = \lambda_2 \id$, we get that $\lambda_1 d
                        \leq \lambda_2$. In particular, $\lambda_1 \leq \lambda_2$.
                      \end{proof}

                      \section{Local normal form of $f$} \label{ParLocalNormalForm}

                      We are now ready to proof Theorem \ref{ThmLocalFormOfMapAndDynamicalDegreAreQuadraticInteger}.

                      \begin{proof}[Proof of Theorem \ref{ThmLocalFormOfMapAndDynamicalDegreAreQuadraticInteger}]

                        Suppose $v_*$ is infinitely singular. From Proposition \ref{PropInfinitelySingularPreparation}, there
                        exists a completion $X$ such that $c_X (v_*) =: p \in E$ is a free point,
                        $f : X \dashrightarrow X$ is defined at $p$, $f_* (\cV_X (p)) \Subset \cV_X (p)$ and $f$
                        contracts $E$ to $p$. We need
                        to show that the germ of holomorphic functions induced by $f$ at $p$ is contracting.

                        First we show that if $\car \k = 0$ and $\k$ is complete, $f$ can be made rigid. It is
                        clear that $E \subset \Crit (f)$ (Recall the notations from \S \ref{SubSecContractingRigidGerm}). If
                        $\Crit (f)$ admits another irreducible component, it induces a curve valuation in $\cV_X (p)$, we can
                        blow up $p$ to get another completion above $X$ satisfying Proposition
                        \ref{PropInfinitelySingularPreparation} such that $\cV_X (p)$ does not contain any curve valuation
                        associated an irreducible component of $\Crit(f)$. Thus, $f$ is rigid at $p$ it remains to show that
                        it is contracting.

                        We do not suppose $\car \k = 0$ anymore but we still assume that $\k$ is complete. We show the
                        germ is contracting. Let $(x,y)$ be local coordinates at $p$ such that
                        $x = 0$ is a local equation of $E$. We must
                        have that $f^*x = x^a \phi$ with $a \geq 1$ and $\phi \in \OO_{X,p}^\times$ because no other
                        germs of curve is contracted to $p$ or sent to $E$ since $f$ is an endomorphism of $X_0$.
                        Since $v_* (E) > 0$ and $f_* v_* = \lambda_1 v_*$ we get that
                        \begin{equation}
                          \lambda_1 v_* (x) = f_* v_* (x) = v_* (x^a \phi) = a v_* (x).
                          \label{<+label+>}
                        \end{equation}
                        Thus, $\lambda_1 = a$ is an integer. Now, since $E$ is contracted by $f$, we get that $f^* y = x \psi$
                        with $\psi \in \OO_{X, p}$ but we must have $\psi \in \m_p$ because otherwise $y=0$ is a germ of
                        curve contained in $X_0$ that is sent to $E$ and this contradicts the fact that $f$ is an
                        endomorphism of $X_0$. Thus, factoring the maximal power of $x$ in $\psi$ we get that
                        \begin{equation}
                          f(x,y) = (x^{\lambda_1} \phi, x^c (x \psi_1(x,y) + y \psi_2(x,y))
                          \label{<+label+>}
                        \end{equation}
                        with $\phi \in \OO_{X, p}^\times$ and $\psi_2(0,y) \neq 0$. This is the local normal form
                        \eqref{EqPseudoFormeInfSing} with $a = \lambda_1$. Consider the norm $\parallel{(x,y)} \parallel = \max
                        (\left| x \right|, \left| y \right|)$ associated to the coordinates $x,y$ and let $U^*$ be the ball of center $p$ and
                        radius $\epsilon >0$. If $\epsilon >0$ is small enough, then $U^*$ is $f$-invariant and $f(U^*)
                        \Subset U^*$, so $f$ is contracting at $p$. Finally, there are no $f$-invariant germ of curves at $p$.
                        Indeed, if
                        $\phi \in \hat{\OO_{X, p}}$ is $f$-invariant, then $f_\bullet v_{\phi} = v_{\phi}$. But we have by
                        Proposition \ref{PropInfinitelySingularPreparation} that $f_\bullet^n v_{\phi} \rightarrow v_*$ and
                        this is a contradiction. If $\car \k = 0$,
                        looking at the classification of the rigid contracting germs in
                        dimension 2, we see that $f$ is in Class 4 of Table 1 in
                        \cite{favreClassification2dimensionalContracting2000} hence of type
                        \eqref{EqLocalNormalFormInfinitelySingular} thus there exists local analytic coordinates
                        $(z,w)$ at $p$
                        \begin{equation}
                        \hat f (z,w) = (z^a, \lambda z^c w + P (z)) \end{equation}
                        where $a \geq 2, c \geq 1, \lambda \in \C^\times$ and $P$ is a polynomial such that $P(0) = 0$. Since $E$
                        is the only germ of curve contracted by $f$ (all the other germs of analytic curves are contained in
                        $X_0$ they cannot be contracted to $p$ by $f$ since $f$ is an endomorphism of $X_0$), we have that $z
                        = 0$ is a local equation of $E$. Furthermore, since $f$ does not have any invariant germ of analytic
                        curve, we get that $P \not \equiv 0$.

                        Suppose now that $v_*$ is irrational, by Proposition \ref{PropGoodCompactificationIrrationalCase},
                        there exists a completion $X$ of $X_0$ such that the lift $ f: X \dashrightarrow X$ is
                        defined at $p = c_X (v_*)$, $X$ contains two divisors at infinity $E,F$ such that $p = E \cap F$
                        and $\hat f$ contracts both $E$ and $F$ at $p$. It remains to show that $f$ is contracting
                        at $p$.

                        First, we show that if $\car \k =0$ and $\k$ is complete, $f$ can be made rigid at $p$. As in the infinitely
                        singular case, we have $\Crit (f) \supset \left\{ xy = 0 \right\}$. If $\Crit f$ has other
                        irreducible components, they define curve valuations and we can blow up more to ensure that
                        $\cV_X (p)$ does not contain them. Therefore $\Crit(f) \cap X_0 = \emptyset$ and $f$
                        is rigid at $p$.

                        We do not suppose $\car \k = 0$ anymore. Since both $E,F$ are contracted to $p$, $f$ is
                        contracting. Finally, there are no
                        $f$-invariant germs of curves at $p$ since for all $\mu \in \cV_X (p; \m_p), f_\bullet^n \mu
                        \rightarrow v_*$ by Proposition \ref{PropGoodCompactificationIrrationalCase}. Let $(z,w)$ be local
                        coordinates at $p$ associated to $(E,F)$. We have that $f$ is of the pseudomonomial form
                        \begin{equation}
                          f(z,w) = \left( z^{a} w^{b} \phi, z^{c} w^{d} \psi \right).
                        \end{equation}
                        with $\phi, \psi \in \OO_{X, p}^\times$ and $a,b,c,d \geq 0$. Notice that
                        $f_* \ord_E = v_{a,b}$ and $f_* \ord_F = v_{c,d}$. Consider the segment of monomial valuations $I$
                        centered at $p$ inside $\cV_X (p ; \m_p)$ we have that $f_\bullet : I \rightarrow I$ is injective,
                        therefore $(a,b)$ is not
                        proportional to $(c,d)$ and $ad - bc \neq 0$.
                        Furthermore the open subset $U^*$ corresponding to
                        the ball of radius $\epsilon >0$ is $f$-invariant for $\epsilon >0$ small enough and $f(U^*)
                        \Subset U^*$.
                        In that case, we show that $\lambda_1 (f)$ is the spectral radius of the invertible matrix $A =
                        \begin{pmatrix}
                          a & b \\
                          c & d
                        \end{pmatrix}$, hence a Perron number of degree 2. Indeed, by Lemma
                        \ref{LemmeActionSurMonomiale}, $v_* = v_{s,t}$ where $(s,t)$ is an eigenvector of $A$ for the
                        eigenvalue $\lambda_1$. Since $v_*$ is irrational, we have $s/t \not \in \Q$ and therefore
                        $\lambda_1 \not \in \Q$. Now, when we iterate $f$, we get that $f^n$ has a pseudomonomial form
                        at $p$ given by the matrix $A^n$, hence we get
                        \begin{equation}
                          \lambda_1^n \begin{pmatrix}
                            v_* (z) \\
                            v_* (w)
                          \end{pmatrix} = A^n
                          \begin{pmatrix}
                            s \\ t
                        \end{pmatrix} \end{equation}
                        If $f$ is tamely ramified, we show that the local normal form can be made monomial. It
                        suffices to show that $ad -bc$ is not divisible by $ \car
                        \k$. Otherwise, there would be positive integers $s,t$ such that $f_* v_{s,t} = (\car \k) v_{s',
                        t'}$ and this
                        contradicts the fact that $f$ is tamely ramified because the value group of $f_* v_{s,t}$ in the value
                        group of $v_{s', t'}$ has index divisible by $\car \k$. Thus, the normal form of $f$ at $p$
                        is analytically conjugated to the monomial form
                        \begin{equation}
                          f(z,w) = \left( z^{a} w^{b}, z^{c} w^{d} \right).
                        \end{equation}

                        Now finally, suppose that $v_*$ is divisorial. Take a completion $X$ as in Proposition
                        \ref{PropGoodCompactificationDivisorialCase}. Let $p = E \cap E_0$ with $v_* = \ord_E$. The lift $f :
                        X \dashrightarrow X$ is defined at $p$.

                        If $\car \k = 0$, up to further blow-ups we can
                        suppose that $\Crit (f) \cap X_0 = \emptyset$ as in the other cases. Therefore, $\Crit (f)
                        \subset E \cup E_0$ which is totally invariant as $f_* \cV_X (p) \Subset \cV_X (p)$ so $f$ is
                        rigid at $p$.

                        We do not suppose $\car \k = 0$ anymore. There are no $f$-invariant germs of curves apart from
                        $E$ at $p$ since for all $\mu \in \cV_X (p; E), f_\bullet^n \mu \rightarrow \ord_E$ by Proposition
                        \ref{PropGoodCompactificationDivisorialCase}. Let $(x,y)$ be local coordinates at $p$ associated to
                        $(E, E_0)$.  Since $f_* \ord_E = \lambda_1 \ord_E$
                        with $\lambda_1 \geq 2$ we have $f^* x = x^{\lambda_1} \phi$ with $\phi \in \OO_{X, p}$ not
                        divisible by $x$. Since no
                        germ of curves is sent to $E$ apart from $E_0$, we have that $f^*x = x^{\lambda_1} y^b \phi $
                        with $\phi \in \OO_{X,p}^\times$ and $b \geq 1$ since $E_0$ is contracted to $p$. Then, $E_0$ is
                        contracted to $p$ so $f^* y = y^c \psi$ with
                        $\psi \in \OO_{X,p}^\times$ and $c = 1$ if $p$ is a noncritical fixed point of $f_{|E}$.
                        Hence, in these coordinates the local normal form of $f$ is
                        \eqref{EqLocalNormalFormDivisorial}:
                        \begin{equation}
                          \hat f (x,y) = \left( x^a y^b \phi, y^c \psi \right)
                        \end{equation}
                        with $a = \lambda_1 \geq 2, b,c \geq 1, \phi, \psi \in \OO_{X,p}^\times$.
                      \end{proof}

                      \section{The dynamical spectrum of the complex algebraic torus}\label{SecDynamicalSpectrum}
                      For an algebraic variety $V$, we have defined in the introduction the dynamical spectrum of $V$ by
                      \begin{equation}
                        \Lambda (V) := \left\{ \lambda_1 (f) : f \in \End(V) \right\}.
                        \label{<+label+>}
                      \end{equation}
                      Recall the definition of Perron numbers, given in the introduction.
                      \begin{prop}\label{PropDynamicalSpectrumTorus}
                        For any field $\k$, $\Lambda (\G_m^2)$ is the set of Perron numbers of degree $\leq 2$.
                      \end{prop}
                      \begin{proof}
                        Any matrix $A = \begin{pmatrix}
                          a & b \\ c & d
                        \end{pmatrix} \in M_2 (\Z)$ yields a monomial automorphism $f_A$ of $\G_m^2$ given by
                        \begin{equation}
                          f_A (x,y) = \left( x^a y^b, x^c y^d \right).
                          \label{<+label+>}
                        \end{equation}
                        Any endomorphism $f$ of $\G_m^2$ is given by the composition of a monomial transformation $f_A$ and a
                        translation. Let $A$ be the matrix associated to the monomial transformation of $f$. Then, $\lambda_1
                        (f)$ is equal to the spectral radius $\rho$ of $A$ (see e.g \cite{favreDegreeGrowthMonomial2012,
                        linPullingBackCohomology2012, hasselblattDegreegrowthMonomialMaps2007,
                      jonssonStabilizationMonomialMaps2011, favreApplicationsMonomialesDeux2003,
                    bedfordLinearRecurrencesDegree2008}). We show that $\rho$ is a Perron number of degree $\leq 1$.
                        Let $P = T^2 - (\Tr A) T + \det A$ be the characteristic polynomial of $A$. Set $\Delta = (\Tr A)^2 - 4
                        \det A$ the discriminant.

                        If $\Delta <0$, then $\det A > 0$ and the two roots of $P$ are complex conjugate and their modulus
                        is $\sqrt{\det A}$, so $\rho = \sqrt{\det A}$ which is a Perron number of degree 2 if
                        $\det A$ is not a square in $\Z$, otherwise it is a positive integer.

                        If $\Delta = 0$, then $(\Tr A)^2 = 4 \det A$. Therefore $\Tr A$ is even and  $P = (T - \frac{\Tr
                        A}{2})^2$, so $\rho = \left| \frac{\Tr A}{2} \right|$ which is a positive integer.

                        If $\Delta > 0$, set $a := \Tr A$. If $a \geq 0$, then $\rho = \frac{a + \sqrt \Delta}{2}$ which is the
                        largest root of $P$ and so $\rho$ is a Perron number of degree 2. If $a < 0$, then
                        $\rho = \frac{-a + \sqrt \Delta}{2}$ which is a Perron number of degree 2 as
                        it is the largest root of $T^2 + a T + \det A$.
                      \end{proof}
                      By Theorem \ref{BigThmDynamicalDegreesEng}, any normal affine surface satisfies $\Lambda (X_0) \subset
                      \Lambda(\G_m^2)$. Thus, $\Lambda (\G_m^2)$ is maximal and one might
                      think that this is a characterisation of the algebraic torus but this is not the case. We now prove
                      Theorem \ref{BigThmDynamicalSpectrum} which states
                      \begin{equation}
                        \Lambda (\A^2) = \Lambda (\G_m^2).
                      \end{equation}

                      \begin{proof}[Proof of Theorem \ref{BigThmDynamicalSpectrum}]
                        By Theorem \ref{ThmLocalFormOfMapAndDynamicalDegreAreQuadraticInteger} and Proposition
                        \ref{PropDynamicalSpectrumTorus} we have $\Lambda (\A^2) \subset \Lambda (\G_m^2)$. We
                        show the equality using the following lemma.
                        \begin{lemme}
                          Every Perron number of degree $\leq 2$ is the spectral radius of a $2
                          \times 2$ matrix with nonnegative integer entries.
                        \end{lemme}
                        Using the lemma, we have that every $\lambda \in \Lambda(\G_m^2)$ is the dynamical degree of a monomial
                        transformation of $\A^2$, thus $\Lambda (\A^2) = \Lambda (\G_m^2)$.
                      \end{proof}

                      \begin{proof}[Proof of the lemma]
                        Let $\lambda$ be a Perron number of degree $\leq 2$.

                        If $\lambda$ is an integer then it is the spectral radius of $\begin{pmatrix} \lambda & 0 \\ 0 & 1
                        \end{pmatrix}$.

                        If $\lambda = \sqrt m$ with $m$ a positive integer which is not
                        a square, then $\lambda$ is the spectral radius of $\begin{pmatrix} 0 & 1 \\ m & 0 \end{pmatrix}$.

                        Finally, suppose $\lambda$ is the largest root of $T^2 - aT + b$ with $a > 0, b \neq 0$. If $b < 0$, then
                        $\lambda$ is the spectral radius of $\begin{pmatrix}
                          a & 1 \\
                        -b & 0 \end{pmatrix}$. If $b > 0$, then the discriminant must satisfy
                        \begin{equation}
                          \Delta = a^2 - 4 b > 0 \Rightarrow \left( \frac{a}{2} \right)^2 > b.
                        \end{equation}
                        If $a = 2k$ is even, then $\lambda$ is the spectral radius of
                        \begin{equation}
                          \begin{pmatrix}
                            k & 1 \\
                            k^2- b & k
                          \end{pmatrix}.
                          \label{<+label+>}
                        \end{equation}
                        If $a = 2k +1 $ is odd, then $(k+ 1/2)^2 > b \Rightarrow k(k+1) \geq b$ and $\lambda$ is the spectral
                        radius of
                        \begin{equation}
                          \begin{pmatrix}
                            k & 1 \\
                            k(k+1) - b & k+1
                          \end{pmatrix}.
                          \label{<+label+>}
                        \end{equation}
                      \end{proof}

                      \chapter{The automorphism case}\label{ChapterAutomorphisms}
                      Here we suppose that $X_0$ is an irreducible normal affine surface that admits a loxodromic
                      automorphism. By Theorem \ref{thm:dynamical_degree_quasi_albanese}, we have that either $X_0 =
                      \G_m^2$ or $\QAlb(X_0) = 0$. In this situation, we can actually deduce a lot more from the result of Chapter
                      \ref{ChapterDynamicsQuasiAlbTrivial}.
                      We change the notation for this section, we will denote $\theta^*$ and $\theta_*$ by $\theta^+$
                      and $\theta^-$ respectively. So that $(f^{\pm 1})^* \theta^\pm = \lambda_1 \theta^\pm$. By
                      Proposition \ref{PropThetaEffectif} and Theorem \ref{ThmExistenceEigenvaluationSurfaces}, we get
                      that
                      \begin{itemize}
                        \item $\theta^+, \theta^- \in \Winf \cap \L2$ and they are both effective.
                        \item $\theta^+ = Z_{v_-}$ and $\theta^- = Z_{v_+}$ where $v_+$ is the eigenvaluation of $f$
                          and $v_-$ the eigenvaluation of ${f}^{-1}$.
                      \end{itemize}

                      \begin{prop}\label{PropTypesEigenValuationsAutomorphisms}
                        Let $X_0$ be a rational affine surface such that $\k[X_0]^\times = \k^\times$ and let $f$ be a
                        loxodromic automorphism of $X_0$, then
                        \begin{enumerate}
                          \item The eigenvaluations $v_+$, $v_-$ of $f$ and ${f}^{-1}$ respectively are of the same type.
                          \item If $\lambda_1 \in \Z_{\geq 0}$, then $v_+$ and $v_-$ are infinitely singular.
                          \item If $\lambda_1 \in \R \setminus \Z_{\geq 0}$ then $v_+$ and $v_-$ are irrational.
                        \end{enumerate}
                      \end{prop}

                      \begin{proof}
                        If the eigenvaluation was divisorial, then we would get by Lemma
                        \ref{LemmaDivisorialConditionOnDegrees} that $\lambda_1 \leq \lambda_2$ and this is absurd because
                        $\lambda_1 > 1$, $f$ being loxodromic. The dichotomy of the type of eigenvaluation follows from Theorem
                        \ref{ThmLocalFormOfMapAndDynamicalDegreAreQuadraticInteger} and the fact that $\lambda_1 (f) =
                        \lambda_1({f}^{-1})$.
                      \end{proof}

                      \begin{cor}\label{CorSelfIntersectionNefClasses}
                        In that case, the nef eigenclasses $\theta^-$ and $\theta^+$ verify \[ (\theta^-)^2 = (\theta^+)^2 = 0 \]
                        and in any completion $X$ of $X_0$ one has $(\theta^\pm_X)^2 >0, \theta^\pm_X \geq 0$ and
                        $\supp \theta^\pm_X = \BD$.
                      \end{cor}

                      \begin{proof}
                        The equalities $(\theta^-)^2 = (\theta^+)^2 = 0$ come from Theorem \ref{ThmEigenclasses}
                        and more precisely from \eqref{EqIntersectionDeTheta}.
                        Since the eigenvaluations are not divisorial, $\theta^-$ and $\theta^+$ are not Cartier divisors
                        by Corollary \ref{CorEbeddingIntoWinf}
                        therefore for any completion $X$ of $X_0$, $(\theta^\pm_X)^2 > 0$. Indeed, if
                        $(\theta_X^\pm)^2 = 0$ then since $\theta^\pm$ is nef, we would get $\theta_X^\pm =
                        \theta^\pm$. Finally, the last statements follow from Theorem \ref{thm:nef-valuation}.
                      \end{proof}

                      Using the classification of dynamics of birational maps of surfaces from
                      \cite{dillerDynamicsBimeromorphicMaps2001} we get the following result.
                      \begin{cor}\label{cor:}
                        Let $X_0$ be an irreducible normal affine surface with a loxodromic automorphism $f$, then $X_0$
                      is rational and $X_0 \simeq \G_m^2 \Leftrightarrow K\left[X_0\right]^\times \neq K^\times$.
                      \end{cor}
                      \begin{proof}
                        If $\QAlb(X_0) \neq 0$ then we have by Theorem \ref{thm:dynamical_degree_quasi_albanese} that
                        $X_0 \simeq \G_m^2$ so $X_0$ is rational and admits a non-invertible regular function. Suppose
                        that $\QAlb(X_0) = 0$ then by Corollary \ref{CorSelfIntersectionNefClasses} we have that on any
                        completion $X$, $(\theta^+_X)^2 > 0$ thus we fall in Class 5 of Table 1 of
                        \cite{dillerDynamicsBimeromorphicMaps2001} and $X_0$ must be rational. Furthermore since
                        $\QAlb(X_0) = 0$ we have that $K[X_0]^\times = K^\times$ by Proposition
                        \ref{prop:carac-trivial-quasi-albanese}.
                      \end{proof}

                      \begin{rmq}\label{rmq:not-regularisable}
                        The fact that $(\theta_X^{\pm})^2 > 0$ implies that $f$ is not regularisable, i.e $f$ is never
                        an automorphism of a completion of $X$. This follows also from
                        \cite{dillerDynamicsBimeromorphicMaps2001}. We can also deduce it from Theorem
                        \ref{ThmLocalFormOfMapAndDynamicalDegreAreQuadraticInteger} because we see that $f$ must
                        contract a divisor at infinity in any completion $X$.
                      \end{rmq}
                      \begin{cor}\label{cor:very-attracting-valutions}
                        For any valuation $v\in \Vinf '$ of finite skewness (e.g for quasimonomial valuations) such that
                        $v \neq v_\mp$, there is a constant $c_\pm > 0$ such that for the strong topology 
                        \begin{equation}
                          \frac{1}{\lambda_1^n} (f^{\pm n}_* v) \xrightarrow[n \rightarrow + \infty]{} c \cdot v_\pm.
                          \label{<+label+>}
                        \end{equation}
                      \end{cor}
                      \begin{proof}
                        From Corollary \ref{CorSelfIntersectionNefClasses}, we have for any completion $X$ of $X_0$ that
                        $\Supp \theta^{\pm}_X = \BD$. Let $v \not v_-$, then there exists a completion $X$ such that
                        $c_X (v) \neq c_X (v_-)$ and we have that 
                        \begin{equation}
                          Z_v \cdot \theta^+ = Z_v \cdot Z_{v_-} = Z_{v,X} \cdot \theta^+_X > 0
                          \label{<+label+>}
                        \end{equation}
                        where the last inequality comes from Corollary
                        \ref{CorDivisorOfValuationWithLocalDivisorOfValuation}. Then putting $c := Z_v \cdot \theta^+
                        >0$ we get the result by Proposition \ref{PropCaracterisationConvergenceVersEigenvaluation}.  
                        The other case is symmetrical.
                      \end{proof}
                      We can thus fully understand the dynamics at infinity of a loxodromic automorphism.
                      \begin{prop}\label{PropBoundaryContracted}
                        For any completion $X$ of $X_0$, there exists an integer $N_0 > 0$ such that for every
                        $N \geq 0, f^{\pm N}$ contracts $\BD$ to $c_X (v_\pm)$.
                      \end{prop}
                      \begin{proof}
                        Let $E$ be a prime divisor at infinity in $X$, by Corollary \ref{cor:very-attracting-valutions}, we have that
                        \begin{equation}
                          \frac{1}{\lambda_1^n} f^{\pm n} \ord_E \rightarrow c_\pm v_\pm
                          \label{eq:<+label+>}
                        \end{equation}
                        for some $c_\pm > 0$.
                        Therefore for $n$ large enough, $c_X ( (f^{\pm n})_* \ord_E) = c_X (v_\pm)$.
                      \end{proof}

                      \section{Gizatullin's work on the boundary and applications}\label{SubSecGizatullinBoundary}
                      In \cite{gizatullinInvariantsIncompleteAlgebraic1971}, Gizatullin considers \emph{minimal
                      completions} of normal affine surfaces. If $X_0$ is a normal affine surface and $X$ is a
                      completion $X$ of $X_0$, then $X$ is \emph{minimal} if it is minimal with respect to the following
                      properties:
                      \begin{itemize}
                        \item There is a smooth open neighbourhood of $\partial_X X_0$. 
                        \item The boundary $\BD$ does not have three prime divisors that intersect at the same point.
                        \item If $\BD$ has a singular irreducible component then $\BD$ consists only of one
                          irreducible curve with at most one nodal singularity.
                      \end{itemize}
                      For such a completion $\iota: X_0 \hookrightarrow
                      X$, Gizatullin defines the curve $E(\iota)$ as the union of the irreducible components $E$ of
                      $\BD$ that are contracted by an automorphism of $X_0$ (the automorphism depends on $E$). His work
                      proves that the dual graph of $E(\iota)$ must be special, we have first to defines zigzags and
                      cycle of rational curves.
                      We call a \emph{zigzag} a chain of rational curves. That is a sequence $(E_1, \cdots,
                      E_r)$ of smooth rational curves such that $E_i \cdot E_{i+1} = 1, i = 1, \cdots, r-1$ and
                      for all $i,j$ such that $\left| i-j \right| \geq 2, E_i \cdot E_j = 0$. In particular the dual graph with
                      respect to the $E_i$'s is of the form
                      \begin{center}
                        \begin{tikzpicture}
                          \draw (0,0) node{$\bullet$};
                          \draw (0,0) node[below]{$E_1$};
                          \draw (1,0) node{$\bullet$};
                          \draw (1,0) node[below]{$E_2$};
                          \draw (2.5,0) node{$\bullet$};
                          \draw (2.5,0) node[below]{$E_{r-1}$};
                          \draw (3.5,0) node{$\bullet$};
                          \draw (3.5,0) node[below]{$E_r$};
                          \draw (1.75,0) node{$\cdots$};

                          \draw (0,0) -- (1.25,0);
                          \draw (2.25, 0) -- (3.5, 0);
                        \end{tikzpicture}
                      \end{center}
                      We will write $E_1 \vartriangleright E_2 \vartriangleright \cdots \vartriangleright E_r$ for the
                      zigzag defined by $(E_1, \cdots, E_r)$.
                      A \emph{cycle} of rational curves is either an irreducible rational curve with one nodal
                      singularity or
                      a sequence $(E_1, \cdots, E_r)$ of smooth rational curves with $r \geq 2$ such that $E_i \cdot E_{i+1} =1$ and $E_1 \cdot
                      E_r = 1$. The dual graph with respect to the $E_i$'s is of the form
                      \begin{center}
                        \begin{tikzpicture}
                          \draw (0:1) node{$\bullet$};
                          \draw (0:1) node[right]{$E_1$};
                          \draw (45:1) node{$\bullet$};
                          \draw (45:1) node[right, above]{$E_2$};
                          \draw (0:1) -- (45:1);
                          \draw (45:1) -- (90:1);
                          \draw (90:1) node{$\bullet$};
                          \draw (135:1) node{$\iddots$};
                          \draw (225:1) node{$\ddots$};
                          \draw (270:1) node{$\bullet$};
                          \draw (270:1) node[below]{$E_{r-1}$};
                          \draw (270:1) -- (315:1);
                          \draw (315:1) -- (0:1);
                          \draw (315:1) node{$\bullet$};
                          \draw (315:1) node[right]{$E_r$};
                        \end{tikzpicture}
                      \end{center}

                      \begin{prop}[\cite{gizatullinInvariantsIncompleteAlgebraic1971}]\label{prop:gizatullin}
                        Let $X_0$ be a normal affine surface and let $X$ be a minimal completion, then there are 3
                        possibilities for the dual graph of $E(\iota)$: 
                        \begin{enumerate}
                          \item $E(\iota) = \emptyset$. 
                          \item $E(\iota)$ is a disjoint union of smooth rational curves of self-intersection 0. 
                          \item $E(\iota)$ is a zigzag. 
                          \item $E(\iota)$ is a cycle of rational curves.
                        \end{enumerate}
                      \end{prop}

                      \begin{prop}\label{prop:}
                        If $X_0$ is a normal affine surface with a loxodromic automorphism and $X$ is a minimal
                        completion, then $E(\iota) = \BD$ is either a zigzag or a cycle of rational curves.
                      \end{prop}
                      \begin{proof}
                        The fact that $E (\iota) = \BD$ follows from Proposition \ref{PropBoundaryContracted}. In
                        particular we get that $E(\iota) \neq \emptyset$. If $\BD$ is a disjoint union of curves of
                        vanishing self-intersection, then any divisor supported on $\BD$ is numerically trivial but this
                        contradicts the fact that $K[X_0]^\times = 0$ and $\Pic^0 (X_0) = 0$. So the two remaining cases
                        are the zigzag and the cycle of rational curves.
                      \end{proof}

                      \begin{thm}\label{ThmAutomorphismCaseDynamicalCompactifications} 
                        \label{THMAUTOMORPHISMCASEDYNAMICALCOMPACTIFICATIONS}
                        Let $X_0= \spec \k[X_0]$ be an irreducible normal affine surface such that $\k[X_0]^\times =
                        \k^\times$ and $\Pic^0 (X_0) = 0$.
                        Suppose that $X_0$ admits an automorphism $f$ with $\lambda_1 (f) >1$.
                        If $X$ is a minimal completion of $X_0$, one has $E(\iota) = \BD = \supp \theta_{X}^{\pm}$.
                        Furthermore we have two mutually excluding cases
                        \begin{enumerate}
                          \item $\lambda_1(f)$ is an integer and in that case $E(\iota)$ is a zigzag.
                          \item $\lambda_1(f)$ is irrational and $E(\iota)$ is a cycle of rational curves.
                        \end{enumerate}
                        Furthermore, there exists a completion $Y$ with two distinct points $p_+,
                        p_- \in \partial_Y X_0$ and an integer $N >0$ such that
                        \begin{itemize}
                          \item $f^{\pm 1} (p_\pm) = p_\pm$.
                          \item $f^{\pm N}$ contracts $\partial_Y X_0$ to $p_{\pm}$.
                          \item $f^{\pm 1}$ has a normal form at $p_\pm$ given by Theorem
                            \ref{ThmLocalFormOfMapAndDynamicalDegreAreQuadraticInteger}, it is pseudomonomial or monomial in
                            the cycle case and of type \eqref{EqLocalNormalFormInfinitelySingular} or
                            \eqref{EqPseudoFormeInfSing} in the zigzag case.
                          \item In the zigzag case, $\partial_Y X_0$ is a tree of rational curves.
                          \item In the cycle case, $\partial_Y X_0$ is a cycle of rational curves.
                        \end{itemize}
                        Furthermore, 
                        \begin{itemize}
                          \item In the zigzag case, the set of completions above $Y$ that satisfy these
                            properties is cofinal in the set of all completions above $Y$.
                          \item In the zigzag case, we have that for any completion $Y$ of $X_0$, $\partial_Y X_0$ is a
                            tree of rational curves. In the cycle case, the set of cyclic completions of $X_0$ satisfying these properties is cofinal in
                            the set of cyclic completions of $X_0$.
                        \end{itemize}
                      \end{thm}
                      This shows Theorem \ref{BigThmAutomorphismIntroEng}. We will prove Theorem
                      \ref{ThmAutomorphismCaseDynamicalCompactifications} in \S \ref{SubSecProofCycleCase} and
                      \S \ref{SubSecProofZigzagCase}. For now we can say that the fact that $E(\iota) = \BD$ for any
                      minimal completion $X$ follows from Proposition \ref{PropBoundaryContracted}.

                      \section{Proof of Theorem \ref{ThmAutomorphismCaseDynamicalCompactifications}, the cycle
                      case}\label{SubSecProofCycleCase}

                      \begin{prop}[\cite{el-hutiCubicSurfacesMarkov1974,
                        cantatCommensuratingActionsBirational2019}]\label{PropCycleCaseIndeterminacyPoints}
                        Let $X$ be projective surface and $U$ an open subset of $X$ such that $X \setminus U$ is a cycle
                        of rational curves. Let $g$ be an automorphism of $U$, then the indeterminacy points of $g$ can only be
                        intersection points of two components of the cycle or at the nodal singularity of $X \setminus
                        U$ if it is a nodal curve.
                      \end{prop}

                      \begin{cor}\label{CorCycleCaseValuationIrrational}
                        If $X$ is a cyclic completion of $X_0$ and $f \in \Aut(X_0)$ is a loxodromic automorphism with
                        one of its eigenvaluations $v$, then $c_X (v)$ is a satellite point and if $(X_n, p_n)$ is the
                        sequence of centers of $v$, then every $X_n$ is a cyclic completion and $p_n$ is a satellite
                        point. The eigenvaluation of a loxodromic automorphism must be irrational and
                        therefore $\lambda_1$ is an algebraic integer of degree 2, in particular it is irrational.
                      \end{cor}

                      \begin{proof}
                        Proposition \ref{PropCycleCaseIndeterminacyPoints} shows that for any completion $X$ of $X_0$,
                        $p_+ = c_X (v_+)$ is a satellite point at infinity. Indeed, by Proposition
                        \ref{PropBoundaryContracted} we have that for $n$ large enough $f^n$ contracts $\BD$ to
                        $p_+$. So that $p_+$ must be an indeterminacy point of $f^n$ and therefore it is a satellite
                        point by Proposition \ref{PropCycleCaseIndeterminacyPoints}. Since blowing up a cyclic
                        completion at a satellite point also yields a cyclic completion, Proposition
                        \ref{PropSequenceOfInfinitelyNearPoints} shows that the eigenvaluations $v_\pm$ are irrational.
                      \end{proof}

                      \begin{proof}[Proof of Theorem \ref{ThmAutomorphismCaseDynamicalCompactifications}]
                        Corollary \ref{CorCycleCaseValuationIrrational} shows the first part of the theorem.
                        We get the normal form at $p_{\pm}$ by blowing up the center of $v_\pm$ enough times. Since
                        these are always intersection points of two prime divisors at infinity we can suppose that
                        $\partial_Y X_0$ is still a cycle.

                        If we blow up any point on $Y$ which is not $p_+$ or $p_-$, then then we get a completion
                        $Y'$ that satisfies the same property as $Y$. If we blow up say $p_+$, then the local normal
                        form at $p_+$ is $f(x,y) = (x^a y^b, x^c y^d)$ up to invertible functions. With $ad - bc = 1$.
                        If $\tilde E$ is the exceptional divisor, then the center of $v_+$ on $Y'$ is $\tilde E \cap F$.
                        Using local coordinates $(u,v)$ at $\tilde E \cap F$ associated to $\tilde E$ and $F$, the blow
                        up is given by $\pi (u,v) = (u,uv) = (z,w)$ where $z = 0$ is a local equation of $E$ and $w = 0$
                        is a local equation of $F$. Conjugating $f$ with $\pi$ we have that $f$ is defined at
                        $c_{Y'} (v_+)$ is a pseudomonomial local normal form. The cofinalness statement follows from the
                        fact that the blow up of a cyclic completion at a satellite point is also cycle.
                      \end{proof}

                      \section{Proof of Theorem \ref{ThmAutomorphismCaseDynamicalCompactifications}, the zigzag
                      case}\label{SubSecProofZigzagCase}
                      Let $X$ be a minimal completion of $X_0$ and suppose that $E(\iota)$ is a zigzag. Then,we have
                      proven that $\BD = E(\iota)$ so that $\BD$ is a zigzag.
                      Following \cite{danilovAutomorphismsAffineSurfaces1975, blancAutomorphisms1fiberedAffine2011}, a zigzag
                      $Z$ is \emph{standard} if
                      it is of the form
                      \begin{equation}
                        Z = F \vartriangleright E \vartriangleright Z'
                      \end{equation}
                      where $F^2 =0, E^2 \leq -1$ and
                      $Z'$ is a \emph{negative} zigzag meaning that every component of $Z'$ has self-intersection $\leq
                      -2$. This is also called a Hirzebruch-Jung string, see \cite{barthCompactComplexSurfaces2004}. If
                      $E^2 =: -m$,we say that the zigzag is $m$-standard. Any zigzag can be put to a standard
                      form via blow-up of points and contractions
                      of (-1)-curves (see \cite{danilovAutomorphismsAffineSurfaces}, \S 1.7). So that we can
                      assume that we have a completion $X$ of $X_0$ such that $\BD$ is a standard zigzag.

                      \subsection{Some technical lemmas about zigzags}\label{SubSecTechnicalLemmasZigZag}

                      \begin{lemme}\label{lemme:center-on-B}
                        Let $X_0$ be a normal affine surface and $X$ a completion such that $\BD$ is an almost standard
                        zigzag.
                        If $B$ is the unique component of nonnegative self-intersection of $\BD$ and $v$ is a valuation
                      such that $Z_v^2 \geq 0$ and $v$ is not divisorial, then $c_X (v) \in B$. 
                    \end{lemme}
                      \begin{proof}
                        By Theorem \ref{thm:nef-valuation}, we have that $Z_{v,X}^2 > 0$. Now if $c_X (v) \not \in B$,
                        then we have that $Z_{v,X} \cdot B = 0$. By the Hodge index theorem this implies either that
                        $Z_{v,X}^2 < 0$ or that $Z_{v,X} = B$ and $Z_{v,X}^2 = B^2 = 0$ both cases yield a
                        contradiction.
                      \end{proof}

                      Following \cite{blancAutomorphisms1fiberedAffine2011}, an \emph{almost standard} zigzag is a zigzag
                      $Z = B_1 \vartriangleright B_2 \vartriangleright \cdots \vartriangleright B_r$ such that
                      \begin{enumerate}
                        \item There exists a unique irreducible component $B_k$ such that $(B_k)^2 \geq 0$.
                        \item There exists at most one component $B_l$ such that $(B_l)^2  =-1$ and in that case we must
                          have $l = k \pm 1$.
                      \end{enumerate}
                      We need to state some technical results for the proof of Theorem
                      \ref{ThmAutomorphismCaseDynamicalCompactifications}. All the results in this Section rely
                      heavily on Proposition \ref{PropTechniqueContractionDeCourbes} and the Castelnuovo criterion.

                      \begin{lemme}[Proposition 3.1.3 of \cite{blancAutomorphisms1fiberedAffine2011}]
                        \label{LemmePointDindeterminationAlmostStandardZigzag}
                        Let $X_0$ be a quasiprojective surface and $X$ a completion of $X_0$ such that $X \setminus X_0$
                        is an almost standard zigzag that has no component of self intersection $-1$. Let $B_k$ be the
                        unique irreducible component of nonnegative self-intersection of $X \setminus X_0$. Let $g$ be an
                        automorphism of $X_0$, then
                        \begin{enumerate}
                          \item $g$ has at most one indeterminacy point $q$ on $X$.
                          \item $q$ has to be on $B_k$ (if it exists).
                          \item If $B_k$ is not on the boundary of the zigzag then $q$ must be the intersection point of
                            $B_k$ with $B_{k + 1}$ or $B_{k-1}$.
                        \end{enumerate}
                      \end{lemme}

                      \begin{proof}
                        Suppose $g$ has at least one indeterminacy point. Let $\pi : Y \rightarrow
                        X$ be the minimal resolution of indeterminacies of $g$ and let $\tilde g$ be the lift of $g$.

                        We first show that $g$ cannot have two indeterminacy points. Suppose $g$ has two indeterminacy
                        point $q_1, q_2$. Let $E_1, E_2$ be two $(-1)$-curves in $Y$ with $E_i$ above $q_i$. Then by
                        minimality of $Y$, $E_1$ and $E_2$ cannot be contracted. So their images in $X$ via $\tilde g$
                        is a curve of self-intersection $\geq -1$ but there is only one such curve in $\BD$ and we
                        cannot have $\tilde g(E_1) = \tilde g(E_2)$ so this is a contradiction.

                        Let $q$ be the unique indeterminacy point of $g$. Since every other curve in $X$ has
                        self-intersection $\leq -2$, the first $(-1)$-curve in $Y$ contracted by $\tilde g$ must be the
                        strict transform of $B_k$ so that $q \in B_k$.

                        Suppose that $B_k$ is not on the boundary of the zigzag and that the
                        indeterminacy point $q$ of $g$ is not an intersection point. Then, the map $\pi$
                        factorizes through the blow-up of $q$ and after contracting the strict transform of
                        $B_k$, we get at infinity three prime divisors that intersect at the same point. But
                        this is a contradiction because $\tilde g^{-1}$ consists only of blow ups of point at
                        infinity and $X \setminus X_0$ does not have three divisors that intersects at the
                        same point and that remains true after any blow-ups.
                      \end{proof}

                      \begin{cor}\label{CorollaryCenterOnAlmostStandardZigzag}
                        Let $X$ be a completion of $X_0$ such that $X \setminus X_0$ is an almost standard
                        zigzag $Z$ and let $f$ be an automorphism of $X_0$. Suppose that $f$ has an
                        indeterminacy point that is a free point on $B_k$, then one of the
                        two sides of $Z$ can be contracted so that $B_k$ becomes a boundary component of the
                        zigzag.
                      \end{cor}

                      \begin{proof}
                        Suppose that $B_k$ is not a boundary component of the zigzag and that $f$ has an
                        indeterminacy point that is a free point on $B_k$. Then, by Lemma
                        \ref{LemmePointDindeterminationAlmostStandardZigzag}, $B_{k-1}$ or $B_{k+1}$ has to
                        be a $(-1)$-curve, suppose it is $B_{k+1}$. We contract it and we obtain an almost standard zigzag
                        and $f$
                        still has an indeterminacy point that is a free point on $B_k$. If $B_k$ is on the
                        boundary we are done, otherwise we apply the same procedure. This algorithm must terminate and
                        we end up with a completion where the boundary is zigzag and $B_k$ is on the boundary. 
                      \end{proof}

                      \begin{lemme}\label{LemmePointDIndeterminationPasSurDeuxMoinsX_0n}
                        Let $X_0$ be a quasiprojective variety and $X$ a completion of $X_0$ such that $X \setminus X_0$ is a
                        zigzag of type $(-m_1, \cdots, -m_k , -1, -1, -m_{k+1}, \cdots,  - m_r)$ such that for all $i, m_i
                        \geq 2$. Let $f$ be an automorphism of $X_0$. Then the intersection point of the two $(-1)$-curves
                        cannot be an indeterminacy point of $f$.

                        If the zigzag is of type $(-1, -2, \ldots, \underbrace{-2}_{F}, \underbrace{-1}_E, -m_{k+1},
                        \cdots,  - m_r)$ with $m_i \geq 2$,
                        then $F \cap E$ cannot be an indeterminacy point of $f$.
                      \end{lemme}

                      \begin{proof}
                        Let $\pi: Z \rightarrow X$ be the minimal resolution of indeterminacy of $f: X \dashrightarrow
                        X$ and let $\tilde f: Z \rightarrow X$ be the lift of $f$. The first curve
                        contracted by $\tilde f$ must
                        be the strict transform of one of the prime divisors at infinity of $X$. But if the intersection
                        of the $(-1)$-curves is an indeterminacy point of $f$, then all the strict transforms of the prime
                        divisors at infinity of $X$ have self-intersections $\leq -2$ and this is a contradiction.

                        If $X \setminus X_0$ is a zigzag $Z$ of type $(-1, -2, \cdots, -2, -1 , -m_{k+1}, \cdots, -
                        m_r)$, suppose that $F \cap E$ is an indeterminacy point of $f$, then the first
                        curve contracted by $\tilde f$ must be the strict transform of the $(-1)$-curve on
                        the left of the zigzag. So we can start by contracting it and we get a zigzag $Z'$
                        of type $(-1, -2 , \cdots, \underbrace{-2}_{F}, \underbrace{-1}_{E}, -m_{k+1}, \cdots, - m_r)$
                        and of size $\# Z -1$. We
                        can repeat this process until we get a zigzag of the form $(\underbrace{-1}_F,
                        \underbrace{-1}_E, -m_{k+1}, \cdots, - m_r)$ and we have that $F \cap E$ cannot be
                        an indeterminacy point of $f$ by the previous case, this is a contradiction.
                      \end{proof}

                      \subsection{Elementary links between almost standard zigzags}
                      All the results of \S \ref{SubSecTechnicalLemmasZigZag} will
                      be applied to the following situation. If $X$ is a completion of $X_0$ such that $\BD$ is an
                      almost standard zigzag and $f$ is a
                      loxodromic automorphism of $X_0$, then first by Lemma \ref{lemme:center-on-B}, $c_X (v_+)$ belongs
                      to the unique component of $\BD$ of nonnegative self-intersection and furthermore some positive
                      iterate of $f$ contracts a component of $\BD$ to $c_{X} (v_+)$. Thus, $c_X (v_+)$ is
                      an indeterminacy point of some positive iterate of ${f}^{-1}$ on $X$.

                      \begin{prop}\label{PropEclatementDuCentreSurZigzag}
                        Let $X$ be a completion of $X_0$ such that $\BD$ is an almost standard zigzag, then
                        one can find a completion $Y$ of $X_0$ with a birational map $\phi: X \dashrightarrow Y$ that
                        is an isomorphism above $X_0$ such that
                        \begin{enumerate}
                          \item $Y \setminus X_0$ is also an almost standard zigzag.
                          \item Let $\tilde X$ be the blow up of $X$ at $c_X (v_+)$, then the lift $\phi :
                            \tilde X \dashrightarrow Y$ is defined at $c_{\tilde X} (v_+)$ and is a local
                            isomorphism there.
                        \end{enumerate}
                      \end{prop}

                      \begin{proof} Let $B$ the unique irreducible component of $\BD$ of nonnegative self
                        intersection.
                        \paragraph{\bf Case: $B$ is on the boundary of the zigzag} $\BD$ is a zigzag of the form $B \triangle E
                        \triangle Z$ where $B^2 \geq 0, E^2 \leq -1$ and $Z$ is a negative zigzag.
                        \begin{itemize}
                          \item \textbf{$c_X (v_+)$ is a free point on $B$.} If $E^2 = -1$, we blow up
                            $c_X (v_+)$ and then contract the strict transform of $E$. Let $Y$ be the new projective
                            surface obtained, it satisfies the proposition.

                            Suppose $E^2 < -1$, If $B^2 >0$ we blow up $B \cap E$ to obtain a new zigzag $B \triangle E '
                            \triangle Z'$ which is still standard. We keep blowing up the strict transform of $B$
                            with the second component of the zigzag until $B^2 = 0$. After all these blowups, let $X '$
                            be the newly obtained projective surface, we have that $X ' \setminus X_0$ is an almost
                            standard zigzag of the form $B \triangle E \triangle Z$ where $B^2 = 0, E^2 = -1$ and $Z$ is a
                            negative zigzag. We blow up
                            $c_{X '} (v_+)$ and let $\tilde E$ be the exceptional divisor, by Lemma
                            \ref{LemmePointDIndeterminationPasSurDeuxMoinsX_0n}, the center of $v_+$ cannot be the
                            intersection point of $\tilde E$ and the strict transform of $B$, therefore it is a free point
                            of $\tilde E$ and we can contract the strict transform of $B$. We call $Y$ the new obtained
                            surface it satisfies the proposition.
                          \item \textbf{$c_X(v_+)$ is the satellite point $B \cap E$.} We blow up $B \cap E$ and call
                            $\tilde E$ the exceptional divisor and $\tilde X$ the newly obtained projective surface. If
                            $B^2 >0$ in $X$, then we still have an almost standard zigzag
                            and we set $Y = \tilde X$. If $B^2=0$ in $X$, then by Lemma
                            \ref{LemmePointDIndeterminationPasSurDeuxMoinsX_0n} $c_{\tilde X} (v_+)$ is a free point of $\tilde E$ and we can
                            contract the strict transform of $B$, we call $Y$ the newly obtained surface.
                        \end{itemize}

                        \paragraph{\bf Case: $B$ is not on the boundary of the zigzag}
                        \begin{itemize}
                          \item \textbf{$c_X (v_+)$ is a free point of $B$.} By Corollary
                            \ref{CorollaryCenterOnAlmostStandardZigzag}, one of the two sides of $\BD$ is
                            contractible, so we contract it and call $X_1$ the newly obtained surface, we can now apply the
                            proof of the boundary case to find $Y$.
                          \item \textbf{$c_X(v_+)$ is the satellite point $B \cap E$.} We can suppose
                            up to contraction that if $X \setminus X_0$ contains a $(-1)$-component, it must be $E$. We start
                            by blowing up $c_X (v_+)$ and let $\tilde E$ be the exceptional divisor.
                            \begin{itemize}
                              \item If $B^2 >0$ in $X$, then we still have an almost standard zigzag and we call $Y$ the newly
                                obtained surface.
                              \item If $B^2=0$ in $X$, then by Lemma \ref{LemmePointDIndeterminationPasSurDeuxMoinsX_0n} the center of
                                $v_+$ cannot be the intersection of $\tilde E$ and the strict transform of $B$ where $\tilde
                                E$ is the exceptional divisor. So we can contract the strict transform of $B$. We get an
                                almost standard zigzag and we call $Y$ the newly obtained surface.
                            \end{itemize}
                        \end{itemize}
                      \end{proof}

                      \begin{cor}\label{cor:sequence-of-centers-zigzag}
                        Let $X$ be a completion of $X_0$ such that $X \setminus X_0$ is an almost standard zigzag and
                        let $f \in \Aut(X_0)$ be a loxodromic automorphism with eigenvaluations $v_\pm$. Let $(X_m,
                        p_m)$ be the sequence of centers and blow-ups associated to $X$ and $v_+$ then there exists a
                        completion $Y_m$ of $X_0$ such that $\partial_{Y_m} X_0$ is an almost standard zigzag and the
                        birational map $Y_m \dashrightarrow X_m$ is an isomorphism in a neighbourhood of $c_{X_m}
                        (v_+)$.
                      \end{cor}

                      \begin{cor}\label{cor:zigzag-inf-sing-eigenvaluation}
                        There exists a completion $Y$ of $X_0$ such that $\partial_Y (X_0)$ is an almost standard
                        zigzag,  $c_Y(v_+) \neq c_Y (v_-)$ and $c_Y (v_+)$ is a free point on a prime divisor at
                        infinity. In particular, $\lambda_1(f)$ is an integer and $v_+$ is infinitely singular.
                      \end{cor}
                      \begin{proof}
                        The assertions on the dynamical degree and the type of the eigenvaluation follow from the
                        following. If $c_Y (v_+) \in F$ is a free point at infinity and $c_Y (v_+) \neq c_Y (v_-)$, then
                        by Proposition \ref{PropBoundaryContracted}, there exists $n \geq 1$ such that $f^{\pm n}$
                        contracts $\partial_Y X_0$ to $p_\pm := c_Y (v_\pm)$. In particular, $f^n (p_+) = p_+$ and
                        $p_-$ is the only indeterminacy point of $f^n$. Let $z = 0$ be a local equation of $F$ at
                        $p_+$, we have that $(f^n)^* z = z^a u$ with $a \geq 2$ and $u$ a germ of invertible regular
                        function. We get that $\lambda_1(f^n)$ is an integer and by Theorem
                        \ref{ThmLocalFormOfMapAndDynamicalDegreAreQuadraticInteger} we have that $v_+$ is infinitely
                        singular and that $\lambda_1 (f)$ is an integer.

                        Start with a completion $X$ such that $\BD$ is an almost standard zigzag. By applying
                        Corollary \ref{cor:sequence-of-centers-zigzag} we can suppose that
                        $c_X (v_+) \neq c_X (v_-)$. Now, we have by Proposition
                        \ref{PropTypesEigenValuationsAutomorphisms} that $v_+$ is either infinitely singular or
                        irrational. If it is infinitely singular, then by applying again  Corollary
                        \ref{cor:sequence-of-centers-zigzag} we get a completion $X$ such that $\BD$ is an almost
                        standard zigzag and $c_X(v_+)$ is a free point on the zigzag. Otherwise, $v_+$ is irrational and
                        we can assume that $X$ is a completion such that $\BD$ is almost standard and for any completion
                        above $X$ the center of $v_+$ is a satellite point. We have by Lemma \ref{lemme:center-on-B}
                        that $c_X (v_+)$ must belong
                        to $B$, the unique component of nonnegative self-intersection of the almost standard zigzag. So
                        we write $c_X (v_+) = E \cap B$ with $B^2 \geq 0$ and $E^2 \leq -1$. First of all,
                        we can assume that if $\BD$ has a $(-1)$-curve then it must be $E$. Indeed, if the other
                        neighbouring curve of $B$ is a $(-1)$-curve, we contract it and we repeat this process until
                        there are no more $(-1)$-curves on the zigzag except maybe $E$. This does not change the fact that $\BD$ is an
                        almost standard zigzag. We show that there is a contradiction.

                        If $B^2 = 0$, let $X_1$ be the blow up of $c_X (v_+) = E \cap B$ and let $E_1$ be the
                        exceptional divisor. We have $c_{X_1} (v_+) \in E_1$ and by Lemma
                        \ref{LemmePointDIndeterminationPasSurDeuxMoinsX_0n}, the center of
                        $v_+$ cannot be $E_1 \cap B$. Since by assumption it must be a satellite point, we must have
                        $c_{X_1} (v_+) = E_1 \cap E$. Let $X_2$ be the blow up $E \cap E_1$ and let $E_2$ be the exceptional
                        divisor. Again, by Lemma \ref{LemmePointDIndeterminationPasSurDeuxMoinsX_0n}, we cannot have
                        $c_{X_2} (v_+) = E_1 \cap E_2$, so that $c_{X_2}(v_+) = E_2 \cap E$. Applying this process we get a sequence of blow-ups $X_1
                        \leftarrow X_2 \leftarrow \cdots \leftarrow X_m \leftarrow \cdots$ such that for every
                        $m \geq 1$ we have that $c_{X_m}(v_+) = E \cap E_m$. However this contradicts Proposition
                        \ref{prop:irrational-centers-not-same-divisor}.

                        If $B^2 > 0$, then by our assumptions $\BD$ is of the form
                        \begin{equation}
                          Z \vartriangleright E \vartriangleright B \vartriangleright Z'
                          \label{eq:<+label+>}
                        \end{equation}
                        with $E^2 \leq -1, B^2 >0$ and $Z, Z'$ are negative zigzags and $c_X (v_+) = B \cap E$. If $Y$
                        is the blow up of $B \cap E$ and $\tilde E$ is the exceptional divisor then $\partial_Y X_0$ is
                        of the form
                        \begin{equation}
                          \tilde Z \vartriangleright \tilde E \vartriangleright B' \vartriangleright Z'
                          \label{eq:<+label+>}
                        \end{equation}
                        and $c_Y (v_+) = B \cap \tilde E$. This shows that we can repeat this procedure of blowing the
                        center of $v_+$ until we have $B^2 = 0$ and then we fall back to the previous case.
                      \end{proof}

                      \begin{proof}[Proof of Theorem \ref{ThmSummaryDynamicalCompactificationAutomorphism} in the zigzag
                        case]
                        Suppose that $X_0$ is a normal affine surface with a loxodromic automorphism $f$ and such that
                        there exists a minimal completion $X$ of $X_0$ such that $\BD$ is a zigzag. Then, by a
                        composition of blow-ups and blow-downs we can assume that $\BD$ is a standard zigag. Corollary
                        \ref{cor:zigzag-inf-sing-eigenvaluation} implies that $\lambda_1(f)$ is an integer and $v_\pm$ are infinitely
                        singular valuations. Notice that if you start with a completion $X$ such that $\BD$ is a tree of
                        rational curves, then any completion above $X$ satisfies the same property. Since $v_+ \neq v_-$
                        there exists a finite sequence of blow-ups above $X$ such that the center of $v_+$ and $v_-$ are distinct. So
                        we replace $X$ by this completion. By Theorem
                        \ref{ThmLocalFormOfMapAndDynamicalDegreAreQuadraticInteger}, there exists a completion $Y$ above
                        $X$ such that $f^\pm$ is defined at $p_\pm := c_Y (v_\pm)$ and the local normal form of
                        $f^\pm$ at $p_\pm$ is given by \eqref{EqPseudoFormeInfSing}. The cofinalness statement also
                        follows from Theorem \ref{ThmLocalFormOfMapAndDynamicalDegreAreQuadraticInteger}.
                      \end{proof}

                      \section{A summary and applications}
                      We sum up the content of Theorem \ref{ThmSummaryDynamicalCompactificationAutomorphism} in Figure
                      \ref{fig:compactification_zigzag_case} and \ref{fig:compactification_irrational_case}
                      \begin{figure}[h]
                        \centering
                        \includegraphics[scale=0.6]{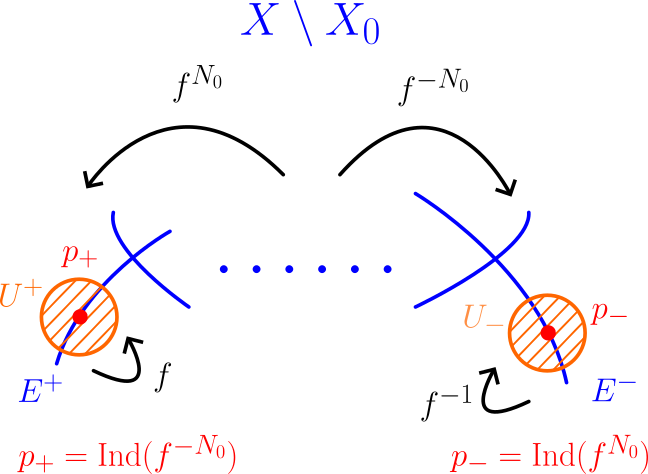}
                        \caption{Dynamics at infinity of $f$ when $\lambda_1 (f) \in \Z_{\geq 0}$}
                        \label{fig:compactification_zigzag_case}
                      \end{figure}

                      \begin{figure}[h]
                        \centering
                        \includegraphics[scale=0.8]{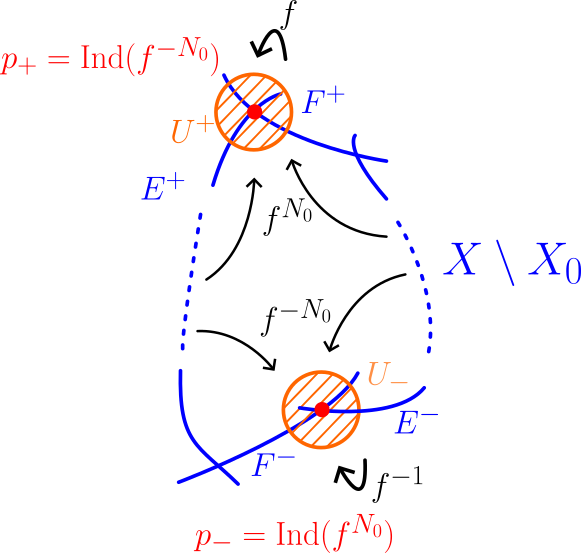}
                        \caption{Dynamics at infinity of $f$ when $\lambda_1 (f) \in \R \setminus \Q$}
                        \label{fig:compactification_irrational_case}
                      \end{figure}
                      \begin{thm}\label{ThmSummaryDynamicalCompactificationAutomorphism}
                        Let $X_0$ be a normal affine surface defined over a field $\k$ such that $\k[X_0]^\times =
                        \k^\times$ and $\Pic^0 (X_0) = 0$. Let $f$ be a loxodromic
                        automorphism of $X_0$.
                        Then, there exists two unique (up to normalization) distinct valuations centered
                        at $v_+, v_- $ such that $f_*^{\pm 1} (v_\pm) = \lambda_1 v_\pm$.
                        Let $\theta^- = Z_{v_+}$ and $\theta^+ = Z_{v_-}$. We have that $\theta^+, \theta^-$
                        are nef, effective and satisfy the following relations
                        \begin{align}
                          f^* \theta^+ &= \lambda_1 \theta^+, \quad f^* \theta^- = \frac{1}{\lambda_1} \theta^- \\
                          f_* \theta^+ &= \frac{1}{\lambda_1} \theta^+, \quad f_* \theta^- = \lambda_1
                          \theta^-.
                          \label{<+label+>}
                        \end{align}
                        Furthermore we have the following intersection relations: $(\theta^+)^2 = \left(
                        \theta^- \right)^2 = 0$ and $\theta^+ \cdot \theta^- = 1$.

                        We can find a completion $X$ of $X_0$ such that if $p_+ := c_X (v_+), p_- :=
                        c_X (v_-)$, then
                        \begin{enumerate}
                          \item $p_+ \neq p_-$.
                          \item some positive iterate of $f^{\pm 1}$ contracts $\BD$ to $p_{\pm}$.
                          \item $f^{\pm 1}$ is defined at $p_{\pm}, f^{\pm 1} = p_\pm$ and $p_\mp$ is the unique
                            indeterminacy point of $f^\pm$.
                          \item There exists an open neighbourhood $U^\pm$ of $p_\pm$ in $X$ and
                            local coordinates at $p_\pm$ such that $f^\pm_{|U^\pm}$ has a local normal form
                            of (pseudo)monomial type \eqref{EqLocalNormalFormMonomial} or (\eqref{EqPseudoFormeMonomiale})
                            if $\lambda_1(f) \not \in \Z_{\geq 0}$ or of type
                            \eqref{EqLocalNormalFormInfinitelySingular} or \eqref{EqPseudoFormeInfSing} if $\lambda_1 (f)
                            \in \Z_{\geq 0}$.
                          \item If $\lambda_1 (f) \in \Z_{\geq 0}$, then $\BD$ is always a tree of rational curves and the
                            set of such completions is cofinal in the set of completions of $X_0$. 
                          \item If $\lambda_1 (f) \not \in \Z_{\geq 0}$ then we can further assume that $\BD$ is a cycle
                            of rational curves and then such completions are cofinal in the set of cyclic completions of
                            $X_0$.
                        \end{enumerate}
                      \end{thm}

                      \begin{proof}
                        Any completion provided by Theorem
                        \ref{ThmAutomorphismCaseDynamicalCompactifications} satisfies Items (1)-(6).
                      \end{proof}

                      \begin{prop}\label{PropNoInvariantCurves}
                        Let $X_0$ be a normal affine surface defined over $\k$. If $f$ is a loxodromic
                        automorphism of $X_0$, then, there are no $f$-invariant algebraic curves in $X_0$.
                      \end{prop}

                      \begin{proof}
                        If $\QAlb(X_0) \neq 0$, then by Theorem \ref{thm:dynamical_degree_quasi_albanese}, $X_0 \simeq \G_m^2$ and this is
                        known.

                        If $\QAlb (X_0) = 0$, let $X$ be a completion of $X_0$ given by
                        Theorem \ref{ThmSummaryDynamicalCompactificationAutomorphism}. We can replace $f$ by one its
                        iterates that contracts $\BD$ to $p_+$ and $p_-$.
                        Suppose that $C \subset X_0$ is an algebraic curve invariant by $f$. Let $\overline C$ be the
                        closure of $C$ in $X$. We must have $\left\{ p_+, p_- \right\} \cap (\overline C \cap \BD) \neq
                        \emptyset$. Indeed, $\overline C \cap \BD$ is not empty so let $p$ be a point in it. If $p \not
                        \in \left\{ p_+, p_- \right\}$, then $f$ is defined at $p$ and $f(p) =p_+$.
                        Since $\overline C$ is $f$-invariant, we get $p_+ \in \overline C$. This means that $C$
                        defines a germ of an analytic curve at $p_+$ that is invariant by $f$ but this is not possible
                        by Theorem \ref{ThmLocalFormOfMapAndDynamicalDegreAreQuadraticInteger}.
                      \end{proof}

                      \begin{rmq}\label{rmq:}
                        This is a direct generalisation of a result of Bedford and Smilie for polynomial automorphism of
                        the affine plane, see \cite{bedfordPolynomialDiffeomorphisms$C^2$1991a,
                        bedfordFatouBieberbachDomainsArising1991}.
                      \end{rmq}

                      \begin{cor}\label{CorPeriodicPointsAreAlgebraic}
                        If $X_0$ is a normal affine surface defined over field $K$ and $f$ is a
                        loxodromic automorphism of $X_0$, then all periodic points of $f$ are contained in
                        $X_0(\overline K)$.
                      \end{cor}

                      \begin{proof}
                        Let $L/K$ be a transcendental extension, we can assume that $L$ contains an algebraic closure
                        $\overline K$ of $K$. Assume that there exists $p \in X_0 (L) \setminus
                        X_0(\overline K)$ such that $f^N (p)
                        = p$. Let $G := \Gal (L / \overline K )$, then for all $q \in G \cdot p$, we have
                        $f^N (q) = q$. Since $p \not \in X_0 (\overline K)$, the orbit $G \cdot p$ is
                        infinite and its Zariski closure $\overline{G \cdot p} \subset X_0 \times \spec L$
                        has dimension $>0$. If $\dim \overline{G \cdot p} = 2$, then $f^N = \id$ and this is
                        impossible because $f$ is loxodromic. If $\dim \overline{G \cdot p} = 1$, then $C =
                        \overline{G \cdot p}$ is an $f^N$-invariant curve of $X_0 \times \spec \C$. This is
                        impossible by Proposition \ref{PropNoInvariantCurves}.
                      \end{proof}

                      Recall that if $\k$ is a complete field and $X$ is a variety, we define the \emph{Euclidian
                      topology} on $X(\k)$ which is the one induced by the absolute value on $\k$.

                      \begin{cor}\label{CorBehabviourOfOrbits}
                        Let $X_0$ be a normal affine surface defined over a complete and algebraically closed field $\k$ such that $\QAlb (X_0) =
                        0$. Let $f$ be a loxodromic automorphism of $X_0$ and let $X$ be a completion of $X_0$ from Theorem
                        \ref{ThmSummaryDynamicalCompactificationAutomorphism}. If $p \in X_0 (\k)$, we have
                        two possibilities.
                        \begin{enumerate}
                          \item The forward $f$-orbit of $p$ is bounded.
                          \item $\left( f^n (p) \right)_{n \geq 0}$ converges towards $p_+$ for the Euclidian topology
                            of $X(\k)$.
                        \end{enumerate}
                      \end{cor}

                      \begin{proof}
                        Suppose that $(f^n (p))_n$ is not bounded. Since $X(\k)$ is compact, $(f^n (p))$ has an
                        accumulation point $q \in \BD$. Let $U_+$ be the open neighbourhood of $p_+$ given
                        by Theorem \ref{ThmSummaryDynamicalCompactificationAutomorphism}. We must have $q
                        \in \left\{ p_+ , p_- \right\}$. Otherwise, since $f(q) = p_+$, if $f^{N_0} (p)$ is sufficiently
                        close to $q$, then for all $N \geq N_0 +1, f^N (p) \in U_+$ and $q$ cannot be an accumulation
                        point. Suppose that $q = p-$. Let $(x,y)$ be the local coordinates at $p_-$ over
                        $U^-$ given by Theorem \ref{ThmSummaryDynamicalCompactificationAutomorphism}.
                        Consider the norm $\max(\left| x \right|, \left| y \right|)$ over $U^-$. Looking at the normal form of
                        $f$, for any $\epsilon >0$ small
                        enough, the ball $B (p_-, \epsilon)$ of center $p_-$ and radius $\epsilon$, with
                        respect to this norm, is ${f}^{-1}$-invariant and we have ${f}^{-1} B(p_- , \epsilon )
                        \Subset B(p_-, \epsilon)$. Therefore if $f^{N_0} (p) \in B(p_-, \epsilon)$,
                        we have $p \in B(p_-, \epsilon)$. Letting $\epsilon \rightarrow 0$ we get $p = p_-$
                        and this is a contradiction. Therefore, the only accumulation point of $(f^N (p))_N$
                        is $p_+$ and it is the limit of this sequence.
                      \end{proof}

                      The dynamics of $f$ established in Theorem \ref{ThmSummaryDynamicalCompactificationAutomorphism}
                      is a direct generalisation of the dynamics of Hénon maps, see \cite[Lemma
                      5.1]{hubbardHenonMappingsComplex1994} and \cite[Lemma
                      2.2]{bedfordPolynomialDiffeomorphisms$C^2$1991a}. This allows one to construct the \emph{Green
                      functions} of $f$, namely let $\k$ be a complete algebraically closed field and fix an embedding
                      $X_0 \hookrightarrow \A^N$ for some large $N$. Then, the Green functions of $f$ are defined by 
                      \begin{equation}
                        \forall p \in X_0 (\k), \quad G^\pm_f (p) := \lim_{n \rightarrow +\infty}
                        \frac{1}{\lambda_1(f)^n} \log^+ \parallel f^{\pm n} (p) \parallel.
                        \label{eq:convergence-Green-functions}
                      \end{equation}
                      We have the following properties: 
                      \begin{enumerate}
                        \item $G^{\pm}_f$ is well defined, continuous, $\geq 0$ and the convergence in
                          \eqref{eq:convergence-Green-functions} is locally uniform.
                        \item $G^{\pm}_f \circ f^{\pm 1} = \lambda_1 (f) G^{\pm}_f$. 
                        \item $G^+_f (p) = 0$ if and only if $\left\{ f^n (p) : n \geq 0 \right\}$ is relatively compact
                          in $X_0(\k)$.
                        \item If $\k = \C$, then $G^{\pm}_f$ is plurisubharmonic and harmonic over $\left\{
                          G^{\pm}_f > 0 \right\}$.
                      \end{enumerate}
                      This construction is done in \cite{abboudRigidityPeriodicPoints2024} where the \emph{canonical
                      height} of $f$ is constructed using the results in this section. This construction is
                      actually compatible with the theory of \emph{adelic divisors over quasiprojective varieties}
                      developed in the recent work of Yuan and Zhang in \cite{yuanAdelicLineBundles2026}.

                      \section{Affine surfaces with a cycle at infinity}
                      Let $X_0$ be a normal affine surface and suppose that there exists a loxodromic automorphism $f$ of
                      $X_0$ such that $\lambda_1(f) \not \in \Z$. Then, by Theorem
                      \ref{ThmAutomorphismCaseDynamicalCompactifications}, for any minimal completion $X$ of $X_0$,
                      $\BD$ is a cycle of rational curves and to study the dynamics of a loxodromic automorphism it
                      suffices to consider completions where the boundary remains a cycle of rational curves.

                      \subsection{The circle at infinity}
                      Let $X$ be
                      such a completion and let $E_1, \cdots, E_r$ be the irreducible components of $\BD$. Define $\cC_X
                      \subset \hat \Vinf$ by
                      \begin{equation}
                        \cC_X = \bigcup_{i = 1}^r [\ord_{E_i}, \ord_{E_{i+1}}].
                        \label{<+label+>}
                      \end{equation}
                      $\cC_X$ consists only of quasimonomial valuations hence it is a subset of $\hat \Vinf '$, the
                      subset of valuations of finite skewness. It can therefore be equipped with the strong topology.
                      \begin{prop}
                        For every completion $X$ such that $\BD$ is a cycle of rational curves, one has
                        \begin{enumerate}
                          \item $\cC_X =: \cC$ does not depend on $X$.
                          \item $\cC$ is homeomorphic to $\CS^1$.
                          \item $\cC$ is characterised as follows: for every continuous embedding $c: \CS^1 \hookrightarrow
                            \hat \Vinf$, $c(\CS^1) = \cC$.
                        \end{enumerate}
                      \end{prop}
                      \begin{proof}
                        For (1) we show that if $\pi : Y \rightarrow X$ is the blow up of a satellite point, then
                        $\cC_Y = \cC_X$. Let $p = E \cap F$ be the center of the blow up and let $\tilde E$ be the
                        exceptional divisor. Then, $[\ord_E, \ord_F] = [\ord_E, \ord_{\tilde E}] \cup [\ord_{\tilde E},
                        \ord_F]$ and we see that $\cC_X = \cC_Y$.

                        For (2), recall that the segment $[\ord_E, \ord_F)$ is naturally a subsegment of $\cV_X
                        (p; E)$ parametrised with the skewness function $\alpha_E$. By Proposition
                        \ref{PropMonomialValuationIsSegment}, we have that $\alpha_E = {\alpha_F}^{-1}$ over $(\ord_E,
                        \ord_F)$. Thus, $\C$ is homeomorphic to a finite union of segments $I_i = [a_i, b_i]$ such
                        that $b_i = a_{i+1}$, hence it is homeomorphic to the circle $\CS^1$.

                        For (3), let $c : \CS^1 \hookrightarrow \hat \Vinf$ be a continuous embedding. Suppose that
                        $c(\CS^1) \neq \cC$ this means that there exists a completion $X$ and $t_0 \in \CS^1$ such that
                      $c(t_0)$ is centered at a free point $p \in E$ at infinity. Let $I_0 = ]a,b[$ be the largest subsegment of
                        $\CS^1$ containing $t_0$ such that for all $s \in I_0, c(s) \in \cV_X (p; E)$. Because $\cV_X
                        (p;E)$ is open we must have $a = b$ and $c(a) = c(b) = \ord_E$. Therefore $c$ is a continuous embedding
                        of $\CS^1$ into $\cV_X (p;E)$ but this is not possible since $\cV_X(p;E)$ is a tree.
                      \end{proof}

                      \subsection{Farey parametrisation}
                      We introduce here the Farey parametrisation which is an algorithm that allows one to list all the
                      rational numbers in $[0, 1]$. Our presentation is directly inspired by
                      \cite{cantatCommensuratingActionsBirational2019}. 
                      Let $X$ be a completion of $X_0$ and let $E$ be a prime divisor at infinity and let $p \in E$. A
                      \emph{Farey} parametrisation of $\cD_{X, p} \cup \left\{ E \right\}$ is given by the following
                      procedure. Pick positive integers $a_0,b_0$ such that $gcd(a_0,b_0) = 1$ and set $\Far_{(E,a_0,b_0)}(E)
                      = (a_0,b_0)$,
                      then do the following. Suppose that $\pi : Y \rightarrow X$ is a completion exceptional above $p$ such
                      that $\Far_{(E,a_0,b_0)}(F)$ has been defined for every $F \in \Gamma_{\pi, E}$. Then, if $q \in F$
                      is a free point with respect to $\Gamma_{\pi, E}$, set
                      \begin{equation}
                        \Far_{(E,a_0,b_0)}(\tilde F) = (a_F +1, b_F)
                      \end{equation}
                      where
                      $\Far_{(E,a_0,b_0)}(F) = (a_F, b_F)$. If $q = F \cap F'$ is a satellite point with respect to $\Gamma_{\pi, E}$
                      then set
                      \begin{equation}
                        Far(\tilde F) = \left( a+a', b+ b' \right)
                        \label{<+label+>}
                      \end{equation}
                      where $\Far_{(E,a_0,b_0)}(F) = (a,b)$ and $\Far_{(E,a_0,b_0)}(F') = (a',b')$.

                      \begin{prop}
                        Let $\Far_{(E,a_0,b_0)}$ be a Farey parametrisation of $\cD_{X, p} \cup \left\{ E \right\}$.
                        \begin{enumerate}
                          \item Set $A_{(E,a_0,b_0)} (F) = \frac{a}{b}$ where $\Far_{(E,a_0,b_0)}(F) = (a,b)$, then $A$ is
                            a parametrisation of
                            $\Gamma_E$.
                          \item For any $F_1, F_2 \in \Gamma_E$ that are adjacent such that $v_{F_1} < v_{F_2}$ we have
                            \begin{equation}
                              a_2 b_1 - a_1 b_2 = 1
                              \label{<+label+>}
                            \end{equation}
                            where $\Far_{(E,a_0, b_0)}(F_i) = (a_i, b_i)$.
                          \item If $M = \begin{pmatrix}
                              \alpha & \beta \\ \gamma & \delta
                            \end{pmatrix} \in \PSL_2 (\Z)$, then $M \circ Far (F) := (\alpha a + \beta, \gamma b +
                            \delta)$ is another Farey parametrisation of $\cD_{X, p}$.
                        \end{enumerate}
                      \end{prop}

                      \begin{prop}
                        Let $X$ be a completion of $X_0$ and $p = E \cap F$ a satellite point at infinity. Then, the
                        skewness function $\alpha_E$ is the restriction of the Farey parametrisation $\Far_{E, 0,1}$ to
                        $[\ord_E, v_F)$
                      \end{prop}
                      \begin{proof}
                        This uses another parametrisation of the valuative tree defined in \cite{favreValuativeTree2004}
                        called the \emph{thinness} function. The thinness function $A_E$ of the valuative tree $\cV_X
                        (p;E)$ is defined by the Farey parametrisation starting with $Far(E) = (1,1)$. The relation
                        between $A_E$ and $\alpha_E$ is the following. Define the multiplicity function $m_E$ by
                        \begin{equation}
                          \forall \phi \in \hat{\OO_{X, p}}, \quad m_E (\phi) = E \cdot_p \left\{ \phi = 0 \right\}.
                          \label{<+label+>}
                        \end{equation}
                        The multiplicity of a valuation is defined as
                        \begin{equation}
                          m_E (v) := \min \{m_E (\phi) : v_\phi \geq v\}.
                          \label{<+label+>}
                        \end{equation}
                        and we have
                        \begin{equation}
                          A_E (v) = 1 + \int_{\ord_E}^v m_E (\mu) d\alpha_E (\mu).
                          \label{<+label+>}
                        \end{equation}
                        see \cite{favreValuativeTree2004} Definition 3.64. It is clear that on the segment $[\ord_E,
                          v_F[$ we get that $m_E$ is constant equal to 1. Hence, over this segment $A_E = 1 + \alpha_E$
                            and $\alpha_E$ is a Farey parametrisation of the segment and it satisfies $\alpha_E
                            (\ord_E) = 0$ and $\alpha_E (v_F) = + \infty$.
                          \end{proof}
                          An interval $[a,b]$ is a \emph{Farey interval} if $a = p/q, b = r/s $ and $qr - ps = 1$. We
                          make the convention that $\pm \infty = \pm 1 / 0$ such that for every integer $a, [-\infty, a] $ and
                          $[a, + \infty]$ are Farey intervals. The image of a Farey interval by an element of $\PGL_2
                          (\Z)$ is a Farey interval (even for infinite ones).

                          \begin{prop}\label{PropFareyIntervalSkewnessComposedWithMobius}
                            Let $E,F$ be two prime divisors at infinity and let $p = E \cap F$. Let $a_E, b_E, a_F, b_F$ be
                            nonnegative integers such that $a_F b_E - a_E b_F = 1$. If $M = \begin{pmatrix}
                            a_F & a_E \\ b_F & b_E \end{pmatrix} \in \PSL_2 (\Z)$, then
                            \begin{equation}
                              \psi := M \circ \alpha_E
                              \label{<+label+>}
                            \end{equation}
                            induces a homeomorphism between $[\ord_E, \ord_F]$ and the Farey interval
                            $[\frac{a_E}{b_E}, \frac{a_F}{b_F}]$ such that $\psi (\ord_E) = \frac{a_E}{b_E}$ and
                            $\psi (\ord_F) = \frac{a_F}{b_F}$.
                          \end{prop}

                          In particular, let $X$ be a completion such that $\BD$ is a cycle of rational curves. Let $E_1,
                          \dots, E_r$ be the irreducible components of $\BD$. Let $x_1, \dots, x_{r-1}$ be rational
                          numbers such that $[- \infty, x_1] , [x_{r-1}, + \infty], [x_i, x_{i+1}], i =1, \dots r_2$ are
                          Farey intervals (in particular $x_1$ and $x_{r-1}$ are integers). We identify $\CS^1$ with the
                          interval $[-\infty, + \infty] / (- \infty = + \infty)$. Write $x_i = p_i/ q_i$ with $\gcd (p_i,
                          q_i) = 1$, by Proposition \ref{PropFareyIntervalSkewnessComposedWithMobius}, we have a
                          homeomorphism $\phi : \cC_X \rightarrow \CS^1$ given by
                          \begin{align}
                            \phi_{|[\ord_{E_r}, v_{E_1})} &= \begin{pmatrix}
                              x_1 & -1 \\
                              1 & 0
                            \end{pmatrix} \circ \alpha_{E_r} \\
                            \phi_{|[\ord_{E_i}, v_{E_{i+1}})} &= \begin{pmatrix}
                              p_{i+1} & p_i \\
                              q_{i+1} & q_i
                            \end{pmatrix} \circ \alpha_{E_i} \\
                            \phi_{|[\ord_{E_{r-1}}, v_{E_r})} &= \begin{pmatrix}
                              1 & x_{r-1} \\
                              0 & 1
                            \end{pmatrix} \circ \alpha_{E_{r-1}}.
                            \label{<+label+>}
                          \end{align}
                          The image by $\phi$ of every interval $[\ord_{E_i}, \ord_{E_{i+1}}]$ is a Farey interval.

                          \subsection{The Thompson group}
                          Identify $\CS^1$ with the segment $[0,1] / \{0 =1 \}$. Traditionally, the Thompson group is a
                          subgroup of the group of homeomorphism of $\CS^1$ defined as follows. A
                          homeomorphism $g$ is in the Thompson group if there exists two subdivisions $\cup_{i=1}^r I_i,
                          \cup_{i = 1}^r J_i$ of $[0,1] / \{0=1 \}$ into Farey intervals such that $g$ sends $I_i$ to $J_i$
                          and $g_i : I_i \rightarrow J_i$ is given by a Möbius transformation with integer
                          coefficients (i.e given by a matrix of $\PGL_2 (\Z))$. We will instead define the Thompson
                          group with the same definition but with the identification $\CS^1 =
                        [-\infty, + \infty] / \{-\infty = + \infty \}$. We can go from one representation to the other by
                          conjugating with the (non canonical) homeomorphism $\phi : [0, 1] / \{ 0=1 \} \rightarrow
                          [-\infty, + \infty] / \{ - \infty = + \infty \}$ defined piecewise by
                          \begin{align}
                            \phi :& s \in[ 0, 1/2]  \mapsto  \frac{2s-1}{s} \in [- \infty, 0] \\
                            & s \in [1/2, 1]  \mapsto  \frac{2s-1}{1-s} \in [0, + \infty]
                            \label{<+label+>}
                          \end{align}
                          The advantage with this representation is that we can identify naturally $[\infty, +
                          \infty] /  \{ -\infty = + \infty \}$ with the boundary of the Poincaré half-plane $\HH$.
                          Thus, the group $\PGL_2 (\Z)$ acting by isometries on $\HH$ via Möbius transformations is a
                          subgroup of the Thompson group via its action on $\partial \HH$. It was first observed by
                          Diller and Lin that this group has a connection with algebraic dynamics, see
                          \cite{dillerRationalSurfaceMaps2016, dillerRotationNumbersElements2017}. We will denote by
                          $\mathbf T$ the Thomson group.

                          \begin{thm}\label{ThmAutomorphismThompsonGroup}
                            If $X_0$ is an affine surface such that $X \setminus X_0$ is a cycle of rational curves, then
                            every automorphism of $X_0$ acts on $\cC \simeq \CS^1$ via an element of the Thompson
                            group $\mathbf T$.
                          \end{thm}

                          This was already shown in \cite{cantatCommensuratingActionsBirational2019} Theorem 8.5.
                          Here, we show that the point of view of valuations is a natural one to prove this theorem.
                          This was mentioned in Remark 8.2 of \cite{cantatCommensuratingActionsBirational2019}.

                          \begin{proof}
                            This is a consequence of Proposition \ref{PropCycleCaseIndeterminacyPoints}. Fix a
                            completion $X$ and let $\phi : \cC_X \rightarrow \CS^1$ be the homeomorphism constructed
                            after Proposition \ref{PropFareyIntervalSkewnessComposedWithMobius}. Let $f \in
                            \Aut(X_0)$ be an automorphism. Suppose that $Y$ is a completion above $X$ such that
                            the lift $F : Y \rightarrow X$ is regular. Then, satellite points of $Y$ must be sent to
                            satellite points of $X$. Let $p = E \cap F$ be a satellite point
                            of $Y$ and $q = F(p) = E' \cap F' \in X$. Since $F$ is birational, $F$ is a composition of
                            blow-ups and of an automorphism of $Y$ (see e.g \cite[Proposition
                            V.5.3]{hartshorneAlgebraicGeometry1977}.  If $(x,y)$ are local coordinates associated to
                            $(E,F)$ and $(z,w)$ local coordinates associated to $(E', F')$, then $f$ is of the form
                            \begin{equation}
                              f(x,y) = (x^a y^b \phi, x^c y^d \psi).
                              \label{<+label+>}
                            \end{equation}
                            with $ad - bc = \pm 1$ and $\phi, \psi$ invertible. By Lemma \ref{LemmeActionSurMonomiale}
                            We get $f_*v_{s,t} = v_{as + bt, cs +dt}$. Hence,
                            $f_\bullet$ sends the Farey subinterval $[\ord_E, \ord_F]$ of $\cC_X$ to
                            another Farey subinterval of $\cC_X$ via a Möbius transformation. It acts as an element of
                            the Thompson group. This representation does not depend on the choice of the completion
                            $Y$. However it does depend on the initial choice of $X$ and of the homeomorphism $\cC_X
                            \simeq \CS^1$.
                          \end{proof}

                          Using this representation, we obtain a refinement of Gizatullin's theorem on automorphism
                          groups of affine surfaces completable by a cycle of rational curves from
                          \cite{gizatullinQuasihomogeneousAffineSurfaces1971}.
                          \begin{thm}\label{ThmAutoGroupIntoThompsonGroup}
                            There is a group homomorphism $\Aut(X_0) \rightarrow \mathbf T$.
                            The kernel is up to finite index an algebraic torus of dimension $d \leq 2$. And we have the
                            following
                            \begin{enumerate}
                              \item If $d = 2$, then $X_0 \simeq \G_m^2$.
                              \item If $X_0 \not \simeq \G_m^2$ and $\Aut(X_0)$ contains a loxodromic element, then $d = 0$.
                                In particular, the kernel is finite and $\Aut(X_0)$ is countable.
                              \item If $d = 1$, then up to finite index
                                \begin{equation}
                                  \Aut (X_0) \simeq \G_m, \text{ or } \Aut(X_0) \simeq \G_m \times A
                                  \label{<+label+>}
                                \end{equation}
                                where $A$ is a subgroup of $\Aut(\A^1)$ or $\Aut(\G_m)$.
                            \end{enumerate}
                          \end{thm}
                          \begin{proof} Notice that the kernel $K$ contains automorphism of $X_0$ that extend to an
                            automorphism of any cyclic completion $X$ of $X_0$. Since $X$ is projective, $\Aut
                            (X)$ is an algebraic group and $K$ is an algebraic subgroup. We denote by $K^0$ its connected component
                            of the identity. In the notations of \cite{gizatullinQuasihomogeneousAffineSurfaces1971} we have that
                            $K = \mathfrak U_1 (C)$ where $C = \BD$. Gizatullin showed
                            that the connected component of $\mathfrak U_1 (C)$ must be an algebraic torus of
                            dimension $d \leq 2$ and $d = 2$ if and only if $X_0 \simeq \G_m^2$. 

                            Now, Suppose $X_0 \not \simeq \G_m^2$. By
                            \cite{gizatullinQuasihomogeneousAffineSurfaces1971} Proposition 1, $X_0$ is rational,
                            therefore we can suppose $\Aut(X_0) \subset \Bir (\P^2)$. If $\Aut(X_0)$ contains a
                            loxodromic element, then \S 7 of \cite{delzantKahlerGroupsReal2012} shows that if $K$ is
                            infinite, then $K^0$ must have an open orbit in $X_0$ and therefore $\dim K^0 = 2$ which
                            is a contradiction. Thus $\dim K^0 = 0$ and $K$ is finite, $\Aut(X_0)$ is countable as
                            $\mathbf T$ is.

                            Finally, if $\dim K^0 = 1$, i.e $K^0 = \G_m$ then $\Aut(X_0)$ does not contain loxodromic elements by the
                            same argument as in the previous paragraph. We show that $\Aut(X_0)$ must preserve a
                            fibration over a curve. Since $K^0 = \G_m$ is reductive, we can form the GIT (Geometric
                            Invariant Theory) quotient $C:= X_0 // K^0$, it will be an affine curve because $K^0$ has an
                            orbit of dimension 1. Concretly, the ring of invariants $K[X_0]^{K^0}$ is a finitely
                            generated $K$-algebra and 
                            \begin{equation}
                              X_0 // K^0 := \spec K[X_0]^{K^0}.
                              \label{<+label+>}
                            \end{equation}
                            Any automorphism $f \in \Aut (X_0)$ acts on $K^0$ by conjugation, but
                            the action is even given by an algebraic group automorphism because $\Aut(X_0)$ has the
                            structure of an ind-algebraic group.
                             But the group of algebraic group
                            automorphism of $K^0$ is $\{\pm 1\}$ because $K^0 \simeq \G_m$. Therefore, up to an
                            index 2 subgroup, every element of $\Aut(X_0)$ commutes with the elements of $K^0$ and
                            therefore preserves the fibration to $C$. We have thus a
                            group homomorphism $\Aut(X_0) \rightarrow \Aut (C)$. We can assume that $C$ is normal by
                            taking its normalisation and let $\overline C$ be the unique
                            projective curve that is a completion of $C$.

                            If $g(\overline C) \geq 2$ (this can only happen in positive characteristic since $X_0$ is
                            rational), then $\Aut(C)$ is finite because $\overline C$ is of general type
                            and up to finite index $\Aut (X_0) \simeq \G_m$.

                            If $g(\overline C) = 1$ (again only possible in positive characteristic), then $\Aut (C)$ is
                            the subgroup of $\Aut(\overline C)$ that preserves $\overline C \setminus C$. This is a
                            finite subgroup, so up to finite index we also get $\Aut (X_0) \simeq \G_m$.

                            Finally, if $g (\overline C) = 0$, then $\overline C \simeq \P^1$ and $\Aut(C)$ is the
                            subgroup of $\Aut (\P^1) \simeq \PGL_2 (\C)$ that preserves $\P^1 \setminus C$. If $\# (\P^1
                            \setminus C) \geq 3$, then $\Aut (C)$ is finite and we get $\Aut (X_0) \simeq \G_m$ up to
                            finite index. Otherwise, $\#\overline C \setminus C = 2$ and $C \simeq \G_m$ or $\# \overline
                            C \setminus C = 1$ and $C \simeq \A^1$. So, up to finite index $\Aut(X_0) \simeq \G_m
                            \rtimes A$ where $A$ is a subgroup of $\Aut (\G_m)$ or $\Aut (\A^1)$.
                          \end{proof}

                          \section{An example: The Markov surface}\label{SecMarkovSurface}
                          Let $\k$ be an algebraically closed field with $\car k \neq 2$, let $D \in \k$ and consider the affine
                          surface $M_D \subset \A^3_\k$ of equation
                          \begin{equation}
                            x^2 + y^2 + z^2 = xyz +D.
                            \label{<+label+>}
                          \end{equation}
                          For any $D \in \k$, $M_D$ satisfies $\QAlb (M_D) = 0$. This is because if we consider the Zariski closure
                          $\oM_D$ of $M_D$ in $\P^3$, it is defined by the equation
                          \begin{equation}
                            T \left( X^2 + Y^2 + Z^2  \right) = XYZ + DT^3.
                            \label{<+label+>}
                          \end{equation}
                          Thus, $\oM_D \setminus M_D$ is the triangle of lines defined by the equations $\left\{ T=0, XYZ = 0
                          \right\}$. One shows that each line has self intersection $-1$, thus the matrix of the intersection form
                          at infinity is given by
                          \begin{equation}
                            \begin{pmatrix}
                              -1 & 1 & 1 \\
                              1 & -1 & 1 \\
                              1 & 1 & -1
                            \end{pmatrix}
                            \label{<+label+>}
                          \end{equation}
                          which is nondegenerate. Therefore, $M_D$ does not admit nonconstant invertible regular functions and it
                          admits loxodromic automorphisms, thus by Theorem
                          \ref{thm:dynamical_degree_quasi_albanese} $M_D \neq \G_m^2$ and $\QAlb(M_D) = 0$.

                          If $D \neq 0,4$, this is a smooth affine surface. If $D = 0$, $M_0$ is a normal affine surface with a
                          singularity at $(0,0,0)$. If $D = 4$, $M_4$ is a normal affine surface with 4 singularities at the points
                          \begin{equation}
                            \left( \pm 2, \pm 2, \pm 2 \right)
                            \label{<+label+>}
                          \end{equation}
                          where two of the signs must be equal.

                          We see that each surface $M_D$ falls into the category of the surface with a cycle at
                          infinity. Thus by
                          Theorem \ref{ThmAutoGroupIntoThompsonGroup}, there is a group homomorphism $\Aut (M_D) \rightarrow
                          G_{Thompson}$ with finite kernel. In the case of the Markov surface there is a very explicit
                          description
                          of the automorphism group and its image in the Thompson group.

                          Let $\sigma_x$ be the involution of $M_D$ defined by
                          \begin{equation}
                            \sigma_x (yz -x, y ,z).
                            \label{<+label+>}
                          \end{equation}
                          If we fix the coordinates $y,z$, then the equation defining $\cM_D$ becomes a polynomial equation of
                          degree 2 with respect to $x$, $\sigma_x$ permutes the 2 roots of this equation. We can define
                          $\sigma_y, \sigma_z$ in the same way. Then, $\sigma_x, \sigma_y, \sigma_z$ generate a free group
                          isomorphic to $(\Z /2 \Z) * (\Z /2 \Z) * (\Z /2 \Z)$ which is of finite index in $\Aut (\cM_D)$
                          (see \cite{el-hutiCubicSurfacesMarkov1974}).

                          Now, let $\Gamma^* \subset \PGL_2 (\Z)$ be the subgroup of matrices congruent to $\id \mod 2$.
                          Then, $\Gamma^*$ is the group generated by
                          \begin{equation}
                            M_x = \begin{pmatrix}-1 & -2 \\ 0 & 1 \end{pmatrix}, \quad M_y = \begin{pmatrix}
                              1 & 0 \\ 0 & -1
                            \end{pmatrix},
                            \quad M_z = \begin{pmatrix}
                              1 & 0 \\ -2 & -1
                            \end{pmatrix}                            \label{EqMatrixOfGenerators}
                          \end{equation}
                          and this three matrices have no non trivial relations apart from $M_x^2 = M_y^2 = M_z^2 =
                          \id$. Thus, we have an isomorphism $\langle \sigma_x, \sigma_y, \sigma_z \rangle \simeq \Gamma^*$ by
                          sending $s_u$ to $M_u$ for $u = x,y,z$. Now, $\Gamma^*$ acts by isometry on the Poincaré half
                          plane $\HH$ and also on its boundary $\partial \HH = \R \cup \left\{ \infty \right\} \simeq
                          \CS^1$.

                          \begin{thm}[\cite{cantatBersHenonPainleve2009}]\label{ThmAutoMarkov}
                            Up to finite index $Aut(M_D) \simeq \langle \sigma_x, \sigma_y, \sigma_z \rangle \simeq
                            \Gamma^*$. The group homomorphism $\langle \sigma_x, \sigma_y, \sigma_z \rangle \rightarrow
                            \mathbf T$ is injective and there
                            exists a homeomorphism $\phi : \cC \rightarrow [-\infty, + \infty]/ (- \infty = + \infty)
                            \simeq \partial \HH$ that conjugates the action of $\langle \sigma_x, \sigma_y,
                            \sigma_z \rangle$ on $\cC$ to the action of $\Gamma^*$ on $\partial \HH$.
                          \end{thm}
                          \begin{proof}
                            Take the completion $X = \overline M_D \subset \P^3$. We write $E_x$ for the prime divisor at
                            infinity $E_x = \left\{ X = T = 0 \right\}$ and define $E_y, E_z$ similarly. In $\CS^1 \simeq
                            \partial \HH$ define the three points $j_x, j_y, j_z = 0, -1, \infty$. Following the
                            construction after Proposition \ref{PropFareyIntervalSkewnessComposedWithMobius}, we define the
                            homeomorphism $\phi : \cC \rightarrow \CS^1$ as follows
                            \begin{align}
                              \psi ([\ord_{E_x}, v_{E_z})) &= \alpha_{E_x} \label{EqFareySurXZ} \\
                              \psi ([\ord_{E_y}, v_{E_x})) &= \begin{pmatrix}
                                0 & -1 \\ 1 & 1
                              \end{pmatrix} \circ \alpha_{E_y} \label{EqFareySurYX} \\
                              \psi \left( [\ord_{E_z}, v_{E_y}) \right) &= \begin{pmatrix}
                                -1 & -1 \\
                                1 & 0
                              \end{pmatrix}.
                              \label{EqFareySurZY}
                            \end{align}
                            For $u = x,y,z$, we have $\phi (\ord_{E_u}) = j_u$ and the segments
                            \begin{equation}
                              [\ord_{E_x}, \ord_{E_z}], \quad [\ord_{E_z}, \ord_{E_y}], \quad [\ord_{E_y}, \ord_{E_x}]
                            \end{equation}
                            are sent respectively to the Farey intervals
                            \begin{equation}
                              [0, + \infty] = [j_x, j_z], \quad [-\infty, -1] = [j_z, j_y], \quad [-1, 0] = [j_y, j_x].
                            \end{equation}
                            We show that $\phi$ is the desired homeomorphism by following the proof of Theorem
                            \ref{ThmAutomorphismThompsonGroup}. We will do the explicit computations only for
                            $\sigma_x$ as they are analogous for $\sigma_y$ and $\sigma_z$.

                            The only indeterminacy point of $\sigma_x $ on $X$ is the satellite point $E_y \cap E_z$.
                            Define $\tilde F_x$ to be the exceptional divisor above $E_y \cap E_z$. Let $Y$ be the
                            completion above $X$ obtained by blowing up $E_y \cap E_z$. For $u = x,y,z$, we write $F_u$ for
                            the strict transform of $E_u$ in $Y$. The lift $S_x: Y \rightarrow X$ of $\sigma_x$ is regular
                            and we have
                            \begin{equation}
                              S_x (F_x)  = E_y \cap E_z, \quad   S_x (\tilde F_x) = E_x, \quad S_x (F_y)  = E_y, \quad  S_x
                              (F_z) = E_z.
                              \label{<+label+>}
                            \end{equation}
                            We have $[\ord_{E_z}, \ord_{E_y}] = [\ord_{F_z}, \ord_{\tilde F_x}] \cup [\ord_{\tilde F_x},
                            \ord_{E_y}]$ and under the homeomorphism $\phi$ we get $\phi (\ord_{\tilde F_x}) = -2$ and
                            \begin{align}
                              \phi \left( \left[ \ord_{F_z}, \ord_{\tilde F_x} \right] \right) &= \left[ -\infty, -2 \right]
                              \\
                              \phi \left( \left[ \ord_{\tilde F_x}, \ord_{F_y}\right] \right) &= \left[ -2, -1 \right].
                              \label{<+label+>}
                            \end{align}
                            The satellite point $F_x \cap F_z$ is sent to $E_y \cap E_z$. If $(u,v)$ are local coordinates
                            associated to $(F_x, F_z)$ and $(z,w)$ local coordinates associated to $(E_y, E_z)$, then up to
                            multiplication by invertible regular functions we have
                            \begin{equation}
                              S_x (u,v) = \left( u , uv \right)
                              \label{<+label+>}
                            \end{equation}
                            and
                            \begin{equation}
                              (S_x)_* (v_{1,t}) = v_{1, 1+t}.
                            \end{equation}
                            In particular, $\alpha_{E_z}((S_x)_\bullet v_{1,t}) = \frac{1}{1+t}$.
                            Thus, using \eqref{EqFareySurXZ} and \eqref{EqFareySurZY}, $\sigma_x$ sends the interval $[0, +
                            \infty] = \phi ([\ord_{E_x}, \ord_{E_z}])$ to the interval $[-\infty, -2] = \phi ([\ord_{E_z},
                            \ord_{\tilde F_z}])$ via the Möbius map
                            \begin{equation}
                              t = \alpha_{E_x} (v_{1,t}) \mapsto -t -2 = M_x (t).
                              \label{<+label+>}
                            \end{equation}
                            One checks similarly using \eqref{EqFareySurXZ}, \eqref{EqFareySurYX} and \eqref{EqFareySurZY}
                            that, under the homeomorphism $\phi : \cC_X \simeq \CS^1$, the other intervals are sent to Farey
                            intervals by $\sigma_x$ via the Möbius map $M_x$.
                          \end{proof}
                          This sheds a new point of view on \S 2.6 of \cite{cantatBersHenonPainleve2009}. Every loxodromic
                          isometry of $g \in \Gamma^*$ admits two fixed point $\alpha (g), \omega(g) \in \partial \HH$ with
                          $\alpha(g)$ repulsive and $\omega(g)$ attracting for $g$ (they are called hyperbolic isometries in
                          \cite{cantatBersHenonPainleve2009}). Under the homeomorphism $\phi$
                          constructed in Theorem \ref{ThmAutoMarkov}, if we still denote by $g$ the automorphism induced on
                          $M_D$, then $\alpha (g)$ corresponds to the eigenvaluation $v_-$ and $\omega(g)$ to the
                          eigenvaluation $v_+$. In particular, we can recover Proposition 2.3 of
                          \cite{cantatBersHenonPainleve2009} using Theorem \ref{ThmAutoMarkov} and Lemma
                          \ref{LemmeConditionRegularité}.

                          \section{A second example, the complement of a nodal cubic}\label{sec:complement-nodal-cubic}
                          Let $C \subset \P^2_\C$ be a cubic with a nodal singularity and write $X_0 = \P^2 \setminus C$.
                          This is a smooth affine surface satisfying $\QAlb(X_0) = 0$. We are going to show that the
                          automorphism group of $X_0$ is virtually cyclic. This was proven in
                          \cite{yoshiharaProjectivePlaneCurves1985}. 

                          We have that $C^2 = 9$ and $C$ is in
                          fact rational. Let $X$ be the blow-up of $\P^2$ at the node and let $E$ be the exceptional
                          divisor. We write $C'$ for the strict transform of $C$ in $X$. We have 
                          \begin{equation}
                            Z_{\ord_{C'}} = \frac{1}{9} \pi^* C = \frac{1}{9} C' + \frac{2}{9} E, \quad Z_{\ord_E} =
                            \frac{2}{9} C' - \frac{5}{9} E.
                            \label{<+label+>}
                          \end{equation}
                          Notice that $Z_{\ord_E}^2 < 0$ so that $Z_{\ord_E}$ is not nef. Let $\left\{ p_1, p_2 \right\}
                          := C' \cap E$. We look at the monomial valuations $v_{1,s}$ centered at $p_1$, their
                          divisors are given by 
                          \begin{equation}
                            Z_s = Z_{\ord_{C'}} + s Z_{\ord_E} + Z_{v_{1,s},X,p}.
                            \label{<+label+>}
                          \end{equation}
                          And we have 
                          \begin{equation}
                            Z_s^2 = \frac{1}{9} + \frac{4s}{9} - \frac{5}{9} s^2 -s.
                            \label{<+label+>}
                          \end{equation}
                          This polynomial has only one positive root which is 
                          \begin{equation}
                            s_0 = \frac{3 \sqrt 5 - 5}{10}.
                            \label{<+label+>}
                          \end{equation}
                          The same is true for the other intersection point of $C' \cap E$. This means that the circle
                          at infinity $\cC$ is the union of two segments $I_+ \cup I_-$ such that $\ord_{C'} \in I_+$
                          and $I_+$ consists of nef valuations, $\ord_E \in I_-$ and $I_-$ consists of non-nef
                          valuations. Furthermore the intersection of these two segments is $I_+ \cap I_- = \left\{
                          v_1, v_2 \right\}$ where $v_i$ is the monomial valuation centered at $p_i$ with weight
                          $1, s_0$.

                          Now any automorphism $f \in \Aut(X_0)$ must preserve the space of nef valuations and non-nef
                          valuations because it preserves the intersection product. This means that $I_+ \cap I_- =
                          \left\{ v_1, v_2 \right\}$ must be invariant. Thus, any loxodromic automorphism of $X_0$ has
                          eigenvaluations $v_1, v_2$ and they all have the same axis. Looking at the classification of
                          birational maps of surfaces in \cite{cantatDynamiqueGroupeDautomorphismes2001}, we have that
                          any automorphism $f \in \Aut(X_0)$ is either loxodromic or elliptic. Indeed, the only remaning
                          case are parabolic transformations but they have only one invariant nef Weil class $\theta \in
                          \wNS $ which in fact is Cartier (it is the class of a fibration to a curve) and satisfies
                          $\theta^2 = 0$ so they cannot appear here because we cannot have $\theta = Z_{v_1}$ or
                          $Z_{v_2}$. Furthermore, define $\Aut^+ (X_0) \subset \Aut(X_0)$ which is the stabiliser
                          of $v_1$. For any $g \in \Aut^+(X_0)$ there
                          exists a positive number $\lambda (g) > 0$ such that 
                          \begin{equation}
                            g^* Z_{v_1} = \lambda (g) Z_{v_1}.
                            \label{<+label+>}
                          \end{equation}
                          It follows that either $\lambda (g)$ or $\lambda (g)^{-1}$ is the first dynamical degree of
                          $g$ and $\lambda (g) = 1$ if and only if $g$ is elliptic.

                          \begin{lemme}\label{lemme:complement-nodal-cubic-has-loxodromic-auto}
                            The group $\Aut^+(X_0)$ contains a loxodromic automorphism $f$.
                          \end{lemme}
                          \begin{proof}
                            Let $X$ be the blow-up of $\P^2$ at the node of $C$, then
                            $\partial_{X} X_0 = C' \cup E$ where $E$ is the exceptional divisor, $(C')^2 = 5$ and
                            $C' \cap E = \left\{ p_1, p_2 \right\}$. We do the following procedure: We blow-up $p_1$ and
                            let $E_1$ be the exceptional divisor and $X_1$ the new completion. We then do the following
                            algorithm: if the strict transform $C'$ of $C$ in $X_k$ has self-intersection we stop,
                            otherwise we blow up $X_k$ at $C' \cap E_k$ and call $X_{k+1}$ the new completion with
                            exceptional divisor $E_{k+1}$. This procedure stops after 6 blow-ups counting the blow of
                            $p_1$. So in $X_6$, the boundary $\partial_{X_6} X_0$ is a cycle of rational curves
                            $(C', E_6, E_5, E_4, E_3,E_2,E_1, E)$ with self-intersections $(-1,-1,-2,-2,-2,-2,-2,-2)$.
                            Then, we contract $C' \cup E \cup E_1 \cdots E_5 $ starting with the blow-down of $C'$. This
                              procedure yields an automorphism $f$ of $X_0$ which is not an automorphism of $\P^2$. We
                              show that $f$ is loxodromic. To do so, we have the following statement on the dynamics of
                              $f$ over $X$: 
                              \begin{enumerate}
                                \item The only indeterminacy point of $f$ is $p_1$. 
                                \item $f$ contracts $C'$ and $E$ to $p_2$.
                                \item We have 
                                  \begin{equation}
                                    f_X^* C' = 6 C' + 5E, \quad f_X^* E = C'+E.
                                    \label{<+label+>}
                                  \end{equation}
                              \end{enumerate}
                              So it is clear that $v_2$ is fixed by $f$ and therefore $v_1$ as well. At the point
                              $p_2$ the germ of $f$ is therefore given by 
                              \begin{equation}
                                f(x,y) = \left( x^6 y^5, xy \right)
                                \label{<+label+>}
                              \end{equation}
                              where $(x,y)$ are associated to $(C', E)$. In particular, the first dynamical degree of
                              $f$ is the spectral radius of the matrix $\begin{pmatrix}
                                6 & 5 \\ 1 & 1
                              \end{pmatrix}$ which is equal to $\frac{7 + 3 \sqrt 5}{2} > 1$. And in fact the direct
                              computation shows that 
                              \begin{equation}
                                f_* v_{1, s_0} = \frac{7 + 3 \sqrt 5}{2} v_{1,s_0}
                                \label{<+label+>}
                              \end{equation}
                              where $v_2 = v_{1,s_0}$.
                          \end{proof}
                          Thus we have a group homomorphism 
                          \begin{equation}
                            g \in \Aut^+ (X_0) \mapsto \log \lambda (g) \in \R.
                            \label{<+label+>}
                          \end{equation}
                          And the kernel consists only of elliptic elements. By Lemma 7.3 of
                          \cite{urechSubgroupsEllipticElements2021}, since $X_0 \neq \G_m^2$ we must have that the
                          kernel is finite. In particular, if $f\in \Aut^+(X_0)$ is the loxodromic automorphism from
                          Lemma \ref{lemme:complement-nodal-cubic-has-loxodromic-auto} then the subgroup
                          generated by $f$ is a finite index subgroup of $\Aut^+ (X_0)$.

                          \section{Berkovich skeleton and integral affine structure}\label{sec:}
                          In this section, we develop the discussion on the circle at infinity and the Thompson group
                          using Berkovich theory. The goal here is to help the reader that knows about Berkovich theory
                          to reconnect what is appearing in this memoir. We won't provide proofs. 

                          Let $X$ be a completion of $X_0$ and let $\Gamma_X$ be the dual graph of $\BD$. We assume that
                          $\BD$ is with simple normal crossings and furthermore that for any two irreducible components
                          $E,F$ that $E \cap F$ is irreducible or empty. This is always possible after a finite number
                          of blowups. If we realise
                          $\Gamma_X$ as a simplex in $\R^d$ where $d$ is the number of irreducible components of
                          $\BD := E_1 \bigcup \cdots \bigcup E_d$, then write $[E_i]$ for the point associated to
                          $E_i$ in $\R^d$. Then, $\Gamma_X$ is the simplex: 
                          \begin{equation}
                            \Gamma_X = \left\{ [E_1], \cdots, [E_d] \right\} \cup \left\{ \alpha_i [E_i] + \beta_j
                            [E_j] : i \neq j, E_i \cap E_j \neq \emptyset, 0 < \alpha_i, \beta_j, \alpha_i + \beta_j = 1
                          \right\}.
                            \label{<+label+>}
                          \end{equation}
                          The cone over $\Gamma_X$ has then two types of points, either of the form $\alpha [E_i]$ for
                          some $i =1, \dots, d$ with $\alpha > 0$ or $\alpha [E_i] + \beta [E_j]$ with $\alpha, \beta >
                          0$, we denote it by $\left| \Gamma_X \right|$. There exists a continuous map $i_X : \left|
                          \Gamma_X \right| \hookrightarrow \cV_\infty^{\qm}$ which is a homeomorphism on its image given by 
                          \begin{equation}
                            \alpha [E] \mapsto \alpha \ord_E, \quad \alpha [E] + \beta [F] \mapsto v_{\alpha, \beta}
                            \label{<+label+>}
                          \end{equation}
                          where $v_{\alpha, \beta}$ is the monomial valuation centered at $E \cap F$ with weight
                          $\alpha$ and $\beta$. One can check that on the image of $\left| \Gamma_X \right|$ the weak
                          and the strong topology coincide. We have an integral affine structure on $\left| \Gamma_X
                          \right|$ given by its realisation as the cone over the simplex $\Gamma_X$. It is the unique
                          affine structure on $\left| \Gamma_X \right|$ such that the evaluation maps $v \mapsto
                          v(P)$ are piecewise affine for every $P \in K[X_0]$. For more details on integral affine
                          structures and Berkovich spaces, see \cite{kontsevichAffineStructuresNonArchimedean2006}.

                          Now, if $\pi : Y \rightarrow X$ is a morphism of completions, there is a canonical map induced
                          by $\pi$ on the dual complexes which yields $\left| \pi \right|: \left| \Gamma_Y \right|
                          \rightarrow \left| \Gamma_X \right|$, it is continuous and in fact it is piecewise affine. If
                          $X$ and $Y$ are cyclic completions, this map is in fact an isomorphism of integral affine
                          structures. Then, $\left| \Gamma_X \right|$ is a union of upper-half quadrants in $\R^2$
                          patched along their boundaries, it is thus homeomorphic to $\R^2 \setminus \left\{ 0
                          \right\}$ and the projectivisation $\P (\left| \Gamma_X \right|) \subset \hat \cV_\infty$ is thus isomorphic to the
                          circle $\CS^1$ with the integral affine structure given by $\CS^1 \simeq \R / \Z$.

                          Suppose $X_0$ is a normal affine surface completable by a cycle of rational curves, then by
                          the work of Gizatullin we know that every minimal completion of $X_0$ is a cyclic completion.
                          Thus, for any completion $X$, the image of $\P (\left| \Gamma_X \right|)$ in $\hat \cV_\infty$
                          contains a circle $\cC$ which is canonical and can be defined as the image of $\P \left(
                            \left| \Gamma_Z \right|
                          \right)$ for any cyclic completion $Z$ of $X_0$. The subset $\cC$ is called the
                          \emph{skeleton} of $\hat \cV_\infty$. Now, we have seen that any automorphism $f \in
                          \Aut(X_0)$ must preserve $\cC$ and thus induces a homeomorphism of $\CS^1 = \R / \Z$ which
                          preserves the integral affine structure, these are exactly the piecewise affine homeomorphism of
                          $\CS^1$ which are exactly the elements of the Thompson group. We therefore obtain a group
                          homomorphism $\Aut(X_0) \rightarrow \mathbf T$.

                          \chapter{Examples}\label{ChapterExamples}
                          \section{An affine surface with a lot of nonproper endomorphisms}
                          \subsection{A family of rational affine surface with no loxodromic automorphisms}
                          In \cite{duboulozCompletionsNormalAffine2004} Example 2.23, Dubouloz gives an infinite family of
                          examples of rational complex affine surfaces that admit a minimal completion for which the dual
                          graph of the curve at infinity is neither a zigzag nor a cycle. This means by Theorem
                          \ref{ThmAutomorphismCaseDynamicalCompactifications} that these surfaces do not admit loxodromic
                          automorphism. The result is the following: Consider the affine surface $S_0 \subset \A^3_\C$ given
                          by the equation
                          \begin{equation}
                            x^n y = P(z)
                            \label{<+label+>}
                          \end{equation}
                          where $n \geq 2$ and $P$ is a degree $r$ polynomial with $r \geq 2$ distinct roots. Then, $S_0$
                          admits a minimal completion for which the dual graph at infinity is given by
                          \begin{equation}
                            \begin{tikzpicture}
                              \draw (-1,0) node{$\bullet$};
                              \draw (0,0) node{$\bullet$};
                              \draw (1,0) node{$\bullet$};
                              \draw (2,0.5) node{$\bullet$};
                              \draw (2,-0.5) node{$\bullet$};
                              \draw (3,0.5) node{$\square$};
                              \draw (3,-0.5) node{$\square$};
                              \draw (-1,0) -- (0,0);
                              \draw (0,0) -- (1,0) -- (2, 0.5) -- (2.9, 0.5);
                              \draw (1,0) -- (2, -0.5) -- (2.9, -0.5);
                            \end{tikzpicture}
                            \label{<+label+>}
                          \end{equation}
                          where $\square$ is a zigzag of $(-2)$-curves of length $n-3$ if $n \geq 3$ and $\square = \emptyset$
                          otherwise.

                          \subsection{A subfamily with a lot of endomorphisms}
                          In \cite{duboulozJacobianConjectureFails2018} \S 5A, Dubouloz and Palka study the following family
                          of surfaces
                          \begin{equation}
                            S(n) := \left\{ x^n y = z^n - 1 \right\} \quad (n \geq 2).
                            \label{<+label+>}
                          \end{equation}
                          They fall inside the previous category of affine surfaces; $S(n)$ admits a $\Z / n \Z$ action given
                          by
                          \begin{equation}
                            \forall a \in \Z / n \Z, \quad a \cdot (x,y,z) = (\epsilon^a x, y, \epsilon^{-a} z)
                            \label{<+label+>}
                          \end{equation}
                          where $\epsilon$ is a primitive $n$-th root of unity. The quotient $S(n) / (\Z / n \Z)$ is an affine
                          surface $S ' (n)$ of equation
                          \begin{equation}
                            S' (n) = \left\{ u(1+uv) = w^n \right\}
                            \label{<+label+>}
                          \end{equation}
                          and the quotient map $\pi : S(n) \rightarrow S' (n)$ is given by
                          \begin{equation}
                            \pi(x,y,z) = \left( x^n, y, xz \right).
                            \label{<+label+>}
                          \end{equation}
                          We have the surprising result
                          \begin{prop}[{\cite[Corollary 5.2]{duboulozJacobianConjectureFails2018}}]
                            For every $n \geq 2$ the affine surface $S ' (n)$ admits a strict open embedding into $S(n)$ given by
                            the following formula
                            \begin{equation}
                              j(u,v,w) = \left( w, v R_0 (-uv), R_1 (-uv) \right)
                              \label{<+label+>}
                            \end{equation}
                            where $R_1 (t) = (\epsilon-1)t + 1$ and $R_0 (t) = \frac{R_1(t)^n - 1}{t (t-1)} \in \C [t]$ where
                            $\epsilon \neq 1$ is an $n$-th root of unity. Different choices of $\epsilon$ lead to different
                            embeddings that are not conjugated by the $\Z / n \Z$ action over $S(n)$.
                          \end{prop}

                          Hence we can define the endomorphism $f : S(n) \rightarrow S(n)$ defined by $f= j \circ \pi$. This
                          yields a nonproper endomorphism of $S(n)$ of topological degree $n$. We can twist this example using
                          the following result

                          \begin{prop}[{\cite[Example 2.11 and Proposition 5.5]{duboulozJacobianConjectureFails2018}}]\label{PropAutoSN}
                            Let $n \geq 2$ be an integer. Every polynomial $P \in \C[x]$ yields an automorphism $g_P$ of
                            $S(n)$ defined by
                            \begin{equation}
                              g_P (x,y,z) := \left( x, y + \frac{(z + P(x) x^n)^n - z^n}{x^n}, z + P(x) x^n \right).
                              \label{<+label+>}
                            \end{equation}
                          \end{prop}

                          \subsection{The surface $S(2)$}
                          We treat in details the example of $S(2) = \left\{ x^2 y = z^2 -1 \right\}$. The $\Z/ 2 \Z$ action is
                          given by $(-1) \cdot (x,y,z) = (-x,y, -z)$. To find a minimal completion of $S(2)$ we follow the
                          computations of \cite{duboulozCompletionsNormalAffine2004} Example 2.23. Consider the birational morphism
                          \begin{equation}
                            \phi : (x,y,z) \in S(2) \mapsto (x,z) \in \A^2 \subset \P^1 \times \P^1.
                            \label{EqPhi}
                          \end{equation}
                          Define the following curves in $S(2), C_\epsilon = \left\{ x = 0, z = \epsilon \right\}$ where
                          $\epsilon = \pm 1, F_0 := \left\{ 0 \right\} \times \P^1, F_\infty = \left\{ \infty
                          \right\} \times \P^1$ and $L = \P^1 \times \left\{ \infty \right\}$, then

                          \begin{equation}
                            \phi : S(2) \setminus (C_1
                            \cup C_{-1}) \rightarrow \P^1 \times \P^1 \setminus (F_0 \cup F_\infty \cup L)
                          \end{equation}
                          is an isomorphism with inverse given by
                          \begin{equation}
                            {\phi}^{-1} (u,v) = \left( u, \frac{v^2 - 1}{u^2}, v \right).
                            \label{<+label+>}
                          \end{equation}
                          The curve $C_\epsilon$ is contracted by $\phi$ to $(0, \epsilon) \in F_0$. Let $F_\epsilon$ be the
                          exceptional divisor above $(0,\epsilon)$. The lift of $\phi$ contract $C_\epsilon$ to a free point
                          on $F_\epsilon$ that we call $p_\epsilon$. Let $X$ be the blow up of $p_1$ and $p_{-1}$, then
                          $\phi$ induces an open embedding $\phi : S(2) \hookrightarrow X$ as $C_\epsilon$ is sent by $\phi$
                          to the exceptional divisor above $p_\epsilon$. Hence, $X$ is a completion of $S(2)$ and the dual
                          graph of the boundary is

                          \begin{equation}
                            \begin{tikzpicture}
                              \draw (0,0) node{$\bullet$};
                              \draw (0,0) node[above]{0};
                              \draw (0,0) node[below]{$F_\infty$};
                              \draw (1,0) node{$\bullet$};
                              \draw (1,0) node[above]{0};
                              \draw (1,0) node[below]{$L$};
                              \draw (2,0) node{$\bullet$};
                              \draw (2,0) node[above]{-2};
                              \draw (2,0) node[below]{$F_0$};
                              \draw (3.5,1) node{$\bullet$};
                              \draw (3.5,1) node[above]{-2};
                              \draw (3.5,1) node[below]{$F_1$};
                              \draw (3.5,-1) node{$\bullet$};
                              \draw (3.5,-1) node[above]{-2};
                              \draw (3.5,-1) node[below]{$F_{-1}$};
                              \draw (0,0) -- (1,0) -- (2,0) -- (3.5,1);
                              \draw (2,0) -- (3.5, -1);
                            \end{tikzpicture}
                            \label{EqDualGraphX}
                          \end{equation}
                          Here, we still denote by $F_\infty, F_0, L$ their strict transform in $X$. In particular,
                          $F_\infty$ is not linearly equivalent to $F_0$ but to $F_0 + F_1 + F_{-1} + \phi (C_1) + \phi
                          (C_{-1})$ which is the strict transform of the "original" $F_0 = \left\{ 0 \right\} \times \P^1
                          \subset \P^1 \times \P^1$.

                          \begin{prop}\label{PropCenterEigenvaluationOnS2}
                            The surface $S(2)$ satisfies $\QAlb (S(2)) = 0$.
                            For every endomorphism $f$ of $S(2)$ such that $\lambda_1(f)^2 > \lambda_2(f)$, the eigenvaluation
                            $v_*$ of Theorem \ref{ThmExistenceEigenvaluationSurfaces} satisfies one of the three
                            following possibilities
                            \begin{enumerate}
                              \item $v_* = \ord_L$.
                              \item $c_X (v_*)$ is a free point on $L$.
                              \item $c_X (v_*) = F_\infty \cap L$.
                            \end{enumerate}
                          \end{prop}
                          \begin{proof}
                            Since $S(2)$ is birational to $\A^2$ we have that it is a rational surface, hence $\Pic^0 (S(2)) =
                            0$. It suffices to show that $S(2)$ does not admit any nonconstant invertible regular function. To
                            do so, we consider the intersection form on $\Div_\infty (X) =  \Z F_\infty \oplus \Z L \oplus \Z
                            F_0 \oplus \Z F_1 \oplus \Z F_{-1}$. It suffices to show that it is non
                            degenerate. The matrix of the intersection form in the basis $(F_\infty, L, F_0, F_1, F_{-1})$is given by
                            \begin{equation}
                              M = \begin{pmatrix}
                                0 & 1 & 0 & 0 & 0 \\
                                1 & 0 & 1 & 0 & 0 \\
                                0 & 1 & -2 & 1 & 1 \\
                                0 & 0 & 1 & -2 & 0 \\
                                0 & 0 & 1 & 0 & -2
                              \end{pmatrix}.
                              \label{<+label+>}
                            \end{equation}
                            It is invertible with inverse given by
                            \begin{equation}
                              {M}^{-1} =
                              \frac{1}{4} \begin{pmatrix}
                                -4 & 4 & 4 & 2 & 2 \\
                                4 & 0 & 0 & 0 & 0 \\
                                4 & 0 & -4 & -2 & -2 \\
                                2 & 0 & -2 & -3 & -1 \\
                                2 & 0 & -2 & -1 & -3
                              \end{pmatrix}.
                              \label{<+label+>}
                            \end{equation}
                            Hence the intersection form in nondegenerate on $\DivInf(X)$ which shows that $\QAlb
                            (S(2)) = 0$.

                            Therefore, we are in the condition of Theorem \ref{ThmExistenceEigenvaluationSurfaces}. Let $f$ be
                            a dominant endomorphism with $\lambda_1 (f)^2 > \lambda_2 (f)$. Let $v_*$ be its eigenvaluation.
                            Then, the invariant class $\theta_* \in \mathbb L^2 (S(2))$ is of the form
                            \begin{equation}
                              \theta_* = w + Z_{v_*}.
                              \label{<+label+>}
                            \end{equation}
                            with $w \in \Pic(X) \cap \DivInf(X)^\perp$.
                            Therefore, we must have $(Z_{v_*})^2 \geq 0$ as $w^2 \leq 0$ and since $\theta_*$ is nef, we have
                            $(\theta_*)^2 \leq (\theta_{*, Y})^2$ for every completion $Y$. This implies that
                            $(Z_{v_*, X})^2 \geq 0$.

                            Now, for all prime divisor $E$ of $X$ at infinity, $Z_{\ord_E}$ is given by a column of
                            ${M}^{-1}$. Indeed, $M$ is the matrix of the intersection form in the basis $(F_\infty, L, F_0, F_1,
                            F_{-1})$ and therefore ${M}^{-1}$ is the matrix of the intersection form in the dual basis
                            \begin{equation}
                              \left( Z_{\ord_{F_\infty}}, Z_{\ord_L}, Z_{\ord_{F_0}}, Z_{\ord_{F_1}}, Z_{\ord_{F_{-1}}}
                              \right).
                            \end{equation}
                            For example $Z_{\ord_L} = F_\infty$ and $Z_{\ord_{F_\infty}} = -F_\infty + L + F_0 +
                            \frac{1}{2} F_1 + \frac{1}{2} F_{-1}$. In particular, we have that $L$ is the unique prime divisor
                            at infinity of $X$ such that $Z_{\ord_L}^2 \geq 0$. This implies that $c_{X} (v_*)$ cannot be
                            a free point on a prime divisor $E \neq L$ otherwise we would get $(Z_{v_*, X})^2 < 0$. If
                            $c_X (v_*)$ is a satellite point, then it cannot
                            be $F_0 \cap F_\epsilon$ because in that case
                            \begin{equation}
                              Z_{v_*, X} = \lambda Z_{\ord_{F_0}} + \mu Z_{\ord_{F_\epsilon}}
                              \label{<+label+>}
                            \end{equation}
                            with $\lambda, \mu > 0$ and looking at the last three rows and columns of ${M}^{-1}$ we would get
                            $(Z_{v_*, X})^2 < 0$. For the same reason, we cannot have $c_X (v_*) = L \cap F_0$ because
                            otherwise
                            \begin{equation}
                            (Z_{v_*, X})^2 = \left( \lambda Z_{\ord_L} + \mu Z_{\ord_{F_0}} \right)^2
                            \label{<+label+>}
                          \end{equation}
                          and $(Z_{v_*, X})^2 = - \mu^2 < 0$, this is a contradiction.
                          Hence $c_X (v_*) \in L \setminus \left\{ L \cap F_0 \right\}$ or $c_X (v_*) = L$.
                        \end{proof}
                        {\textbf{First example of endomorphism.--}}
                        The endomorphism $f = j \circ \pi $ is equal to
                        \begin{equation}
                          f(x,y,z) = \left( xz, 4y, 2 z^2 -1 \right).
                          \label{<+label+>}
                        \end{equation}
                        Using the map $\phi : S(2) \rightarrow \A^2$ from Equation \eqref{EqPhi}, we have that $f$ is
                        conjugated to
                        \begin{equation}
                          \eta: (u,v) \in \A^2 \mapsto (uv, 2 v^2 - 1) \in \A^2
                          \label{<+label+>}
                        \end{equation}
                        Hence we get that $\lambda_1(f) = \lambda_2(f) = 2$. Consider the completion $X$ of $S(2)$ defined
                        above with dual graph given by Equation \eqref{EqDualGraphX}. We know that the eigenvaluation $v_*$
                        of $f$ must be centered on $L = \P^1 \times \left\{ \infty \right\}$. Therefore, we can study the
                        local dynamics of $f$ on $L$ using $\eta$. Let $[U:T], [V:W]$ be the homogeneous coordinates of $\P^1
                        \times \P^1$ such that $u = \frac{U}{T}$ and $v = \frac{V}{W}$. In homogeneous coordinates we get
                        \begin{equation}
                          \eta \left( [U:T] , [V:W] \right) = \left( [UV : TW ] , [2V^2 - W^2 : W^2] \right).
                          \label{<+label+>}
                        \end{equation}
                        Consider the affine coordinates $t = \frac{T}{U}$ and $w = \frac{W}{V}$. In particular, $t = 0$ is a
                        local equation of $F_\infty$ and $w = 0$ is a local equation of $L$. Then, in these coordinates we
                        have
                        \begin{equation}
                          \eta (t,w) = \left( tw, \frac{w^2}{2 - w^2} \right).
                          \label{<+label+>}
                        \end{equation}

                        Hence, we get that $(0,0) = F_\infty \cap L$ is a fixed point. From $\eta^* t = tw$ we infer $f_*
                        \ord_{F_\infty} = \ord_{F_\infty}$. But $\ord_{F_\infty}$ is \emph{not} the eigenvaluation by
                        Proposition \ref{PropCenterEigenvaluationOnS2}.
                        We have that $L$ is contracted to $(0,0)$ so $v_*$ must be centered at
                        $(0,0)$. We blow up $(0,0) = F_\infty \cap L$. Let $\tilde E$ be the exceptional divisor and let $s,
                        s'$ be local coordinates at $\tilde E \cap L$ (we still denote by $L$ its strict transform), the
                        blow up map is given by
                        \begin{equation}
                          \pi (s, s') = (s, s s');
                          \label{<+label+>}
                        \end{equation}
                        $s' = 0$ is a local equation of $L$ and $s = 0$ a local equation of $\tilde E$. At $\tilde E \cap
                        L$, we get
                        \begin{equation}
                          f(s, s') = \left( s^2 s', \frac{s' }{2 - s^2 (s')^2} \right).
                          \label{<+label+>}
                        \end{equation}
                        Thus, $f^* s = s^2 s'$ and therefore $f_* \ord_{\tilde E} = 2 \ord_{\tilde E}$. This implies that
                        $v_* = \ord_{\tilde E}$.

                        {\textbf{Second example.--}} Consider Proposition \ref{PropAutoSN} with $P = 1$ and $f$ the
                        endomorphism of $S(2)$ from the previous paragraph. Define $g = g_1 \circ f$, then
                        \begin{equation}
                          g(x,y,z) = \left( xz, x^2 z^2 + 4 z^2 + 4y -2, x^2 z^2 + 2 z^2 -1 \right)
                          \label{<+label+>}
                        \end{equation}
                        Let $\A^2 \subset \P^1 \times \P^1$ with affine coordinates $(u,v)$ and the birational morphism
                        $\phi : S(2) \rightarrow \A^2$. Then, $g$ is conjugated by $\phi$ to
                        \begin{equation}
                          \eta (u,v) = (uv, u^2 v^2 + 2 v^2 - 1).
                          \label{<+label+>}
                        \end{equation}
                        It is an endomorphism of $S(2)$ of topological degree 2. By Proposition
                        \ref{PropCenterEigenvaluationOnS2}, if $v_*$ is the eigenvaluation of $g$, then its center must belong to
                        $L$. Consider the affine coordinates $t = 1/u, w = 1/v$ centered at $F_\infty \cap L$.
                        In these coordinates we have
                        \begin{equation}
                          \eta (t,w) = (tw, \frac{t^2 w^2}{1 + 2 t^2 - t^2 w^2}).
                          \label{<+label+>}
                        \end{equation}
                        Hence, $F_\infty$ and $L$ are both contracted to $L \cap F_\infty = p$ so it must be equal
                        to $c_{X} (v_*)$. Blow up $p$ and let $E_1$ be the exceptional divisor. Blow up again the
                        intersection point of $E_1$ and the strict transform of $L$ and let $E_2$ be the exceptional
                        divisor. Then there exists local algebraic coordinates $(u,v)$ at $E_2 \cap L$ associated to
                        $(E_2, L)$ such that
                        \begin{equation}
                          g(u,v) = (u^3 v, \frac{1}{1 + u^2 (2 - u^4 v^2)})
                          \label{<+label+>}
                        \end{equation}
                        we see that the point $(0,1)$ in these coordinates is fixed. Consider the local analytic coordinates
                        at this point given by $(u, w) = (u, v-1)$. Then,
                        \begin{equation}
                          g(u,w) = \left( u^3( 1 + w), -2 u^2 + u^2\psi(u,w) \right)
                          \label{<+label+>}
                        \end{equation}
                        where $\psi(u,w)$ is a holomorphic function with $\psi(0,0) = 0$. We have $g^* u = u^3 (1 +w)$ this
                        implies that $\lambda_1(g) = 3$ and since $\lambda_1 (g) = 3 > \lambda_2 (g) = 2$ we have by
                        Lemma \ref{LemmaDivisorialConditionOnDegrees} that $v_*$ is not divisorial. Therefore, we get that
                        $v_*$ is infinitely singular and the center of $v_*$ is a free point on this completion.

                        \begin{rmq}\label{rmq:}
                          The fact that $v_*$ is infinitely singular also follows from \cite{favreEigenvaluations2007}
                          since the endomorphism $g$ is conjugated to a polynomial endomorphism of $\C^2$. Also the
                          computation of the first dynamical degree could have been done directly over $\P^1 \times
                          \P^1$ since the rational map $g : \P^1 \times \P^1 \dashrightarrow \P^1 \times \P^1$ is
                          algebraically stable.
                        \end{rmq}

                        \section[An elliptic curve at infinity with an action by translation]{An affine surface with an
                          elliptic curve at infinity with an action by
                        translation}\label{sec:example-elliptic-case}
                        We show that the Elliptic case in Theorem
                        \ref{ThmLocalFormOfMapAndDynamicalDegreAreQuadraticInteger} can happen. Start with a generic (2,2,2)
                        divisor $V$ in $\P^1 \times \P^1 \times \P^1$. This is a K3 surface. It is
                        given by one equation in the variables $(x,y,z)$ which is of degree 2 with respect to each $x,y,z$.
                        The projection on the plane $\P^1 \times \P^1$ induced by forgetting the coordinate $x$ yields a 2:1
                        cover $V \rightarrow \P^1 \times \P^1$. Indeed, we can rewrite the equation of $V$ as
                        \begin{equation}
                          a(y,z) x^2 + b(y,z) x + c (y,z) = 0
                          \label{<+label+>}
                        \end{equation}
                        where $b,c$ are rational functions in $y,z$ of degree 2 with respect to each variable. Let $\sigma_x$ be
                        the involution of $V$ that switches the folds of the cover. We can define similarly the involutions
                        $\sigma_y, \sigma_z$. The group generated by $\sigma_x, \sigma_y, \sigma_z$ is $\Z /2 \Z * \Z /2
                        \Z * \Z /2 \Z$ (see \cite{cantatDynamiqueGroupeDautomorphismes2001}). Now, we restrict our
                        attention to the subgroup generated by $\sigma_x, \sigma_y$. Consider the family of curves
                        defined by the hyperplane sections
                        $E_\alpha := V \cap \left\{ z = \alpha \right\}$. The involutions $\sigma_x, \sigma_y$ preserve
                        $E_\alpha$ for every $\alpha$. Thus, for a very general parameter $\alpha$ the subgroup of
                        $\Aut(E_\alpha)$ generated by ${\sigma_x}_{|E_\alpha}, {\sigma_y}_{|E_\alpha}$ is isomorphic to $\Z
                        /2 \Z * \Z / 2 \Z$.

                        \begin{prop}
                          For $\alpha \in \C$ very general, set $E:= E_\alpha, X_0 = \P^1 \times \P^1 \setminus E$ where we
                          have identified $\P^1 \times \P^1 \simeq \P^1 \times \P^1 \times \P^1 \bigcap \left\{z =
                          \alpha\right\}$. Then, $X_0$ is a smooth
                          affine surface and there exists two endomorphisms $f,g$ of $X_0$ such that
                          \begin{enumerate}
                            \item $\lambda_1(f) = \lambda_1 (g) = \lambda_2(f) = \lambda_2 (g) =2$.
                            \item $f_{|E} = \sigma_y$
                            \item $g_{|E} = \sigma_x$
                            \item let $h = g \circ f$, then $\lambda_1(h) = \lambda_2(h) = 4$ and $h_{|E}$ is a
                              translation by a non-torsion element.
                          \end{enumerate}
                        \end{prop}
                        \begin{proof}
                          We look again at Example
                          \ref{ExEllipticCurveAtInfinity}.  We write
                          the equation of $E$ in two different ways. There exists degree two rational fractions $P(x), Q(x)
                          \in \C(x), R(y), S(y) \in \C(y)$ such that the equation of $E$ is of the form
                          \begin{equation}
                            y^2 - P(x) y + Q(x) = 0
                            \label{<+label+>}
                          \end{equation}
                          and
                          \begin{equation}
                            x^2 - R(y) x + S(y) = 0
                            \label{<+label+>}
                          \end{equation}
                          Let $X = \P^1 \times \P^1$ and let $X_0 = X \setminus E$. Let $k : X_0 \rightarrow X_0$ be the
                          endomorphism from Example \ref{ExEllipticCurveAtInfinity}. That is $k$ commutes with $\pr_x$ and
                          acts as $z \mapsto z^2$ on each fiber with $0, \infty \in \P^1$ being the
                          intersection points of $E$ with the fiber. We let $f$ be the endomorphism of
                          $X_0$ that preserves the fibration $\pr_x$ and acts on each fiber $(\simeq \P^1)$ as $z \mapsto
                          1/z^2$ where $0$ and $\infty$ are the intersection points of the fiber and $E$. This defines an
                          endomorphism with $\lambda_1 = \lambda_2 = 2$ and $f_{|E}$ is an involution. Indeed, we have $f^2 =
                          k^2$, therefore the eigenvaluation of $f$ must be $\ord_E$ and $\lambda_1 (k^2) = 4$. The four points of $E$
                          where the discriminant with respect to $y$ vanishes are the four indeterminacy points of $f$ and
                          they are fixed points of $f_{|E}$. In coordinates
                          $(x,y), f$ is of the form
                          \begin{equation}
                            f(x,y) = \left( x, \frac{P(x) y^2 -2 (P(x)^2 - 2 Q(x))y +P(x) (P(x)^2 - Q(x))}{y^2 - 2 P(x) y +
                            (P(x)^2 - 2 Q(x))} \right).
                            \label{EqFormeEndormorphiseElliptique}
                          \end{equation}
                          It is clear that $f_{|E} = {\sigma_y}_{|E}$.

                          Now, if we exchange the role of the coordinates $x,y$, we obtain an endomorphism $g$ with an
                          expression similar to \eqref{EqFormeEndormorphiseElliptique} with $(R(y), S(y))$ instead of
                          $(P(x), Q(x))$. Set $h = g \circ f$, then $\lambda_1(h) = \lambda_2 (h) = 4$. Indeed, $\lambda_2
                          (h) = \lambda_2 (f) \lambda_2(g)$. And we have that $h_* Z_{\ord_E} = g_* f_* (Z_{\ord_E}) = 4
                          Z_{\ord_E} + w$ with $w \in \DivInf(X)^\perp$. Since $Z_{\ord_E} = \frac{1}{8} E$ is nef, $\ord_E$
                          must be the eigenvaluation of $h$ and $\lambda_1 (h) = 4$.

                          Now, let $\Lambda \subset \C$ be a lattice such that $E \simeq \C / \Lambda$.
                          An involution of $E$ lifts to a linear map over $\C$ of the form
                          \begin{equation}
                            z \mapsto -z + b
                            \label{EqForme1Involution}
                          \end{equation}
                          with $b \in \C$ or
                          \begin{equation}
                            z \mapsto z + b
                            \label{EqForme2Involution}
                          \end{equation}
                          with $2b \in \Lambda$. However, \eqref{EqForme2Involution} is impossible for
                          ${\sigma_x}_{|E}$ or ${\sigma_y}_{|E}$ because they admit fixed points. So, they are
                          both of the form \eqref{EqForme1Involution} and we get that ${\sigma_x}_{|E} \circ
                          {\sigma_y}_{|E}$ is a translation of infinite order because $\langle {\sigma_x}_{|E},
                          {\sigma_y}_{|E} \rangle \simeq \Z /2 \Z * \Z /2 \Z$.
                        \end{proof}

                        \begin{cor}
                          The endomorphism $h = g \circ f$ does not admit a preperiodic curve or preserve a fibration over a curve.
                        \end{cor}

                        \begin{proof}
                          Let $X = \P^1 \times \P^1$.
                          We have that $\hat \cV_\infty$ is a real tree rooted at $\ord_E$. Indeed, as in the case of
                          $\A^2 \subset \P^2$, $\cV_\infty$ is the union of the valuation trees $\cV_X(p; E)$ for $p \in E$ which are all rooted at
                          $\ord_E$. Now, for a divisorial valuation $w = \ord_F \in \hat \cV_\infty$, we have that
                          $F$ is a rational curve because it is obtained by blowing up a sequence of smooth points. This
                          implies that $\left\{ \ord_E \right\}$ is totally invariant under the action of $h = g \circ
                          f$ because no such $F$ can be sent dominantly to $E$ and that is true actually for any
                          dominant endomorphism of $X_0$ because $m^{-1}
                          (\cV_\infty) = \cV_\infty$ for any dominant endomorphism $m : X_0 \rightarrow X_0$. Now
                          suppose that $h$ admits an invariant curve $C$, up to replacing $h$ by an iterate we can
                          assume that all the branches at infinity of $C$ are fixed, call $\phi$ one of them and let
                          $v_{\phi}$ be the associated curve valuation and let $q \in E$ be its center. We have that $h_\bullet (\ord_E) = \ord_E$ and
                          $h_{\bullet} (v_\phi) = v_{\phi}$. Now, let $P_1, \dots, P_r \in \C[X_0]$ be generators of
                          $\C[X_0]$ as a $\C$-algebra, we have that 
                          \begin{equation}
                            D_h := \left\{ v \in \cV_\infty : h_* (v) \in \cV_\infty \right\} = \left\{ v : \min_i \left(
                                v(h^* (P_i)) < 0
                              \right)\right\}.
                            \label{<+label+>}
                          \end{equation}
                          Indeed, we have by Lemma \ref{LemmeCaracValuationCentreeAlInfini} that $h_*v \in \cV_\infty$
                          if and only if there exists $p \in \C[X_0]$ such that $h_* v (P) = v (h^* P) < 0$ and for any
                          $Q \in \C[X_0], v (Q) \geq \min_i v (P_i)$ by properties of valuations. Now, we have seen that
                          $\hat \cV_\infty$ is a real tree rooted at $\ord_E$. We claim that
                          $D_h$ is a subtree of $\hat \cV_\infty$. Indeed, let $P \in \C[X_0]$ and take $p \in E$, then
                          at $p$ the rational function $P$ is locally of the form 
                          \begin{equation}
                            P = z^{-\alpha} \phi_1 \cdots \phi_r
                            \label{<+label+>}
                          \end{equation}
                          where $z$ is a local equation of $E, \alpha > 0$ and the $\phi_i$'s are local regular
                          functions at $p$. We then have 
                          \begin{equation}
                            \forall v \in \cV_X (p; E), \quad v(P) = - \alpha + \sum_i v (\phi_i)
                            \label{<+label+>}
                          \end{equation}
                          and we see that the set $\left\{ v \in \cV_X (p; E) : v(P) < 0 \right\}$ is a subtree rooted at
                          $\ord_E$. Doing this
                          for every $p \in E$ we get that $\left\{ v (P) < 0 \right\}$ is a subtree of $\hat \cV_\infty$ and
                          finally 
                          \begin{equation}
                            D_h = \bigcup_i \left\{ v (h^* (P_i)) < 0  \right\}
                            \label{<+label+>}
                          \end{equation}
                          is a subtree because all these subtrees have the same root $\ord_E$. Since $\ord_E,
                          v_{\phi} \in D_h$, this implies that $[\ord_E, v_\phi] \subset D_h$ and therefore
                          \begin{equation}
                            h_\bullet \left( \left[ \ord_E, v_\phi \right] \right) \subset \hat \cV_\infty.
                            \label{<+label+>}
                          \end{equation}
                           And since $\left\{ \ord_E \right\}$ is totally invariant and
                          $v_\phi$ is centered at $q$ we actually have 
                          \begin{equation}
                            h_\bullet \left( ] \ord_E, v_{\phi} ] \right) \subset \cV_X (q;E) \setminus \left\{ \ord_E
                            \right\}
                            \label{<+label+>}
                          \end{equation}
                          Because $\cV_X (q;E) \setminus \left\{ \ord_E \right\}$ is the connected component of $\hat
                          \cV_\infty \setminus \left\{ \ord_E \right\}$ containing $v_\phi$.

                          This implies that for $v \in ] \ord_E, v_\phi]$ sufficiently close to $\ord_E$ we have that
                          $h_* v$ is centered at $q$ but this is a contradiction because it should be centered at
                          $h_{|E} (q)$ but $h_{|E} (q) \neq q$ since $h_{|E}$ is a translation.

                           Now, suppose that there exists a curve $C$ and a
                          rational transformation $\phi : C \dashrightarrow C$ such that the diagram
                          \begin{center}
                            \begin{tikzcd}
                              X_0 \ar[r, "g \circ f"] \ar[d, "\pi"] & X_0 \ar[d, "\pi"] \\
                              C \ar[r, dashed, "\phi"] & C
                            \end{tikzcd}
                          \end{center}
                          commutes. Let $\overline C$ be the unique completion of $C$. We have the commutative diagram
                          \begin{center}
                            \begin{tikzcd}
                              X \ar[r, "g \circ f", dashed] \ar[d, dashed, "\pi"] & X \ar[d, dashed, "\pi"] \\
                              \overline C \ar[r,"\phi"] & \overline C
                            \end{tikzcd}.
                          \end{center}
                          We cannot have $\text{genus}(\overline C) \geq 1$ because $X$ is rational. Now suppose that
                          $\text{genus} \left(
                            \overline C \right) =0 $ i.e that $C = \P^1$. We have that $\phi : \P^1 \rightarrow \P^1$
                            must have a fixed point $q \in \P^1$. The fiber $\pi^{-1} (q)$ is therefore an invariant curve
                            of $h : X \dashrightarrow X$ and the only possibility is therefore that $\pi^{-1}
                            (q) = E$. Since $q$ must be the only fixed point of $\phi$ we have up to composing
                            $\pi$ with an automorphism of $\P^1$ that $q = \infty \in
                            \P^1$ and $\phi : \A^1 = \P^1 \setminus \left\{ \infty \right\} \rightarrow \A^1$ is a
                            translation. Let $\eta : Y \rightarrow X$ be a morphism of completions such that $\pi: Y
                            \rightarrow \overline C$ is a morphism, then $\partial_Y X_0 = E \cup F_1 \cup \cdots F_r$.
                            By the Stein factorisation theorem from Proposition \ref{PropSteinFactorization}, we can
                            assume that $\pi: Y \rightarrow \P^1$ has connected fibers. We show that the generic fiber
                            must be irreducible, indeed let $\eta$ be the generic point of $\P^1$, if $Y_\eta =
                            \pi^{-1} (\eta)$ is not irreducible then it is a finite union of curves which have
                            intersection points because $Y_\eta$ is
                            connected. But $h$ acts on the irreducible components of $Y_\eta$ so some iterates of $h$
                            must fix all of them and therefore must fix the intersection points of the irreducible
                            components of $Y_\eta$, but they induce sections of $\pi$ which must be $h$-invariants
                            curves and this is a contradiction. So we have that the generic fiber of $\pi$ is an
                            irreducible curve which degenerates to the union of an elliptic curve and some rational
                            curves. In particular, the genus of $Y_\eta$ must be $\geq 1$. If $\text{genus}(Y_\eta) \geq 2$, then
                            $Y_\eta$ is of general type and some iterate of $h$ acts trivially on $Y_\eta$ but this
                            contradicts the fact that the topological degree of $h$ is 4. If $\text{genus}(Y_\eta) = 1$, then
                            the generic fiber is an elliptic curve, let $\Sigma$ be the set of prime divisors $F$ at
                            infinity in $Y$ such that $F$ defines a multi-section of $\pi$. For a general parameter
                            $p \in \A^1, E_p := \pi^{-1} (p)$ is an elliptic curve and 
                            \begin{equation}
                              X_0 \cap E_p = E_p \setminus \left( \bigcup_{F \in \Sigma} F \cap E_p \right).
                              \label{<+label+>}
                            \end{equation}
                            Furthermore, every $G \subset \partial_Y X_0 \setminus (\Sigma \cup \left\{ E
                            \right\} )$ is contained in a fiber of $\pi$.
                            Take a general parameter $p \in \A^1$ and consider the restriction $h_p : E_p \rightarrow
                            E_{h(p)}$ of $h$. No points of $E_p \cap X_0$ can be sent by $h_p$ to a point of
                            $ \bigcup_{F \in \Sigma} E_{h(p)} \cap F$. If we replace $h$ by a large enough iterate,
                            the topological degree of $h$ goes to infinity but not the size of $\Sigma$ and not the number
                            of intersection points of the elements of $\Sigma$ with the fiber of a general parameter,
                            this implies that $h_p$ restricted to a general $p \in \A^1$ must be of topological degree 1 but this
                            contradicts the fact that the topological degree of $h$ is not 1.
                        \end{proof}

%% file: main.bbl
\begin{thebibliography}{BHPVDV04}

\bibitem[Abb24]{abboudRigidityPeriodicPoints2024}
Marc Abboud.
\newblock Rigidity of periodic points for loxodromic automorphisms of affine surfaces, June 2024.

\bibitem[Abh56]{abhyankarValuationsCenteredLocal1956}
Shreeram Abhyankar.
\newblock On the {{Valuations Centered}} in a {{Local Domain}}.
\newblock {\em American Journal of Mathematics}, 78(2):321--348, 1956.

\bibitem[Abr94]{abramovichSubvarietiesSemiabelianVarieties1994}
Dan Abramovich.
\newblock {Subvarieties of semiabelian varieties}.
\newblock {\em Compositio Mathematica}, 90(1):37--52, 1994.

\bibitem[BC13]{blancDynamicalDegreesBirational2013}
J{\'e}r{\'e}my Blanc and Serge Cantat.
\newblock Dynamical degrees of birational transformations of projective surfaces.
\newblock {\em Journal of the American Mathematical Society}, 29, July 2013.

\bibitem[BD11]{blancAutomorphisms1fiberedAffine2011}
J{\'e}r{\'e}my Blanc and Adrien Dubouloz.
\newblock Automorphisms of {{$\mathbb{A}$}}{$^1$}-fibered affine surfaces.
\newblock {\em Transactions of the American Mathematical Society}, 363(11):5887--5924, 2011.

\bibitem[BD13]{blancAffineSurfacesHuge2013}
J{\'e}r{\'e}my Blanc and Adrien Dubouloz.
\newblock Affine {{Surfaces With}} a {{Huge Group}} of {{Automorphisms}}.
\newblock {\em International Mathematics Research Notices}, February 2013.

\bibitem[BDJ20]{bellTranscendentalDynamicalDegree2020}
Jason~P. Bell, Jeffrey Diller, and Mattias Jonsson.
\newblock A transcendental dynamical degree.
\newblock {\em Acta Mathematica}, 225(2):193--225, 2020.

\bibitem[Ber12]{berkovichSpectralTheoryAnalytic2012}
Vladimir Berkovich.
\newblock Spectral {{Theory}} and {{Analytic Geometry}} over {{Non-Archimedean Fields}}.
\newblock volume~33 of {\em Mathematical {{Surveys}} and {{Monographs}}}, Providence, Rhode Island, August 2012. American Mathematical Society.

\bibitem[BFJ08]{boucksomDegreeGrowthMeromorphic2008}
S{\'e}bastien Boucksom, Charles Favre, and Mattias Jonsson.
\newblock Degree growth of meromorphic surface maps.
\newblock {\em Duke Mathematical Journal}, 141(3), February 2008.

\bibitem[BFJ09]{boucksomDifferentiabilityVolumesDivisors2009}
S{\'e}bastien Boucksom, Charles Favre, and Mattias Jonsson.
\newblock Differentiability of volumes of divisors and a problem of {{Teissier}}.
\newblock {\em Journal of Algebraic Geometry}, 18(2):279--308, 2009.

\bibitem[BHPVDV04]{barthCompactComplexSurfaces2004}
Wolf~P. Barth, Klaus Hulek, Chris A.~M. Peters, and Antonius Van De~Ven.
\newblock {\em Compact {{Complex Surfaces}}}, volume~4 of {\em Ergebnisse Der {{Mathematik}} Und Ihrer {{Grenzgebiete}}. 3. {{Folge}} / {{A Series}} of {{Modern Surveys}} in {{Mathematics}}}.
\newblock Springer, Berlin, Heidelberg, 2004.

\bibitem[BJ22]{boucksomGlobalPluripotentialTheory2022}
S{\'e}bastien Boucksom and Mattias Jonsson.
\newblock {Global pluripotential theory over a trivially valued field}.
\newblock {\em Annales de la Facult\'e des sciences de Toulouse : Math\'ematiques}, 31(3):647--836, 2022.

\bibitem[BK06]{bedfordPeriodicitiesLinearFractional2006}
Eric Bedford and Kyounghee Kim.
\newblock Periodicities in linear fractional recurrences: {{Degree}} growth of birational surface maps.
\newblock {\em Michigan Mathematical Journal}, 54(3):647--671, November 2006.

\bibitem[BK08]{bedfordLinearRecurrencesDegree2008}
Eric Bedford and Kyounghee Kim.
\newblock Linear recurrences in the degree sequences of monomial mappings.
\newblock {\em Ergodic Theory and Dynamical Systems}, 28(5):1369--1375, October 2008.

\bibitem[Bot23]{botSmoothComplexRational2023}
Anna Bot.
\newblock A smooth complex rational affine surface with uncountably many real forms, February 2023.

\bibitem[BS91a]{bedfordFatouBieberbachDomainsArising1991}
Eric Bedford and John Smillie.
\newblock Fatou-{{Bieberbach Domains Arising}} from {{Polynomial Automorphisms}}.
\newblock {\em Indiana University Mathematics Journal}, 40(2):789--792, 1991.

\bibitem[BS91b]{bedfordPolynomialDiffeomorphisms$C^2$1991a}
Eric Bedford and John Smillie.
\newblock Polynomial {{Diffeomorphisms}} of {$\C^2$}. {{II}}: {{Stable Manifolds}} and {{Recurrence}}.
\newblock {\em Journal of the American Mathematical Society}, 4(4):657--679, 1991.

\bibitem[Can01a]{cantatDynamiqueGroupeDautomorphismes2001}
S.~Cantat.
\newblock {Sur la dynamique du groupe d'automorphismes des surfaces K3}.
\newblock {\em Transformation Groups}, 6(3):201--214, September 2001.

\bibitem[Can01b]{cantatDynamiqueAutomorphismesSurfaces2001}
Serge Cantat.
\newblock {Dynamique des automorphismes des surfaces K3}.
\newblock {\em Acta Mathematica}, 187(1):1--57, 2001.

\bibitem[Can09]{cantatBersHenonPainleve2009}
Serge Cantat.
\newblock Bers and {{H\'enon}}, {{Painlev\'e}} and {{Schr\"odinger}}.
\newblock {\em Duke Mathematical Journal}, 149(3):411--460, 2009.

\bibitem[Can11]{cantatGroupesTransformationsBirationnelles2011}
Serge Cantat.
\newblock {Sur les groupes de transformations birationnelles des surfaces}.
\newblock {\em Annals of Mathematics}, 174(1):299--340, July 2011.

\bibitem[CD]{cantatRationalSurfacesLarge}
Serge Cantat and Igor Dolgachev.
\newblock Rational surfaces with a large group of automorphisms.

\bibitem[CdC19]{cantatCommensuratingActionsBirational2019}
Serge Cantat and Yves de~Cornulier.
\newblock Commensurating actions of birational groups and groups of pseudo-automorphisms.
\newblock {\em Journal de l'\'Ecole polytechnique --- Math\'ematiques}, 6:767--809, 2019.

\bibitem[CLC13]{CantatNormalsubgroupsCremona2013}
Serge Cantat, St{\'e}phane Lamy, and Yves Cornulier.
\newblock Normal subgroups in the {{Cremona}} group.
\newblock {\em Acta Mathematica}, 210(1):31--94, 2013.

\bibitem[CP00]{cutkoskyMonomialResolutionsMorphisms2000}
Dale Cutkosky and Olivier Piltant.
\newblock Monomial resolutions of morphisms of algebraic surfaces.
\newblock {\em Communications in Algebra}, 28(12):5935--5959, January 2000.

\bibitem[DF01]{dillerDynamicsBimeromorphicMaps2001}
J~Diller and C~Favre.
\newblock Dynamics of bimeromorphic maps of surfaces.
\newblock {\em American Journal of Mathematics}, 123(6):1135--1169, 2001.

\bibitem[DF21]{dangSpectralInterpretationsDynamical2021}
Nguyen-Bac Dang and Charles Favre.
\newblock Spectral interpretations of dynamical degrees and applications.
\newblock {\em Annals of Mathematics}, 194(1):299--359, July 2021.

\bibitem[DG]{danilovAutomorphismsAffineSurfaces}
Vladimir Danilov and Marat Gizatullin.
\newblock Automorphisms of affine surfaces. {{II}}.
\newblock {\em Mathematics of the USSR-Izvestiya}, 11:51.

\bibitem[DG75]{danilovAutomorphismsAffineSurfaces1975}
V.~I. Danilov and M.~H. Gizatullin.
\newblock Automorphisms of affine surfaces. {{I}}.
\newblock {\em Mathematics of the USSR-Izvestiya}, 9(3):493, June 1975.

\bibitem[DL16]{dillerRationalSurfaceMaps2016}
Jeffrey Diller and Jan-Li Lin.
\newblock Rational surface maps with invariant meromorphic two-forms.
\newblock {\em Mathematische Annalen}, 364(1):313--352, February 2016.

\bibitem[DL17]{dillerRotationNumbersElements2017}
Jeffrey Diller and Jan-Li Lin.
\newblock Rotation numbers of elements in {{Thompson}}'s group \textbackslash bf {{T}}.
\newblock {\em Topology Proceedings}, 50:203--212, 2017.

\bibitem[DN05]{dinhMixedHodgeRiemannBilinear2005}
Tien-Cuong Dinh and Vi{\^e}t-Anh Nguy{\^e}n.
\newblock The mixed {{Hodge-Riemann}} bilinear relations for compact {{Kahler}} manifolds.
\newblock February 2005.

\bibitem[DN11]{dinhComparisonDynamicalDegrees2011}
Tien-Cuong Dinh and Vi{\^e}t-Anh Nguy{\^e}n.
\newblock Comparison of dynamical degrees for semi-conjugate meromorphic maps.
\newblock {\em Commentarii Mathematici Helvetici}, 86(4):817--840, September 2011.

\bibitem[DP12]{delzantKahlerGroupsReal2012}
Thomas Delzant and Pierre Py.
\newblock K\"ahler groups, real hyperbolic spaces and the {{Cremona}} group. {{With}} an appendix by {{Serge Cantat}}.
\newblock {\em Compositio Mathematica}, 148(1):153--184, January 2012.

\bibitem[DP18]{duboulozJacobianConjectureFails2018}
Adrien Dubouloz and Karol Palka.
\newblock The {{Jacobian Conjecture}} fails for pseudo-planes.
\newblock {\em Advances in Mathematics}, 339:248--284, December 2018.

\bibitem[DS03]{dinhBorneSuperieurePour2003}
Tien-Cuong Dinh and Nessim Sibony.
\newblock Une borne sup\'erieure pour l'entropie topologique d'une application rationnelle.
\newblock {\em Annals of mathematics, ISSN 0003-486X, Vol. 161, Nº 3, 2005, pags. 1637-1644}, March 2003.

\bibitem[Dub04]{duboulozCompletionsNormalAffine2004}
Adrien Dubouloz.
\newblock Completions of normal affine surfaces with a trivial {{Makar-Limanov}} invariant.
\newblock {\em Michigan Mathematical Journal}, 52(2):289--308, August 2004.

\bibitem[{\`E}H74]{el-hutiCubicSurfacesMarkov1974}
M~H {\`E}l'-Huti.
\newblock Cubic surfaces of {{Markov}} type.
\newblock {\em Mathematics of the USSR-Sbornik}, 22(3):333--348, April 1974.

\bibitem[Fav00]{favreClassification2dimensionalContracting2000}
Charles Favre.
\newblock Classification of 2-dimensional contracting rigid germs and {{Kato}} surfaces: {{I}}.
\newblock {\em Journal de Math\'ematiques Pures et Appliqu\'ees}, 79(5):475--514, May 2000.

\bibitem[Fav03]{favreApplicationsMonomialesDeux2003}
Charles Favre.
\newblock Les applications monomiales en deux dimensions.
\newblock {\em Michigan Mathematical Journal}, 51(3):467--475, December 2003.

\bibitem[FJ04]{favreValuativeTree2004}
Charles Favre and Mattias Jonsson.
\newblock {\em The {{Valuative Tree}}}.
\newblock Lecture {{Notes}} in {{Mathematics}}. Springer-Verlag, Berlin Heidelberg, 2004.

\bibitem[FJ07]{favreEigenvaluations2007}
Charles Favre and Mattias Jonsson.
\newblock Eigenvaluations.
\newblock {\em Annales Scientifiques de l'\'Ecole Normale Sup\'erieure}, 40(2):309--349, March 2007.

\bibitem[FJ11]{favreDynamicalCompactificationsMathbf2011}
Charles Favre and Mattias Jonsson.
\newblock Dynamical compactifications of {$\mathbf{C}^2$}.
\newblock {\em Annals of Mathematics}, 173(1):211--249, January 2011.

\bibitem[FTX22]{favreTopologicalEntropyRational2022}
Charles Favre, Tuyen Truong, and Junyi Xie.
\newblock {\em Topological Entropy of a Rational Map over a Complete Metrized Field}.
\newblock August 2022.

\bibitem[Ful98]{fultonIntersectionTheory1998}
William Fulton.
\newblock {\em Intersection {{Theory}}}.
\newblock Springer, New York, NY, 1998.

\bibitem[FW12]{favreDegreeGrowthMonomial2012}
Charles Favre and Elizabeth Wulcan.
\newblock Degree {{Growth}} of {{Monomial Maps}} and {{McMullen}}'s {{Polytope Algebra}}.
\newblock {\em Indiana University Mathematics Journal}, 61(2):493--524, 2012.

\bibitem[Giz69]{gizatullinApplicationNumericalCriterion1969}
M.~H. Gizatullin.
\newblock An application of a numerical criterion of ampleness.
\newblock {\em Akademiya Nauk SSSR. Matematicheskie Zametki}, 5:149--154, 1969.

\bibitem[Giz71a]{gizatullinInvariantsIncompleteAlgebraic1971}
M.~H. Gizatullin.
\newblock Invariants of incomplete algebraic surfaces that can be obtained by means of completions.
\newblock {\em Izvestiya Akademii Nauk SSSR. Seriya Matematicheskaya}, 35:485--497, 1971.

\bibitem[Giz71b]{gizatullinQuasihomogeneousAffineSurfaces1971}
M.~H. Gizatullin.
\newblock Quasihomogeneous affine surfaces.
\newblock {\em Izvestiya Akademii Nauk SSSR. Seriya Matematicheskaya}, 35:1047--1071, 1971.

\bibitem[Goo69]{goodmanAffineOpenSubsets1969}
Jacob~Eli Goodman.
\newblock Affine {{Open Subsets}} of {{Algebraic Varieties}} and {{Ample Divisors}}.
\newblock {\em Annals of Mathematics}, 89(1):160--183, 1969.

\bibitem[GR21]{gignacLocalDynamicsNoninvertible2021}
William Gignac and Matteo Ruggiero.
\newblock {\em Local Dynamics of Non-Invertible Maps near Normal Surface Singularities}, volume 272 of {\em Memoirs of the {{American Mathematical Society}}}.
\newblock American Mathematical Society, July 2021.

\bibitem[Gro90]{gromovConvexSetsKahler1990}
M.~Gromov.
\newblock Convex sets and k\"ahler manifolds.
\newblock In {\em Advances in {{Differential Geometry}} and {{Topology}}}, pages 1--38. WORLD SCIENTIFIC, November 1990.

\bibitem[Gro03]{gromovEntropyHolomorphicMaps2003}
Mikhail Gromov.
\newblock On the entropy of holomorphic maps.
\newblock {\em L'Enseignement Math\'ematique. IIe S\'erie}, 49, January 2003.

\bibitem[Gue05]{guedjEntropieTopologiqueApplications2005}
Vincent Guedj.
\newblock Entropie topologique des applications m\'eromorphes.
\newblock {\em Ergodic Theory and Dynamical Systems}, 25(6):1847--1855, December 2005.

\bibitem[Har77]{hartshorneAlgebraicGeometry1977}
Robin Hartshorne.
\newblock {\em Algebraic {{Geometry}}}.
\newblock Graduate {{Texts}} in {{Mathematics}}. Springer-Verlag, New York, 1977.

\bibitem[HO94]{hubbardHenonMappingsComplex1994}
John~H. Hubbard and Ralph~W. {Oberste-Vorth}.
\newblock H\'enon mappings in the complex domain {{I}}: {{The}} global topology of dynamical space.
\newblock {\em Publications Math\'ematiques de l'Institut des Hautes \'Etudes Scientifiques}, 79(1):5--46, December 1994.

\bibitem[HP07]{hasselblattDegreegrowthMonomialMaps2007}
Boris Hasselblatt and James Propp.
\newblock Degree-growth of monomial maps.
\newblock {\em Ergodic Theory and Dynamical Systems}, 27(5):1375--1397, October 2007.

\bibitem[Iit77]{iitakaLogarithmicKodairaDimension1977}
S.~Iitaka.
\newblock On {{Logarithmic Kodaira Dimension}} of {{Algebraic Varieties}}.
\newblock In T.~Shioda and W.~L.~Jr Baily, editors, {\em Complex {{Analysis}} and {{Algebraic Geometry}}: {{A Collection}} of {{Papers Dedicated}} to {{K}}. {{Kodaira}}}, pages 175--190. Cambridge University Press, Cambridge, 1977.

\bibitem[Iit79]{iitakaNumericalCriterionQuasiabelian1979}
Shigeru Iitaka.
\newblock A numerical criterion of quasi-abelian surfaces.
\newblock {\em Nagoya Mathematical Journal}, 73:99--115, March 1979.

\bibitem[JM12]{jonssonValuationsAsymptoticInvariants2012}
Mattias Jonsson and Mircea Musta{\c t}{\u a}.
\newblock Valuations and asymptotic invariants for sequences of ideals.
\newblock {\em Annales de l'institut Fourier}, 62(6):2145--2209, 2012.

\bibitem[JW11]{jonssonStabilizationMonomialMaps2011}
Mattias Jonsson and Elizabeth Wulcan.
\newblock Stabilization of monomial maps.
\newblock {\em Michigan Mathematical Journal}, 60(3), November 2011.

\bibitem[JW12]{jonssonCanonicalHeightsPlane2012}
Mattias Jonsson and Elizabeth Wulcan.
\newblock Canonical heights for plane polynomial maps of small topological degree.
\newblock {\em Mathematical Research Letters}, 19(6):1207--1217, 2012.

\bibitem[Koj01]{kojimaOpenSurfacesLogarithmic2001}
Hideo Kojima.
\newblock Open surfaces of logarithmic {{Kodaira}} dimension zero in arbitrary characteristic.
\newblock {\em Journal of the Mathematical Society of Japan}, 53(4):933--955, October 2001.

\bibitem[KS06]{kontsevichAffineStructuresNonArchimedean2006}
Maxim Kontsevich and Yan Soibelman.
\newblock Affine structures and non-{{Archimedean}} analytic spaces.
\newblock In {\em The Unity of Mathematics}, volume 244 of {\em Progr. {{Math}}.}, pages 321--385. Birkh\"auser Boston, Boston, MA, 2006.

\bibitem[Laz04]{lazarsfeldPositivityAlgebraicGeometry2004}
Robert Lazarsfeld.
\newblock {\em Positivity in {{Algebraic Geometry I}}: {{Classical Setting}}: {{Line Bundles}} and {{Linear Series}}}.
\newblock January 2004.

\bibitem[Lin12]{linPullingBackCohomology2012}
Jan-Li Lin.
\newblock Pulling back cohomology classes and dynamical degrees of monomial maps.
\newblock {\em Bulletin de la Soci\'et\'e Math\'ematique de France}, 140(4):533--549, 2012.

\bibitem[McM07]{mcmullenDynamicsBlowupsProjective2007}
Curtis~T. McMullen.
\newblock Dynamics on blowups of the projective plane.
\newblock {\em Publications math\'ematiques de l'IH\'ES}, 105(1):49--89, June 2007.

\bibitem[Miy82]{miyanishiAffineruledIrrationalSurfaces1982}
Masayoshi Miyanishi.
\newblock On affine-ruled irrational surfaces.
\newblock {\em Inventiones mathematicae}, 70(1):27--43, February 1982.

\bibitem[Nov]{novacoskiValuationsCenteredTwodimensional}
Josnei Novacoski.
\newblock Valuations centered at a two-dimensional regular local domain: Infima and topologies \textbar{} {{EMS Press}}.
\newblock https://ems.press/books/ecr/197/3745.

\bibitem[Rug13]{ruggieroContractingRigidGerms2013}
Matteo Ruggiero.
\newblock {Contracting rigid germs in higher dimensions}.
\newblock {\em Annales de l'Institut Fourier}, 63(5):1913--1950, 2013.

\bibitem[Rus81]{russellAffineruledRationalSurfaces1981}
Peter Russell.
\newblock On affine-ruled rational surfaces.
\newblock {\em Mathematische Annalen}, 255(3):287--302, September 1981.

\bibitem[Ser01]{serreExposesSeminaires195019992001}
J.~Serre.
\newblock Expos\'es de s\'eminaires, 1950-1999.
\newblock 2001.

\bibitem[Sil09]{silvermanArithmeticEllipticCurves2009}
Joseph~H. Silverman.
\newblock {\em The {{Arithmetic}} of {{Elliptic Curves}}}, volume 106 of {\em Graduate {{Texts}} in {{Mathematics}}}.
\newblock Springer, New York, NY, 2009.

\bibitem[Spi90]{spivakovskyValuationsFunctionFields1990}
Mark Spivakovsky.
\newblock Valuations in {{Function Fields}} of {{Surfaces}}.
\newblock {\em American Journal of Mathematics}, 112(1):107, February 1990.

\bibitem[Tru15]{truongRelativeDynamicalDegrees2015}
Tuyen~Trung Truong.
\newblock ({{Relative}}) dynamical degrees of rational maps over an algebraic closed field.
\newblock {\em arXiv:1501.01523 [math]}, January 2015.

\bibitem[Ure21]{urechSubgroupsEllipticElements2021}
Christian Urech.
\newblock Subgroups of elliptic elements of the {{Cremona}} group.
\newblock {\em Journal f\"ur die reine und angewandte Mathematik (Crelles Journal)}, 2021(770):27--57, January 2021.

\bibitem[Vaq00]{vaquieValuations2000}
Michel Vaqui{\'e}.
\newblock Valuations.
\newblock In Herwig Hauser, Joseph Lipman, Frans Oort, and Adolfo Quir{\'o}s, editors, {\em Resolution of {{Singularities}}: {{A}} Research Textbook in Tribute to {{Oscar Zariski Based}} on the Courses given at the {{Working Week}} in {{Obergurgl}}, {{Austria}}, {{September}} 7--14, 1997}, Progress in {{Mathematics}}, pages 539--590. Birkh\"auser, Basel, 2000.

\bibitem[Xie17]{xieExistenceZariskiDense2017}
Junyi Xie.
\newblock The existence of {{Zariski}} dense orbits for polynomial endomorphisms of the affine plane.
\newblock {\em Compositio Mathematica}, 153(8):1658--1672, August 2017.

\bibitem[Yom87]{yomdinVolumeGrowthEntropy1987}
Y.~Yomdin.
\newblock Volume growth and entropy.
\newblock {\em Israel Journal of Mathematics}, 57(3):285--300, October 1987.

\bibitem[Yos85]{yoshiharaProjectivePlaneCurves1985}
Hisao Yoshihara.
\newblock Projective plane curves and the automorphism groups of their complements.
\newblock {\em Journal of the Mathematical Society of Japan}, 37(1):87--113, 1985.

\bibitem[YZ26]{yuanAdelicLineBundles2026}
Xinyi Yuan and Shou-Wu Zhang.
\newblock {\em Adelic Line Bundles on Quasi-Projective Varieties}.
\newblock Princeton, NJ: Princeton University Press, 2026.

\bibitem[ZS60]{zariskiCommutativeAlgebra1960}
Oscar Zariski and Pierre Samuel.
\newblock {\em Commutative {{Algebra}}}, volume~29 of {\em Graduate {{Texts}} in {{Mathematics}}}.
\newblock Springer, Berlin, Heidelberg, 1960.

\end{thebibliography}
